\documentclass[final,onefignum,onetabnum]{siamart190516}

\usepackage{lipsum}
\usepackage{amsfonts}
\usepackage{graphicx,slashbox}
\usepackage{algorithm}
\usepackage{algpseudocode}

\ifpdf
  \DeclareGraphicsExtensions{.eps,.pdf,.png,.jpg}
\else
  \DeclareGraphicsExtensions{.eps}
\fi

\newsiamremark{remark}{Remark}
\newsiamremark{hypothesis}{Hypothesis}
\crefname{hypothesis}{Hypothesis}{Hypotheses}
\newsiamthm{claim}{Claim}
\graphicspath{{./figs/}} 

\headers{An efficient method for Dirichlet partition}{D. Wang}

\title{An efficient unconditionally stable method for Dirichlet partitions in arbitrary domains\thanks{Submitted to the editors DATE.
\funding{This work is supported by the Guangdong Provincial Key Laboratory of Big Data Computing, The Chinese University of Hong Kong, Shenzhen. Dong Wang acknowledges the support from National Natural Science Foundation of China (Grant No. 12101524) and the University Development Fund from The Chinese University of Hong Kong, Shenzhen (UDF01001803).  }}}

\author{Dong Wang\thanks{School of Science and Engineering, The Chinese University of Hong Kong, Shenzhen, Guangdong 518172, China (\email{wangdong@cuhk.edu.cn}).}
}

\usepackage{amsopn}

\ifpdf
\hypersetup{
  pdftitle={An efficient method for Dirichlet partition},
  pdfauthor={D. Wang}
}
\fi

\begin{document}

\maketitle

\begin{abstract}
A Dirichlet $k$-partition of a domain is a collection of $k$ pairwise disjoint open subsets such that the sum of their first Laplace--Dirichlet eigenvalues is minimal. In this paper, we propose a new relaxation of the problem by introducing auxiliary indicator functions of domains and develop a simple and efficient diffusion generated method to compute Dirichlet $k$-partitions for arbitrary domains. The method only alternates three steps: 1. convolution, 2. thresholding, and 3. projection. The method is simple, easy to implement, insensitive to initial guesses and can be effectively applied to arbitrary domains without any special discretization. At each iteration, the computational complexity is linear in the discretization of the computational domain. Moreover, we theoretically prove the energy decaying property of the method. Experiments are performed to show the accuracy of approximation, efficiency and unconditional stability of the algorithm. We apply the proposed algorithms on both 2- and 3-dimensional flat tori, triangle, square, pentagon, hexagon, disk, three-fold star, five-fold star, cube, ball, and tetrahedron domains to compute Dirichlet $k$-partitions for different $k$ to show the effectiveness of the proposed method. Compared to previous work with reported computational time, the proposed method achieves hundreds of times acceleration.
\end{abstract}

\begin{keywords}
Dirichlet partition, energy stability, thresholding
\end{keywords}

\begin{AMS}
49Q10, 
49R05,
05B45 
\end{AMS}

\section{Introduction}
For $d\geq 2$, let $\Omega$ be either an open bounded domain in $\mathbb R^d$ with Lipschitz boundary or a closed, smooth, $d$-dimensional manifold. For $k\geq 2$ fixed, the Dirichlet $k$-partition problem for
$\Omega$ is to choose a $k$-partition, {\it i.e.}, $k$ disjoint quasi-open sets
$\Omega_1, \Omega_2, \ldots, \Omega_k \subseteq \Omega$, that attains
\begin{equation} \label{eq:ContDirPart}
\min_{\Omega = \cup_{\ell \in [k]} \Omega_\ell} \ \sum_{\ell \in [k]} \lambda_1(\Omega_\ell), 
\qquad \textrm{where} \quad
\lambda_1(D) := \min_{\substack{u \in H^1_0(D)\\ \|u\|_2=1}} E(u).
\end{equation} 
Here, $E(u) =  \int_D |\nabla u|^2 \  dx$ is the Dirichlet energy and $\lambda_1(D)$ is the first
Dirichlet eigenvalue of the  Laplace operator, $-\Delta$, on $D$ with Dirichlet boundary conditions imposed on $\partial D$.  

The existence of optimal partitions in the class
of quasi-open sets was proved in \cite{bucur1998existence}. The properties of
optimal partitions including the regularity of the partition interfaces and the asymptotic behavior  as
$k \to \infty$ have been investigated in \cite{caffarelli2007optimal,Helffer2010b,bourdin2010optimal}. The consistency of Dirichlet partitions has been rigorously studied in \cite{Osting_2017}. Dirichlet partitions have been applied into the study of Bose--Einstein condensates \cite{bao2004ground,bao2004computing,chang2004segregated} and 
models for interacting agents \cite{conti2002nehari,conti2003optimal,chang2004segregated,cybulski2005,cybulski2008}. 

%

The development of efficient numerical methods for finding such partitions attracts much attention in recent years, especially when the dimension is high or number of partitions is large. Essentially, this is an interface related optimization problem subject to global constraints, numerical considerations usually start with the representation of interfaces. Corresponding numerical methods are mainly developed along the directions of phase field based approaches \cite{Du_2008},  level set based approaches \cite{Chu_2021}, and other optimization based approaches \cite{bourdin2010optimal}. 

Along the direction of phase field based approaches, following \cite{caffarelli2007optimal}, for fixed $\varepsilon > 0$, problem \eqref{eq:ContDirPart} can be relaxed to minimizing a relaxed energy,  
\begin{equation}
E_\varepsilon(u) = \sum_{\ell \in [k]} \frac{1}{2}\int_\Omega |\nabla u_\ell|^2 \  dx + \frac{1}{4\varepsilon^2} \int_\Omega  F_k(u)  \ dx.  \label{eq:relax1}
\end{equation}
over fields that take values in $\mathbb R^k$ where $u = (u_1,u_2,\cdots, u_k)$, $[k]$ denotes the set $\{1,2,\cdots,k\}$,
and \[F_k(u) = \sum_{i\neq j \in [k]}   \  u_i^2u_j^2. \]
Then the minimization reads
\begin{align}
\min_{u \in H^1 (\Omega, \mathbb R^k) } & \ E_\varepsilon(u)  \label{min1}\\ 
\textrm{s.t.} & \ \int_\Omega u_\ell^2 \ dx= 1, \qquad  \forall \ell \in [k].  \nonumber
\end{align}

Note that the penalty term in the objective functional tries to penalize that the support of each function $u_\ell$ has no overlap with others. Based on this, in \cite{Du_2008}, Du and Lin proposed an efficient normalized gradient descent method to approximately find the minimizer of \eqref{eq:relax1}. The method is initialized with an initial condition $u^0 \in H^1(\Omega,\mathbb{R}^k)$ and alternates the following three steps until convergence. 
In the first step, the Cauchy problem for the gradient flow of the first term in $E_\varepsilon$, {\it i.e.},
\begin{equation}
\partial_t u(x,t) = \Delta u(x,t), 
\end{equation}
is computed until time $\tau > 0$, with initial condition, $u(x,t=0) = u^0(x)$. Let $\tilde u_\ell(x) = u_\ell(x,\tau)$ for $\ell \in [k]$.
In the second step, for each $x \in \Omega$, the following system of ordinary differential equations is solved until time $\tau$,
\begin{equation} \label{Dulinproj}
\frac{d}{dt} {u}_\ell = \frac{1}{\varepsilon^2} \left( \sum_{j \neq \ell} {u}^2_j \right) {u}_\ell, 
\qquad \qquad   
\ell \in [k], 
\end{equation}
with initial condition given by $u_\ell(x,t=0) = \tilde u_\ell(x)$. 
This is precisely the gradient flow of the second term of the relaxed energy. 
Numerically, this system is solved using the  Gauss-Seidel method. Let $\hat u_\ell(x) =  u_\ell(x,\tau)$ for $\ell = [k]$. Finally, in the third step, each component of $u$ is normalized to satisfy the $L^2(\Omega)$ norm constraint, {\it i.e.}, 
\begin{equation}
u_\ell = \frac{\hat u_\ell} {\| \hat u_\ell \|_2}. 
\end{equation}
In this method, the small parameter $\varepsilon$ thickens the interface between any two partitions, restricting the mesh size and making the convergence relatively slow. Recently, a scalar auxiliary variable (SAV) approach \cite{Shen_2018,Shen_2019} shows its great advantage on solving systems of gradient flow and a SAV based method for preserving global constraints is proposed in \cite{Cheng_2020} to solve \eqref{min1}. However, in such a specific class of problems, the solution of interest is the minimizer instead of the dynamics from an initial guess to the minimizer. The method takes a lot of iterations to converge to the minimizer and seems to be difficult to find regular solutions, especially when $k$ is large.

To accelerate the convergence, Wang and Osting \cite{Wang_2019} proposed a diffusion generated method to compute Dirichlet partitions approximately. The main novelty in the method is replacing the second step of Du and Lin's method ({\it i.e.}, \eqref{Dulinproj}) by direct projection to make $F_k(u) =0$. This is based on the observation that any solution satisfies that $ \cup_\ell \ {\rm supp} (u_\ell) = \Omega $ by the monotonicity of Dirichlet eigenvalues (and also the relaxed energy). To be more precise, for each $x$, the projection is simply done by comparing the values among $(u_1, u_2, \cdots, u_k)$, keeping the largest one and projecting other values to $0$. This approach dramatically accelerates the speed of convergence from random initial guesses based on the numerical observations. It can simply and quickly find Dirichlet $k$-partitions in 4-dimensional space with different $k$ even on a laptop. However, it alternates a diffusion step, a projection step, and a normalization step. No theoretical guarantee could be provided on the convergence or energy decaying properties. In addition, the results presented in \cite{Wang_2019} are limited to periodic cases or closed surfaces, it is not obvious on how to effectively extend to arbitrary domains with Dirichlet boundary conditions.


Another approach, developed first in \cite{bourdin2010optimal}, is based on a Schr\"odinger operator relaxation of \eqref{eq:ContDirPart} and was further used in \cite{bogosel2017efficient,Bogosel_2016}. 
The idea here is to replace the shape optimization problem for a partition $\Omega = \cup_\ell \Omega_\ell$ in \eqref{eq:ContDirPart}  with the following relaxed optimization problem for a collection of functions $\{\varphi_\ell\}_{\ell \in [k]}$:
\begin{equation} \label{eq:Schro}
\min_{\{ \varphi_\ell\} \in  K} \ \sum_{\ell \in [k]} \lambda^\alpha_1(\varphi_\ell). 
\end{equation} 
Here, the constraint set is given by 
\begin{equation}
K = \left\{ \{\varphi_\ell\}_{\ell \in [k]} \colon \sum_\ell \varphi_\ell(x) = 1 \ \ \textrm{and} \ \ \varphi_\ell(x) \in [0,1]  \ \ a.e. \ x \in \Omega \right\}. \label{eq:defnK}
\end{equation}
For $\alpha >0$,   $\lambda^\alpha_1(\phi)$ is defined as the first eigenvalue for the Schr\"odinger  operator $- \Delta + \frac{1}{\alpha} (1-\phi)$ by
\begin{equation}\label{energy2}
  \lambda^\alpha_1(\phi) := \min_{\substack{u \in H^1(\Omega)\\ \|u\|_2=1}} \frac{1}{2}\int_\Omega |\nabla u|^2 + \frac{1}{2 \alpha} (1-\phi) u^2 \ dx.
\end{equation}
It was shown that if $\phi = \chi_D$ where $\chi_D$ denotes the indicator function of domain $D$, then $  \lambda^\alpha_1(\phi) \to \lambda_1(D) $ as $\alpha \rightarrow 0$  \cite{bourdin2010optimal} (see also in \cite{osting2013minimal}).  Furthermore, the objective functional in \eqref{eq:Schro} is concave with respect to $\varphi$, so the minimum in \eqref{eq:Schro} is attained at extreme points of $K$, which are exactly indicator functions of domain, giving partition solutions. One could interpret the second term in \eqref{energy2} as a penalty term to penalize the support of $u$ to be the region where $\phi = 1$ if $\phi$ is an indicator function of a domain. 


In this paper, we propose a novel relaxation to the Dirichlet partition problem. Instead of considering the relaxation using the Schr\"odinger  operator, we propose to approximate $\lambda_1(\phi)$ using a small $\tau$ as follows
\[\lambda_1^\tau(\phi) = \min_{\substack{v\in L^2(\Omega) \\  \|v\|_2 = 1}}  \frac{1}{\tau}-\frac{1}{\tau}\int_\Omega \phi |e^{\frac{\tau}{2} \Delta} v|^2 \ dx  \]
where $e^{\frac{\tau}{2} \Delta} v $ denotes the solution at $t = \tau/2$ of the following free space heat diffusion equation:
\begin{equation}
\begin{cases}
\partial_t u = \Delta u,\\
u(x,t=0)= v(x)
\end{cases}
\end{equation}
which can also explicitly be written by 
\[u(x,\tau/2) = G_{\tau/2}*v, \quad G_\tau(x) = \frac{1}{(4\pi\tau)^{d/2}} \exp(-\frac{|x|^2}{4\tau}).\]

Based on the new relaxed problem, we derive a novel and simple iterative method for finding the approximate solution. The method only alternates three steps: 1. convolution, 2. thresholding, and 3. projection. Because of the use of auxiliary functions, it can be applied into Dirichlet partition problems in arbitrary domains. The convolution is between a free-space heat kernel and a function with finite support, it can be efficiently computed using the fast Fourier transform (FFT) even for the cases of arbitrary domains by simply extending to a relatively larger square domain. Furthermore, we rigorously prove the unconditional stability of the proposed method. In other words, each iteration in the proposed algorithms enjoys the energy decaying property.

The paper is organized as follows. In Section~\ref{sec:relaxedproblem}, we describe the new relaxation of the Dirichlet $k$-partition problem and some basic properties of the relaxed objective functional. We derive the algorithm in Section~\ref{sec:derivation} and describe the detail of implementation in Section~\ref{sec:implement}. We show the performance of the proposed algorithm via extensive numerical experiments in Section~\ref{sec:num} and draw some conclusions and future discussions in Section~\ref{sec:diss}.

\section{New relaxation of the problem}\label{sec:relaxedproblem}

Let $d\geq 2$ and $ \Omega \in \mathbb R^d$ be a bounded open connected set. 
For every open (or quasi-open) subset $A\subset  \Omega$ we denote by $\lambda_1^A$ the first Dirichlet eigenvalue of the Laplace operator:
\begin{equation}
\begin{cases}
-\Delta u = \lambda_1^A u \ \ {\rm in} \ \ A, \\
u = 0 \ \ {\rm on} \ \ \partial A.
\end{cases}
\end{equation}
This can be understood in a weak sense to find $u \in H_0^1(A)$ such that, 
\begin{equation}
\forall v\in H_0^1(A), \ \ \int_A \nabla u \cdot \nabla v \ dx = \lambda_1^A \int_A u v\ dx.
\end{equation}
The eigenvalue can now be given by the minimum principle:
\begin{equation}\label{originalmin}
\lambda_1^A = \min_{u\in H_0^1(A)} \frac{\int_A |\nabla u|^2\ dx}{\int_A |u|^2 \ dx}
\end{equation}

Using the fact $$\int_A |\nabla u|^2 \ dx = \int_A -u \Delta u \ dx,$$
a simple expansion
\begin{equation}
\begin{cases}
e^{\frac{\tau}{2} \Delta} u = u +\frac{\tau}{2} \Delta u+o(\tau) &\quad\tau \rightarrow 0,  \\
|e^{\frac{\tau}{2} \Delta} u|^2 = |u|^2 +\tau u \Delta u+o(\tau) & \quad  \tau \rightarrow 0,
\end{cases}
\end{equation}
and $\int_A |u|^2 \ dx=1$,
one can obtain
%
%
%
%
an approximation to $\lambda^{\tau,A}_1$ through
\begin{equation}\label{originalmin2}
\lambda_1^{\tau,A} =  \min_{\substack{ u\in H_0^1(A) \\  \|u\|_2=1}}  \int_A |\nabla u|^2\ dx =  \min_{\substack{u\in H_0^1(A) \\  \|u\|_2=1}} \frac{1}{\tau}-\frac{1}{\tau}\int_A |e^{\frac{\tau}{2} \Delta} u|^2 \ dx +o(1)  \quad  \tau \rightarrow 0.
\end{equation}

%

Let $\phi: \Omega\rightarrow \{0,1\} $ be a bounded variation function, we further consider a relaxed problem:

\begin{equation}\label{relaxedmin}
\lambda_1^\tau(\phi) =  \inf_{\substack{v\in L^2(\Omega) \\  \|v\|_2= 1}} \frac{1}{\tau} -\frac{1}{\tau}\int_\Omega \phi |e^{\frac{\tau}{2} \Delta} v|^2 \ dx .
\end{equation}

\begin{remark}
1. In the relaxation, we also relax the space for $v$ from $H_0^1(\Omega)$ to $L^2(\Omega)$. \\
2. Furthermore, we note that when $\phi=\chi_A $, the direct relaxation using $\phi$ in the form $$ \inf_{\substack{v\in H_0^1(\Omega) \\  \|v\|_2= 1}} \int_\Omega \phi |\nabla v|^2\ dx \ \ \ 
\textrm{or} \ \ \  \inf_{\substack{v\in L^2(\Omega) \\  \|v\|_2= 1}} \int_\Omega \phi |\nabla v|^2\ dx $$
would obviously attain the global minimum value $0$ at many choices of $v\in H_0^1(\Omega)$ satisfying $\|v\|_2= 1$. For example, one can simply take $\tilde v \in H_0^1(\Omega)$ satisfying $\|\tilde v\|_2= 1$ and ${\rm supp}(\tilde v) \in \Omega\setminus \bar A$  where $\bar A$ is the closure of $A$.
\end{remark}

Denote
\[B = \{u =(u_1,u_2,\cdots,u_k)| u \in L^2(\Omega,\mathbb R^k), \|u_\ell\|_2= 1\},\]
we propose a new relaxation to problem \eqref{eq:ContDirPart} by

\begin{equation}
 \inf_{\varphi\in K} \inf_{u \in B}   E^\tau(\varphi,u) : =\frac{k}{\tau} - \sum_{\ell \in [k]} \frac{1}{\tau}\int_\Omega \varphi_\ell |e^{\frac{\tau}{2} \Delta} u_\ell|^2 \ dx  .  \label{relax2}
 \end{equation}

We first list some properties of $E^\tau(\varphi,u)$ in the following lemma.
\begin{lemma} \label{lem1}
Assume $\tau>0$, then the following properties hold for the functional $E^\tau(\varphi,u)$ defined in \eqref{relaxedmin}.
\begin{itemize}
\item[(i)] $E^\tau(\varphi,u)$ is nonnegative for any $(\varphi, u) \in K\times B$. 
\item[(ii)] Given $\varphi$, $E^\tau(\varphi,u)$ is continuous with respect to $u$ on $L^2(\Omega; \mathbb R^k)$.
\item[(iii)] $E^\tau(\varphi,u)$  is concave with respect to $u$ on $L^2(\Omega; \mathbb R^k)$.
\item[(iv)] The Fr\'echet derivative of $E^\tau(\varphi,u)$ with respect to $u_\ell \in L^2(\Omega; \mathbb R)$ is 
\[\frac{\delta E^\tau(\varphi,u)}{\delta u_\ell} = -\frac{2}{\tau} e^{\frac{\tau}{2}\Delta} (\varphi_\ell e^{\frac{\tau}{2}\Delta} u_\ell) .\]
\end{itemize}
\end{lemma}

\begin{proof}

\begin{itemize}
\item[(i)] For any $(\varphi, u) \in K\times B$ and $\ell\in [k]$, $\varphi_\ell(x) \in [0,1]$ and $\|u_\ell\| = 1$, we then have 
\begin{align*}
-\frac{1}{\tau}\int_\Omega \varphi_\ell |e^{\frac{\tau}{2} \Delta} u_\ell|^2 \ dx \geq & -\frac{1}{\tau}\int_\Omega |e^{\frac{\tau}{2} \Delta} u_\ell|^2 \ dx \\
 \geq & -\frac{1}{\tau}\|e^{\frac{\tau}{2} \Delta}\|^2 \|u_\ell\|_2^2 \geq -\frac{1}{\tau} \|u_\ell\|_2^2 =-\frac{1}{\tau}
\end{align*}
and thus we have $E^\tau(\varphi,u) \geq 0 $. 

\item[(ii)] Let $u, v \in L^2(\Omega, \mathbb R^k)$, direct calculation using the fact that $\|e^{\frac{\tau}{2} \Delta}u\|_2\leq \|u\|_2$ yields
\begin{align*}|E^\tau(\varphi,u) - E^\tau(\varphi,v)| =& \sum_{\ell \in [k]} \frac{1}{\tau}\int_\Omega \varphi_\ell \left||e^{\frac{\tau}{2} \Delta} v_\ell|^2-|e^{\frac{\tau}{2} \Delta} u_\ell|^2\right| \ dx \\
=& \sum_{\ell \in [k]} \frac{1}{\tau}\int_\Omega \varphi_\ell \left|\left(e^{\frac{\tau}{2} \Delta} (u_\ell+v_\ell)\right)\left(e^{\frac{\tau}{2} \Delta} (u_\ell-v_\ell)\right) \right|\ dx \\
\leq & \frac{1}{\tau}\sum_{\ell \in [k]} \|u_\ell+v_\ell\|_2 \|u_\ell-v_\ell\|_2 \\
\leq &  \frac{2\sqrt{k}}{\tau} \|u-v\|_2,
\end{align*}
implying the continuity in the $L^2$ topology.

\item[(iii)] This can be proved by a direct computation.

\item[(iv)] $\forall v \in L^2(\Omega, \mathbb R^k)$ with $v_i = 0 \ (i\neq \ell)$,  direct computation yields
\begin{align*}
\left\langle\frac{\delta E^\tau(\varphi,u)}{\delta u_\ell} , v_\ell  \right\rangle & = \lim_{\epsilon \rightarrow 0} \frac{E^\tau(\varphi,u+\epsilon v) -E^\tau(\varphi,u)}{\epsilon} \\ 
& = -\frac{2}{\tau}\int_\Omega (\varphi_\ell  e^{\frac{\tau}{2} \Delta} u_\ell) e^{\frac{\tau}{2} \Delta}v_\ell \ dx \\
& =  -\frac{2}{\tau}\int_\Omega (e^{\frac{\tau}{2} \Delta} (\varphi_\ell e^{\frac{\tau}{2} \Delta} u_\ell)) v_\ell \ dx \\
& = \left\langle -\frac{2}{\tau} e^{\frac{\tau}{2}\Delta} (\varphi_\ell e^{\frac{\tau}{2}\Delta} u_\ell) , v_\ell  \right\rangle .
\end{align*}
Here, the second to the last equality is by the fact that the operator  $e^{\frac{\tau}{2} \Delta}$ forms a semi-group, that is,
\[\left\langle e^{\frac{\tau}{2}\Delta} f, g \right\rangle =\left\langle  f, e^{\frac{\tau}{2}\Delta} g \right\rangle. \]
\end{itemize}
\end{proof}

In the follows, we first discuss the existence of the solution of the relaxed problem \eqref{relax2} for given $\varphi$.

\begin{theorem}[Existence of $u$]
For a given $\varphi \in K$ and $\tau>0$, the problem \eqref{relax2} admits at least one solution $u \in B$.
\end{theorem}
\begin{proof}
Denote 
\begin{align*} 
\tilde B =  \{\tilde u =(\tilde u_1,\tilde u_2,\cdots,\tilde u_k)| &  u = ( u_1,u_2,\cdots, u_k) \in L^2(\Omega,\mathbb R^k), \\
 &\tilde u_\ell   = \dfrac{e^{\frac{\tau}{2} \Delta} (\phi e^{\frac{\tau}{2} \Delta} u_\ell)}{ \|e^{\frac{\tau}{2} \Delta} (\phi e^{\frac{\tau}{2} \Delta} u_\ell)\|_2},  \|u_\ell\|_2 \leq 1\}
 \end{align*}
and $cl(\tilde B)_{H^1(\Omega, \mathbb R^k)}$ as the closure of $\tilde B$ in $H^1(\Omega, \mathbb R^k)$. Here, we use the fact that $\tilde u_\ell$ is a smooth function and its $H^1$ norm is bounded. Because $H^1(\Omega, \mathbb R^k)$ is compactly embedded  in $L^2(\Omega, \mathbb R^k)$ and $cl(\tilde B)_{H^1(\Omega, \mathbb R^k)}$ is a closed and bounded subset in $H^1(\Omega, \mathbb R^k)$, $cl(\tilde B)_{H^1(\Omega, \mathbb R^k)}$ is then compact in $L^2(\Omega, \mathbb R^k)$. In addition, because $E^\tau(\varphi,u)$ is continuous on $L^2(\Omega, \mathbb R^k)$, there exists at least one $u^\star \in cl(\tilde B)_{H^1(\Omega, \mathbb R^k)}$ such that 
$$E^\tau(\varphi,u^\star) = \inf_{u\in cl(\tilde B)_{H^1(\Omega, \mathbb R^k)}} E^\tau(\varphi,u). $$
Furthermore, it is straightforward to see $\|u^\star\|_2 = 1$ because the concavity of  $E^\tau(\varphi,u)$ and convexity of $cl(\tilde B)_{H^1(\Omega, \mathbb R^k)}$. In other words, the minimum value occurs at the extreme points of the set.

For any $v = (v_1,v_2,\cdots,v_k)\in B$ but not in $cl(\tilde B)_{H^1(\Omega, \mathbb R^k)}$, denote 
$$\hat v_\ell   = \dfrac{e^{\frac{\tau}{2} \Delta} (\phi e^{\frac{\tau}{2} \Delta} v_\ell)}{ \|e^{\frac{\tau}{2} \Delta} (\phi e^{\frac{\tau}{2} \Delta} v_\ell)\|_2},$$ we have 
\begin{align}
& - \left\langle e^{\frac{\tau}{2} \Delta} (\phi e^{\frac{\tau}{2} \Delta} v_\ell) , v_\ell\right\rangle 
\geq - \left\langle e^{\frac{\tau}{2} \Delta} (\phi e^{\frac{\tau}{2} \Delta} v_\ell) , \hat v_\ell \right\rangle  \label{eq:concave1}
\end{align}
and thus
\begin{align}
 - \left\langle e^{\frac{\tau}{2} \Delta} (\phi e^{\frac{\tau}{2} \Delta} v_\ell) , v_\ell \right\rangle 
&\geq  - \left\langle e^{\frac{\tau}{2} \Delta} (\phi e^{\frac{\tau}{2} \Delta} v_\ell) , v_\ell \right\rangle  - 2\left\langle e^{\frac{\tau}{2} \Delta} (\phi e^{\frac{\tau}{2} \Delta} v_\ell) , \hat v_\ell-v_\ell \right\rangle   \nonumber \\
&\geq   - \left\langle e^{\frac{\tau}{2} \Delta} (\phi e^{\frac{\tau}{2} \Delta} \hat v_\ell) , \hat v_\ell \right\rangle. \label{eq:concave2}
\end{align}
The last inequality comes from the fact that the graph of the linearization of a concave functional locates above its graph.
We then have $E^\tau(\varphi,v)\geq E^\tau(\varphi,\hat v)$. This implies that the minimum value can always be attained in $cl(\tilde B)_{H^1(\Omega, \mathbb R^k)}$.
\end{proof}

The existence of $\varphi$ can be argued as follows.

\begin{theorem}[Existence of $\varphi$] Problem \eqref{relax2} admits at least one solution $\varphi \in K$.
\end{theorem}
\begin{proof}
From Lemma~\ref{lem1}(i), we have 
\[\inf_{\varphi \in K} E^\tau(\varphi,u) \in [0,\frac{1}{\tau}].\]
Let $\{\varphi^n\}_{n=1}^\infty \in K$ be a minimizing sequence, from the weak$^\ast L^\infty(\Omega)$ sequential compactness of $K$, we have that there exists a subsequence, which we continue to denote by $\{\varphi^n\}_{n=1}^\infty$ and $\varphi^\star \in K$, such that $\varphi^n \overset{w^\ast- L^\infty}{\rightharpoonup}  \varphi^\star$. Then, because of the fact that $|e^{\frac{\tau}{2} \Delta} u_\ell|^2 \in L^1(\Omega,\mathbb R)$, we have \[\sum_{\ell \in [k]} \int_\Omega \varphi_\ell^n |e^{\frac{\tau}{2} \Delta} u_\ell|^2 \ dx \rightarrow \sum_{\ell \in [k]} \int_\Omega \varphi_\ell^\star |e^{\frac{\tau}{2} \Delta} u_\ell|^2 \ dx \]
implying the infimum attains at $\varphi^\star$.
\end{proof}

Furthermore, because of the form of the objective functional, one could arrive that at least one solution of $\varphi$ gives partition functions. In other words, there exists at least one solution of $\varphi$ whose entries are indicator functions of domains.

\begin{theorem}
Denote \[\tilde K = \{\varphi = (\varphi_1, \varphi_2, \cdots, \varphi_k)| \varphi_\ell: \Omega \rightarrow \{0,1\} \ {\rm measurable \  and} \ \sum_{\ell \in [k]} \varphi_\ell = 1 \ a.e. \ \Omega \} .\] 
\eqref{relax2}  admits at least one solution in $\tilde K$.
\end{theorem}
\begin{proof}
Assume $\varphi$ is an optimal solution and not in $\tilde K$. We assume there exists an $\epsilon>0$, a measurable set $A \subset \Omega$, and $i\neq j \in [k]$, such that $0<|A|<|\Omega|$ and  
\[\varphi_i(x), \varphi_j(x)  \in (\epsilon, 1-\epsilon) \ \ \ \forall x \in A.\]
Considering an arbitrary $\tilde A \subset A$ with a positive measure,
\[\psi_m(x,t) = \varphi_m(x)+ t(\delta_{m,i}-\delta_{m,j}) \chi_{\tilde A}(x) \]
for $m = 1, 2, \cdots, k$. Then we have 
\[\sum_m \psi_m(x,t) = 1\ \ {\rm and} \ \ \psi_m(x,t)\geq 0\]
for $t \in (-\epsilon, \epsilon)$ so that $\psi_m(\cdot,t) \in K$.
Then, we have 
\begin{align*}
\frac{dE^\tau(\varphi,u)}{dt} & =   \frac{1}{\tau} \int_\Omega \chi_{\tilde A} (|e^{\frac{\tau}{2} \Delta} u_j|^2 -  |e^{\frac{\tau}{2} \Delta} u_i|^2) \ dx.
\end{align*}

Because $\varphi$ is an optimal solution, the first order necessary condition gives that \[\int_\Omega \chi_{\tilde A} (|e^{\frac{\tau}{2} \Delta} u_j|^2 -  |e^{\frac{\tau}{2} \Delta} u_i|^2) \ dx = 0, \ \forall \tilde A \subset A {\rm \  with \ a \ positive \ measure}\]
implying that 
\[|e^{\frac{\tau}{2} \Delta} u_j(x)|^2 = |e^{\frac{\tau}{2} \Delta} u_j(x)|^2, \ a.e. {\rm \ in \ } A. \]
This immediately leads us that 
$\varphi_i(x) = 1 \ \forall x\in A$ or $\varphi_j(x) = 1 \ \forall x\in A$ can give the same value of $E^\tau(\varphi,u)$.

If there exists some $\tilde A \subset A$ with positive measure that 
\[\int_\Omega \chi_{\tilde A} (|e^{\frac{\tau}{2} \Delta} u_j|^2 -  |e^{\frac{\tau}{2} \Delta} u_i|^2) \ dx \neq 0.\]
It then quickly implies that $\psi_m(x,0)$ can not be an optimal solution which contradicts with the assumption. 
\end{proof}

Now, we focus on the situation on when $\varphi \in \tilde K$. For any entry in $\varphi$ denoted by $\phi$, when $\phi = \chi_A$ is the indicator function of a set $A \subset \Omega$ , one can intuitively treat $\lambda_1^\tau(\phi)$ as an approximation to $\lambda_1(A)$. 
The following two lemmas show that as $\tau$ becomes small the minimizer $u$ corresponding to $\lambda_1^\tau(\phi)$ becomes strongly localized on $A$ and the approximation is exact in the limit that $\tau \rightarrow 0$.

\begin{lemma}\label{lem:support}
Given $\phi = \chi_A$ where $A$ is a measurable set with positive measure and finite perimeter. Define
\[u^\tau : = \arg\min_{\substack{v \in L^2(\Omega, \mathbb R) \\ \|v\|_2=1}} - \int_\Omega \phi |e^{\frac{\tau}{2} \Delta} v|^2 \ dx.\]
Then, $u^\tau\geq 0, \ \forall x\in \Omega$ (or $u^\tau\leq 0, \ \forall x\in \Omega$)  and $\lim_{\tau\rightarrow 0}\int_\Omega (1-\phi) u^\tau \ dx= 0$.
\end{lemma}
\begin{proof}

Without loss of generality, we prove the non-negativity of $u^\tau$.  If not, we write $u^\tau = u^\tau_{+}-u^\tau_{-}$ where 

\[u^\tau_{+}(x) = \max\{0,u^\tau(x)\},  \ \ u^\tau_{-} =u^\tau_{+} - u^\tau. \]

Writing $\tilde u^\tau = u^\tau_{+}+u^\tau_{-}$, we have $\|\tilde u^\tau\|_2 = 1$, $e^{\frac{\tau}{2} \Delta} u^\tau_{+} \geq 0$,  and $e^{\frac{\tau}{2} \Delta} u^\tau_{-} \geq 0$. Thus, 
\[|e^{\frac{\tau}{2} \Delta} u^\tau (x)| \leq |e^{\frac{\tau}{2} \Delta} \tilde u^\tau (x)|  \]
and 
\[- \int_\Omega \phi |e^{\frac{\tau}{2} \Delta} u^\tau|^2 \ dx  \geq - \int_\Omega \phi |e^{\frac{\tau}{2} \Delta} \tilde u^\tau|^2 \ dx. \]
This implies the positivity of $u^\tau$.

We then prove $\lim_{\tau\rightarrow 0}\int_\Omega (1-\phi) u^\tau \ dx= 0$.

Because $\|u^\tau\|_2= 1$, we have 
\[ \|\phi e^{\frac{\tau}{2} \Delta}  u^\tau\|_2\leq 1  \ \ \ 
\textrm{and}\ \ \  \|e^{\frac{\tau}{2} \Delta} (\phi e^{\frac{\tau}{2} \Delta} u^\tau)\|_2\leq 1\]
using the fact that $\|e^{\frac{\tau}{2} \Delta} \| \leq 1$.

Then, similar to \eqref{eq:concave1} and \eqref{eq:concave2}, we have 
%
\begin{align*}
 - \left\langle e^{\frac{\tau}{2} \Delta} (\phi e^{\frac{\tau}{2} \Delta} u^\tau) , u^\tau \right\rangle 
&\geq   - \left\langle e^{\frac{\tau}{2} \Delta} (\phi e^{\frac{\tau}{2} \Delta} \hat u^\tau) , \hat u^\tau \right\rangle
\end{align*}
where 
\[\hat u^\tau = \dfrac{e^{\frac{\tau}{2} \Delta} (\phi e^{\frac{\tau}{2} \Delta} u^\tau)}{ \|e^{\frac{\tau}{2} \Delta} (\phi e^{\frac{\tau}{2} \Delta} u^\tau)\|_2}.\]
Because $u^\tau$ is the optimal solution, the following equality holds 
\begin{align*}
 - \left\langle e^{\frac{\tau}{2} \Delta} (\phi e^{\frac{\tau}{2} \Delta} u^\tau) , u^\tau \right\rangle  - 2\left\langle e^{\frac{\tau}{2} \Delta} (\phi e^{\frac{\tau}{2} \Delta} u^\tau) , \hat u^\tau-u^\tau \right\rangle  =   - \left\langle e^{\frac{\tau}{2} \Delta} (\phi e^{\frac{\tau}{2} \Delta} \hat u^\tau) , \hat u^\tau \right\rangle,
\end{align*}
which gives 
\[\left\langle \phi e^{\frac{\tau}{2} \Delta} (\hat u^\tau -u^\tau), e^{\frac{\tau}{2} \Delta} (\hat u^\tau -u^\tau) \right\rangle = 0.\]
That is, 
\[ e^{\frac{\tau}{2} \Delta}\hat u^\tau = e^{\frac{\tau}{2} \Delta}u^\tau  \ \  a.e. \ \ \textrm{in} \ \ A  \]
and  
\[ e^{\frac{\tau}{2} \Delta} (\phi e^{\frac{\tau}{2} \Delta} u^\tau) = e^{\frac{\tau}{2} \Delta} (\phi e^{\frac{\tau}{2} \Delta} \hat u^\tau) = \nu \hat u^\tau  \ \  \textrm{in} \ \ \Omega \]
where \[\nu = \|e^{\frac{\tau}{2} \Delta} (\phi e^{\frac{\tau}{2} \Delta} u^\tau )\|_2. \]
In \cite{miranda2007short}, when the perimeter or surface area of the set $A$ is finite and $\psi: \mathbb R^d \rightarrow \mathbb R$ is continuous with compact support, it was rigorously proved that 
\[\lim_{\tau\rightarrow 0} \sqrt{\frac{2\pi}{\tau}} \int_{\mathbb R^d} (1-\phi) e^{\frac{\tau}{2} \Delta}(\phi \psi)  \ dx = \int_\Gamma \psi  \ d\Gamma\]
where $\Gamma$ is the boundary of the set $A$.
It follows that
\[\lim_{\tau\rightarrow 0}\int_\Omega (1-\phi) e^{\frac{\tau}{2} \Delta}(\phi \hat \psi) \ dx = \lim_{\tau\rightarrow 0}\int_\Omega (1-\phi) \hat u^\tau  \ dx =0 \]
by denoting $\hat \psi = e^{\frac{\tau}{2} \Delta} \hat u^\tau$ and thus 
\[\lim_{\tau\rightarrow 0}\int_\Omega (1-\phi) u^\tau \ dx= 0.\]
\end{proof}

\begin{lemma}\label{lem:eigenvalue}
Given $A\subset \Omega$ an open set with positive measure, $\lambda_1^\tau(\chi_A) \leq \lambda_1(A)$ and $\lim_{\tau \rightarrow 0} \lambda_1^\tau(\chi_A)  = \lambda_1(A)$.
\end{lemma}
\begin{proof}
We first prove the boundedness of $\lambda_1^\tau(\chi_A)$. As shown in the proof of Lemma~\ref{lem:support}, $\lambda_1^\tau(\chi_A)$
can be written by 
\begin{align*}
\lambda_1^\tau(\chi_A) = &\min_\gamma \frac{1-\gamma}{\tau} \\
& \textrm{s.t.} \ \ e^{\frac{\tau}{2} \Delta} (\phi e^{\frac{\tau}{2} \Delta}  u) = \gamma u \ \ \textrm{in} \ \  \Omega, \\
&   \ \ \ \ \ \ \|u\|_2=1. 
\end{align*}

Denote by $\tilde u$ the unit eigenfunction corresponding to the first eigenvalue $\lambda_1(A)$ of $-\Delta$ in $A$ with Dirichlet boundary condition and extended by $0$ in $\Omega\setminus A$;
\begin{align*} 
-\Delta \tilde u = \lambda_1(A) \tilde u  \ \ \textrm{in} \ \ A,  \\
 \tilde u= 0 \ \  \textrm{in} \ \ \Omega \setminus A.
\end{align*}
Direct calculation yields $e^{\frac{\tau}{2} \Delta}  \tilde u = e^{-\lambda_1(A) \tau/2} \tilde u $ which is still supported in $A$. Hence, 
\[e^{\frac{\tau}{2} \Delta} (\phi e^{\frac{\tau}{2} \Delta}  \tilde u) =e^{\frac{\tau}{2} \Delta}(e^{-\lambda_1(A) \tau/2} \tilde u ) = e^{-\lambda_1(A) \tau} \tilde u.\]
That is, $\tilde u$ satisfies the constraint with $\gamma = e^{-\lambda_1(A) \tau}$. It implies that 
\[\lambda_1^\tau(\chi_A)  \leq \frac{1-e^{-\lambda_1(A) \tau}}{\tau} \leq \lambda_1(A).\]
In addition, straightforward calculation yields,
\begin{align*}
\dfrac{d \lambda^\tau_1(\phi)}{d\tau} = & -\frac{1}{\tau^2} +\frac{1}{\tau^2}\langle \phi e^{\frac{\tau}{2}\Delta} u, e^{\frac{\tau}{2}\Delta} u \rangle - \frac{1}{\tau}\langle \phi e^{\frac{\tau}{2}\Delta} u, e^{\frac{\tau}{2}\Delta} \Delta u \rangle \\
\leq&  -\frac{1}{\tau^2} +\frac{1}{\tau^2}\langle \phi e^{\frac{\tau}{2}\Delta} u, e^{\frac{\tau}{2}\Delta} (I - \tau\Delta ) u \rangle \\
\leq &-\frac{1}{\tau^2} +\frac{1}{\tau^2}\langle e^{\frac{\tau}{2}\Delta} u, e^{\frac{\tau}{2}\Delta} (I -\tau \Delta ) u \rangle \\
= & -\frac{1}{\tau^2} +\frac{1}{\tau^2} \langle u, e^{\tau\Delta} (I -\tau \Delta ) u \rangle  \leq -\frac{1}{\tau^2} +\frac{1}{\tau^2} \|e^{\tau\Delta} (I -\tau \Delta )\| \leq 0. 
\end{align*}
It is easy to check that the equality holds only when ${\rm supp}(e^{\frac{\tau}{2}\Delta} u) \subset A$, $e^{\frac{\tau}{2}\Delta} (I - \tau\Delta ) u \subset A$ and $u$ is a constant function, which can not happen. Hence $\frac{d \lambda^\tau_1(\phi)}{d\tau} <0$.
Consequently, $\lim_{\tau\rightarrow 0} \lambda_1^\tau(\chi_A) $ exists and $\lim_{\tau\rightarrow 0} \lambda_1^\tau(\chi_A) \leq \lambda_1(A)$.

For the reverse inequality, for the sequence $$u^\tau = \arg\min_{\substack{v \in L^2(\Omega, \mathbb R) \\ \|v\|_2=1}} \frac{1}{\tau}- \frac{1}{\tau}\int_\Omega \phi |e^{\frac{\tau}{2} \Delta} v|^2 \ dx , $$ Lemma~\ref{lem:support} implies that there exists a $\tilde u$ such that $u^\tau \rightarrow \tilde u$ after possibly passing to a subsequence and $\int_\Omega (1-\phi) \tilde u \ dx= 0$. We then observe that $\tilde u$ is in the admissible set of the following minimization problem
\[\min_{\substack{v \in L^2(\Omega, \mathbb R) \\  \|v\|_2=1 \\ \int_\Omega (1-\phi) v  dx= 0}} \frac{1}{\tau}- \frac{1}{\tau}\int_\Omega |e^{\frac{\tau}{2} \Delta} v|^2 \ dx\] 
which indicates that 
\[\lim_{\tau\rightarrow 0} \lambda_1^\tau(\chi_A) \geq \lim_{\tau\rightarrow 0} \frac{1-e^{-\lambda_1(A)\tau}}{\tau} = \lambda_1(A).\]
\end{proof}

\begin{figure}[ht]
\centering
\includegraphics[width = 0.9\textwidth, clip, trim = 30cm 5cm 15cm 8cm]{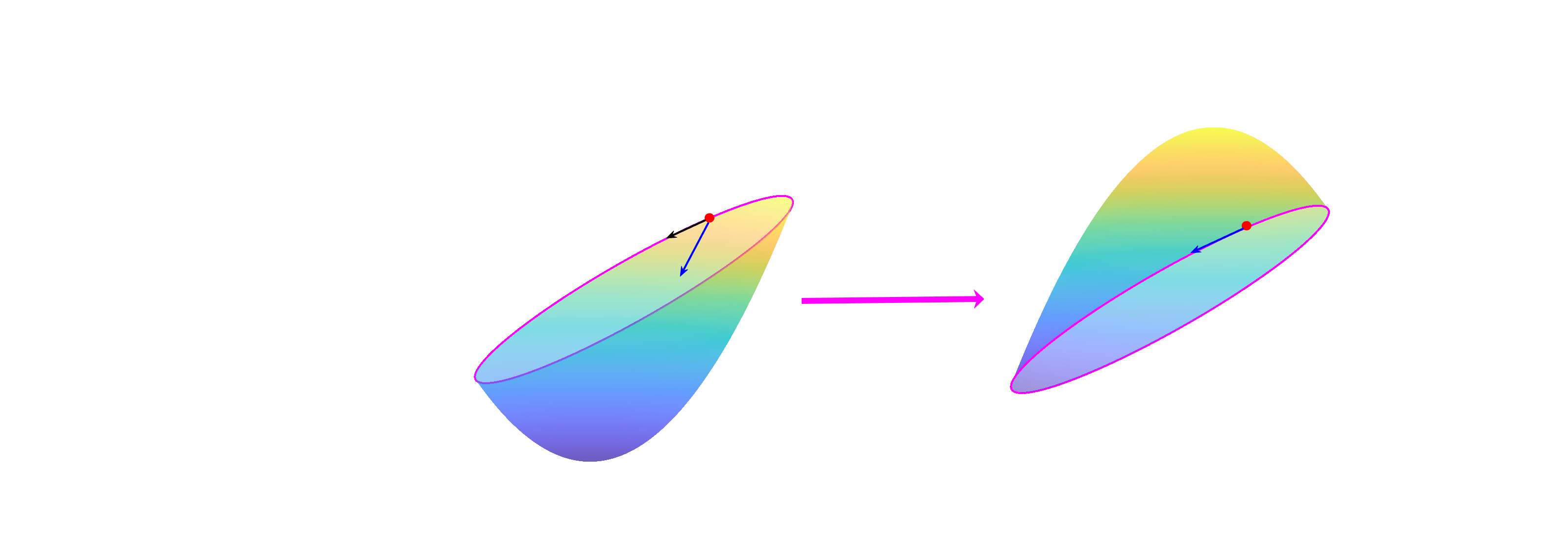} 
\caption{A diagram for the main idea on approximating a convex energy using a concave energy by keeping values on the constraint set. See Section~\ref{sec:relaxedproblem}.}\label{fig:diagconvex2concave}
\end{figure}

In Figure~\ref{fig:diagconvex2concave}, we give a simple demonstration of the main idea on approximating a convex energy using a concave energy by keeping values on the constraint set. Considering the constraint set being a circle, for a convex energy as shown in the left, one could calculate the gradient direction which makes the iteration moves inside the convex hull of the constraint set, then an artificial projection to the circle is necessary which may make the energy either increasing or decreasing slowly. However, the optimization is essentially only on the boundary. If one can find a relaxation to a concave approximation as shown in the right,  the problem can immediately be relaxed to minimizing a concave functional on a convex set, this can be done efficiently by linear sequential programming owing to the fact that the graph of a concave functional always locates under its linearization. Here, we assume the constraint set has a very important feature: the constraint set is exactly the extreme set of its convex hull. Many algorithms can be connected to this idea, for example, the threshold dynamics method \cite{esedoglu2015threshold} whose constraint set is $\{0,1\}^n$ which are the extreme points of its convex hull $[0,1]^n$. It has been applied into many interface related problems such as wetting dynamics \cite{WWX,Xu_2017}, image segmentation \cite{Ma_2021,Wang_2017,wang2019iterative},  foam bubbles \cite{Wang_2019b} and surface reconstruction \cite{Wang_2021}. Some other related work on diffusion generated methods for target-valued harmonic maps can be referred to \cite{Osting_2019,Osting_2020, WOW} where the constraint set is $n\times n$ orthogonal matrix group $O(n)$.

\section{Derivation of the algorithm}\label{sec:derivation}
In this section, for any fixed $k$, we focus on the numerical method for the following minimization problem 
\[\min_{\substack{\varphi\in K\\ u \in B}}  E^{\tau}(\varphi, u)  =\frac{k}{\tau}-\sum_{\ell \in [k]} \frac{1}{\tau}\int_\Omega \varphi_\ell |e^{\frac{\tau}{2} \Delta} u_\ell|^2 \ dx .
\]
We simply use an alternating direction method of minimization to minimize the energy functional with respect to $u = (u_1,u_2,\cdots, u_k)$ and $\varphi = (\varphi_1,\varphi_2, \cdots, \varphi_k)$. To be specific, start with an initial guess $u^0$, we compute the sequence 
\[\varphi^0, u^1, \varphi^1, u^2, \varphi^2, \cdots, \varphi^n, u^n, \cdots\]
by 
\begin{align}
\varphi^n = \min_{\varphi \in K}  E^{\tau}(\varphi, u^n)  \label{probphi}\\
u^{n+1} = \min_{ u \in B} E^{\tau}(\varphi^n, u) \label{probu}
\end{align}

To solve \eqref{probphi}, it's easy to see that the problem is linear in $\varphi_\ell$ and it can be solved in a pointwise manner by a simple comparison among $k$ values at any $x\in \Omega$. That is, for each $x$, 
\begin{equation} \label{updatephi}
\varphi_{\ell}^{n+1}(x) = \begin{cases} 1 & \ \ {\rm if}  \ \  \ell = \min\{\arg\max_{i \in [k]} |e^{\frac{\tau}{2} \Delta} u_\ell^n(x)|^2\}, \\
0 & \ \  {\rm otherwise}.
\end{cases}
\end{equation}

We then quickly have the following lemma by a direct calculation.

\begin{lemma} \label{phidecay}
The $\varphi^{n+1}$ computed using \eqref{updatephi} satisfies 
\[E_k^{\tau}(\varphi^{n+1}, u^n) \leq E_k^{\tau}(\varphi, u^n), \ \ \forall \varphi\in K.\]
\end{lemma}

To solve \eqref{probu}, we note that the functional is quadratic and strictly concave with respect to $u_\ell$. Furthermore, we note that $u_\ell$ are individual to each other, hence \eqref{probu} can be further relaxed to update $u_\ell$ independently,
\begin{align}
u_\ell^{n+1} = \min_{ \substack{u_\ell\in L^2(\Omega,\mathbb R)\\  \|u_\ell\|_2\leq 1}} \frac{1}{\tau}-\int_\Omega \frac{1}{\tau}\varphi_\ell^n |e^{\frac{\tau}{2} \Delta} u_\ell|^2 \ dx. \label{proburelaxed}
\end{align}

The following lemma shows the equivalence between problem \eqref{probu} and problem \eqref{proburelaxed}.
\begin{lemma} 
\[\min_{ u \in B} E_k^{\tau}(\varphi^n, u)  =\min_{ \substack{ u\in L^2(\Omega,\mathbb R^k)\\ \|u_\ell\|_2\leq 1}}E_k^{\tau}(\varphi^n, u) . \]
\end{lemma}

\begin{proof}
If $c=  \|u_\ell\|_2 <1$ for some $\ell \in [k]$, write $\tilde u = (u_1,u_2, \cdots, u_\ell/c, \cdots, u_k)$, we have
\[E^\tau_k(\varphi^n,\tilde u) - E^\tau_k(\varphi^n,u) = -\frac{1-c^2}{c^2} \frac{1}{\tau}\int_\Omega \varphi_\ell |e^{\frac{\tau}{2} \Delta} u_\ell|^2 \ dx \leq 0\]
where the equality holds only when 
\[\int_\Omega \varphi_\ell |e^{\frac{\tau}{2} \Delta} u_\ell|^2\ dx = 0,\]
which is impossible for a minimizer.
Hence, the minimizer for  
\[\min_{ \substack{ u\in L^2(\Omega,\mathbb R^k) \\ \|u_\ell\|_2\leq 1}}E_k^{\tau}(\varphi^n, u)\]
can be attained in $B$.
\end{proof}

Then, problem \eqref{proburelaxed} can be solved by the sequential linear programming approach. That is, we consider
\begin{align}\label{linearprob}
u_\ell^{n+1} = \arg\min_{\substack{u\in L^2(\Omega,\mathbb R^k)\\ \|u_\ell\|_2\leq 1} } L_{\tau}^{u^n_\ell} (u)\end{align}
where \[L_{\tau}^{u^n_\ell} (u) = E_k^{\tau}(\varphi^n, u^n)+\sum_{\ell \in [k]}\left \langle u_\ell-u_\ell^n, -\frac{2}{\tau} e^{\frac{\tau}{2}\Delta} (\varphi_\ell^n e^{\frac{\tau}{2}\Delta} u_\ell^n)\right\rangle.\]

The following lemma shows that the minimization problem \eqref{linearprob} can then be done by a simple projection step.
\begin{lemma}\label{lem:linear}
The minimum value of problem \eqref{linearprob} attains at 
\[u_\ell^{n+1} = \dfrac{e^{\frac{\tau}{2}\Delta} (\varphi_\ell^n e^{\frac{\tau}{2}\Delta} u_\ell^n)}{\|e^{\frac{\tau}{2}\Delta} (\varphi_\ell^n e^{\frac{\tau}{2}\Delta} u_\ell^n)\|_2}.\]
\end{lemma}
\begin{proof}
By rewriting \eqref{linearprob} and dropping constant terms with respect to $u$, problem \eqref{linearprob} is equivalent to 
\[u_\ell^{n+1} =  \arg\min_{\substack{u_\ell\in L^2(\Omega,\mathbb R^k) \\ \|u_\ell\|_2\leq 1} }\left \langle u_\ell, -\frac{2}{\tau} e^{\frac{\tau}{2}\Delta} (\varphi_\ell^n e^{\frac{\tau}{2}\Delta} u_\ell^n)\right\rangle.\]
This is a direct consequence of the constraint $\|u_\ell\|_2=1$ and a simple projection.
\end{proof}

When updating $u$, one can simply iterate one step to find a solution giving a smaller value in problem \eqref{proburelaxed} or iterate to a stationary solution of $u$ for fixed $\varphi^n$ before updating $\varphi^{n+1}$. These are summarized into the following two algorithms ({\it i.e.}, Algorithm~\ref{alg1} and Algorithm~\ref{alg2}), respectively.

\begin{algorithm}[ht!]
\begin{algorithmic}
\State {\bf Input:} {Let $\Omega$ be a given domain, $\tau > 0$, $tol>0$, and $u^0 \in B$.} 
\State {\bf Output:} {$\varphi^n$: representing the partition}
\State Set $s=1$\;
 \While{$\|\varphi^{s+1} -\varphi^{s}\|_2\geq tol$ \ }
 
 {\bf 1.  Diffusion Step.} Compute $u^*_\ell = e^{\frac{\tau}{2}\Delta} u^s_\ell$. 

{\bf 2. Update $\varphi$.} Update $\varphi$ by:
\begin{align*}
\varphi^s_\ell(x) = \begin{cases} 1 & \ \ {\rm if}  \ \  \ell = \min\{\arg\max_{i \in [k]} |u^*_\ell(x)|^2\}, \\
0 & \ \  {\rm otherwise}.
\end{cases}
\end{align*}

{\bf 3. Update $u$.} Update $u$ by:
\begin{align*}
u^{s+1}_\ell = \dfrac{e^{\frac{\tau}{2}\Delta} (\varphi^s_\ell e^{\frac{\tau}{2}\Delta} u_\ell^s)}{\|e^{\frac{\tau}{2}\Delta} (\varphi^s_\ell e^{\frac{\tau}{2}\Delta} u_\ell^s)\|_2}
\end{align*}
Set $s = s+1$\;

 \EndWhile
 \end{algorithmic}
\caption{An iterative method for approximating the solution of problem \eqref{relax2}. } 
\label{alg1}
\end{algorithm}

\begin{algorithm}[ht!]
\begin{algorithmic}
 \State {\bf Input:}  {Let $\Omega$ be a given domain, $\tau > 0$, $tol>0$, and $u^0 \in B$.}
 \State {\bf Output:}  {$\varphi^n$: representing the partition}
 \State Set $s=0$\;
 \While{$\|\varphi^{s+1} -\varphi^{s}\|_2\geq tol$ \  }
 
 \State {\bf 1.  Diffusion Step.} Compute $u^*_\ell = e^{\frac{\tau}{2}\Delta} u^s_\ell$. 

\State {\bf 2. Update $\varphi$.} Update $\varphi$ by:
\begin{align*}
\varphi^s_\ell(x) = \begin{cases} 1 & \ \ {\rm if}  \ \  \ell = \min\{\arg\max_{i \in [k]} |u^*_\ell(x)|^2\}, \\
0 & \ \  {\rm otherwise}.
\end{cases}
\end{align*}

\State {\bf 3. Update $u$.}  Set $\hat u^{s,0} = u^{s}$ and $m = 0$, \While{$\|\hat u^{s,m+1} -\hat u^{s,m}\|\geq tol$}
\begin{align*}
\hat u^{s,m+1}_\ell = \dfrac{e^{\frac{\tau}{2}\Delta} (\varphi^s_\ell e^{\frac{\tau}{2}\Delta} \hat u_\ell^{s,m})}{\|e^{\frac{\tau}{2}\Delta} (\varphi^s_\ell e^{\frac{\tau}{2}\Delta} \hat u_\ell^{s,m})\|_2}.
\end{align*}

\State Set $m=m+1$.

\EndWhile

\State Set $u^{s+1} = \hat u^{s,m}$.\;

\State Set $s = s+1$.\;

  \EndWhile
 \end{algorithmic}
\caption{An iterative method for approximating the solution of problem \eqref{relax2}. } 
\label{alg2}
\end{algorithm}

Intuitively, for Algorithms~\ref{alg1} and \ref{alg2}, a relatively large $\tau$ may make the algorithm be insensitive to initial guesses and a relatively small $\tau$ could increase the accuracy of the method, especially in $u$. Based on these observations, we propose adaptive in time algorithms corresponding to Algorithms~\ref{alg1} and \ref{alg2}  in Algorithms~\ref{alg3} and \ref{alg4}.

\begin{algorithm}[ht!]
\begin{algorithmic}
 \State {\bf Input:}  {Let $\Omega$ be a given domain, $\tau > 0$, $tol >0$, and $u^0 \in B$.}
 \State {\bf Output:}  {$\varphi^n$: representing the partition}
  \State Set $r=0$\;
 \While{$\|\varphi^{r+1} -\varphi^{r}\|_2\geq tol$}
 
 \State Run Algorithm~\ref{alg1} and output $\varphi^r$ and $u^r$.\;

\State Set $r = r+1$ and $\tau \leftarrow \tau/2$.
 
   \EndWhile
 \end{algorithmic}
\caption{An adaptive in time algorithm for Algorithm~\ref{alg1}. } 
\label{alg3}
\end{algorithm}

\begin{algorithm}[ht!]
\begin{algorithmic}
 \State {\bf Input:} {Let $\Omega$ be a given domain, $\tau > 0$, $tol >0$, and $u^0 \in B$.}
  \State {\bf Output:}  {$\varphi^n$: representing the partition}
   \State Set $r=0$\;
 \While{$\|\varphi^{r+1} -\varphi^{r}\|_2\geq tol$}
  \State  Run Algorithm~\ref{alg2} and output $\varphi^r$ and $u^r$.\;
   \State Set $r = r+1$ and $\tau \leftarrow \tau/2$.
 \EndWhile
 \end{algorithmic}
\caption{An adaptive in time algorithm for Algorithm~\ref{alg2}. } 
\label{alg4}
\end{algorithm}

Based on the derivation of above algorithms, we have the following theorem to guarantee that the updating sequence decreases the energy functional ({\it i.e.}, unconditionally stable) and converges in finite steps.

\begin{theorem}
For any given $\tau>0$, considering the sequence $(\varphi^s,u^s)$ ($s =0,1,2,\cdots$) generated from Algorithms~\ref{alg1} or \ref{alg2}, we have 
\[E^\tau(\varphi^{s+1},u^{s+1}) \leq E^\tau(\varphi^{s},u^{s}).\]
\end{theorem}
\begin{proof}
From Lemma~\ref{phidecay}, we have 
\[E^\tau(\varphi^{s+1},u^{s}) \leq E^\tau(\varphi^{s},u^{s}).\]
Combining Lemma~\ref{lem:linear} and the concavity of $E^\tau(\varphi,u)$ with respect to $u$ yields 
\[E^\tau(\varphi^{s+1},u^{s}) \geq L_\tau^{u^n}(\varphi^{s+1},u^{s+1}) \geq E^\tau(\varphi^{s+1},u^{s+1}) \]
where the last inequality comes from the fact that the linearization of a concave functional always locates above the functional. 
Thus we have 
\[E^\tau(\varphi^{s+1},u^{s+1}) \leq E^\tau(\varphi^{s},u^{s}).\]
The equality holds only when $u^{s+1} = u^s$ and $\varphi^{s+1} = \varphi^s$. Because the energy functional has a low bound, the algorithm converges to a stationary solution in finite steps.
\end{proof}

\section{Implementation and discussion on boundary conditions}\label{sec:implement}
In this section, we discuss the implementation of the algorithm on flat tori and with Dirichlet boundary conditions on arbitrary domains.

Note that in the algorithm, the only thing we need to compute is $e^{\frac{\tau}{2}\Delta} u$ for a given $u$. In the follows, we discuss it into two cases.

\begin{itemize}

\item[1.] {\bf Periodic boundary conditions:} In the case where we consider $\Omega$ as flat tori ($[-\pi,\pi]^d$) , we write $e^{\frac{\tau}{2}\Delta }u = G_{\tau/2} *u$ where 
\[G_{\tau/2} = \frac{1}{(2\pi \tau)^{d/2}}\exp(-\frac{|x|^2}{2\tau}).\]

Because of the periodic boundary condition, we simply compute it by 
\[e^{\frac{\tau}{2}\Delta} u  = \mathcal F^{-1} (\exp(-|\xi|^2\tau/2) \mathcal F(u))\]
where $\mathcal F$ is the Fourier transform, $\xi$ is the spectral variable, and $\mathcal F^{-1}$ is the inverse Fourier transform. These can be done efficiently by using the fast Fourier transform (FFT). Note that $\exp(-|\xi|^2\tau/2)$ here is the Fourier transform of $\frac{1}{(\pi \tau)^{d/2}}\exp(-\frac{|x|^2}{2\tau})$.

\item[2.] {\bf Dirichlet boundary conditions:} In the case where we consider $\Omega$ as arbitrary domains with Dirichlet boundary conditions, we consider an extension of $\Omega$ to $\tilde \Omega$ with $\tilde \Omega$ being relatively large and square (See Figure~\ref{dia} for a diagram). The values of $u$ and $\varphi$ are extended from $\Omega$ to $\tilde \Omega$ simply by assigning $0$ in $\tilde \Omega \setminus \Omega$.
In each iteration, we only update $\varphi$ in the domain $\Omega$ which can be simply done by introducing an indicator function of $\Omega$ in the computational domain $\tilde \Omega$. 
To be more precise, we use $\psi$ to  denote the indicator function of the domain and in the update of $\varphi_\ell^n$, we simply set $\varphi^n_\ell = \varphi^n_\ell \psi$.

We note that this does not break the energy decaying property of the algorithm. The proposed algorithm does not need a special discretization for the specific domain, one could compute the Dirichlet $k$-partition in a fixed computational domain for arbitrary shapes by only introducing one auxiliary indicator function $\psi$.

\begin{figure}[ht!]
\centering
\includegraphics[width = 0.4\textwidth]{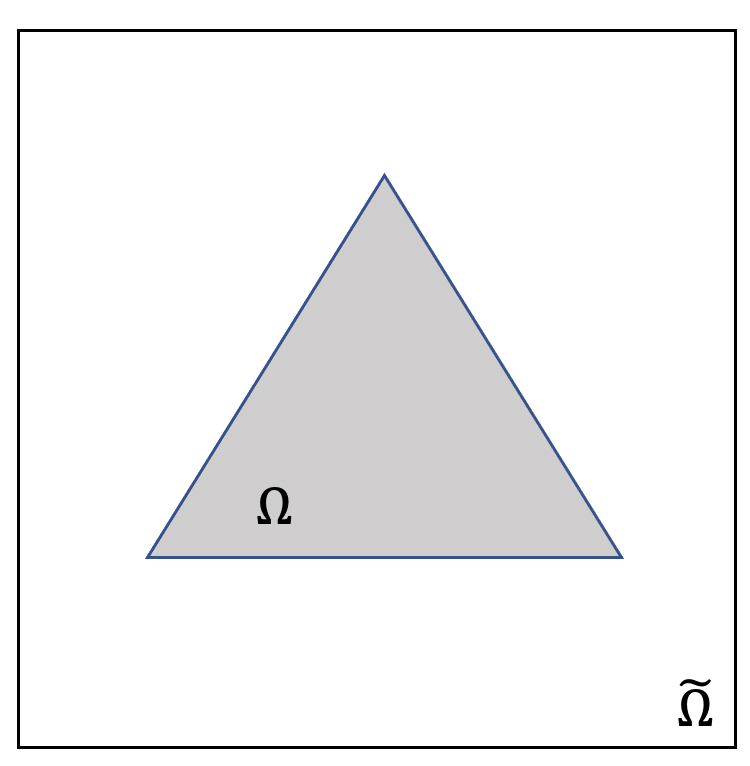}
\caption{A diagram for the extension of a computation domain from $\Omega$ to $\tilde \Omega$.}\label{dia}
\end{figure}

\end{itemize}

\section{Numerical experiments}\label{sec:num}
In this Section, we demonstrate the diffusion generated
method in Algorithms~\ref{alg1}-\ref{alg4} on flat tori and arbitrary domains.  All methods were implemented in MATLAB and results reported below were obtained on a laptop with a 2.7GHz Intel Core i5 processor and 8GB of RAM. 

If there is no other statement,  for all computational results, we set the computational domain as $[-\pi,\pi]^d$ and use a random initial guess for $k$-partition ({\it i.e.} $u^0$) as follows.
\begin{itemize}
\item[Step1.] Generate $k$ random points ({\it seeds}), $x_\ell$ ($\ell\in [k]$), in the computational domain.

\item[Step2.] Compute the Voronoi cell around each seed, $A_\ell$ ($\ell\in [k]$), and denote \[u_\ell^0 = \frac{\chi_{A_\ell}}{(\int_\Omega \chi_{A_\ell} \ dx)^{1/2}}.\]
\end{itemize}

\begin{remark}
If we consider the cases of Dirichlet boundary conditions in arbitrary domains, we restrict $k$ random seeds in the specific domain instead of the whole computational domain. 
\end{remark}

\subsection{Accuracy on the computation of the first eigenvalue and eigenfunction}\label{sec:num1}
In the first experiment, we check the computation accuracy of the proposed algorithm for computing the first eigenvalue and eigenfunction when $\varphi$ is fixed as 
$\varphi = \chi_A$ with $A = [-\pi/2,\pi/2]^2 \subset [-\pi,\pi]^2$.
We then solve \[ u = \arg\min_{v\in B} E^\tau(\chi_A,v) \]
by the Step 3 in Algorithm~\ref{alg2} associated with an adaptive in time technique. To make no confusion, we write the scheme in the follows.

\begin{algorithm}[ht]
\begin{algorithmic}
 \State {\bf Input:} {Let $A \subset \Omega$ be a given domain, $\tau > 0$, $tol >0$, and $u^0 \in B$.}
 \State {\bf Output:} {$u^n$: approximating the first eigenfunction, \\
 $\frac{1}{\tau} - \frac{1}{\tau} \int_\Omega \chi_A |e^{\frac{\tau}{2}\Delta} u^n|^2 \ dx $: approximating the first eigenvalue.}
  \State  Set $s=0$\;
 \While{$\tau \geq tol$}
  \State  Set $\hat u^{s,0} = u^s$ and $m=0$, \While{$\|u^{s,m+1}  - u^{s,m} \|_2 \geq tol$}
\begin{align*}
\hat u^{s,m+1} = \dfrac{e^{\frac{\tau}{2}\Delta} (\chi_A  e^{\frac{\tau}{2}\Delta} \hat u^{s,m})}{\|e^{\frac{\tau}{2}\Delta} (\chi_A e^{\frac{\tau}{2}\Delta} \hat u^{s,m})\|_2}.
\end{align*}
 \State  Set m = m+1.\;
  \EndWhile

 \State Set $\tau \leftarrow \frac{\tau}{2}$.\;

 \State  Set $u^{s+1} = \hat u^{s,m}$.\;

 \State  Set $s = s+1$.\;
 
  \EndWhile
 \end{algorithmic}
\caption{A scheme to approximate the first eigenvalue and eigenfunction of a fixed domain with Dirichlet boundary conditions.} 
\label{alg5}
\end{algorithm}

The exact solution of the first eigenvalue and the corresponding eigenfunction (with unit $L^2(\Omega)$ norm) in $A$ with the zero Dirichlet boundary condition is written by \begin{align}\lambda_1(A) = 2, \  \ u^A_1(x,y) = \begin{cases} \frac{2}{\pi}\cos x \cos y  \ & {\rm if} \  (x,y )\in [-\pi/2,\pi/2]^2,  \\
0 \ &  {\rm otherwise}.
\end{cases}\label{exact1}
\end{align}

We first check the accuracy of the new approximation $$\frac{1}{\tau} - \frac{1}{\tau} \int_\Omega \chi_A |e^{\frac{\tau }{2}\Delta}u|^2 \ dx$$ to the first eigenvalue when $u = u^A_1$. Table~\ref{tab:accuracycheck1} lists the approximate eigenvalue computed with different $\tau$. One can observe that the computed eigenvalue converges to the exact value $2$ as $\tau$ goes to $0$ from below. This verifies the motivation of the new proposed approximation and is consistent with Lemma~\ref{lem:eigenvalue}. We note that this fact is also similar to those in \cite{bourdin2010optimal,osting2013minimal}.
\begin{table}[ht]
\caption{Approximate eigenvalues computed on $\Omega$ with different $\tau = 2^{-4} $- $2^{-15}$. See Section~\ref{sec:num1}.}\label{tab:accuracycheck1} \centering
\begin{tabular}{c|c|c|c|c|c|c}
\hline
$\tau$ & $2^{-4}$ & $2^{-5}$& $2^{-6}$&$2^{-7}$ &$2^{-8}$&$2^{-9}$ \\
\hline
 Approximate $\lambda_1$& 1.8801&  1.9388&$1.9691$ &1.9845 &1.9922&  1.9961 \\
\hline
\hline
$\tau$ &$2^{-10}$ &$2^{-11}$ & $2^{-12}$& $2^{-13}$ & $2^{-14}$ & $2^{-15}$\\
\hline
Approximate $\lambda_1$& 1.9980 & $1.9990$ &  $1.9995$ &  $1.9998$ &   $1.9999$ &   $1.9999$ \\
\hline
\end{tabular}
\end{table}

We then apply Algorithm~\ref{alg5} onto the computational domain $[-\pi,\pi]^2$ discretized by $64\times 64$, $128\times 128$, $256\times 256$, $512\times 512$ and $1024\times 1024$ uniform grid points. Figure~\ref{fig:accuracycheck} shows the approximated solution of the eigenfunction and the difference between the approximate solution and the exact solution~\eqref{exact1}, computed on $512\times 512$ discretized mesh. We observe that the support of the approximate eigenfunction is almost in $[-\pi/2,\pi/2]^2$ and is consistent with the exact solution. Table~\ref{tab:eigenvalue} lists the eigenvalues computed on different discretization with different values of $tol$. It is clear that the eigenvalue converges to the exact value $2$ with a finer mesh and a smaller $\tau$.

\begin{figure}[ht!]
\centering
\includegraphics[width = 0.4\textwidth]{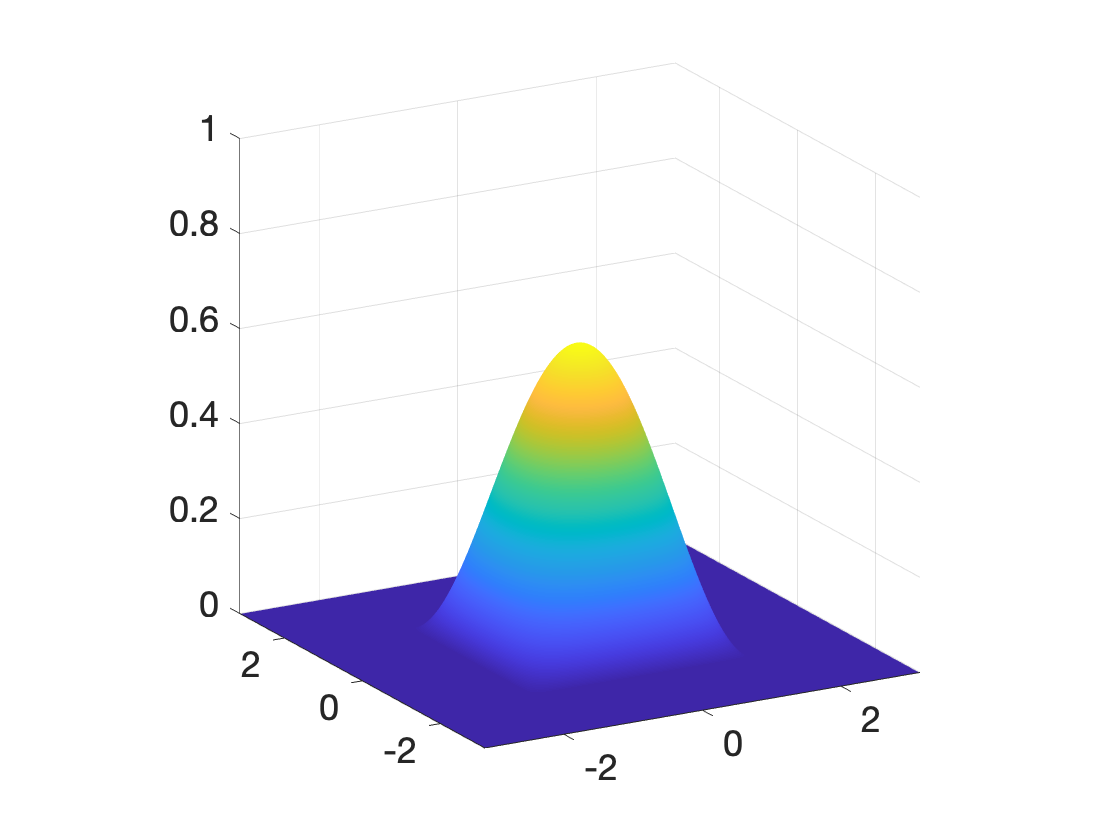}
\includegraphics[width = 0.4\textwidth]{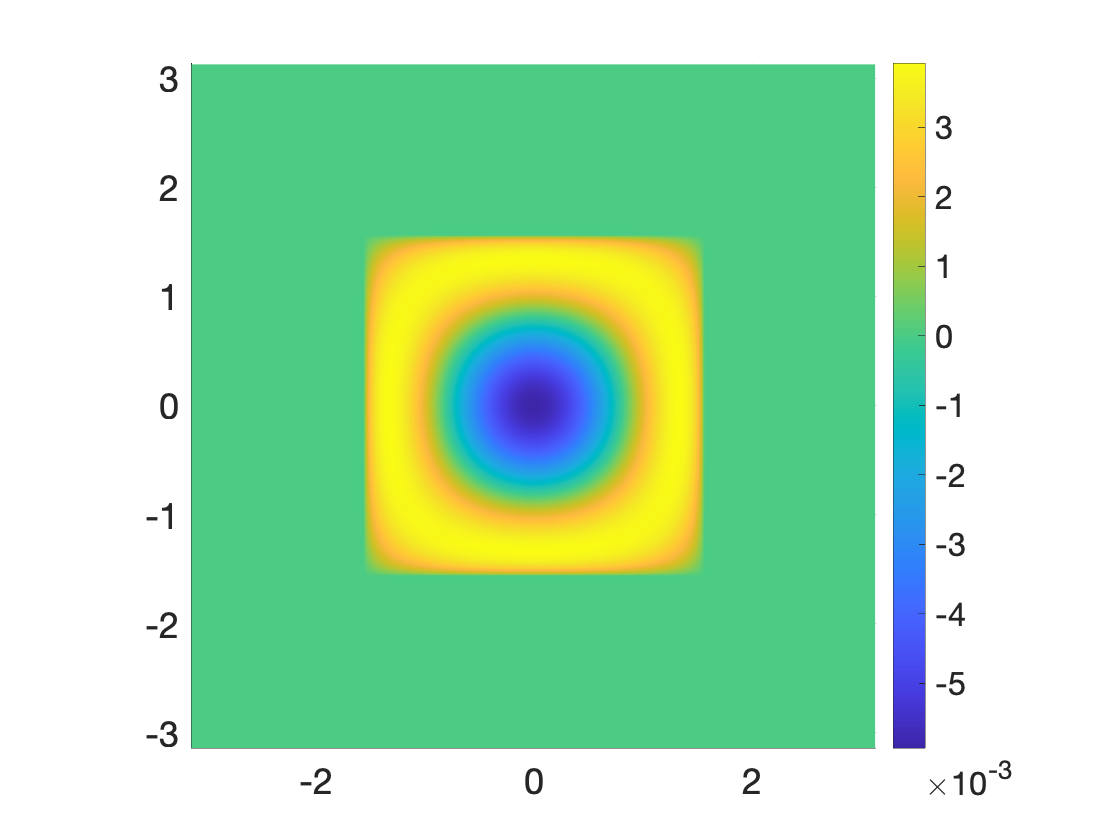}
\caption{Left: the approximate eigenfunction in $[-\pi/2,\pi/2]^2$. Right: The error compared to the exact solution \eqref{exact1}. See Section~\ref{sec:num1}.} \label{fig:accuracycheck}
\end{figure}
\begin{table}[ht!]
\caption{The approximate eigenvalue computed on $\Omega$ with different discretization and different $tol = 10^{-3}$-$10^{-7}$. See Section~\ref{sec:num1}.}\label{tab:eigenvalue}
\centering
\begin{tabular}{c|c|c|c|c|c}
\hline
\backslashbox{$tol$}{Resolution} & $64\times 64$ &$128\times 128$ & $256\times 256$ & $512\times 512$ & $1024 \times 1024$\\
\hline
$10^{-3}$  &  1.7725 & 1.8514 &  1.8863 &   1.9023 &   1.9100 \\
\hline
$10^{-4}$  & 1.7938&  1.8881&  1.9342 & 1.9543 &  1.9631\\
\hline
$10^{-5}$  & 1.7970 & 1.8945& 1.9463 &   1.9727&  1.9852 \\
\hline
$10^{-6}$  &  1.7972 & 1.8949 & 1.9472 &  1.9746 &  1.9890 \\
\hline
$10^{-7}$  & 1.7972 &  1.8950 &  1.9473 &1.9749 &1.9896 \\
\hline
\end{tabular}
\smallskip
\end{table}

\begin{figure}[ht!]
\includegraphics[width = 0.19\textwidth, clip, trim = 11cm 6cm 10cm 5cm]{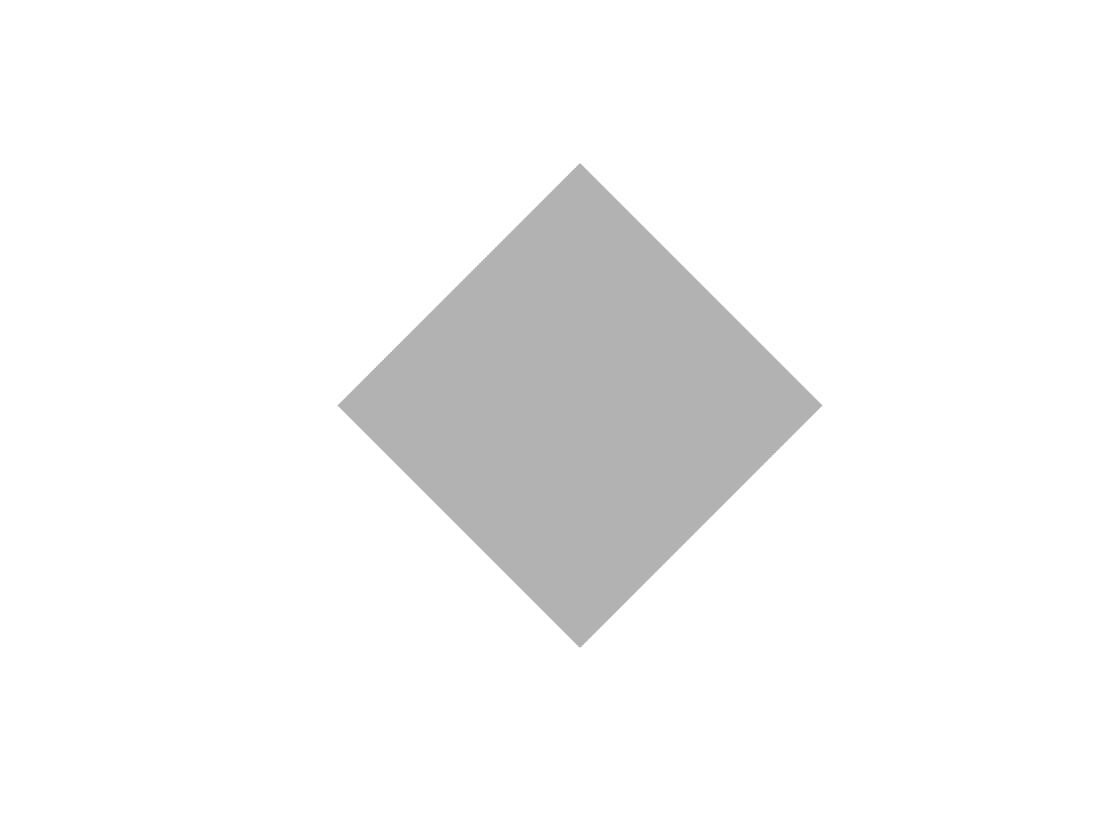}
\includegraphics[width = 0.19\textwidth, clip, trim = 11cm 6cm 10cm 5cm]{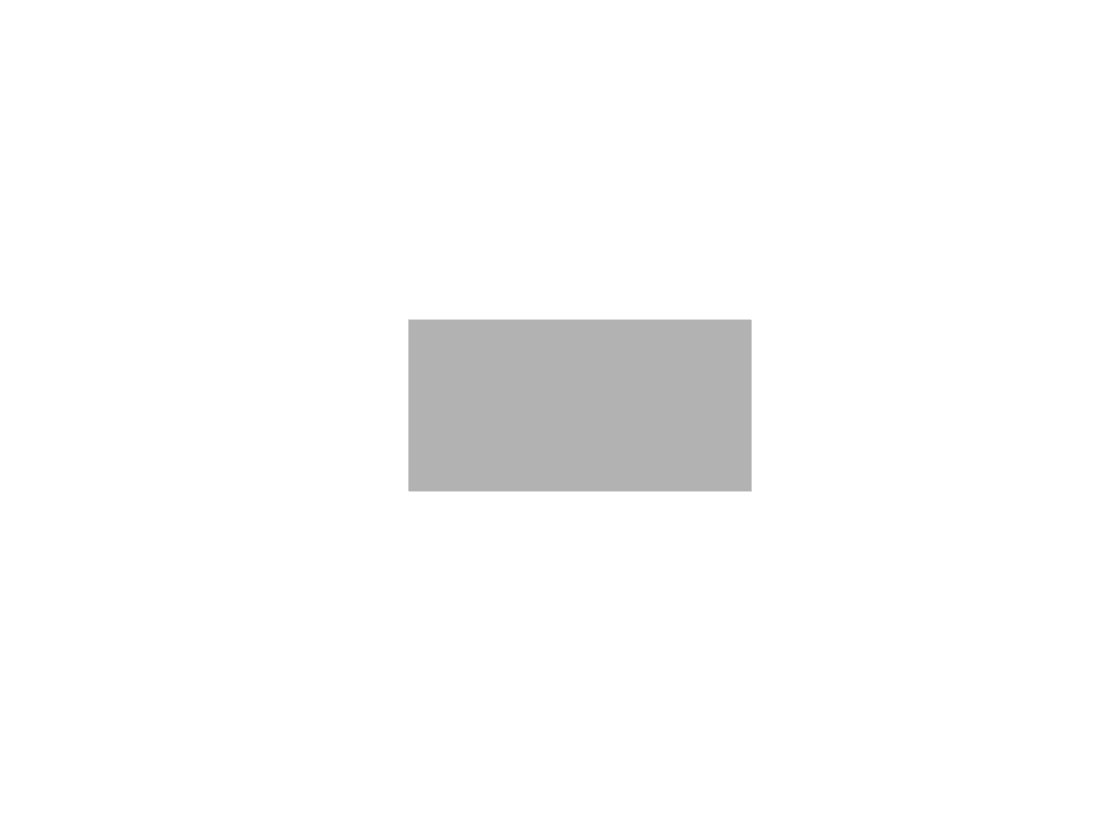}
\includegraphics[width = 0.19\textwidth, clip, trim = 11cm 6cm 10cm 5cm]{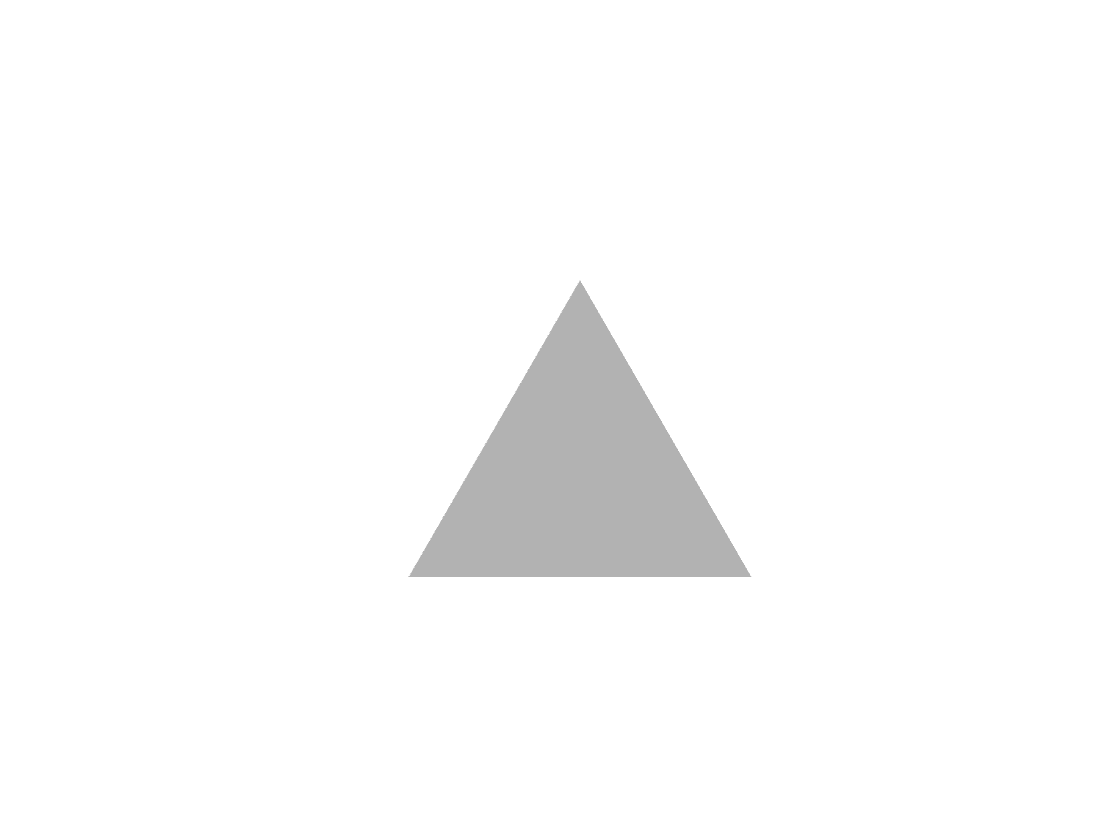}
\includegraphics[width = 0.19\textwidth, clip, trim = 11cm 6cm 10cm 5cm]{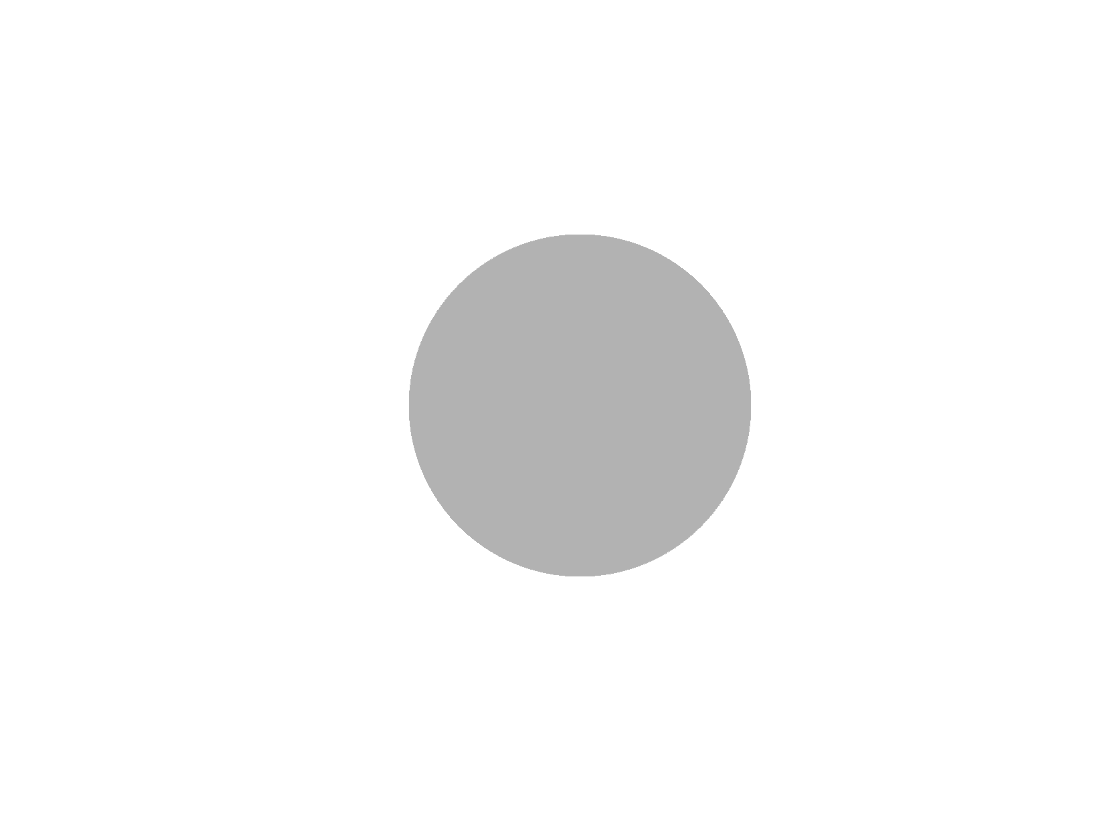}
\includegraphics[width = 0.19\textwidth, clip, trim = 11cm 6cm 10cm 5cm]{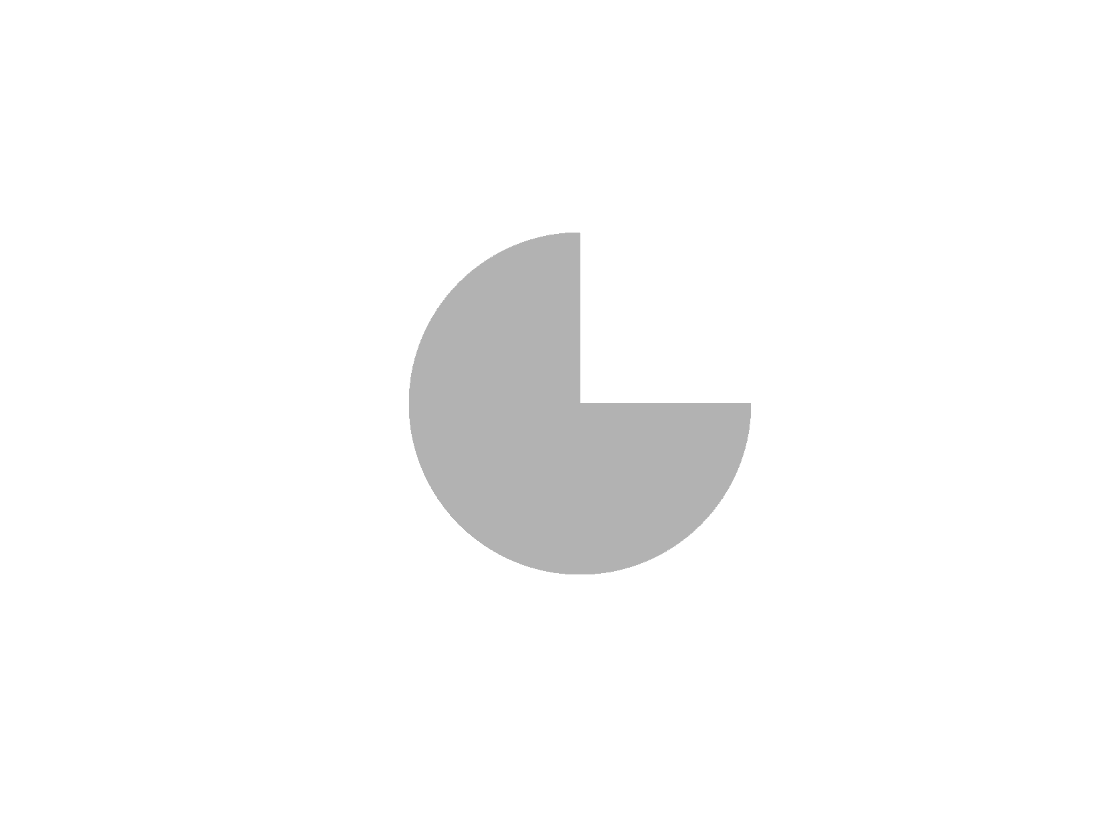}
\caption{Five general domains to check the convergence of approximate eigenvalue as $\tau \rightarrow 0$. See Section~\ref{sec:num1}.}\label{fig:5domains}
\end{figure}

\begin{table}[ht!]
\caption{The approximate eigenvalue for different domains. See Section~\ref{sec:num1}.}\label{tab:eigenvalue2}
\centering
\begin{tabular}{c|c|c|c|c|c}
\hline
\backslashbox{$\tau$}{Domain} & rotated square & rectangle  & triangle & disk & $3/4$ disk \\
\hline
$10^{-3}$  &  1.9285 & 4.6713& 5.0156 &  2.2657 & 4.3111   \\
\hline
$10^{-4}$  & 1.9737&   4.8693&  5.2237& 2.3195&  4.5068  \\
\hline
$10^{-5}$  & 1.9877 & 4.9274& 5.2872 &   2.3360&   4.5667 \\
\hline
$10^{-6}$  &  1.9910 &  4.9384& 5.3008&  2.3397 &  4.5791 \\
\hline
$10^{-7}$  & 1.9915 &  4.9397 &  5.3025 &2.3402 &  4.5806\\
\hline
Reference solution & 2 & 5& 16/3 &2.3438& 4.6182 \\ 
\hline
\end{tabular}
\smallskip
\end{table}

To check the convergence of the relaxation to the exact eigenvalue for arbitrary domains in a same computational domain ({\it i.e.}, $[-\pi,\pi]^2$), we apply Algorithm~\ref{alg5} to consider the approximate eigenvalues in the following several domains: 1. a rotated square domain, 2. a rectangle domain, 3. an equilateral triangle domain, 4. a disk domain, and 5. a three quarter disk domain as displayed in Figure~\ref{fig:5domains}. For the domain where the exact solution is not available, we set the reference solution by the solution computed from finite element method with very fine meshes. In all experiments, we use $1024\times 1024$ uniform grids to discretize the computational domain. In Table~\ref{tab:eigenvalue2}, we list the approximate value obtained by different $tol$ and observe that for general domains, the approximate values converge to the exact values as $\tau \rightarrow 0$ from below, which is consistent with the fact we proved in Lemma~\ref{lem:eigenvalue} and also shows that the relaxation and approximation is insensitive to the domain of consideration.

\subsection{Verification on energy decaying and comparisons among Algorithms~\ref{alg1}-\ref{alg4}} \label{sec:num2} 

In this section, we perform a careful study on the energy decaying properties for Algorithms~\ref{alg1}-\ref{alg4}. For this study, we simply choose $k=3$ with a random initial guess. Consider the computational domain discretized by $256\times 256$ grid points and use $\tau = 1/4$ for Algorithms~\ref{alg1}-\ref{alg2} with the random initial guess as shown in the middle of Figure~\ref{fig:energycheck}. One can observe that Algorithm~\ref{alg1} takes $66 $ iterations to converge while  Algorithm~\ref{alg2} only takes $29$ iterations. This is consistent with the fact that Algorithm~\ref{alg2} always finds the stationary solution for $u^{n+1}$ when $\varphi^n$ is given while Algorithm~\ref{alg1} only iterates one step along the descent direction. However, what is interesting, both algorithms converge to the same stationary solution and Algorithm~\ref{alg1} only takes $0.27$ {\it seconds} CPU time while Algorithm~\ref{alg2} takes $2.52$ {\it seconds} CPU time. One can understand this by that even Algorithm~\ref{alg2} only takes $29$ iterations, in each iteration, it takes many more steps for $u$ to converge.

\begin{figure}[ht!]
\centering
\includegraphics[width = 0.3\textwidth]{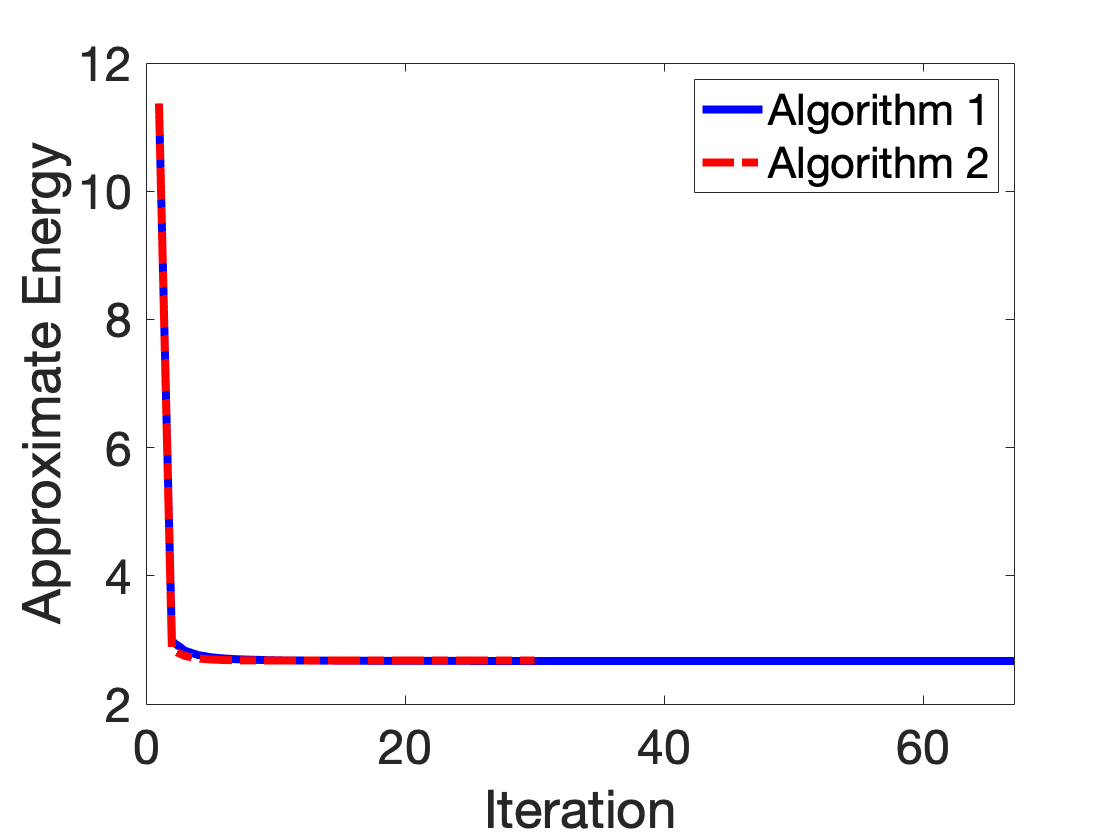}
\includegraphics[width = 0.3\textwidth]{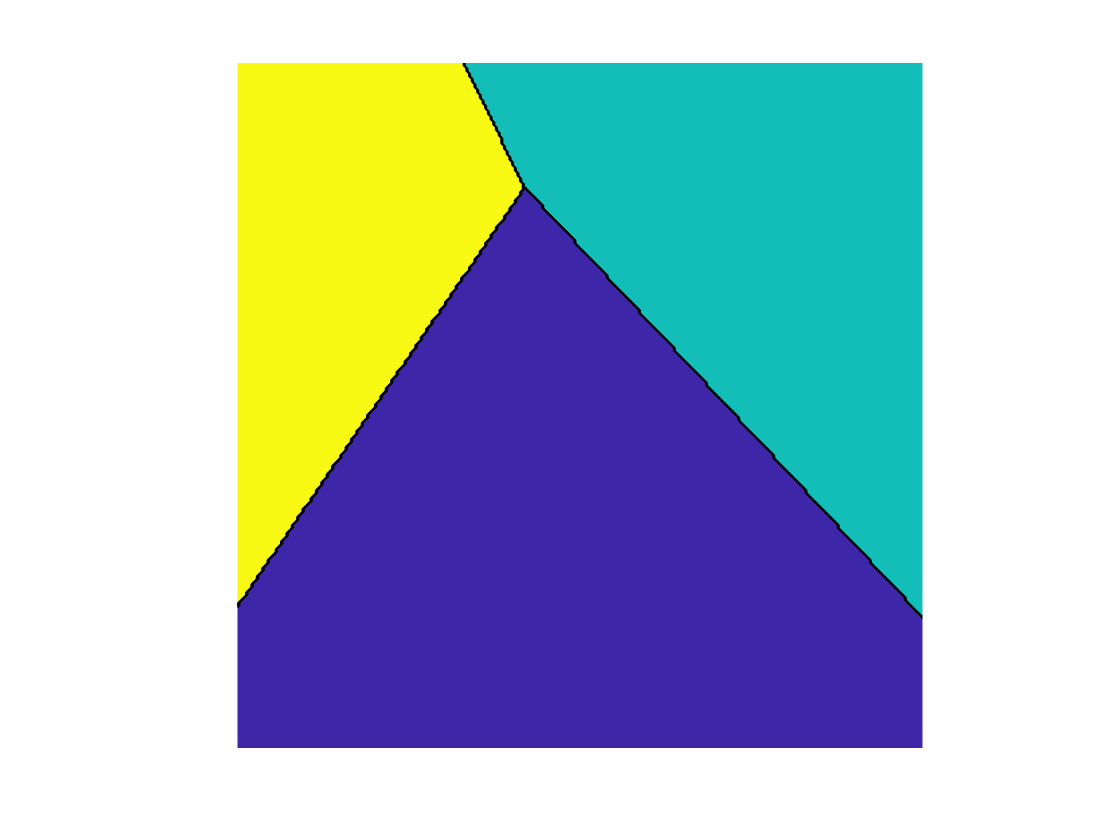}
\includegraphics[width = 0.3\textwidth]{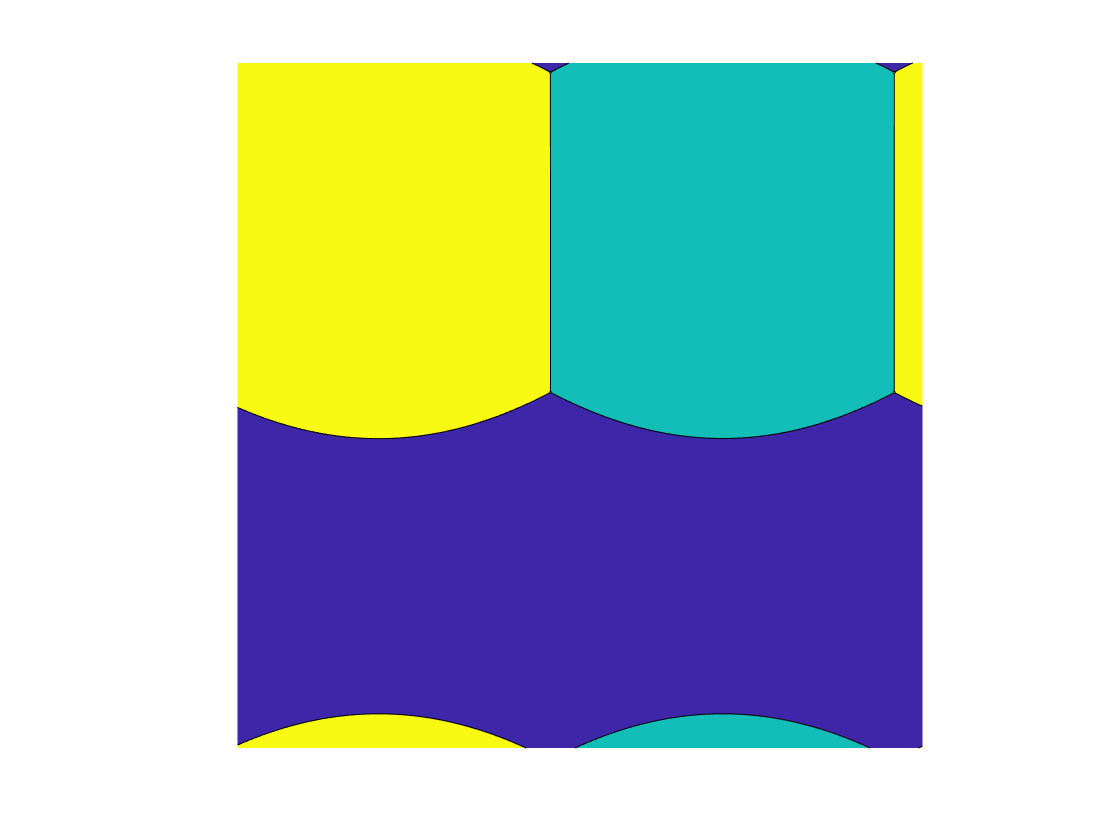}
\caption{Left: The energy curve with respect to the number of iterations for Algorithms~\ref{alg1}-\ref{alg2}. Middle: An initial guess for $3$-partition. Right: The converged solution (same for Algorithm~\ref{alg1} and Algorithm~\ref{alg2}). See Section~\ref{sec:num2}.}\label{fig:energycheck}
\end{figure}

Furthermore, we check the energy decaying properties of Algorithms~\ref{alg3}-\ref{alg4} with an initial $\tau = 1/4$ and $tol_\tau = 10^{-4}$. Because adaptive in time techniques are used, we compute the approximate energy in two ways: 1. using the adaptive $\tau$ and 2. using a fixed relatively small $\tau = 10^{-4}$. Figure~\ref{fig:energy2} list the energy decaying curve of Algorithms~\ref{alg3} and \ref{alg4}. If we use a fixed $\tau$ to calculate the energy, the approximate energy is monotonically decaying (See the left in Figure~\ref{fig:energy2}).  If we use varying $\tau$ to compute the approximate energy, one can observe that in the iteration of each $\tau$, the approximate energy is decaying but the energy jumps to a large value at the iteration when $\tau$ is halved. This is also consistent with the observation from Tables~\ref{tab:accuracycheck1} and \ref{tab:eigenvalue} that the energy converges to the exact value from below (See the middle in Figure~\ref{fig:energy2}).  In the right of Figure~\ref{fig:energy2}, we plot the change between two iterations of $\varphi$, one can see that as $\tau$ decreases, the partition  becomes stationary. In particular, we observe that an initial $\tau = 1/4$ can already efficiently find the stationary solution and decreasing $\tau$ only refines the solution a bit but computes the approximate energy more accurately.

\begin{figure}[ht]
\centering
\includegraphics[width = 0.3\textwidth]{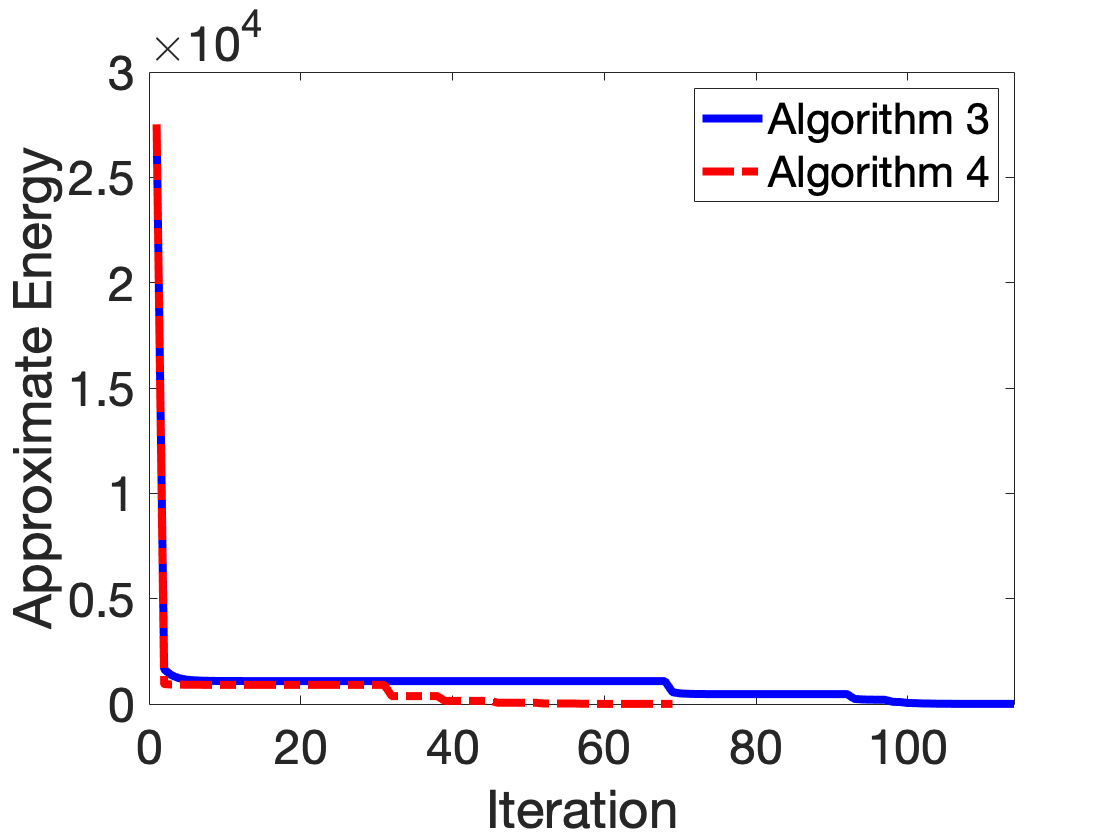}
\includegraphics[width = 0.3\textwidth]{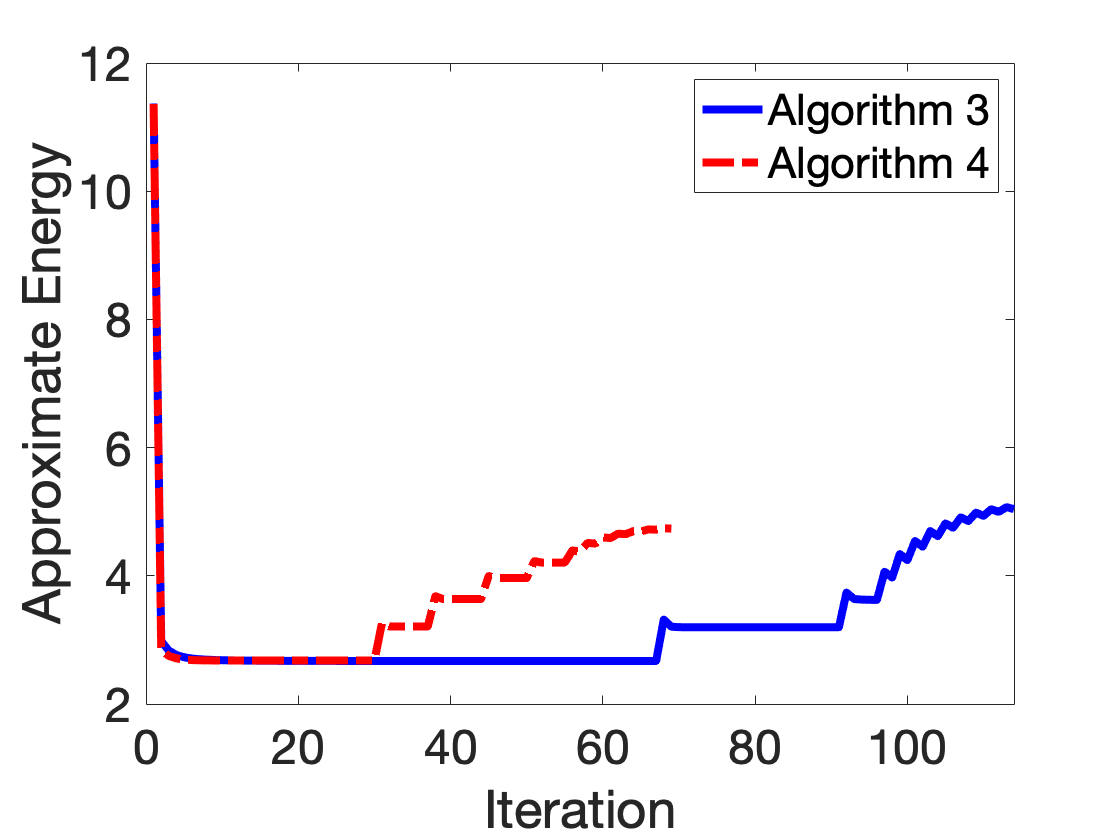}
\includegraphics[width = 0.3\textwidth]{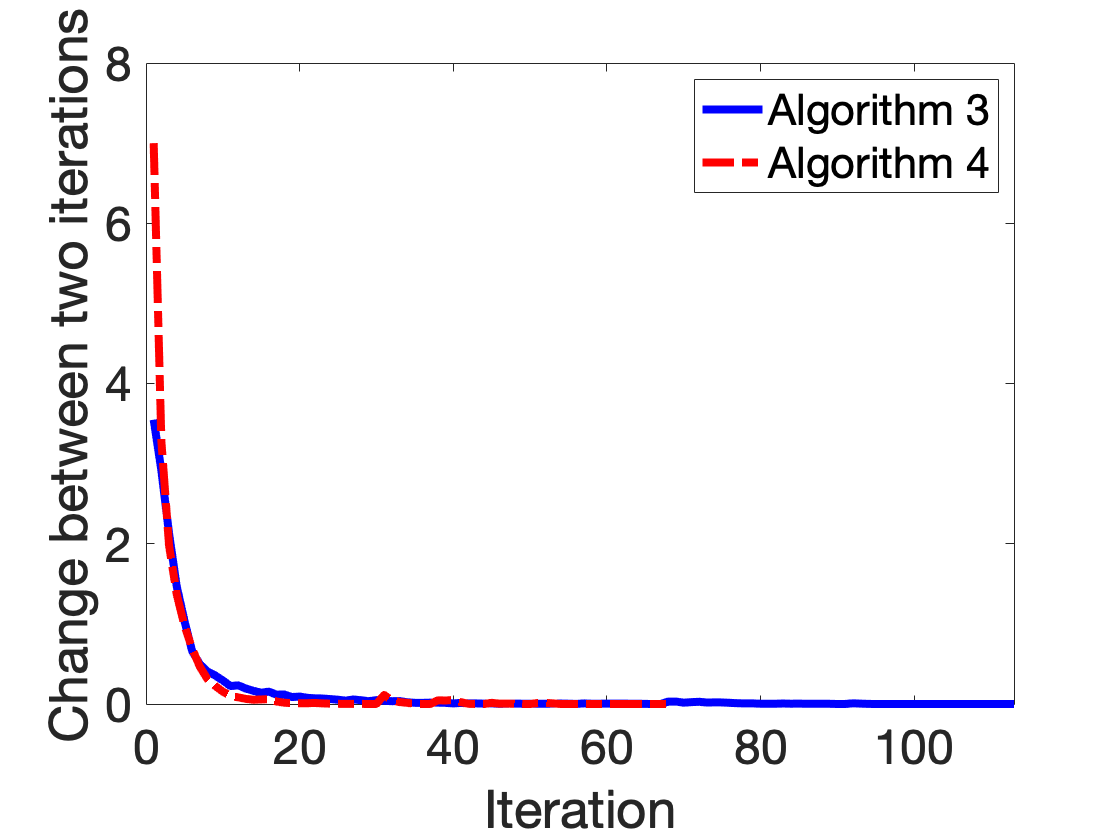}
\caption{Left: Energy decaying curves with respect to iterations when using $\tau = 10^{-4}$ to calculate the approximate energy, Middle: Energy decaying curves with respect to iterations when using the adaptive values of $\tau$ to calculate the approximate energy, Right: Change in $\varphi$ with respect to iterations. See Section~\ref{sec:num2}.} \label{fig:energy2}
\end{figure} 

\subsection{Acceleration by adaptive in time techniques}\label{sec:adptive}

In this section, we check the advantage of adaptive in time techniques through two experiments from a same initialization: 1.) adaptive in time for $\tau$ changes in  $[1/4,1/8,1/16,1/32,1/64]$, 2.) fix $\tau = 1/64$ without adaptive in time. In Figure~\ref{fig:adaptiveintime}, we list the snapshots at different iterations for both experiments and observe that both converge to the same solution. In the first row, snapshots at iteration $67, 77, 85, 88, 90$ correspond to stationary solutions at different $\tau = 1/4, 1/8, 1/16, 1/32,$ and $1/64$. One can observe that in the sense of partition, using large $\tau$ can achieve almost same result as that obtained from small $\tau$, but it accelerate the convergence dramatically. For instance, using a large time step $\tau =1/4$ gives a similar solution after 10 iterations while using $\tau = 1/64$ requires about 80 iterations. 

\begin{figure}[ht!]
\centering
\begin{tabular}{c|c|c|c|c|c|c|c}
\hline 
 1 & 10 & 20  & 67 & 77  & 85& 88  & 90 \\
\includegraphics[width = 0.09\textwidth, clip, trim = 4cm 1cm 3cm 1cm]{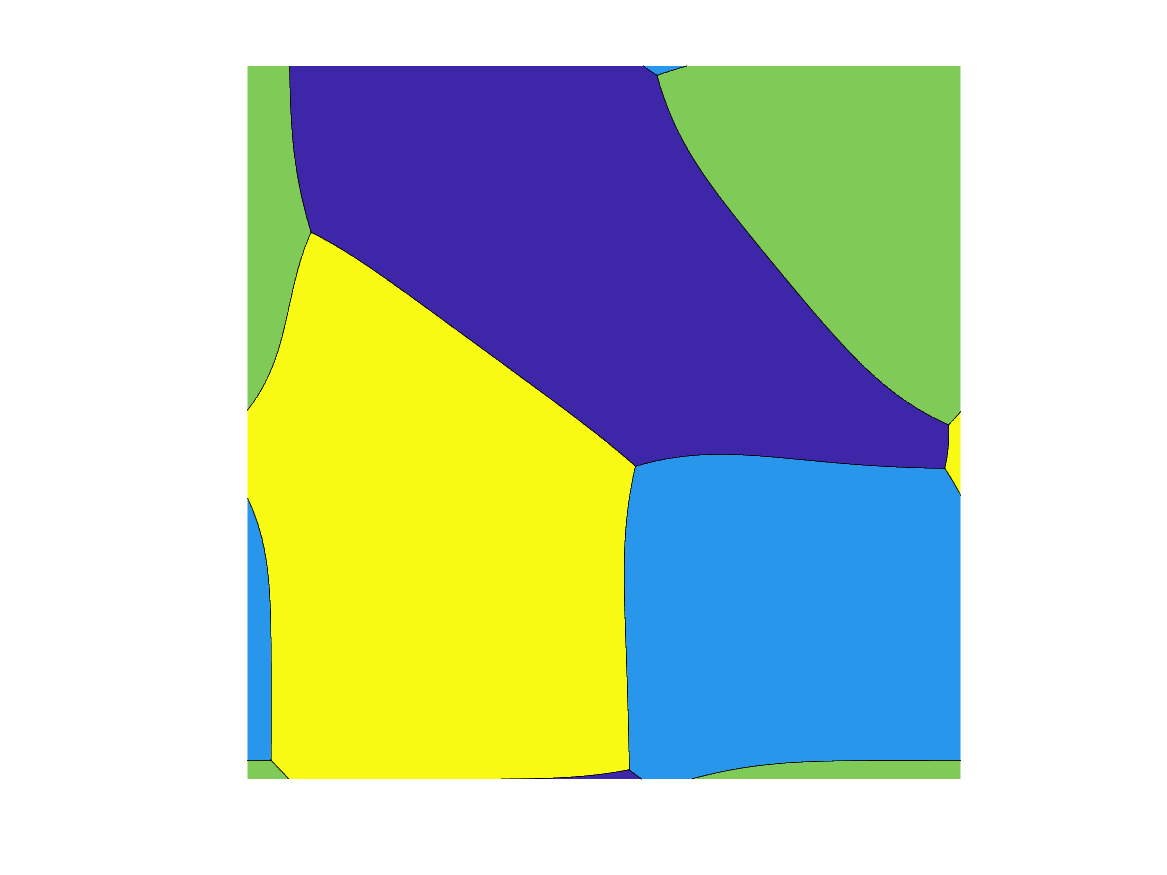}&  
\includegraphics[width = 0.09\textwidth, clip, trim = 4cm 1cm 3cm 1cm]{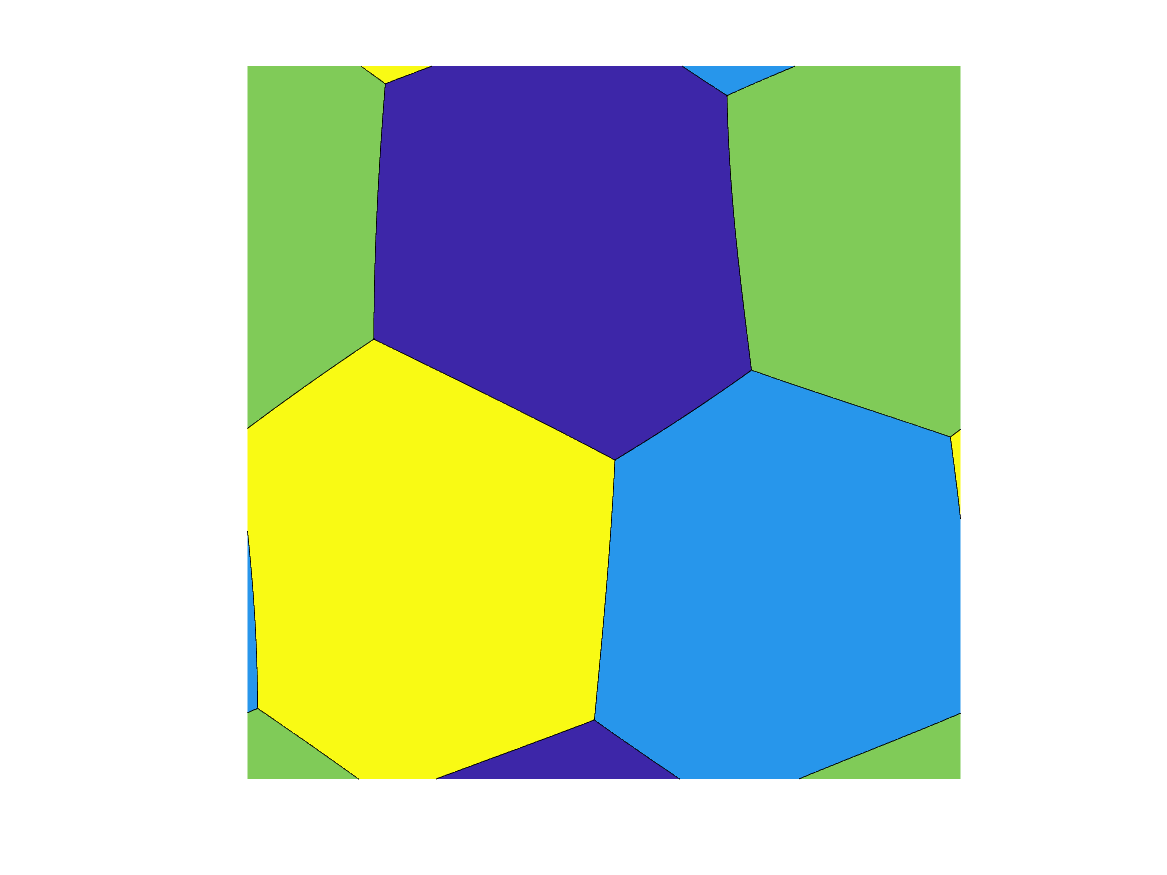}&
\includegraphics[width = 0.09\textwidth, clip, trim = 4cm 1cm 3cm 1cm]{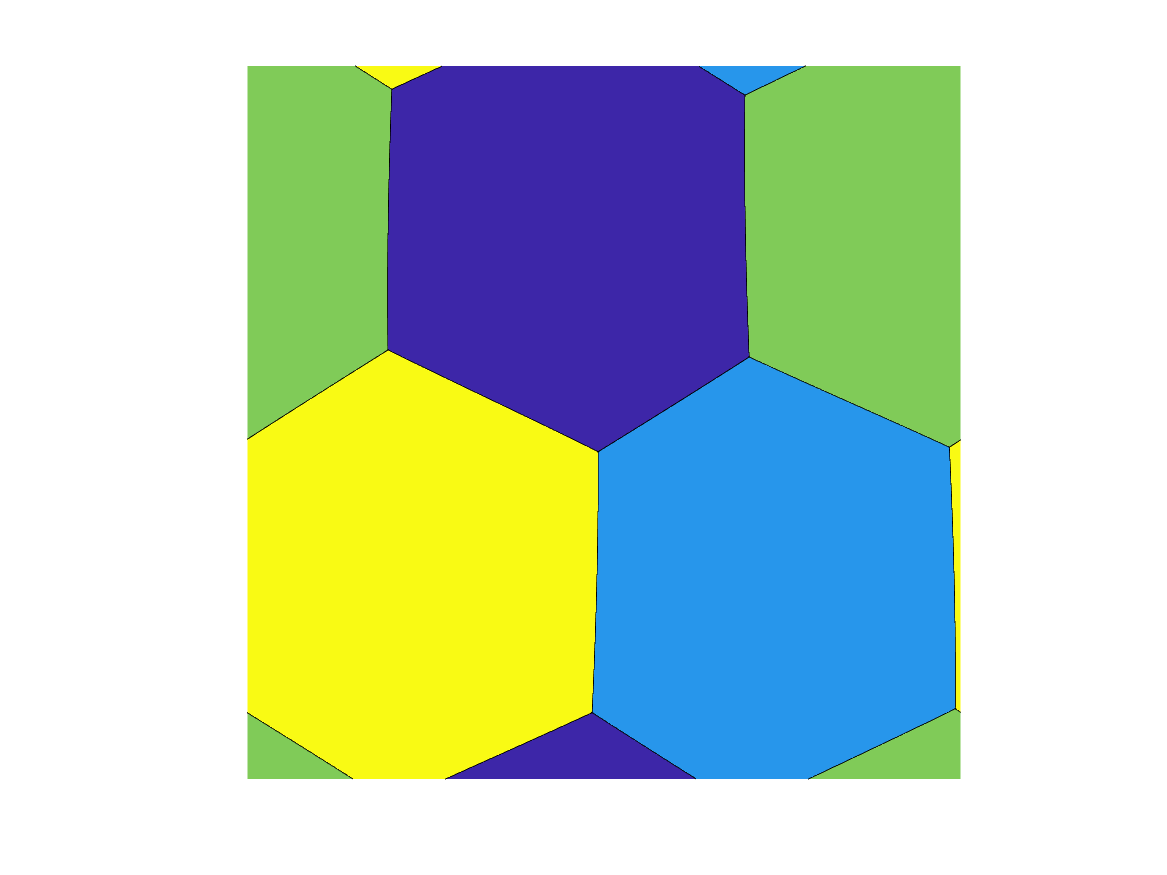}&
\includegraphics[width = 0.09\textwidth, clip, trim = 4cm 1cm 3cm 1cm]{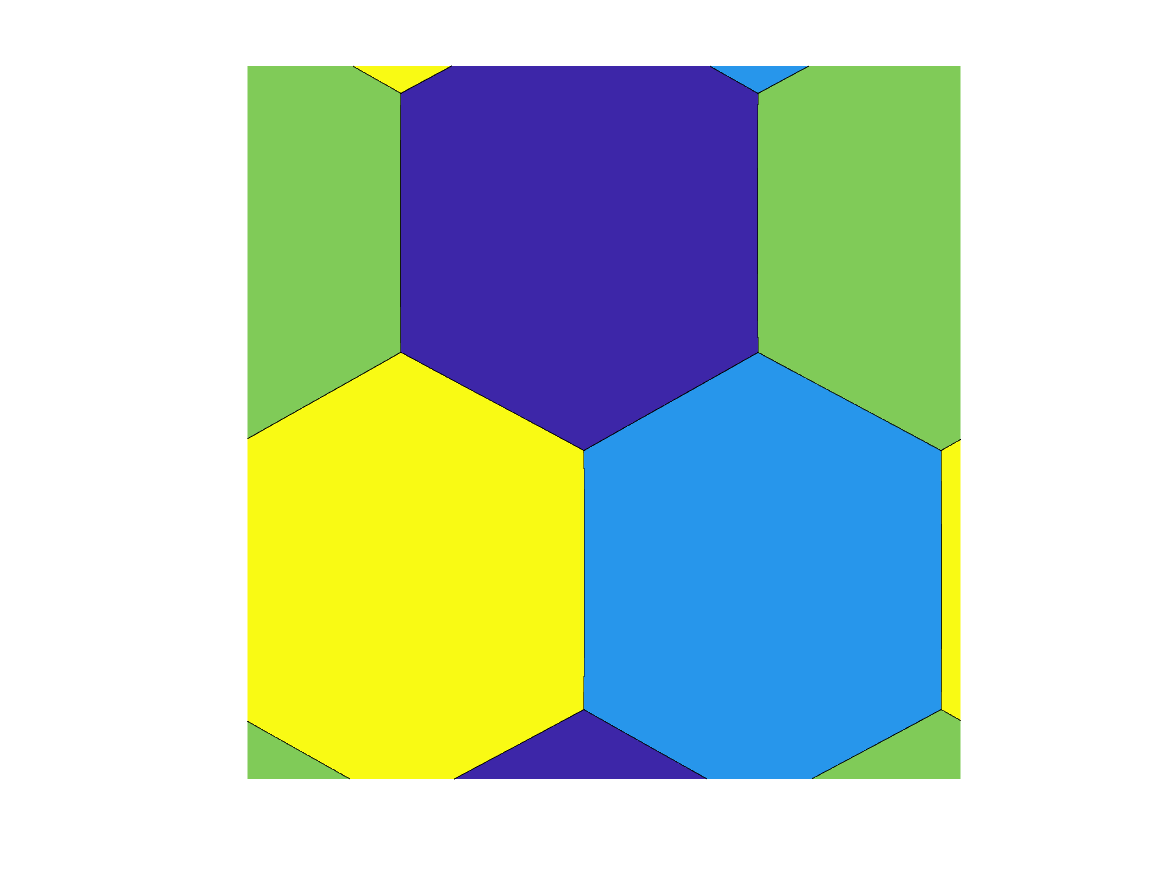}&
\includegraphics[width = 0.09\textwidth, clip, trim = 4cm 1cm 3cm 1cm]{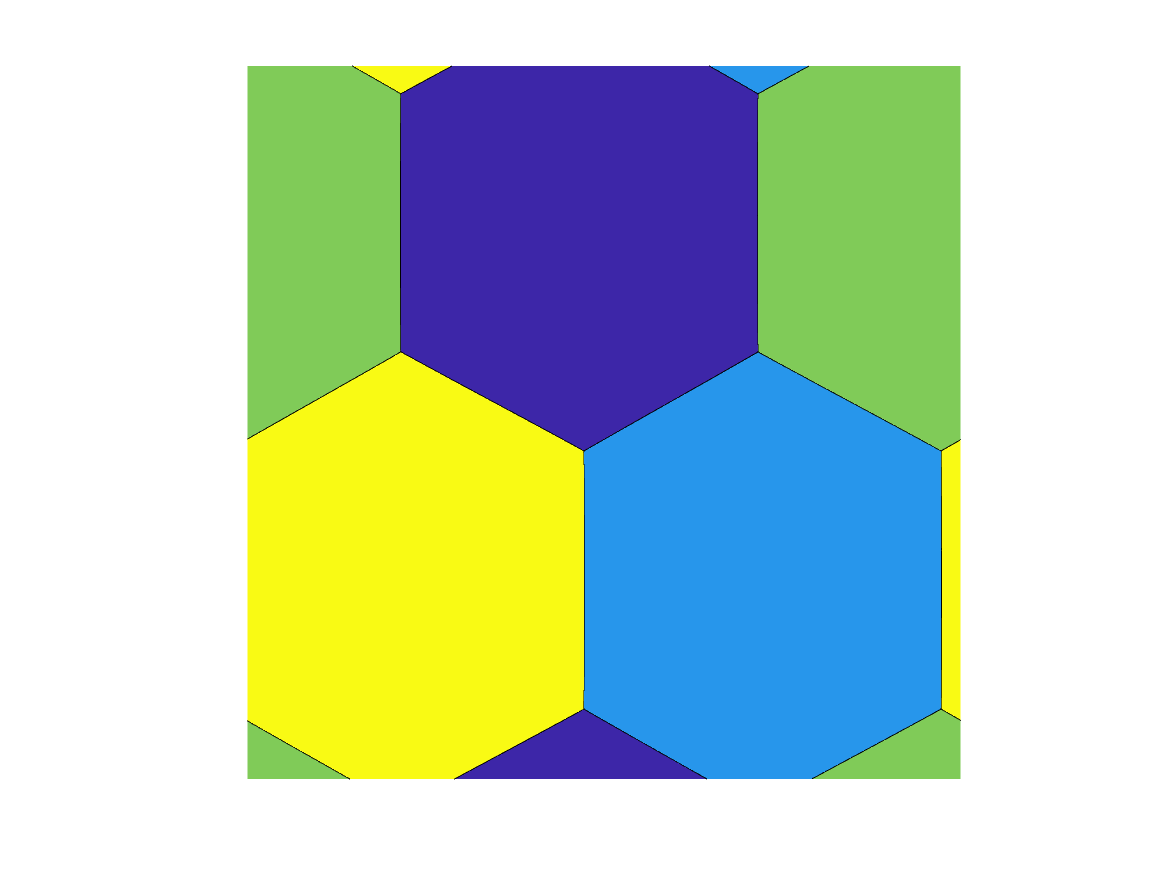}&
\includegraphics[width = 0.09\textwidth, clip, trim = 4cm 1cm 3cm 1cm]{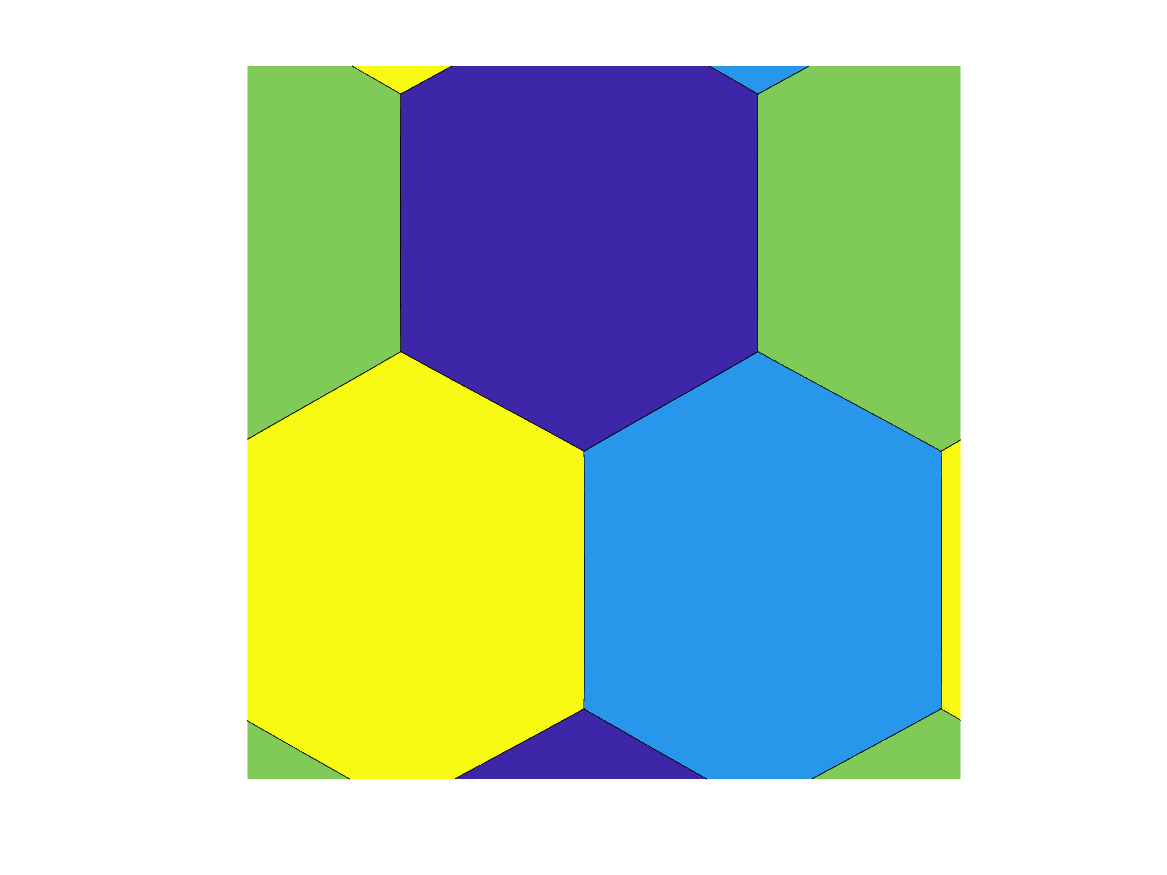}&
\includegraphics[width = 0.09\textwidth, clip, trim = 4cm 1cm 3cm 1cm]{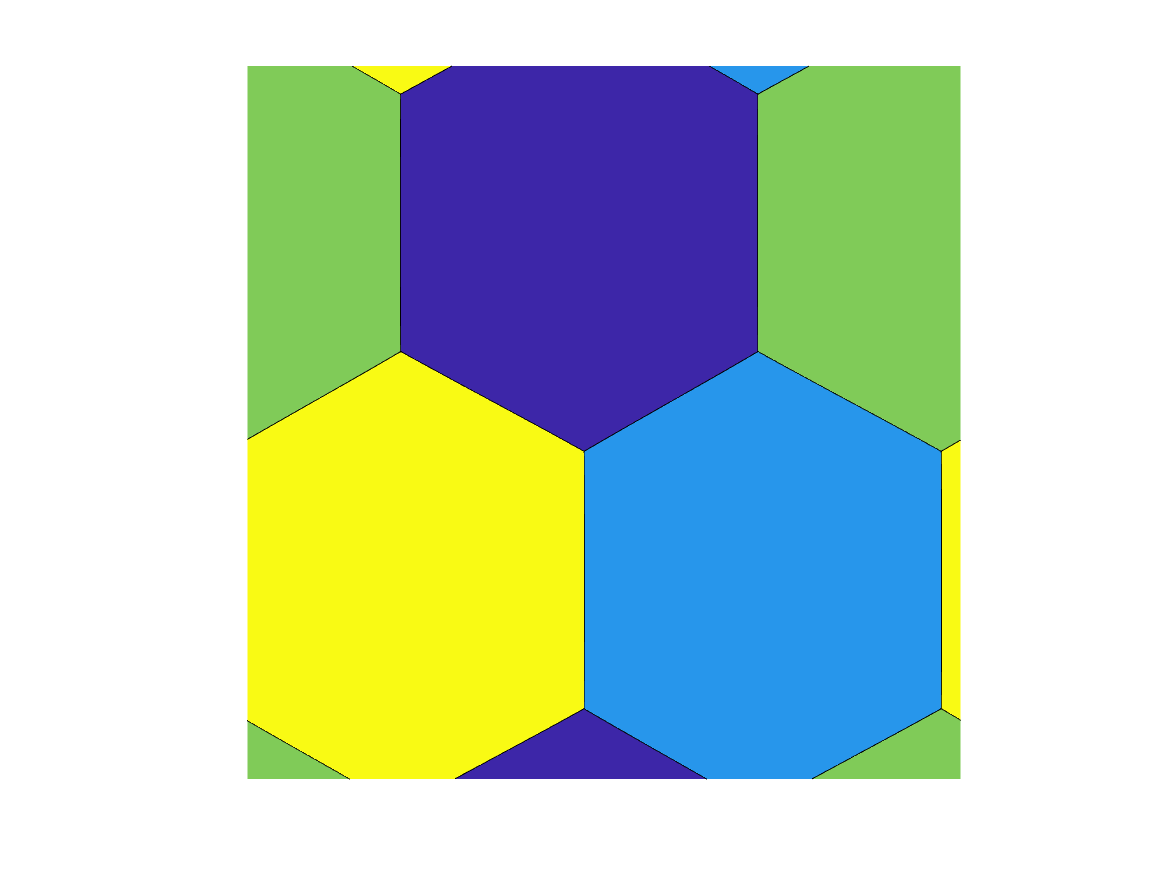}&
\includegraphics[width = 0.09\textwidth, clip, trim = 4cm 1cm 3cm 1cm]{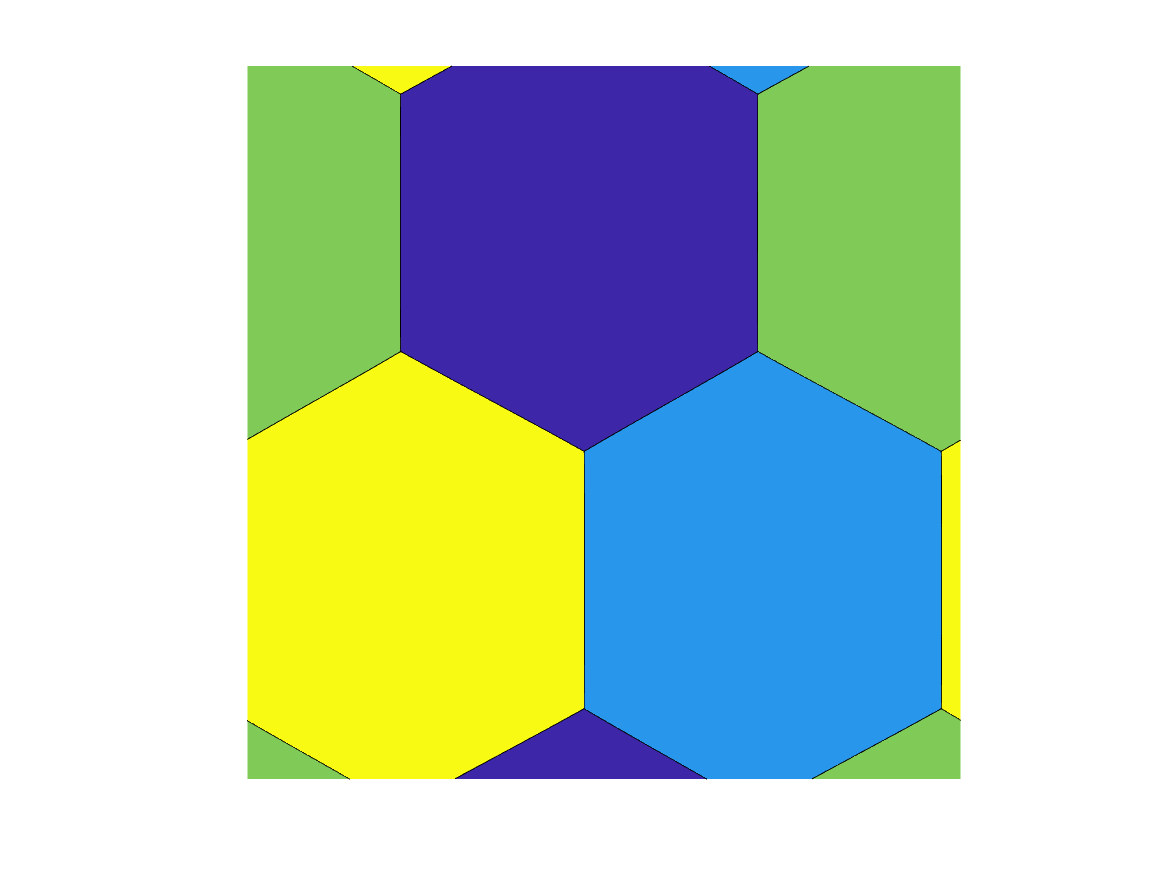} \\ 
\hline
\hline
 1 & 81 & 161  & 241 & 321  & 401& 481  &601 \\
\includegraphics[width = 0.09\textwidth, clip, trim = 4cm 1cm 3cm 1cm]{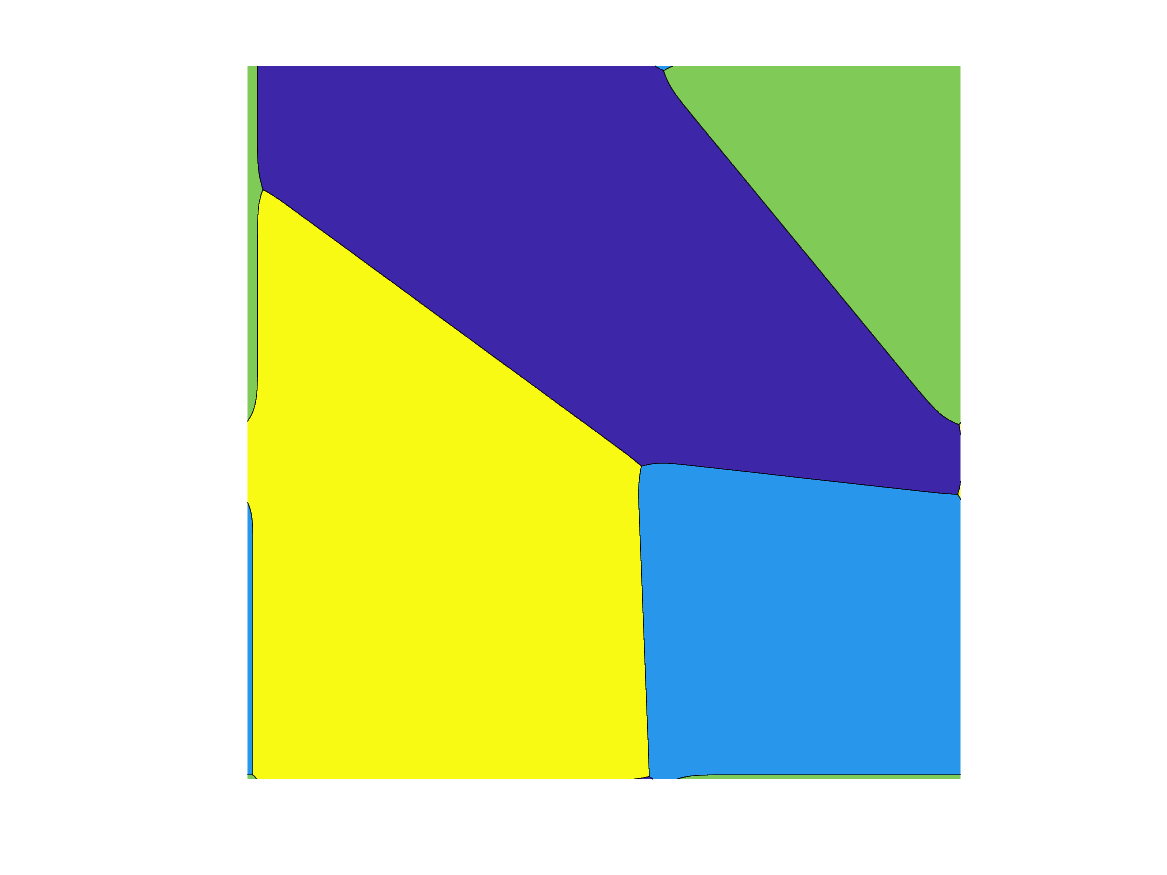}&  
\includegraphics[width = 0.09\textwidth, clip, trim = 4cm 1cm 3cm 1cm]{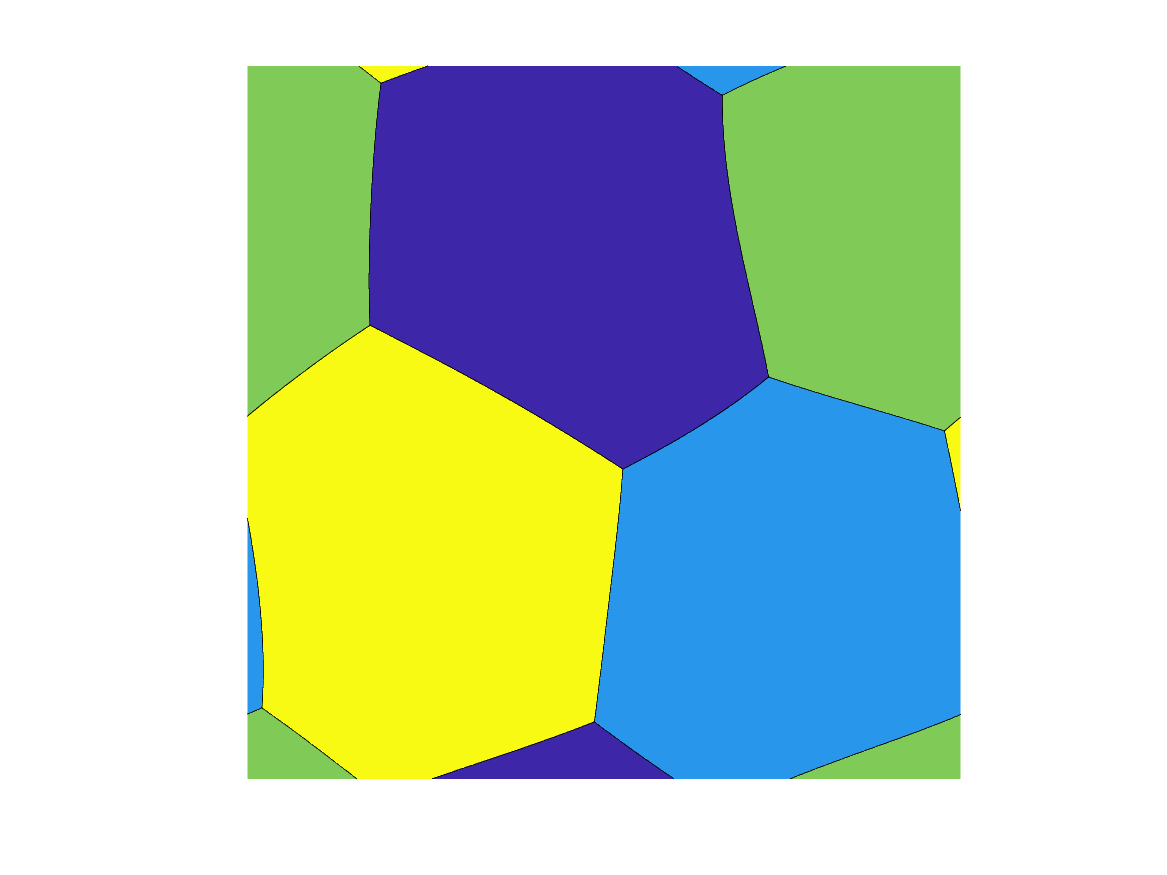}&
\includegraphics[width = 0.09\textwidth, clip, trim = 4cm 1cm 3cm 1cm]{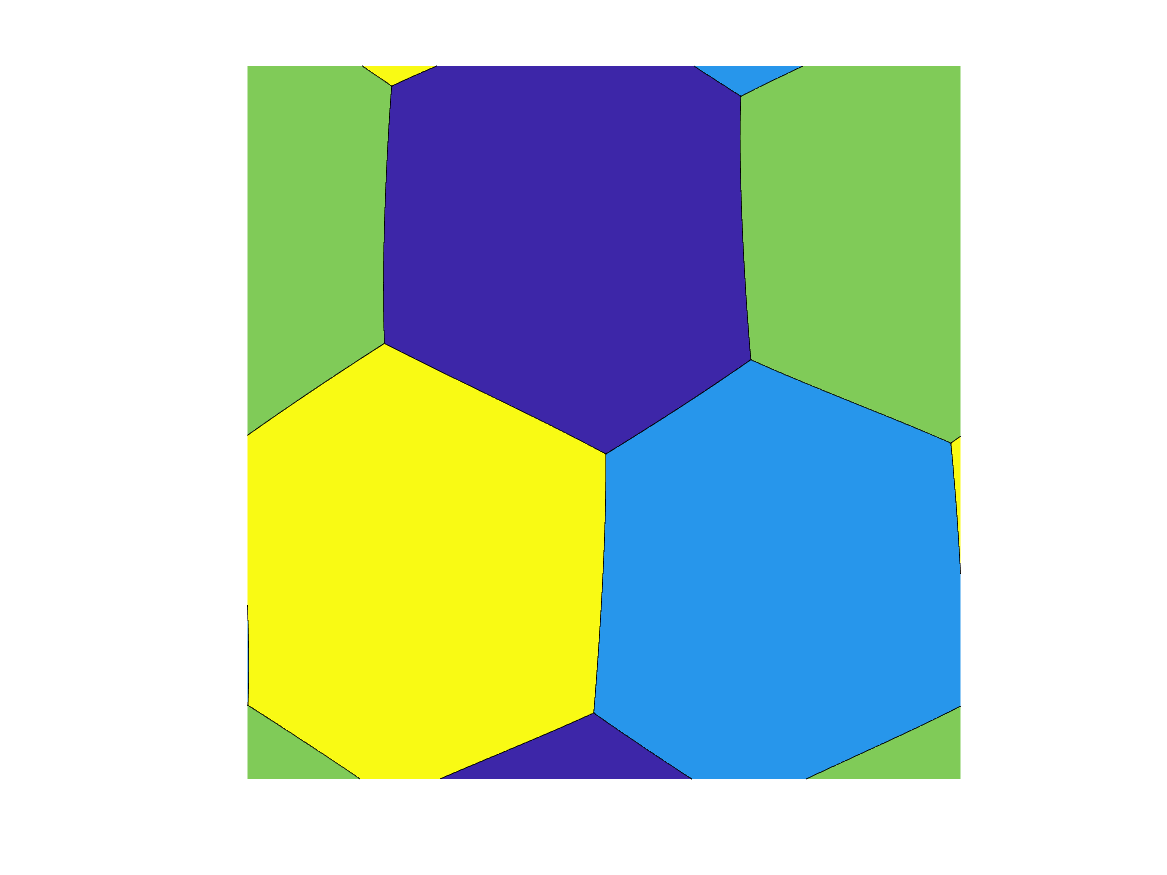}&
\includegraphics[width = 0.09\textwidth, clip, trim = 4cm 1cm 3cm 1cm]{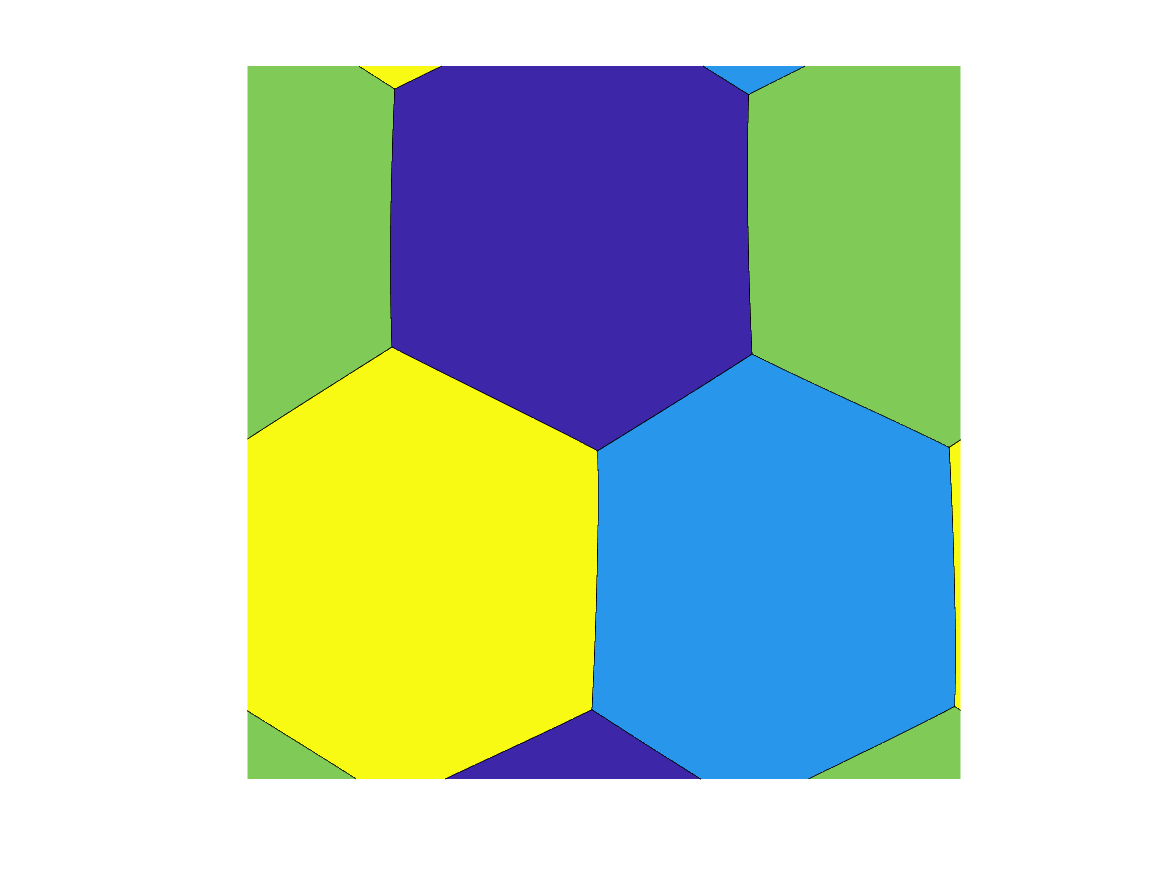}&
\includegraphics[width = 0.09\textwidth, clip, trim = 4cm 1cm 3cm 1cm]{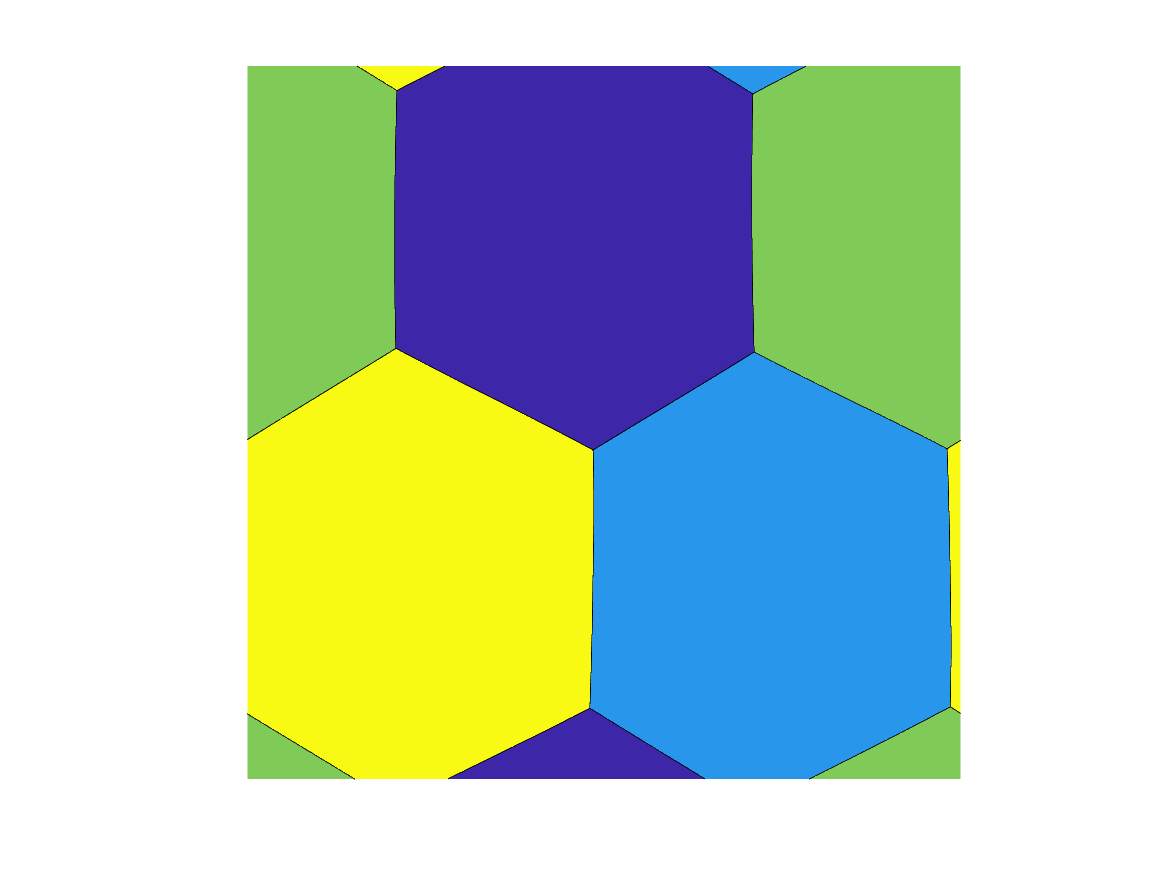}&
\includegraphics[width = 0.09\textwidth, clip, trim = 4cm 1cm 3cm 1cm]{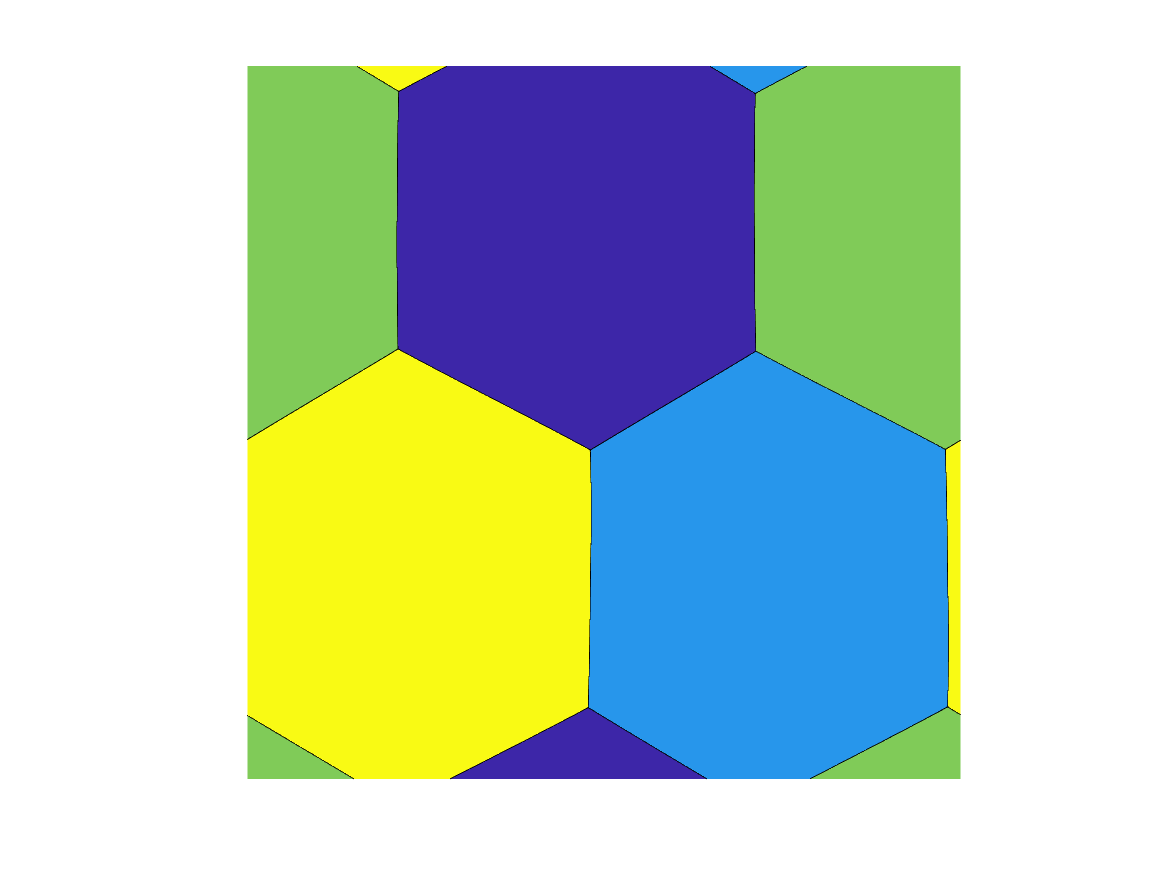}&
\includegraphics[width = 0.09\textwidth, clip, trim = 4cm 1cm 3cm 1cm]{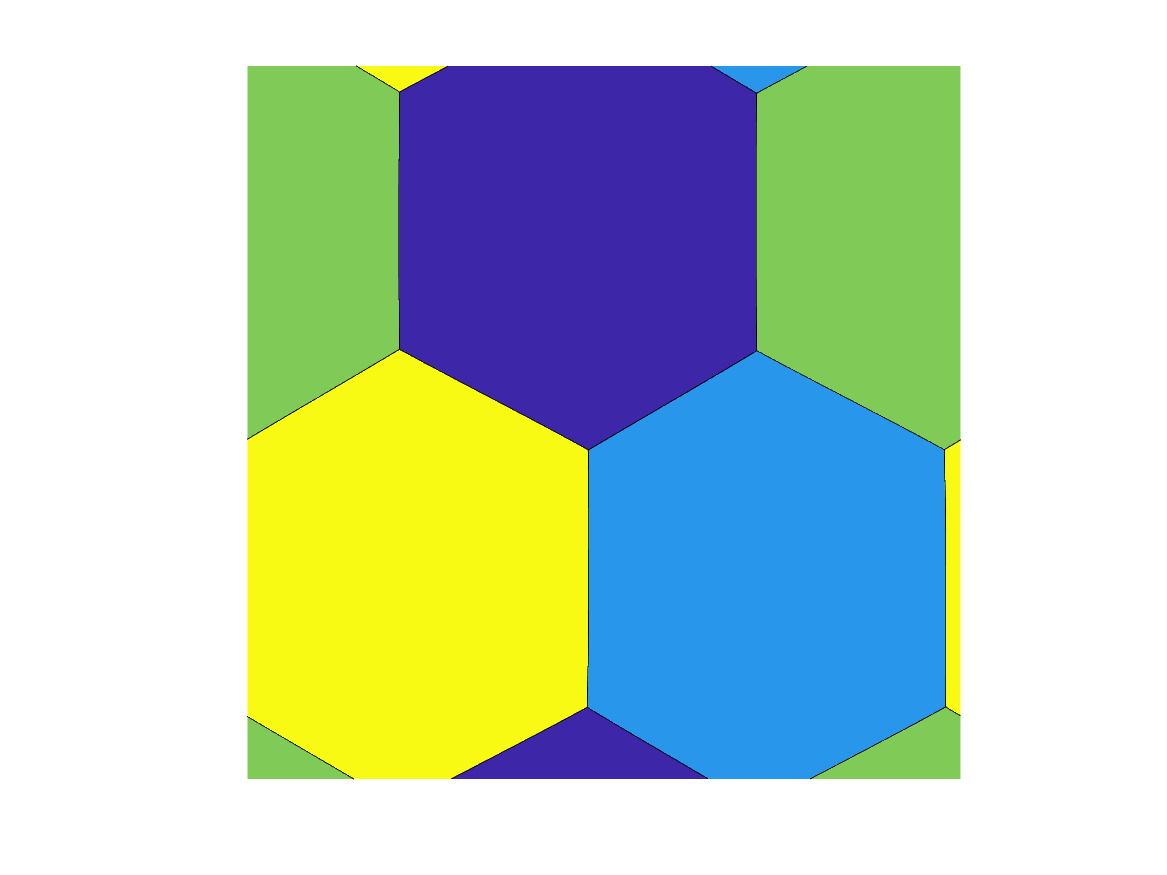}&
\includegraphics[width = 0.09\textwidth, clip, trim = 4cm 1cm 3cm 1cm]{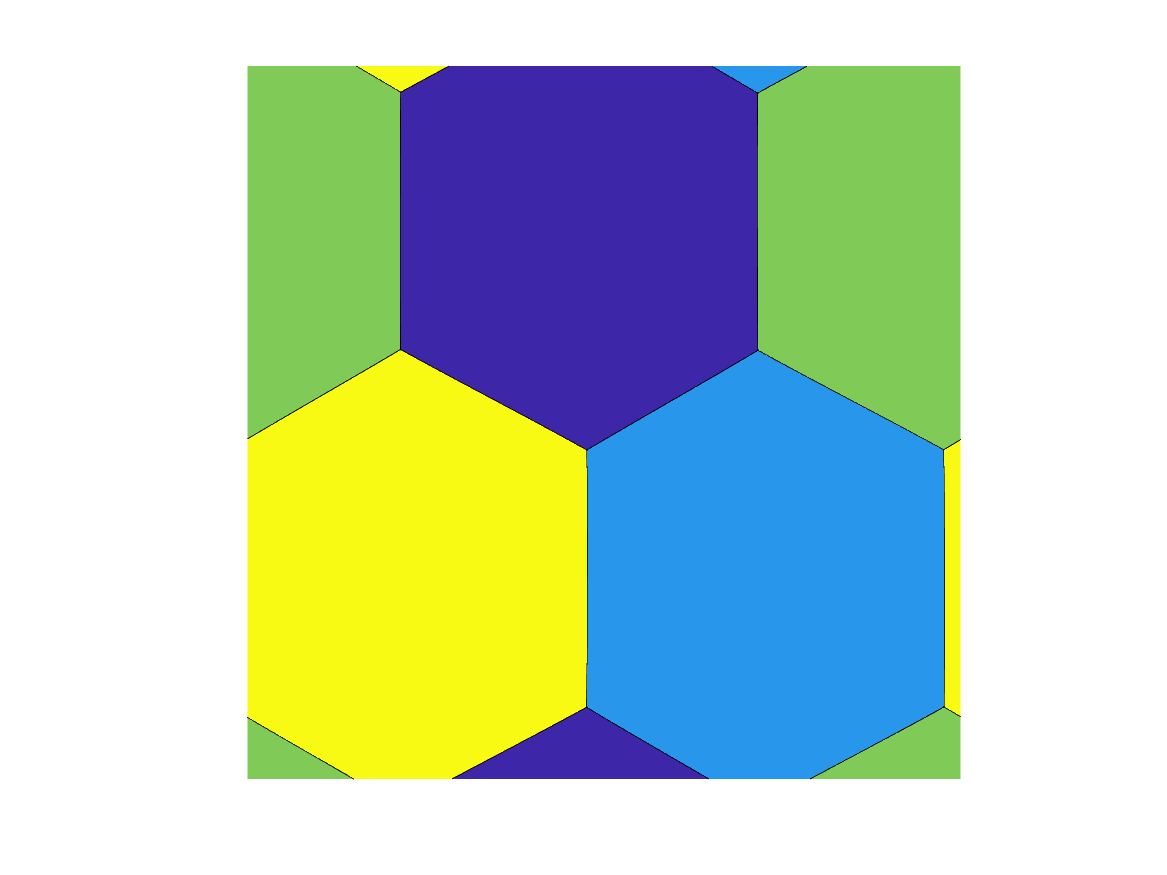} \\ 
\hline
\end{tabular}
\caption{First row: snapshots at different iterations in adaptive in time for $\tau$ changing in $[1/4,1/8,1/16,1/32,1/128]$, second row: snapshots at different iterations for a fixed $\tau=1/64$. See Section~\ref{sec:adptive}.} \label{fig:adaptiveintime}
\end{figure}

\subsection{$k$-partition for a 2-dimensional periodic domain}\label{sec:2dperiodic}
In this section, we simply apply Algorithm~\ref{alg3} on the calculation of $k$-partitions for periodic domains in 2-dimensional spaces. In Figure~\ref{fig:2dperiodic}, we list the solution of $k$-partition for $k = 4-12$, $14-16$, $18$, $20$, $23-25$, $28$, $30$ and $36$. In all results, we discretize the domain with $256^2$ grid points, set an initial $\tau = 1/4$ and $tol_\tau = 1/128$, and start with random initial guesses. From our experimental observation, all experiments converge in fewer than about 2-3 hundreds steps and take about $2$-$75$ {\it seconds} CPU time in average. Here, the average CPU time is the average CPU time of $10$ experiments with individually independent random initial guesses for a fixed $k$. More precisely, the $4$-partition case only takes $2$ seconds and even the $36$-partition computation only takes $75$ {\it seconds}. All reported results are consistent with the results in \cite{Wang_2019}. Besides, we observe that for most $k$, especially when $k$ is large, we get hexagon structures. This tessellation for Dirichlet partition is also consistent with the conjecture proposed in \cite{caffarelli2007optimal}. For k = 5,7,10, irregular structures are observed and periodic extensions are plotted in Figure~\ref{fig:extension}.

\begin{figure}[ht]
\centering
\includegraphics[width = 0.19\textwidth, clip, trim = 4cm 1cm 3cm 1cm]{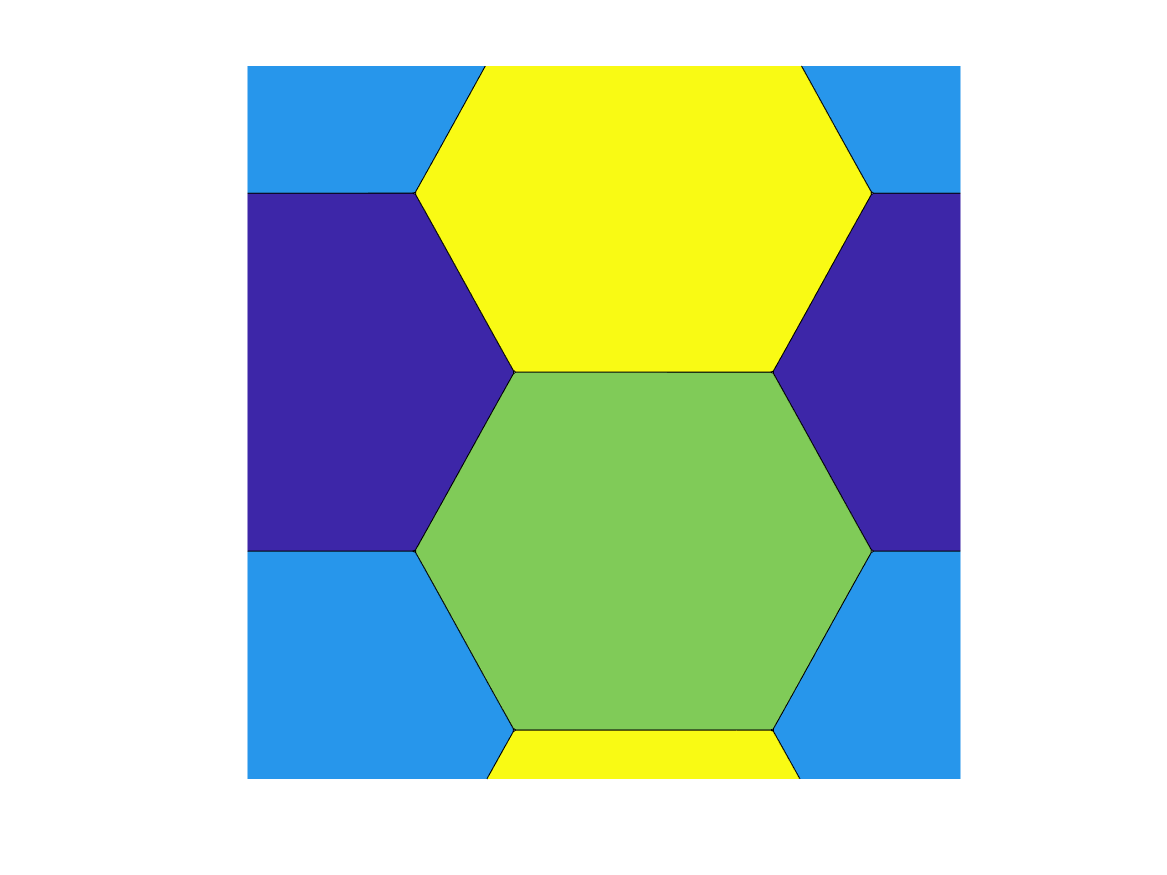}
\includegraphics[width = 0.19\textwidth, clip, trim = 4cm 1cm 3cm 1cm]{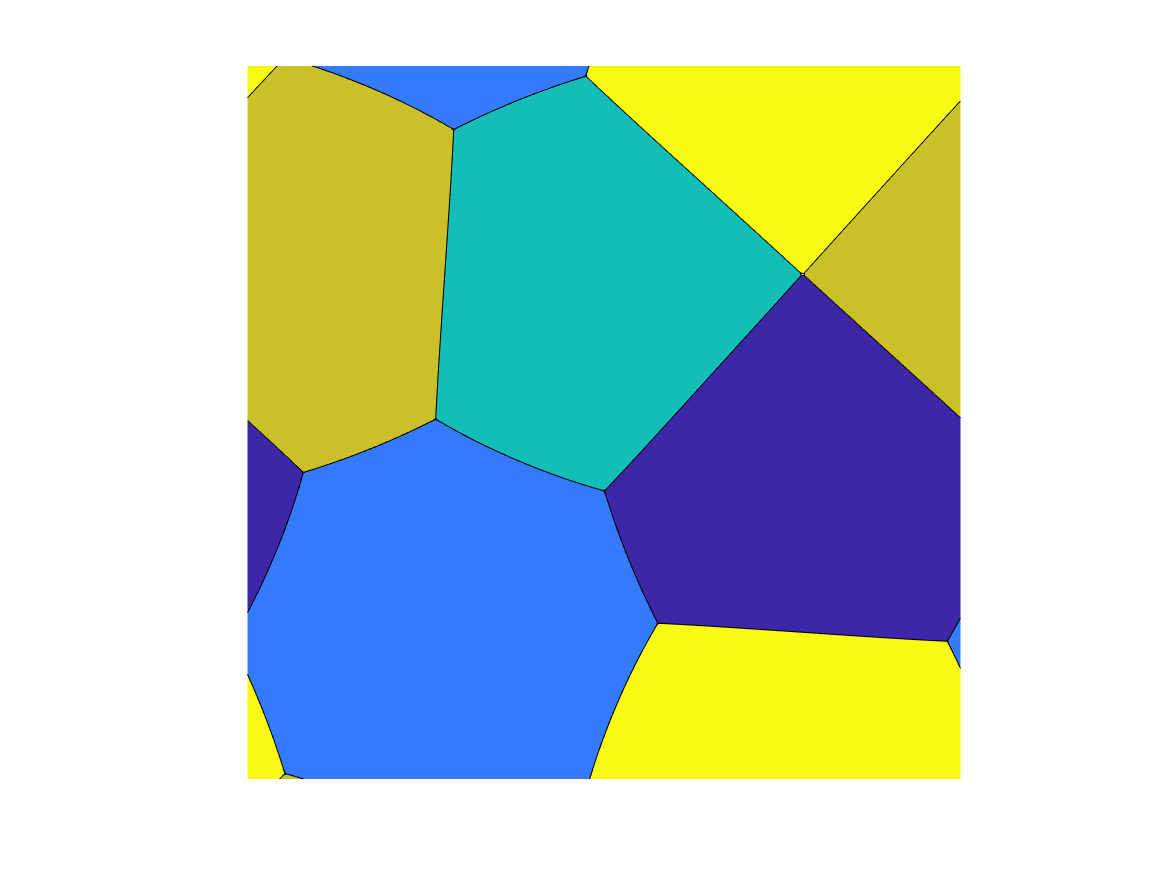}
\includegraphics[width = 0.19\textwidth, clip, trim = 4cm 1cm 3cm 1cm]{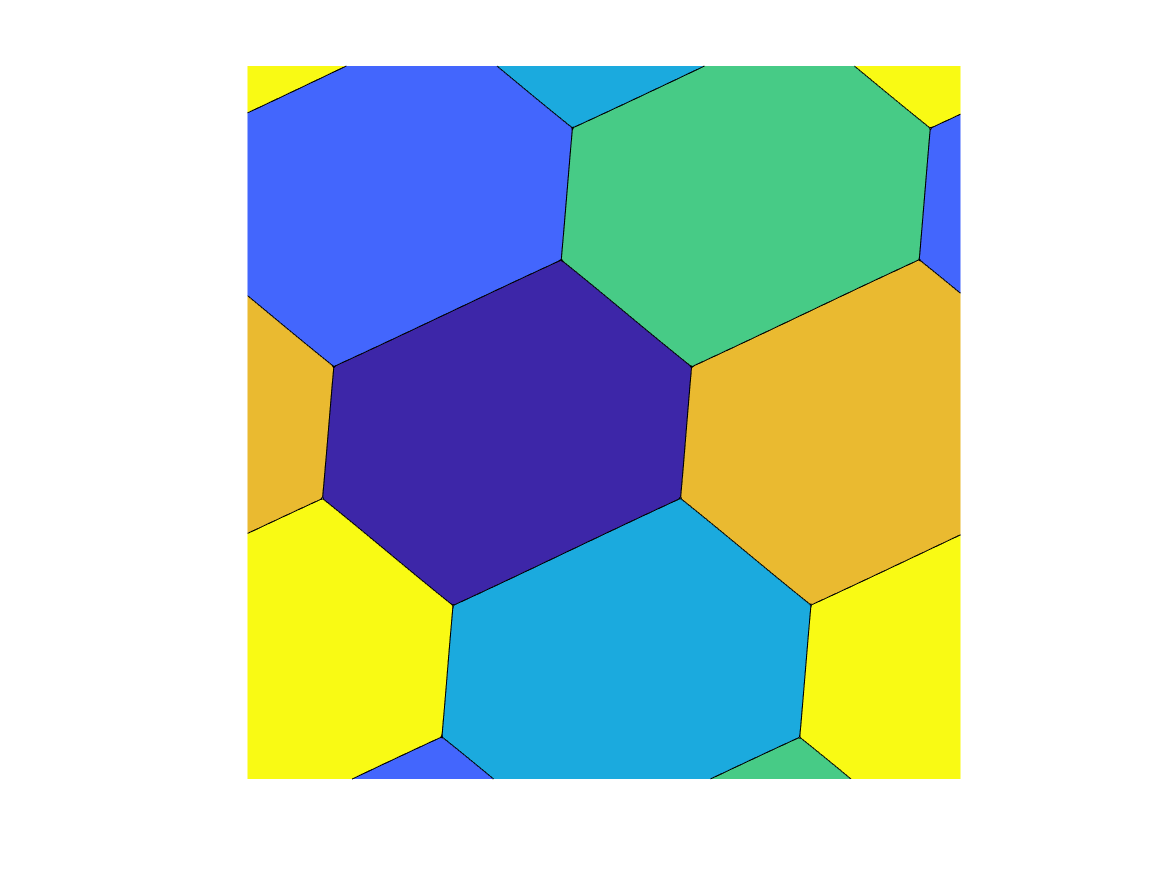}
\includegraphics[width = 0.19\textwidth, clip, trim = 4cm 1cm 3cm 1cm]{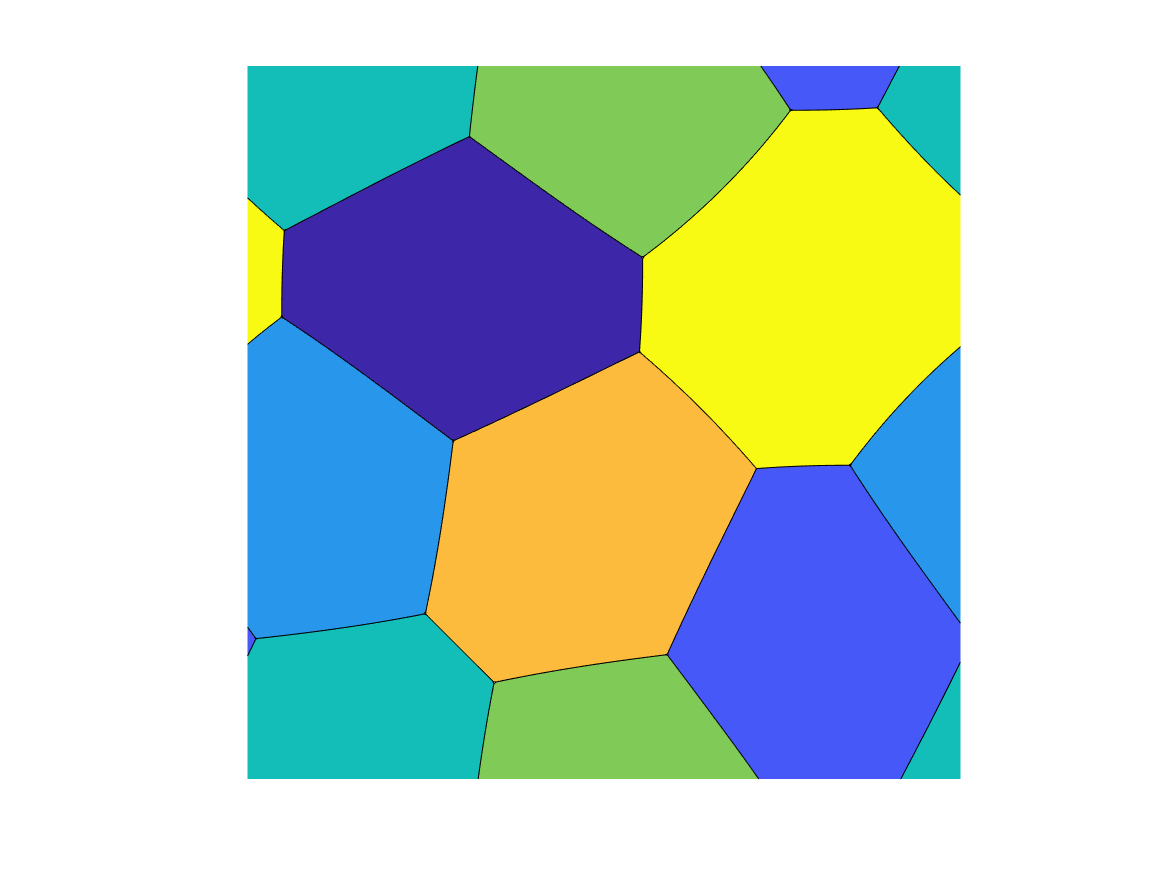}
\includegraphics[width = 0.19\textwidth, clip, trim = 4cm 1cm 3cm 1cm]{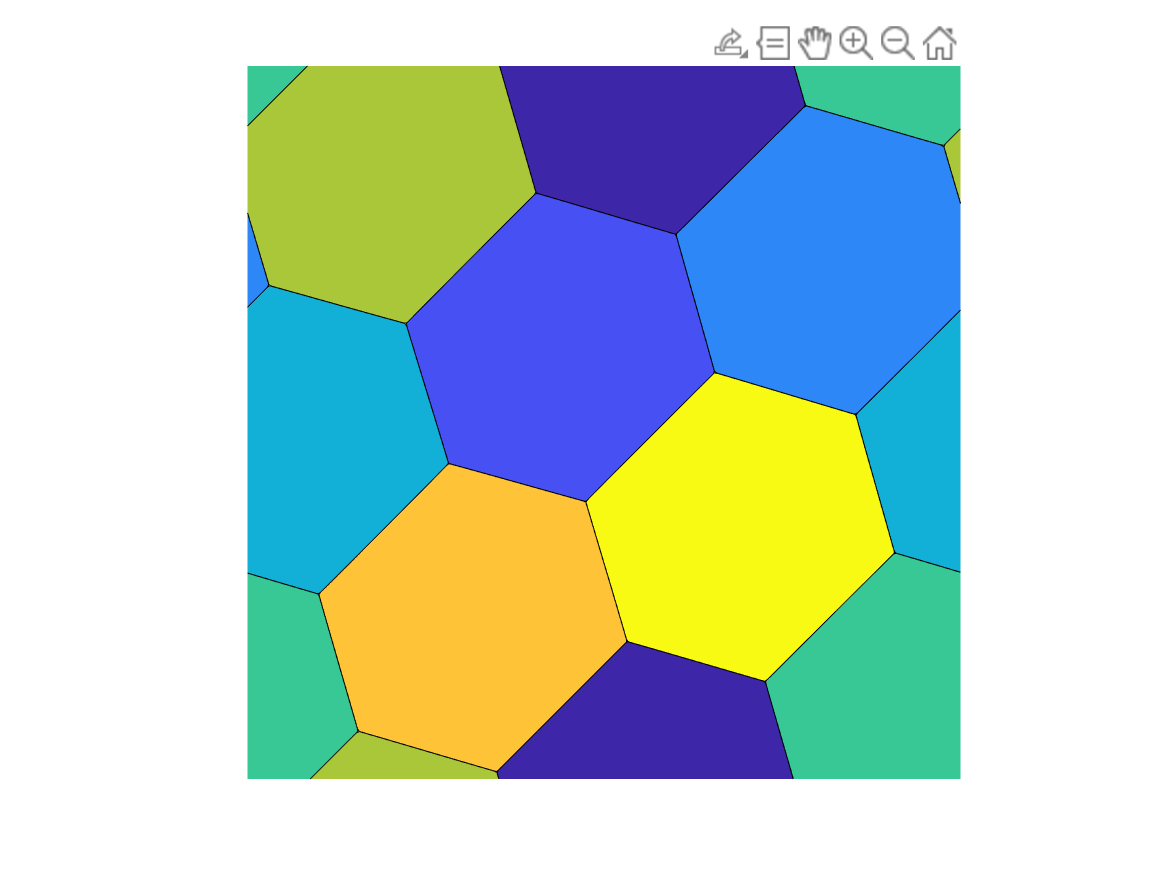}
\includegraphics[width = 0.19\textwidth, clip, trim = 4cm 1cm 3cm 1cm]{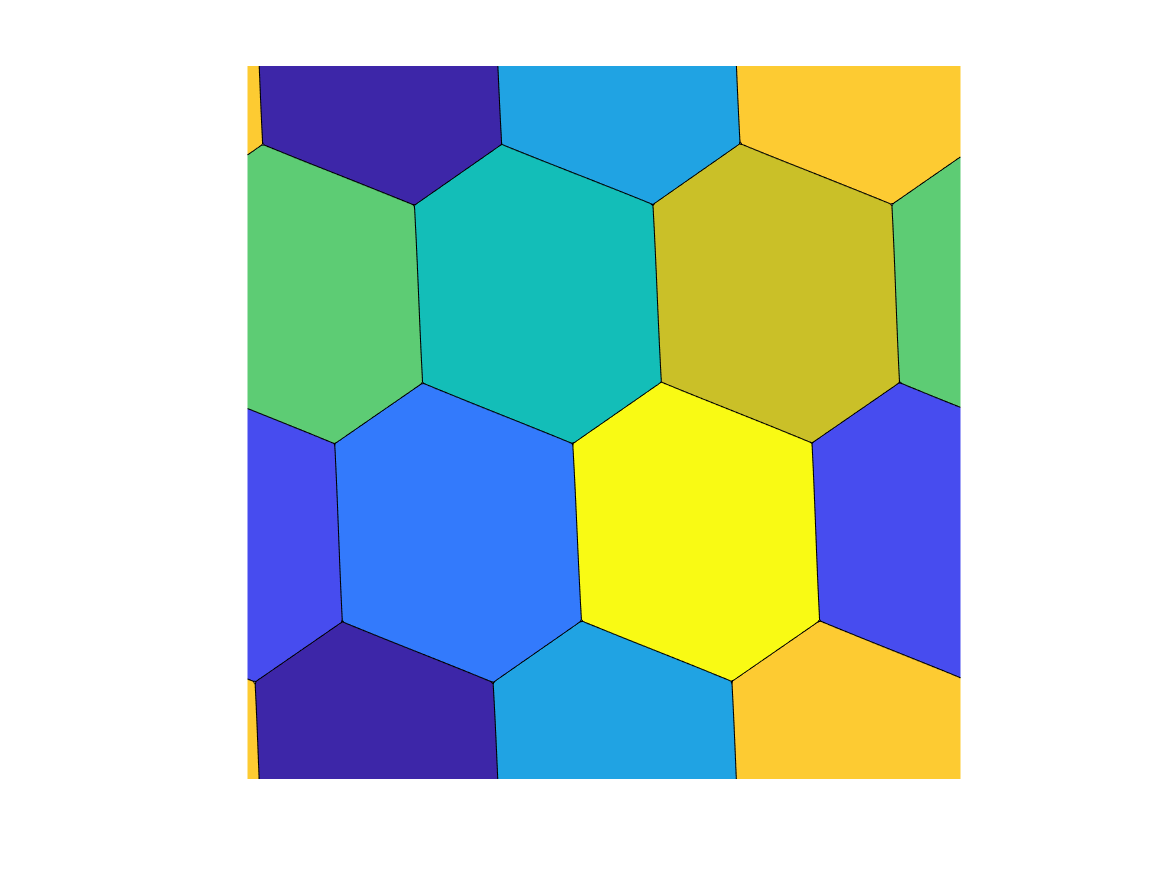}
\includegraphics[width = 0.19\textwidth, clip, trim = 4cm 1cm 3cm 1cm]{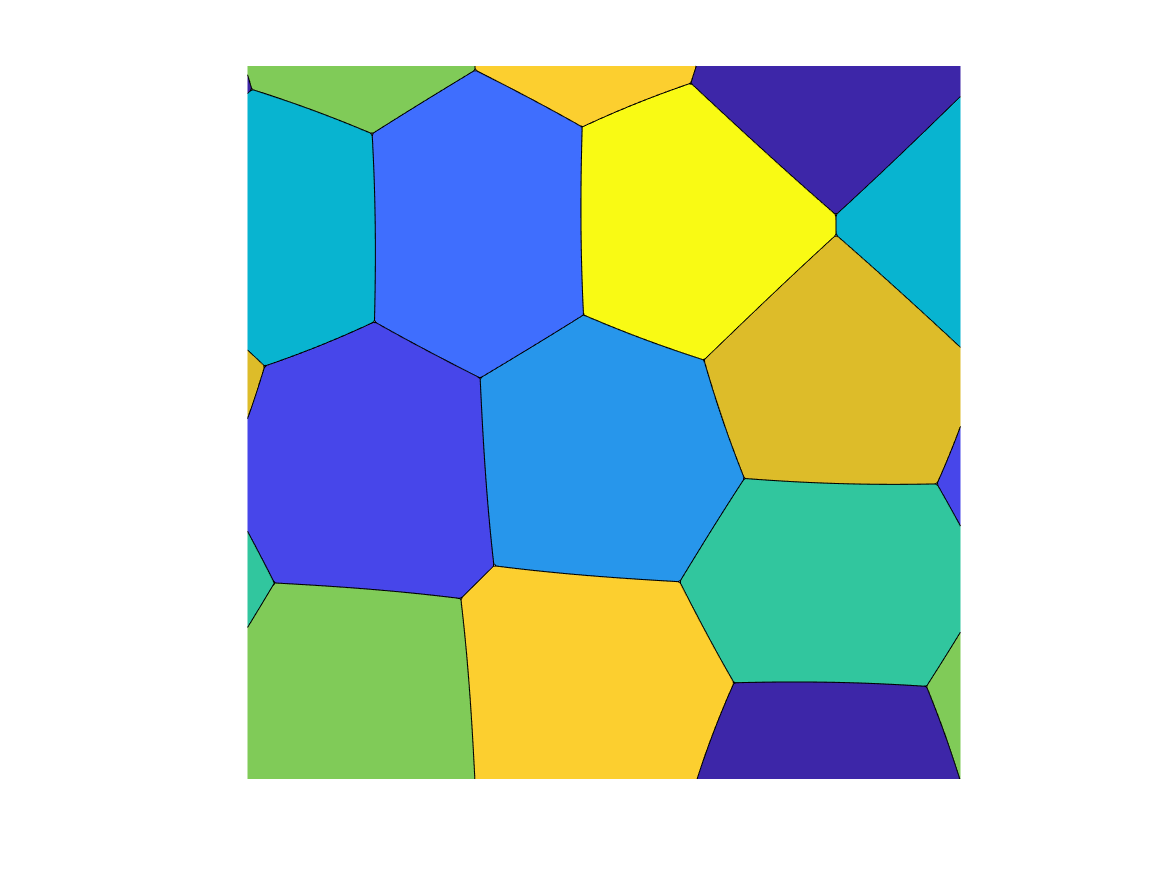}
\includegraphics[width = 0.19\textwidth, clip, trim = 4cm 1cm 3cm 1cm]{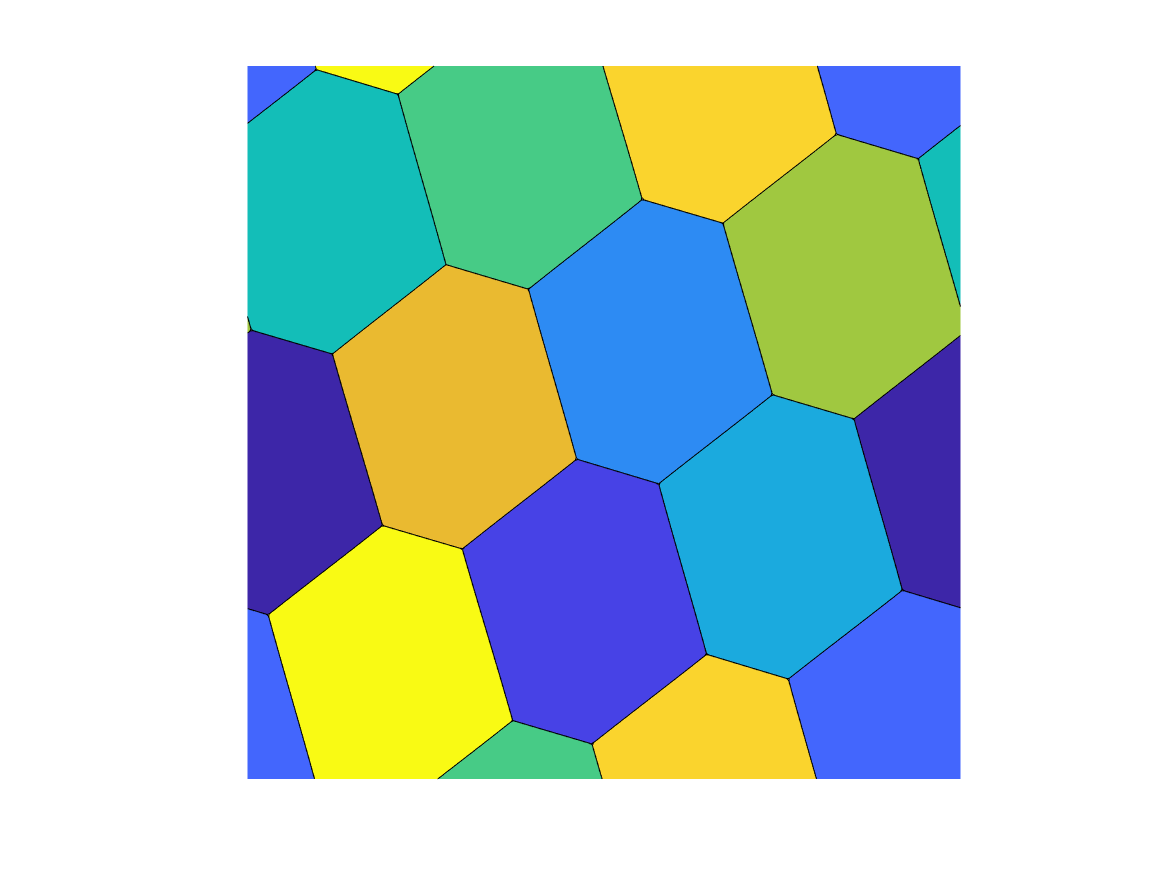}
\includegraphics[width = 0.19\textwidth, clip, trim = 4cm 1cm 3cm 1cm]{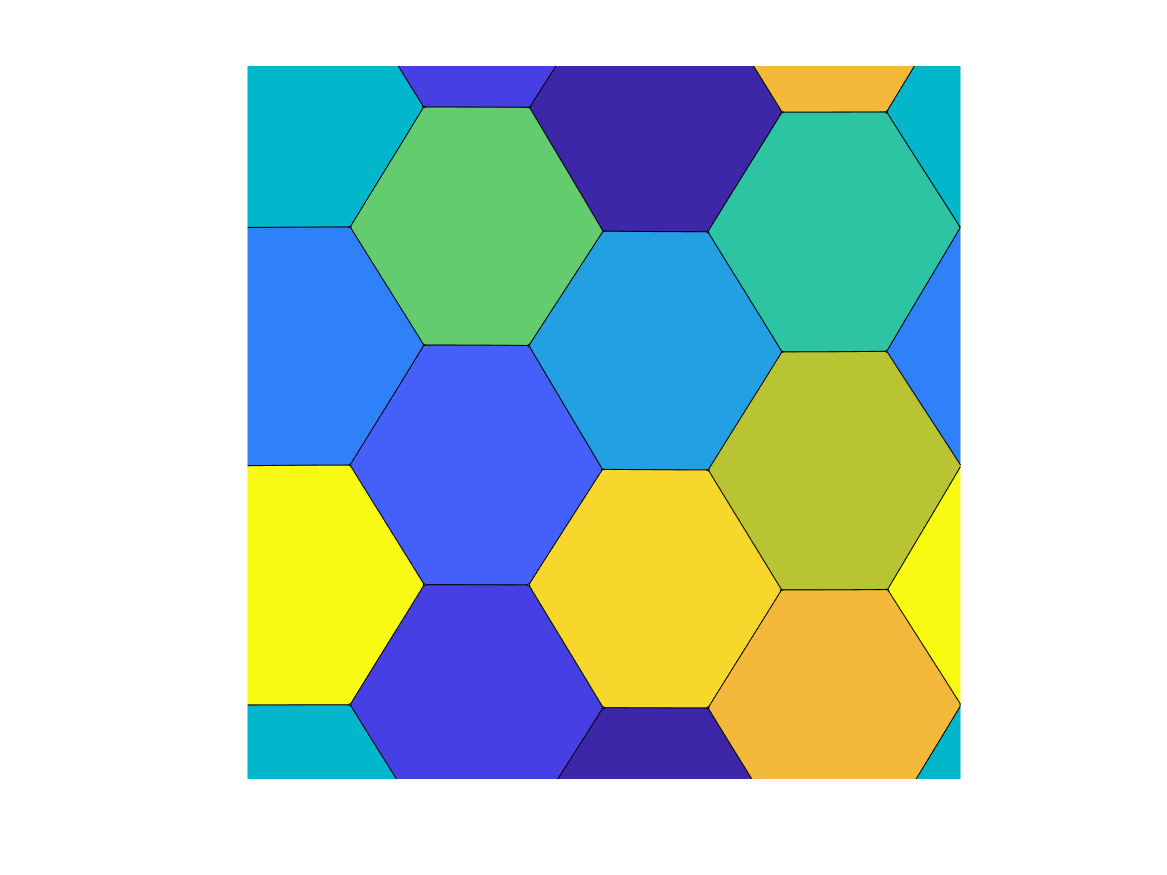}
\includegraphics[width = 0.19\textwidth, clip, trim = 4cm 1cm 3cm 1cm]{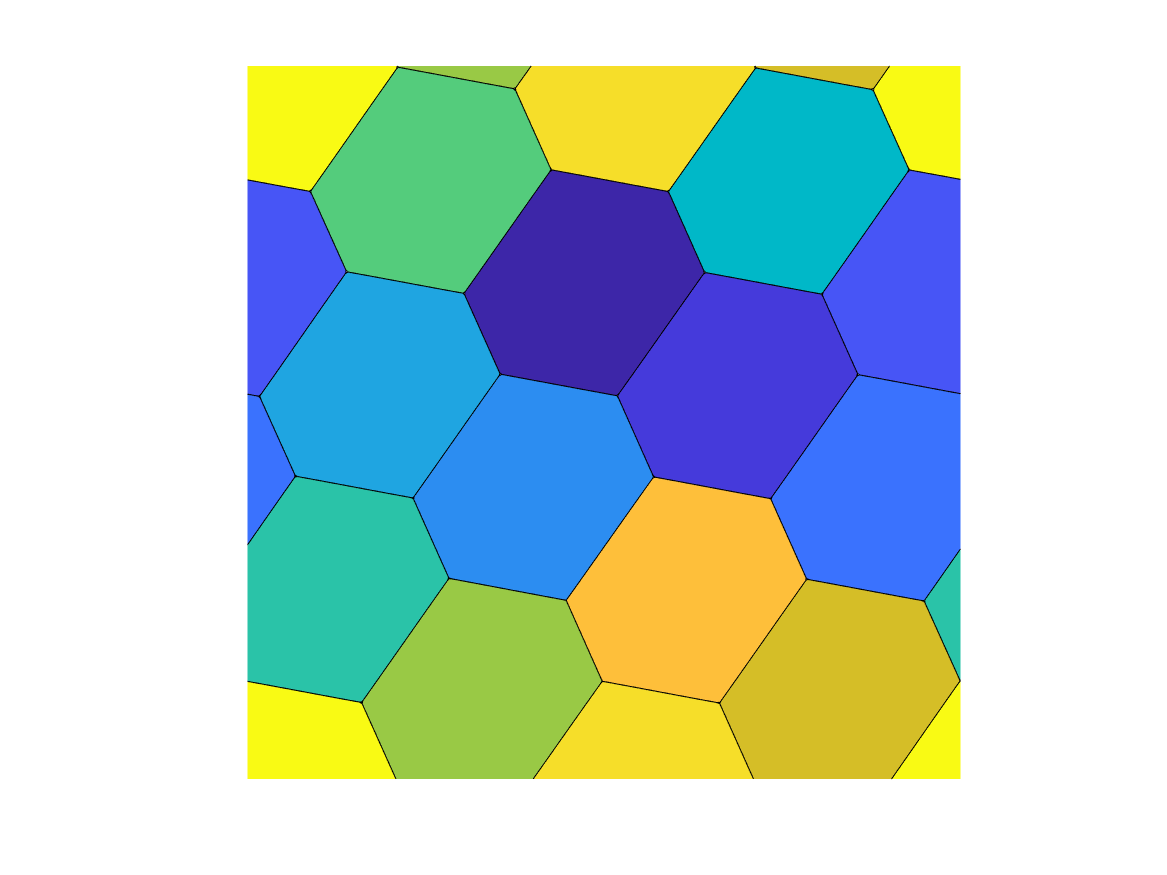}
\includegraphics[width = 0.19\textwidth, clip, trim = 4cm 1cm 3cm 1cm]{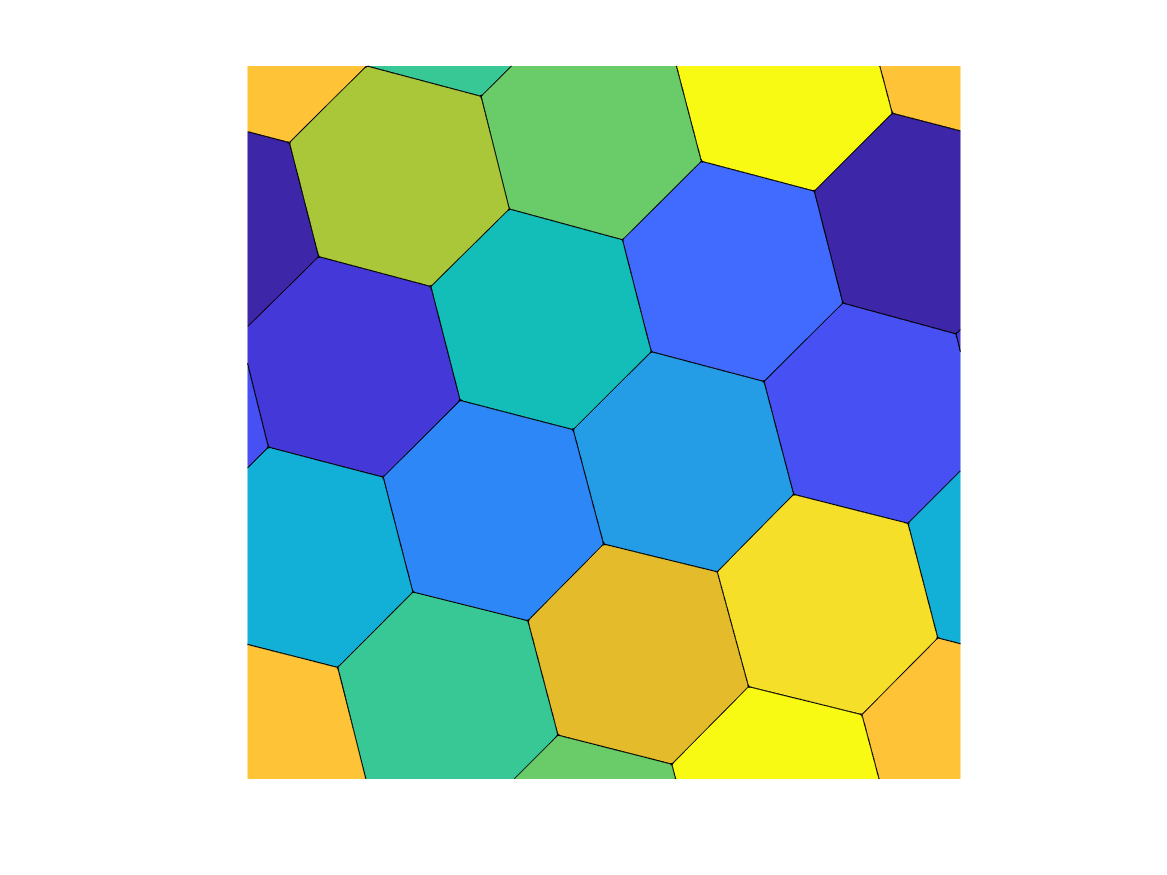}
\includegraphics[width = 0.19\textwidth, clip, trim = 4cm 1cm 3cm 1cm]{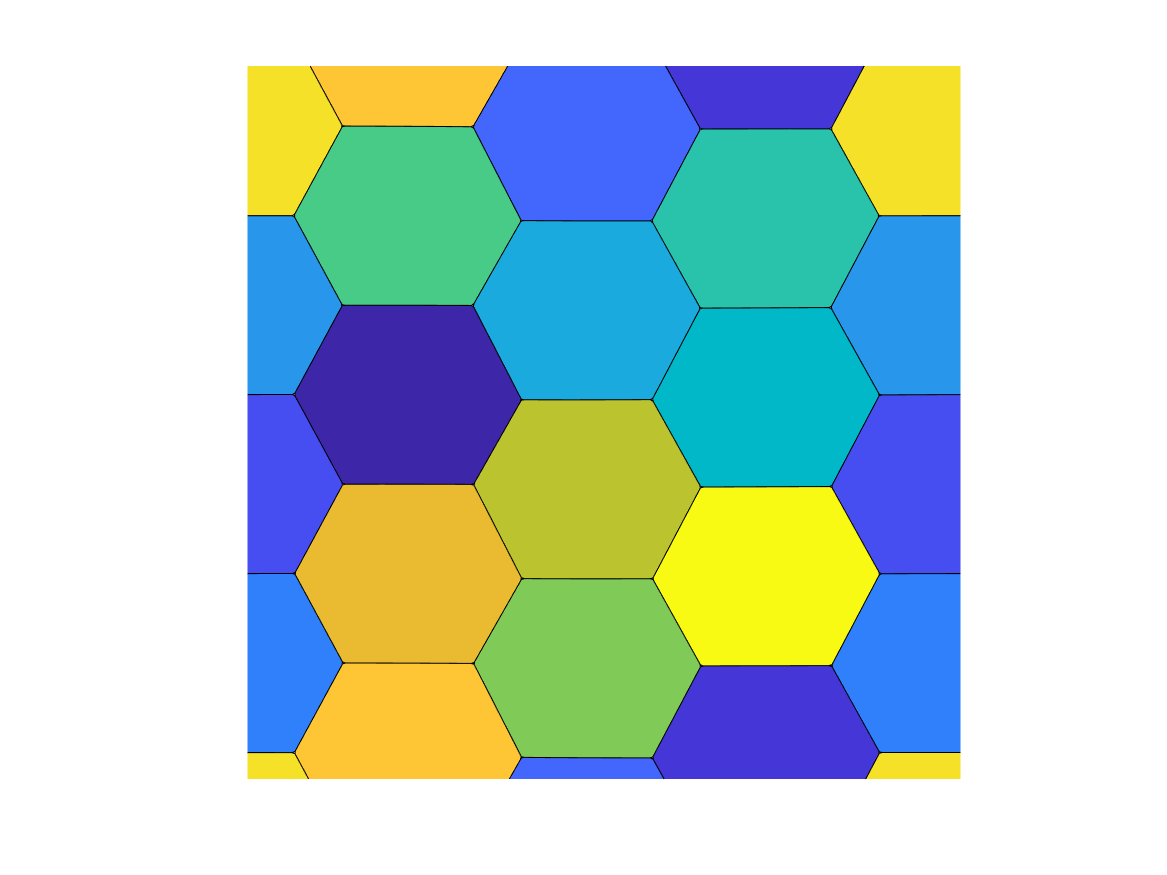}
\includegraphics[width = 0.19\textwidth, clip, trim = 4cm 1cm 3cm 1cm]{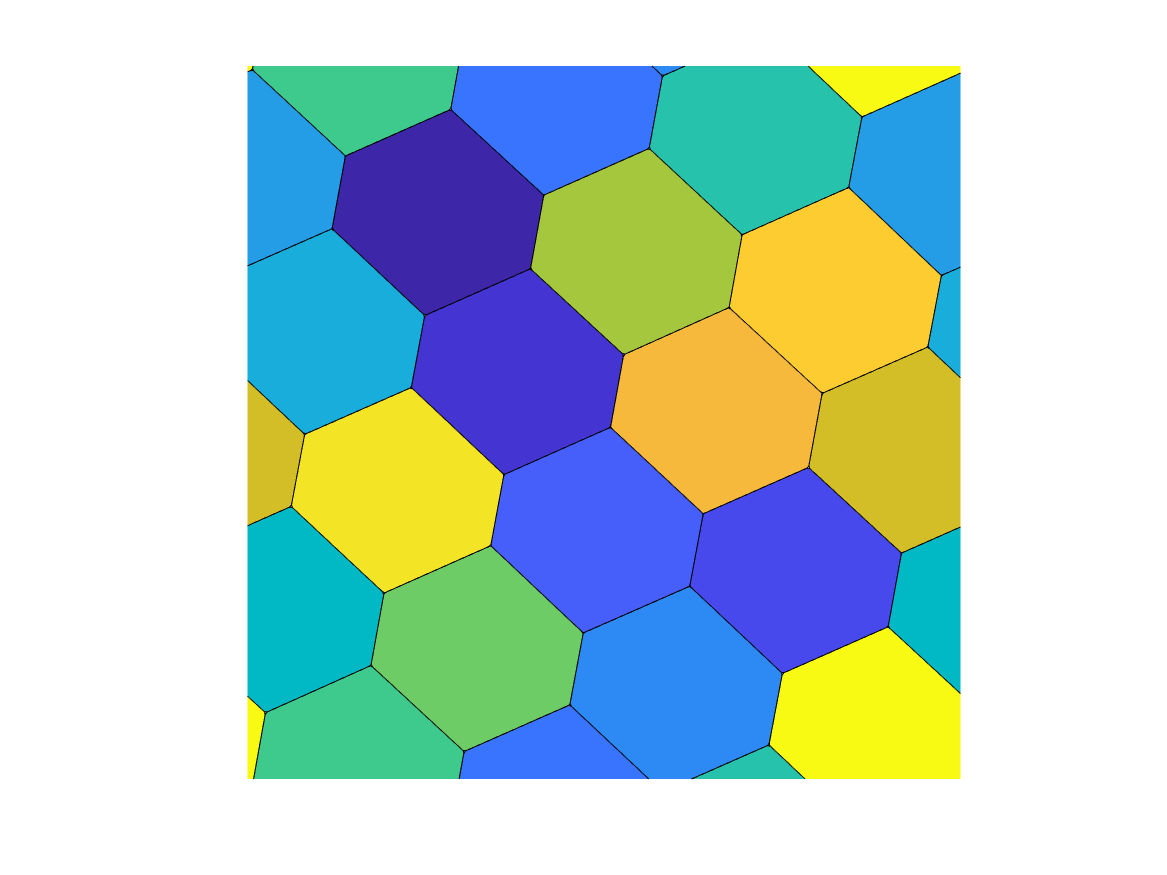}
\includegraphics[width = 0.19\textwidth, clip, trim = 4cm 1cm 3cm 1cm]{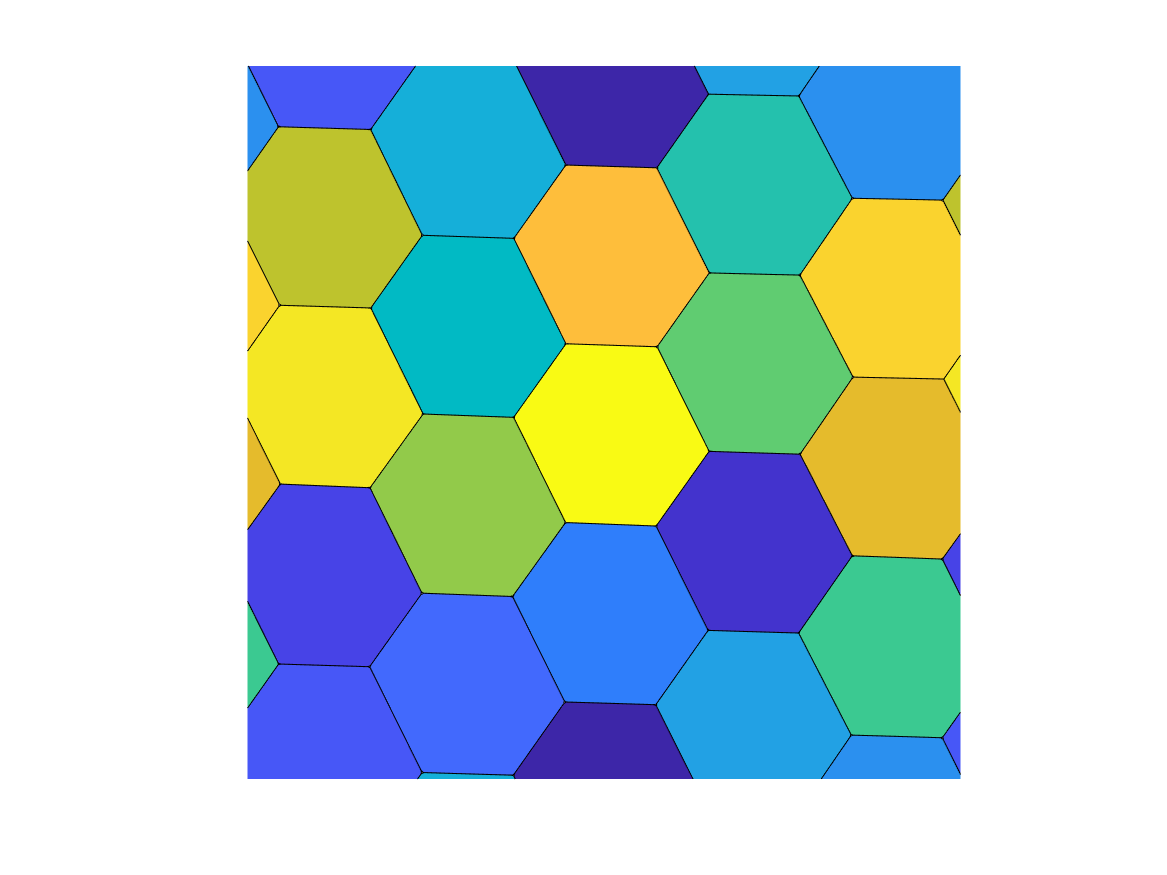}
\includegraphics[width = 0.19\textwidth, clip, trim = 4cm 1cm 3cm 1cm]{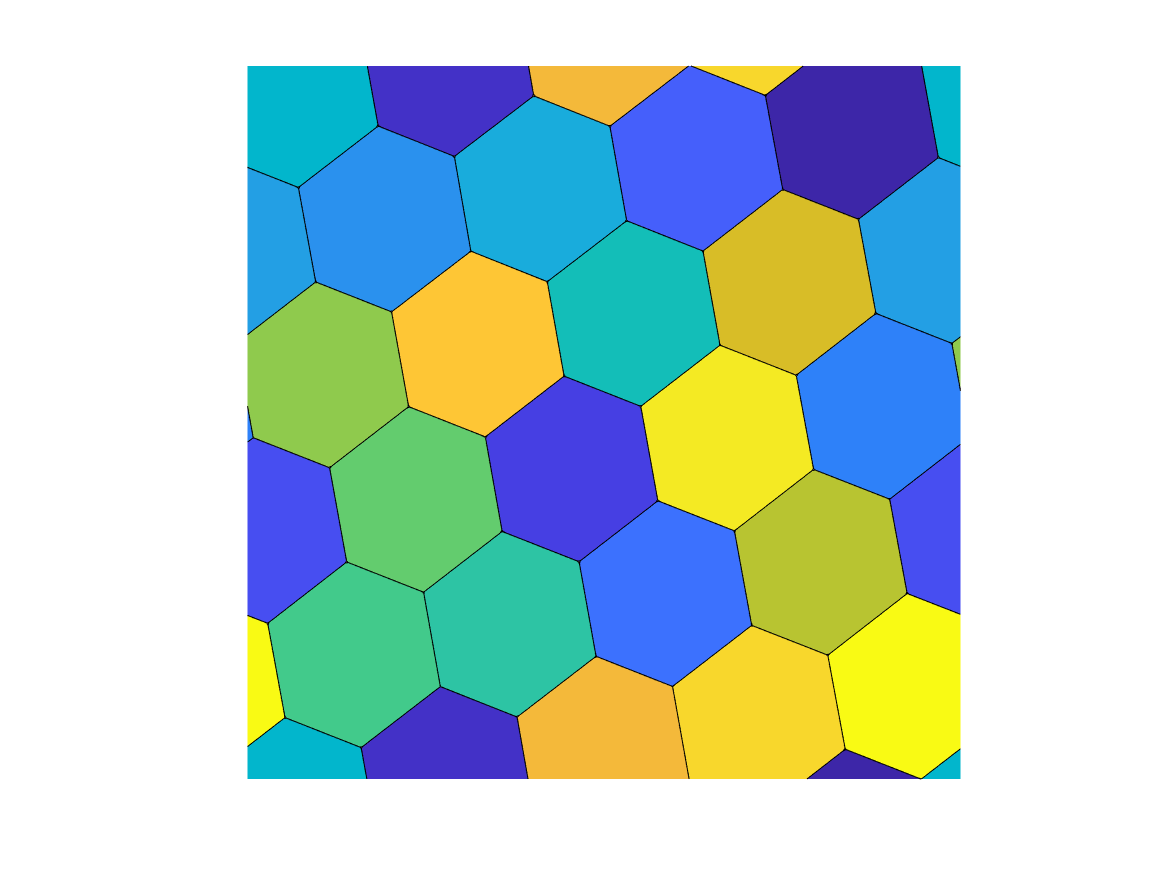}
\includegraphics[width = 0.19\textwidth, clip, trim = 4cm 1cm 3cm 1cm]{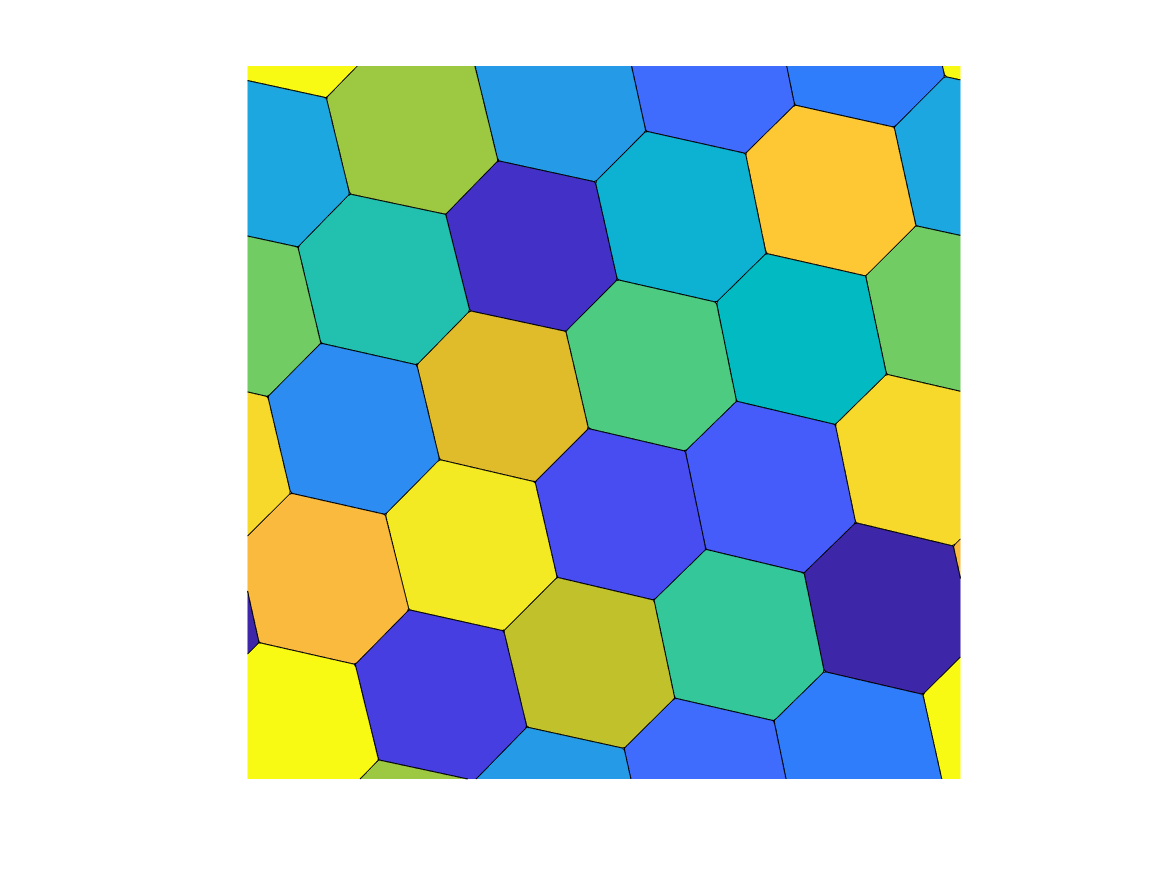}
\includegraphics[width = 0.19\textwidth, clip, trim = 4cm 1cm 3cm 1cm]{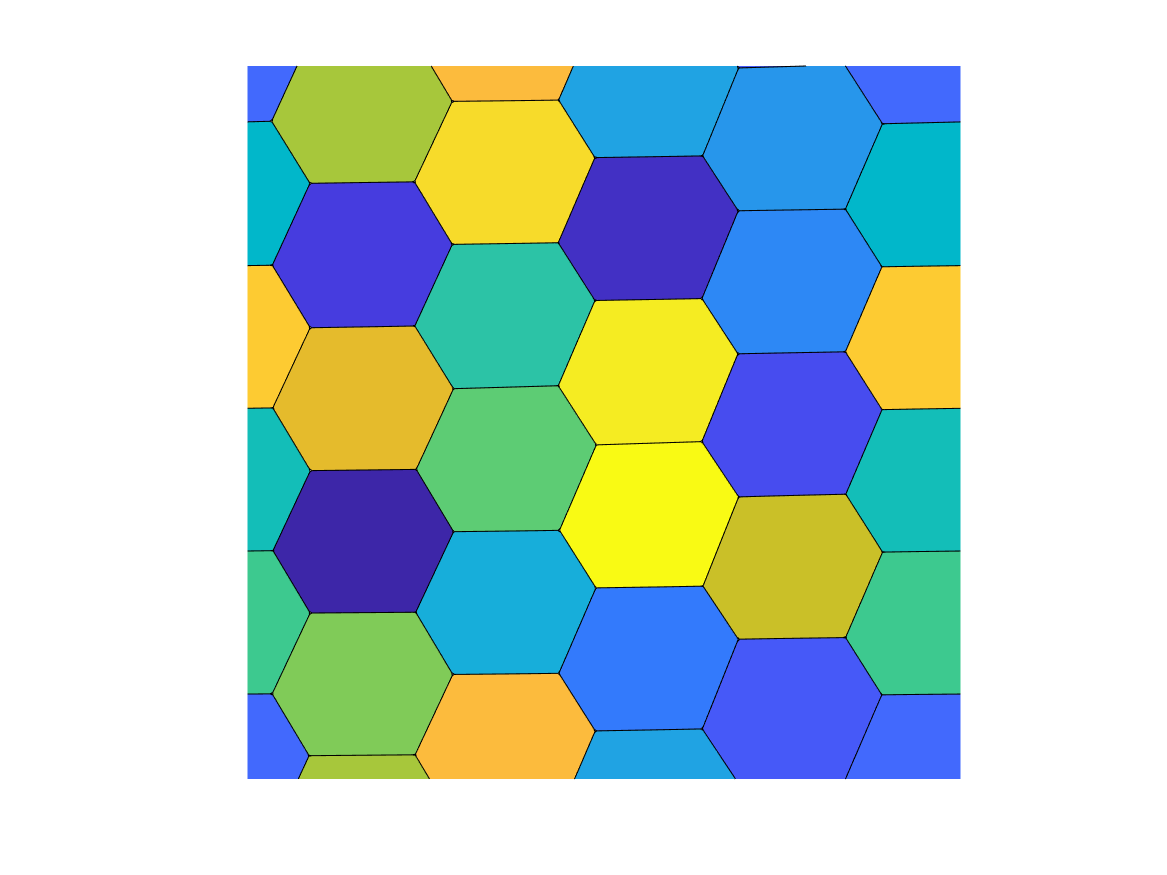}
\includegraphics[width = 0.19\textwidth, clip, trim = 4cm 1cm 3cm 1cm]{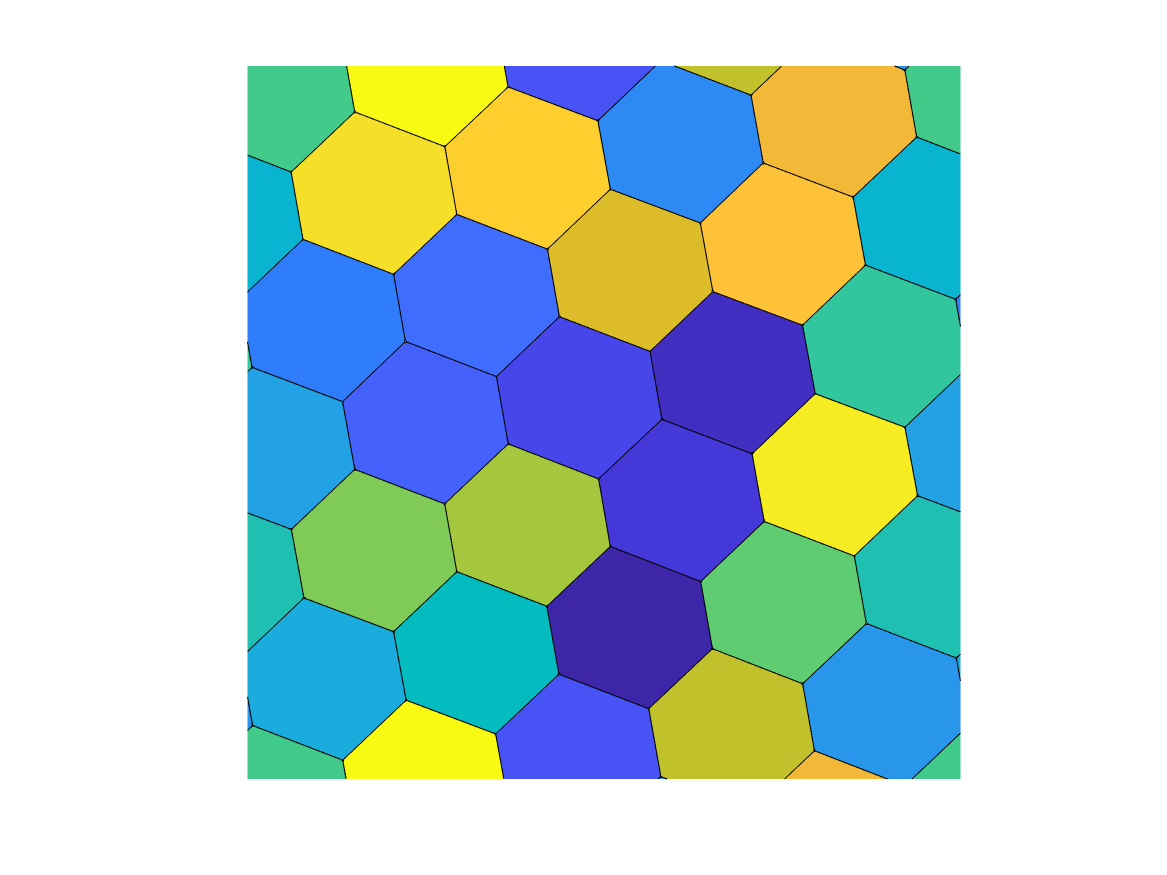}
\includegraphics[width = 0.19\textwidth, clip, trim = 4cm 1cm 3cm 1cm]{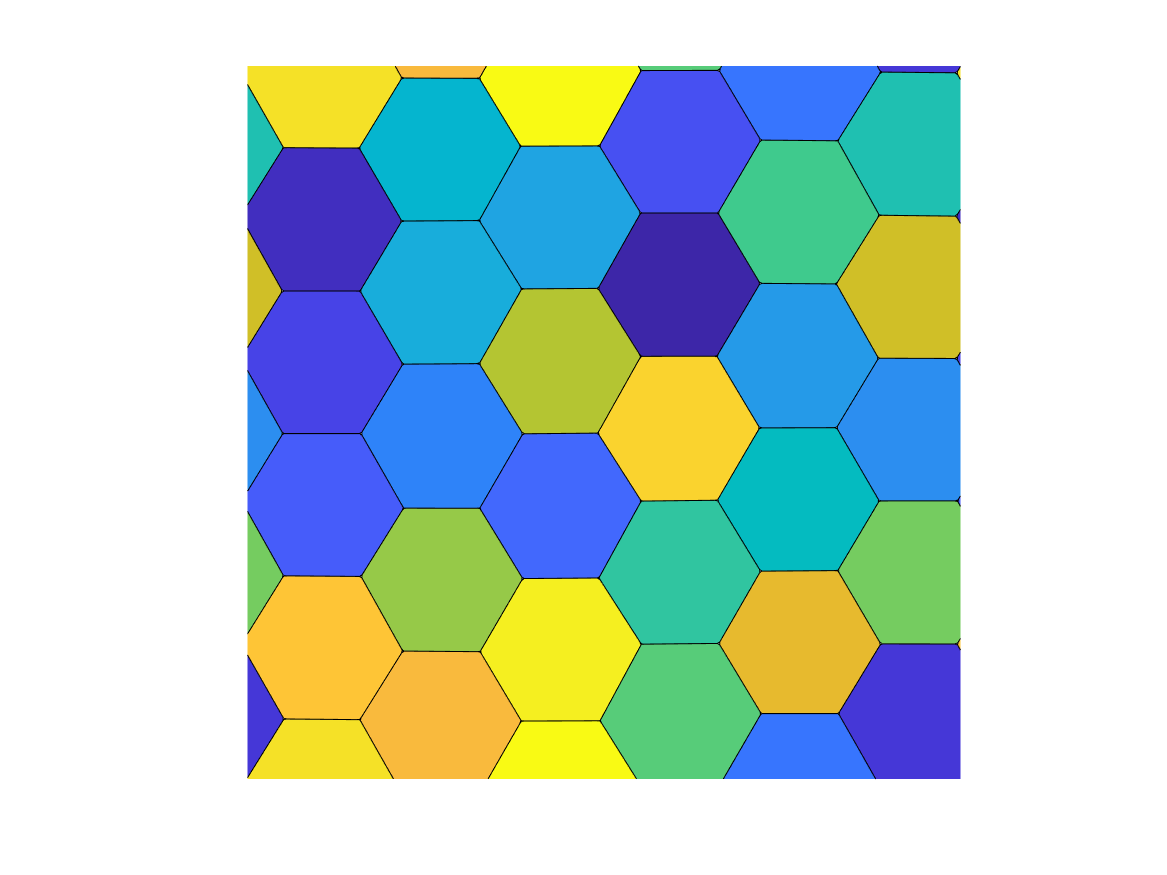}
\includegraphics[width = 0.19\textwidth, clip, trim = 4cm 1cm 3cm 1cm]{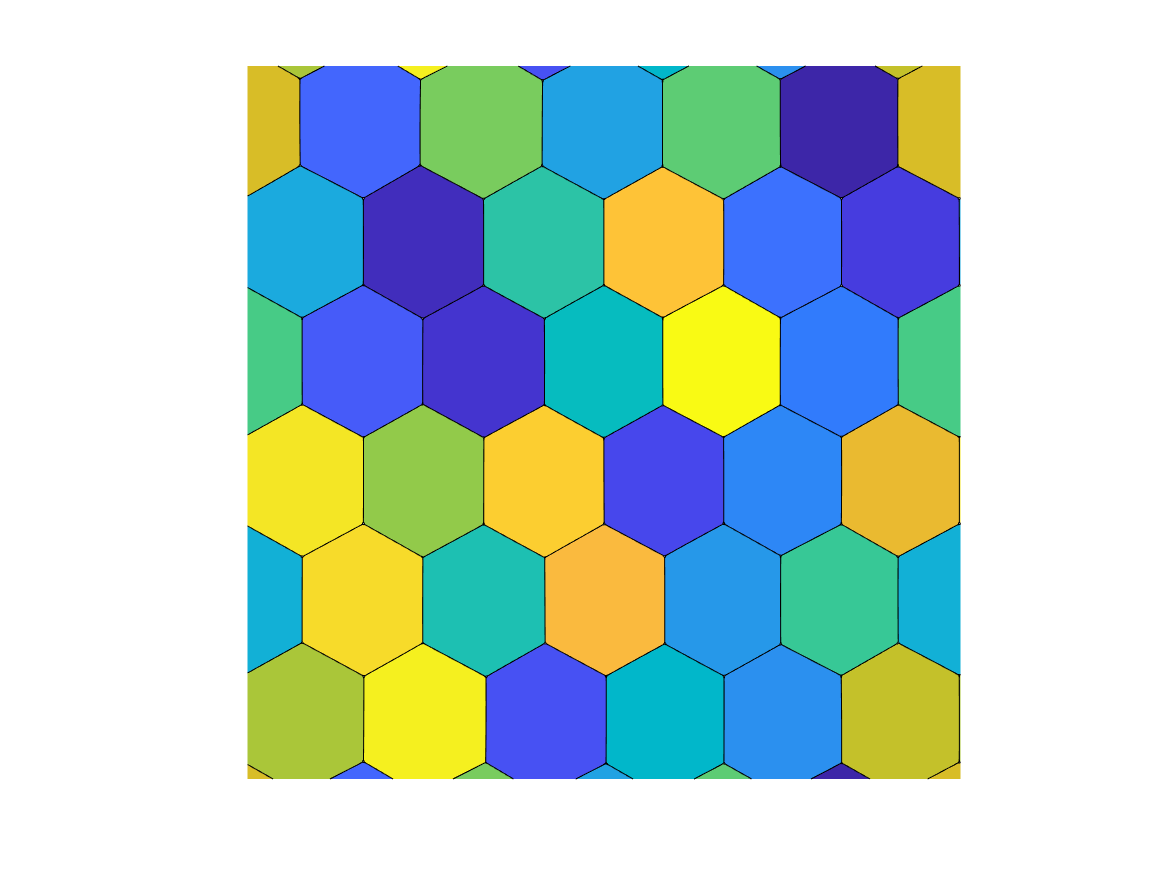}

\medskip 

{\footnotesize
\begin{tabular}{c|c|c|c|c|c|c|c|c|c}
\hline
k=4 & k=5  & k=6 & k=7& k=8 & k=9& k=10 & k=11 & k=12 & k =14  \\ 
\hline 
5.06 & 8.95 & 10.48 & 18.44 &16.81 & 20.56 & 24.67 & 28.77 & 32.81 & 42.20  \\
\hline
\hline
k =15& k = 16 & k=18 & k = 20 & k =23 & k = 24 & k =25  & k=28 & k = 30 & k=36  \\
\hline
46.72 & 51.81 & 62.33  &73.01 & 89.94 & 95.88 & 101.91 &120.10 & 132.47 & 327.67 \\
\hline
\end{tabular}
}

\caption{$k$-partitions in a periodic domain for $k = 4-12$, $14-16$, $18$, $20$, $23-25$, $28$, $30$ and $36$ together with corresponding approximate eigenvalues. See Section~\ref{sec:2dperiodic}.} \label{fig:2dperiodic}
\end{figure}

\begin{figure}[ht]
\centering
\includegraphics[width = 0.3\textwidth, clip, trim = 8cm 3cm 6cm 2cm]{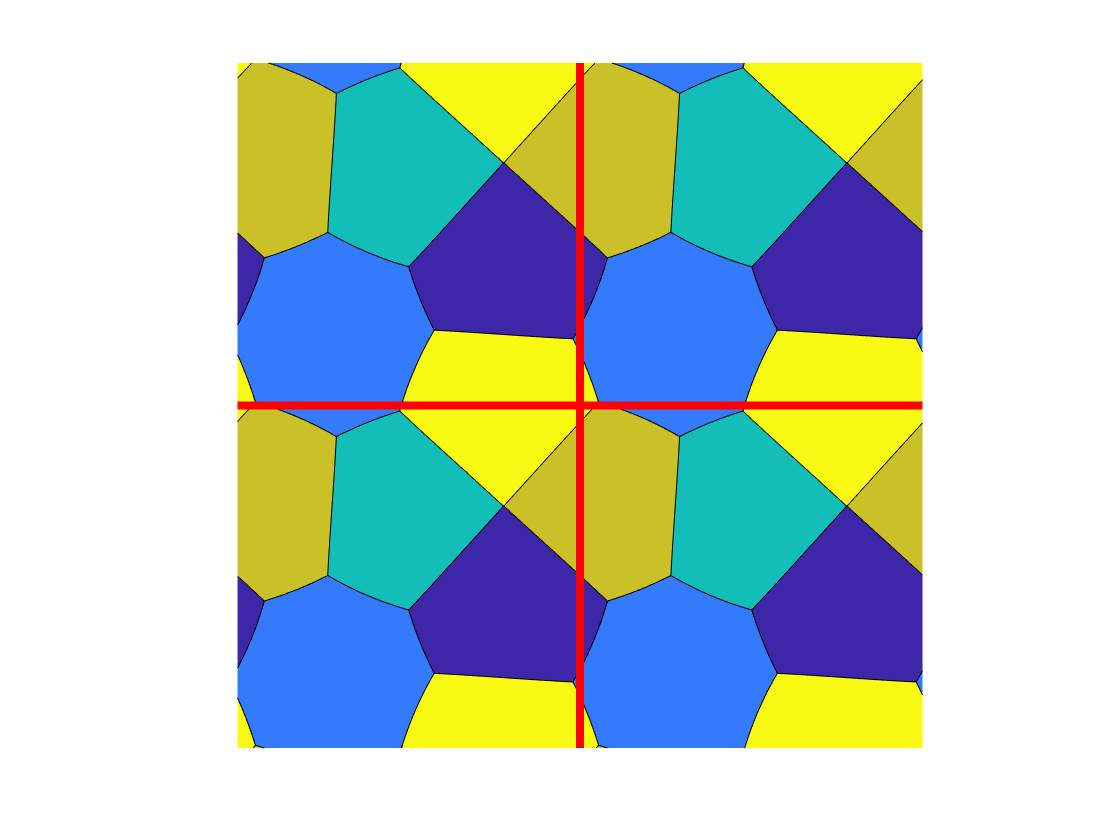}
\includegraphics[width = 0.3\textwidth, clip, trim = 8cm 3cm 6cm 2cm]{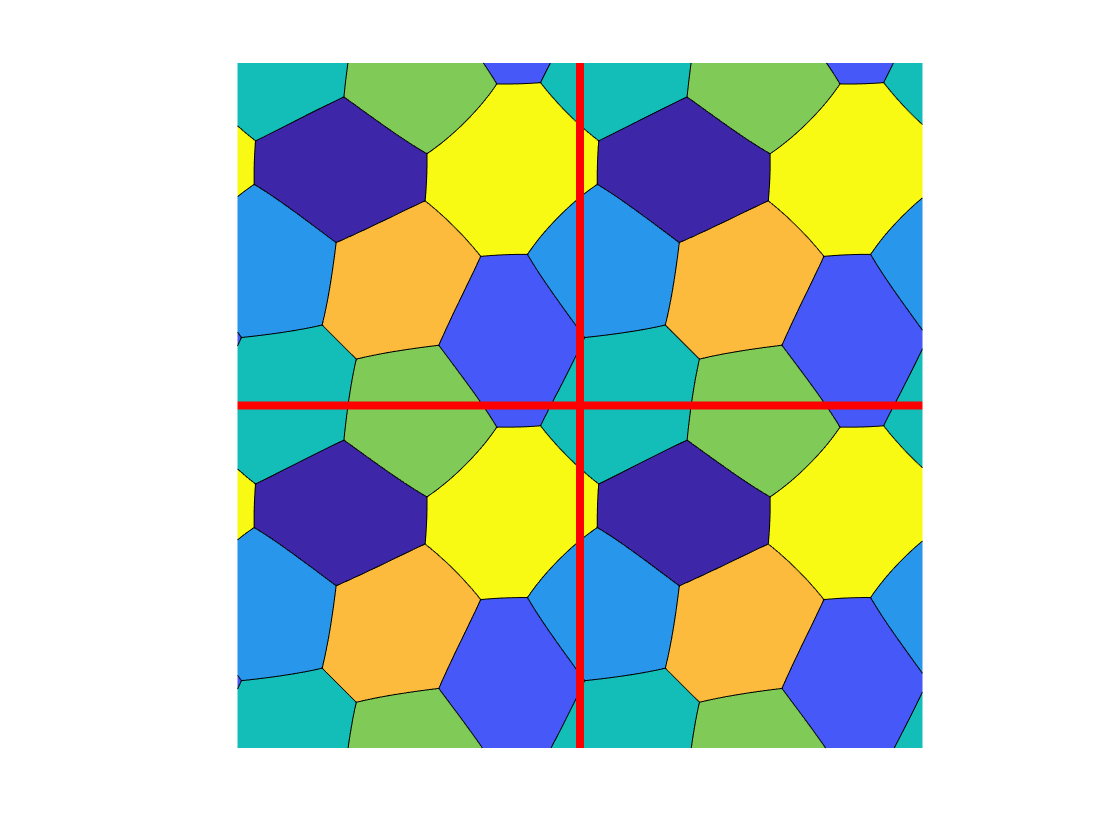}
\includegraphics[width = 0.3\textwidth, clip, trim = 8cm 3cm 6cm 2cm]{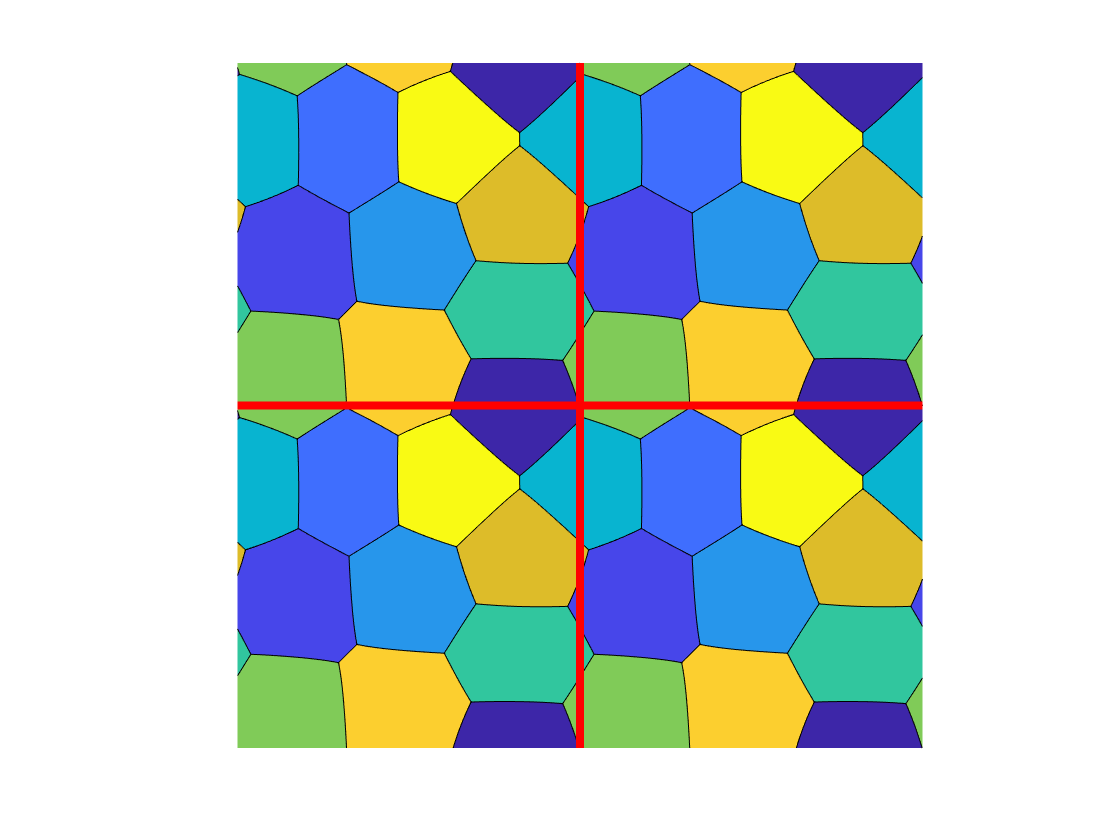}

\caption{Periodic extensions of $k$-partitions with $k = 5,7,10$. See Section~\ref{sec:2dperiodic}.} \label{fig:extension}
\end{figure}

\subsection{$k$-partition for 2-dimensional arbitrary domains.}\label{sec:2darbitrary}
In this section, we compute the Dirichlet $k$-partition in arbitrary domains. We treat the domain as a subset of the computational domain $[-\pi,\pi]^2$ and discretize the computational domain by $512^2$ uniform grid points.

In Figure~\ref{fig:2darbitrary1}, we list $k$-partitions in an equilateral triangle domain for $k = 2-10$, $12$, $13$, $15$, $21$, $28$, $36$, and $45$. For $k = 2-10$, the results are consistent with the results presented in \cite{Bogosel_2016,Chu_2021}. For $k>10$, we select regular structures we observe to list in  Figure~\ref{fig:2darbitrary1}. In particular, when $k = \frac{n(n+1)}{2}, n = 1,2,\cdots$, one can see very regular structures with hexagon tessellations in the interior layer of the partition (for example, $k = 10,15,21,28,36,$ and $45$). For all reported $k$ in Figure~\ref{fig:2darbitrary1}, the average CPU time for each computation is less than hundred {\it seconds} starting with random initial guesses. For small $k$ ({\it e.g.} $k<10$), the computation only takes about $2$ {\it seconds}.

\begin{figure}[ht]
\centering
\includegraphics[width = 0.15\textwidth, clip, trim = 6cm 4cm 5.5cm 3.5cm]{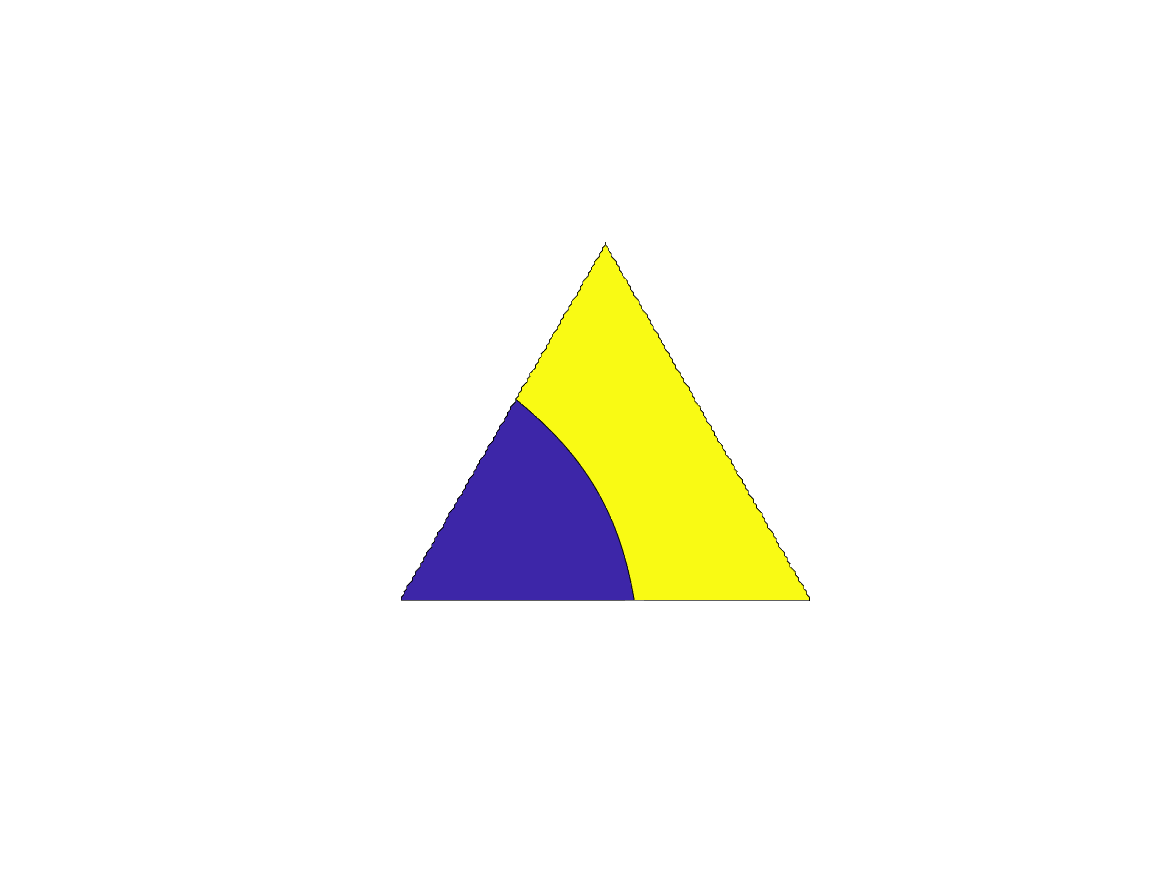}
\includegraphics[width = 0.15\textwidth, clip, trim = 6cm 4cm 5.5cm 3.5cm]{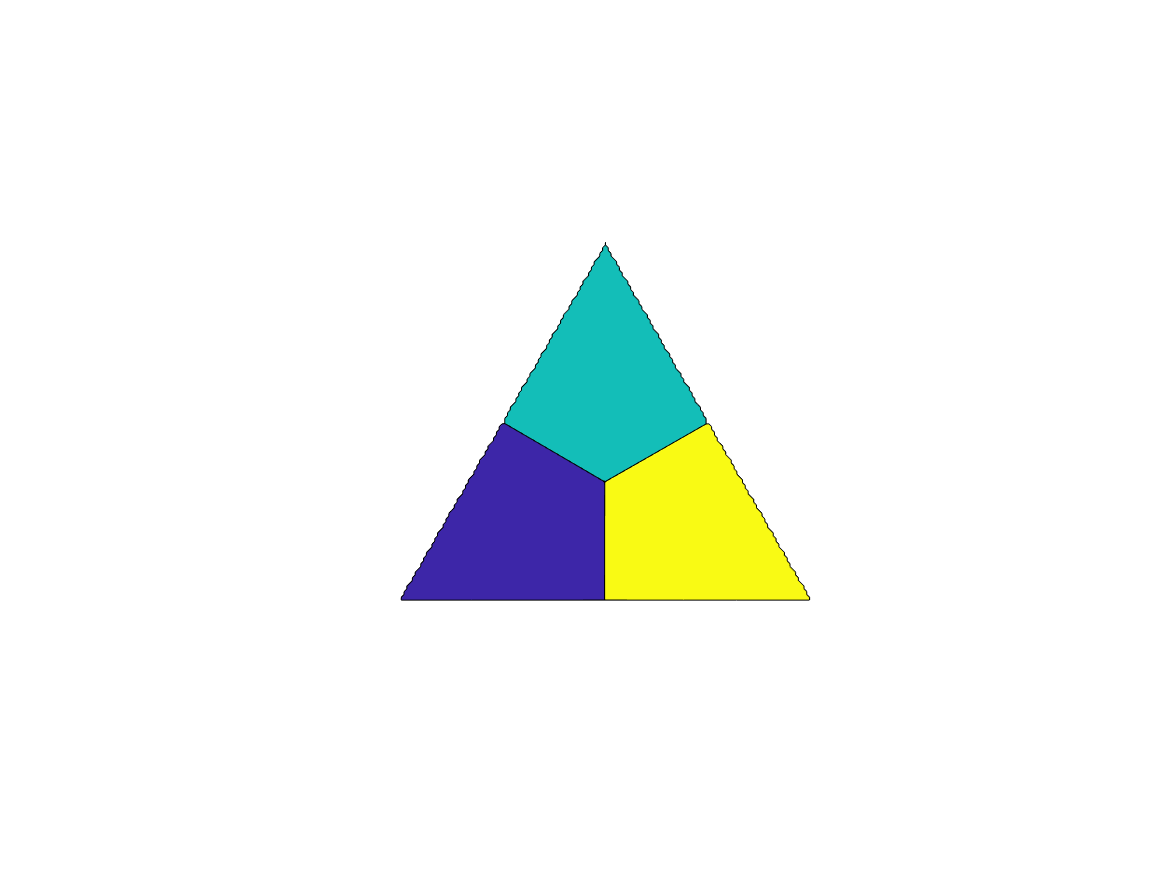}
\includegraphics[width = 0.15\textwidth, clip, trim = 6cm 4cm 5.5cm 3.5cm]{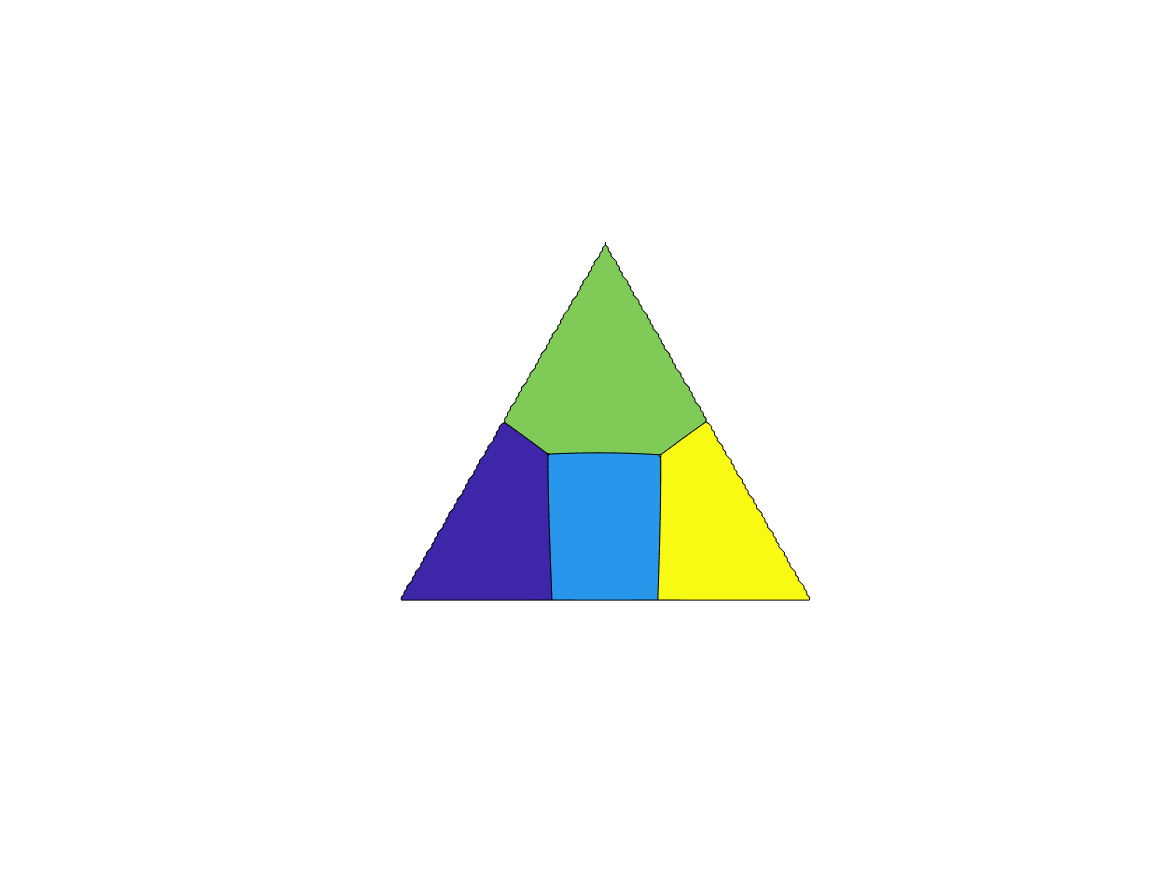}
\includegraphics[width = 0.15\textwidth, clip, trim = 6cm 4cm 5.5cm 3.5cm]{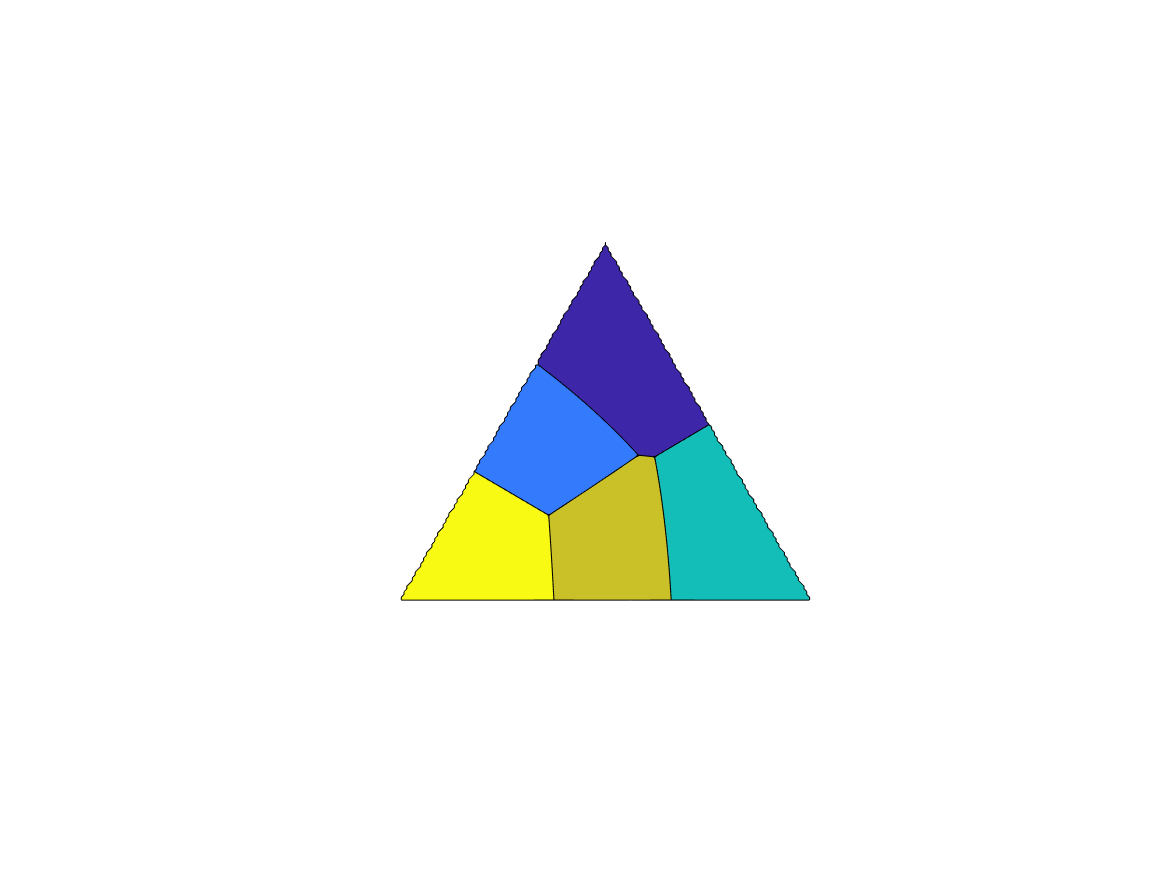}
\includegraphics[width = 0.15\textwidth, clip, trim = 6cm 4cm 5.5cm 3.5cm]{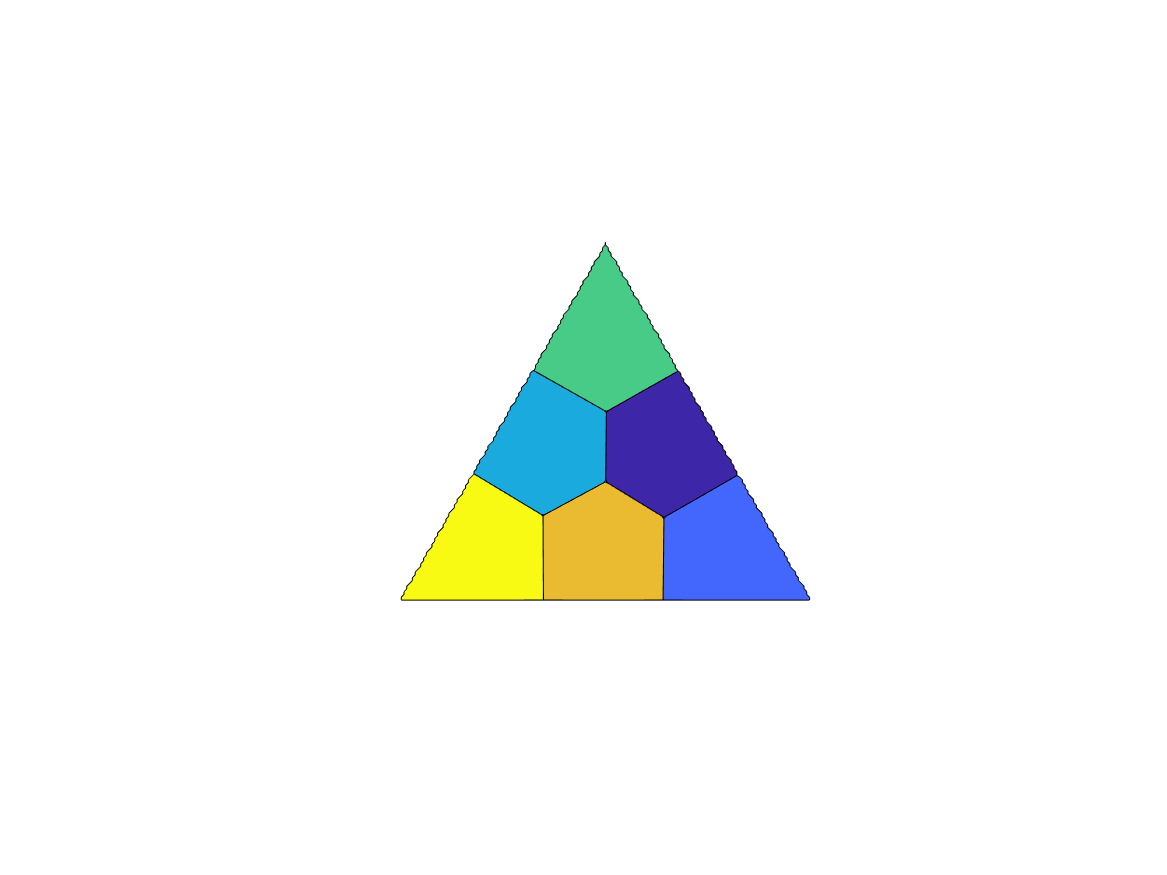}
\includegraphics[width = 0.15\textwidth, clip, trim = 6cm 4cm 5.5cm 3.5cm]{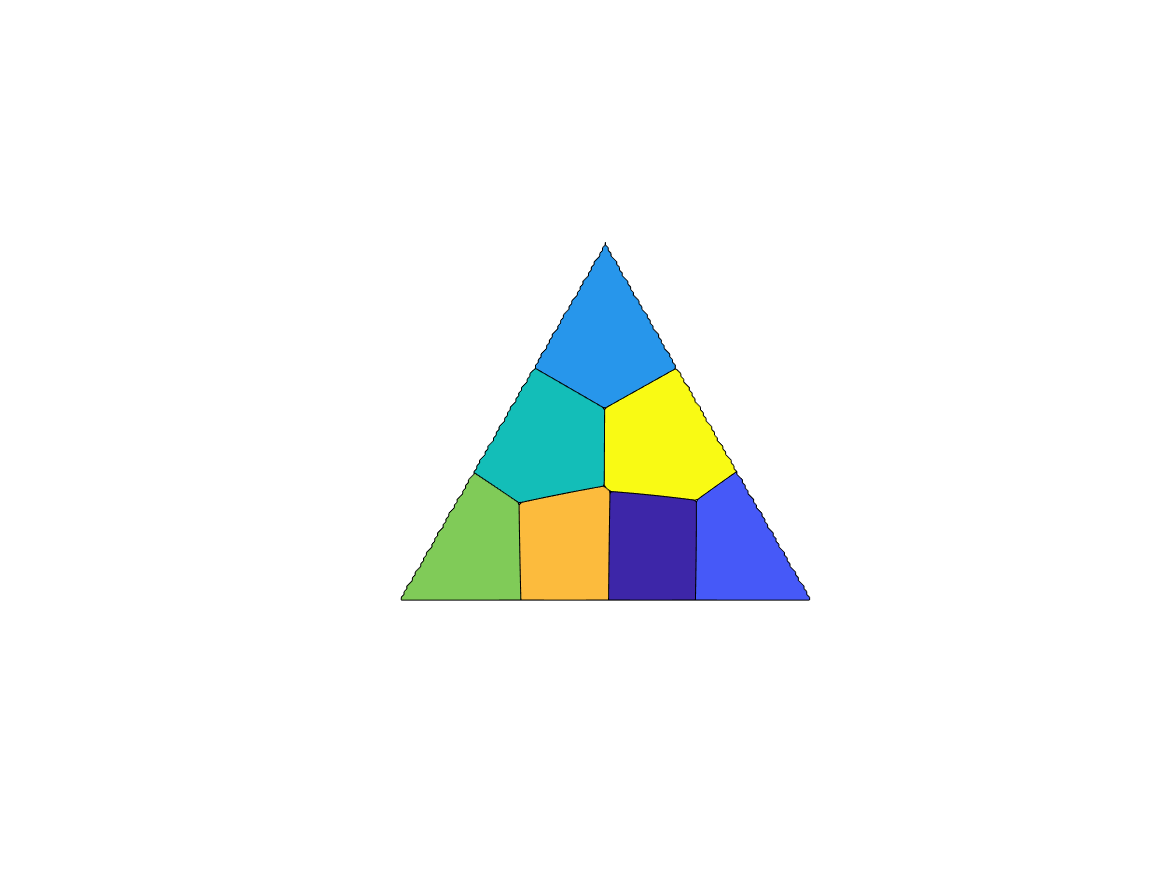}
\includegraphics[width = 0.15\textwidth, clip, trim = 6cm 4cm 5.5cm 3.5cm]{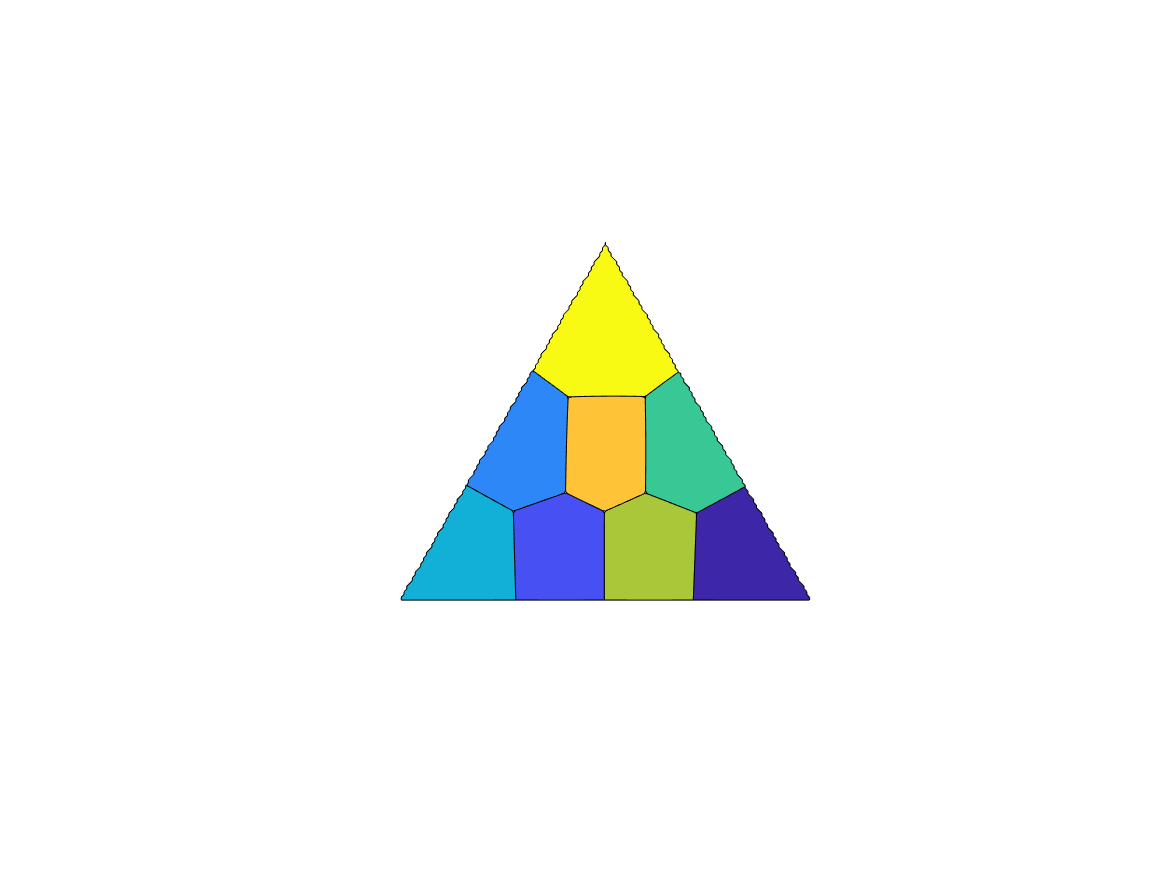}
\includegraphics[width = 0.15\textwidth, clip, trim = 6cm 4cm 5.5cm 3.5cm]{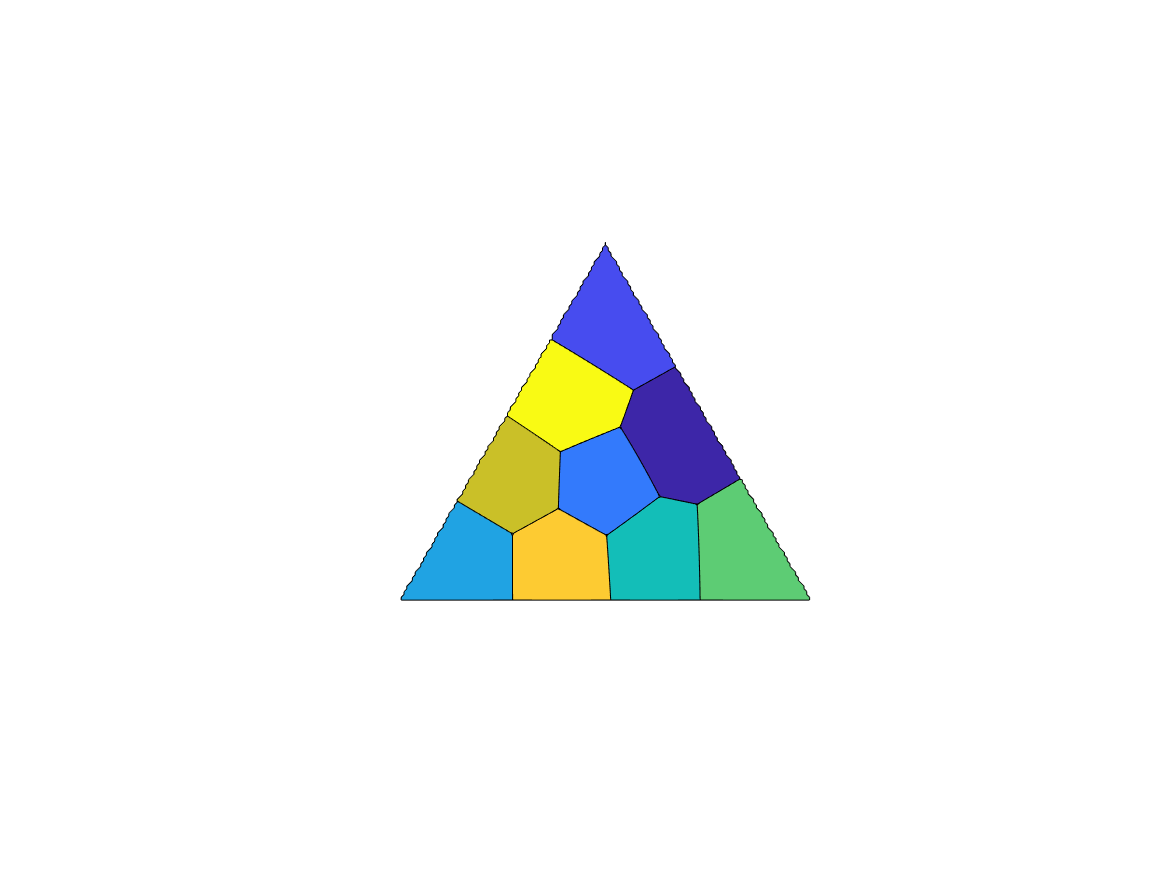}
\includegraphics[width = 0.15\textwidth, clip, trim = 6cm 4cm 5.5cm 3.5cm]{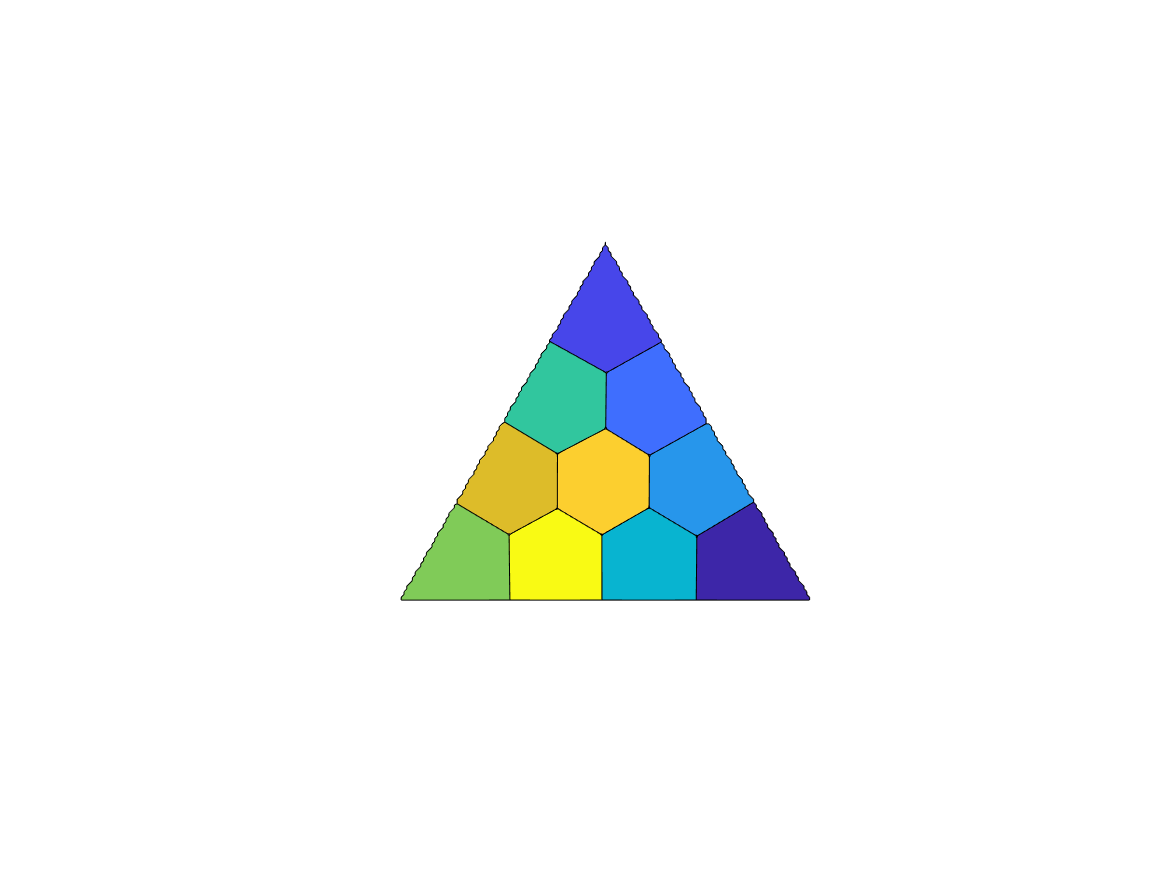}
\includegraphics[width = 0.15\textwidth, clip, trim = 6cm 4cm 5.5cm 3.5cm]{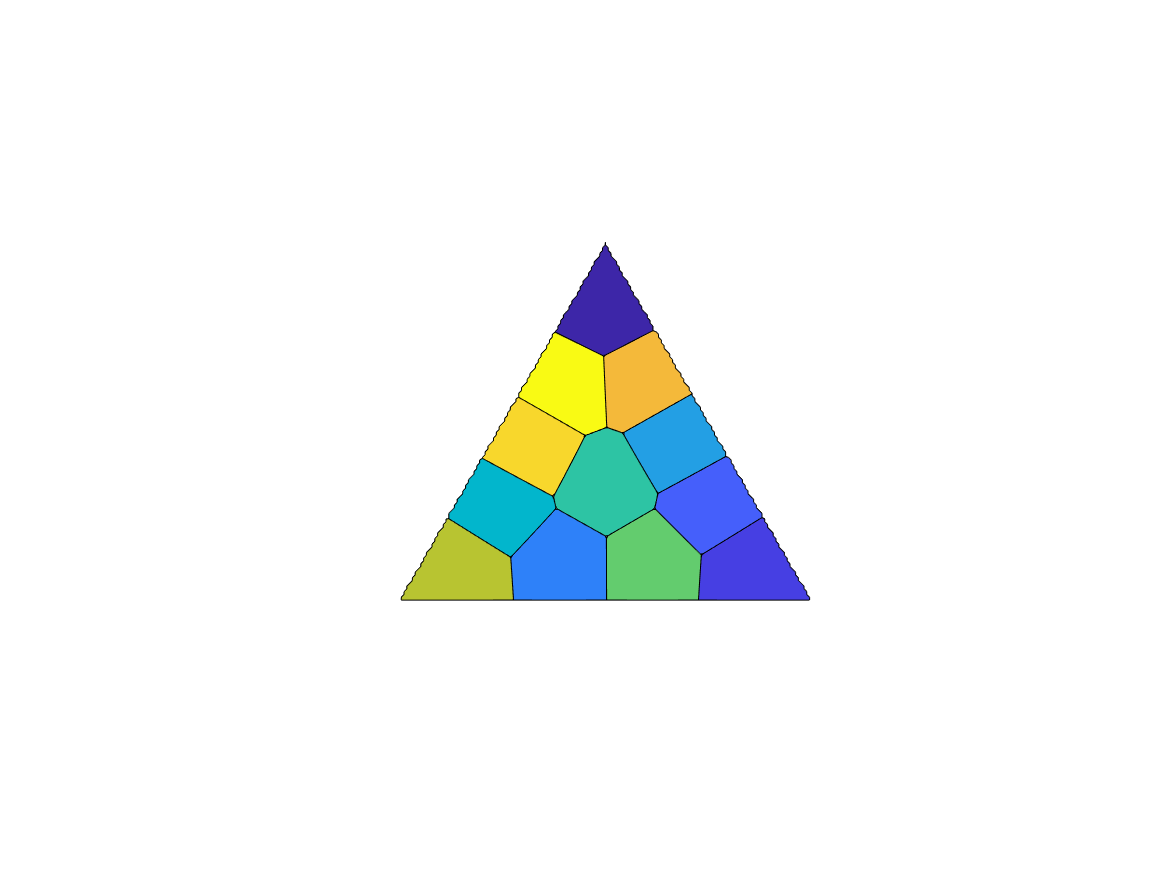}
\includegraphics[width = 0.15\textwidth, clip, trim = 6cm 4cm 5.5cm 3.5cm]{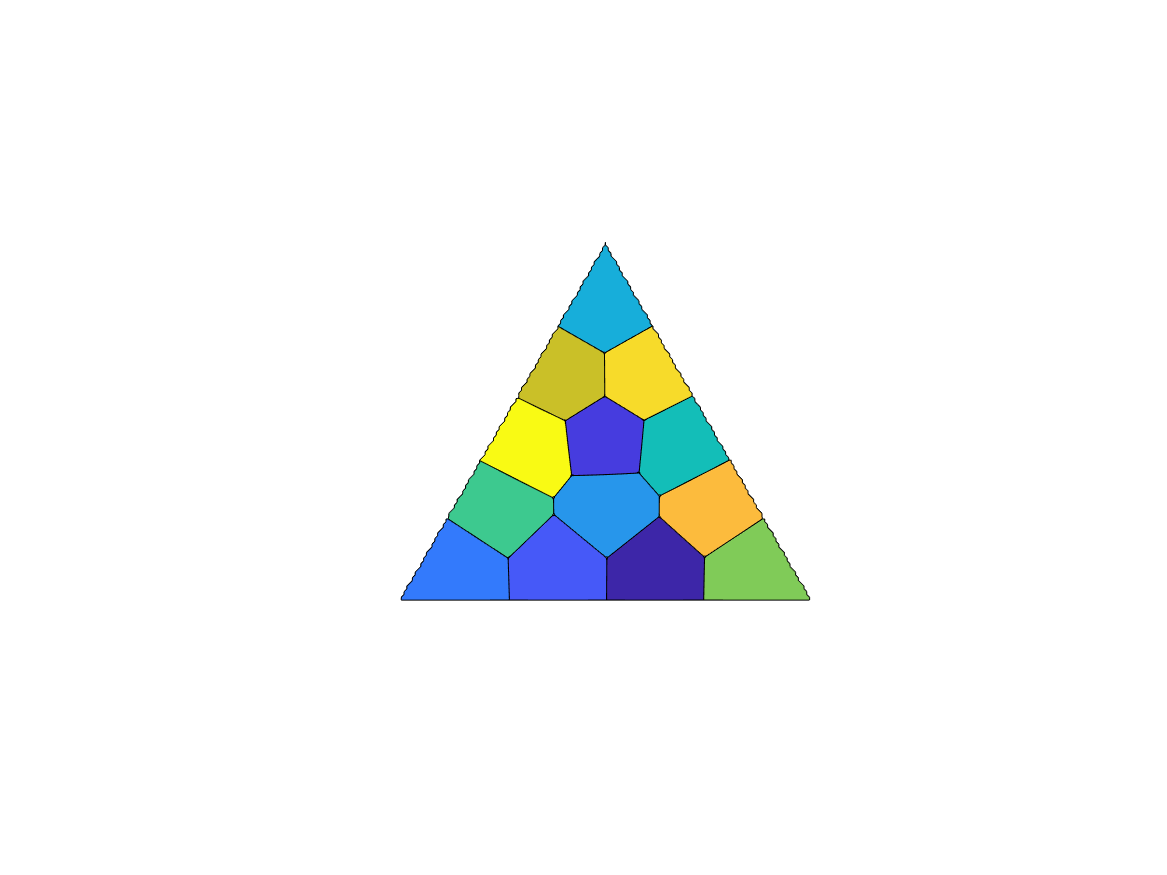}
\includegraphics[width = 0.15\textwidth, clip, trim = 6cm 4cm 5.5cm 3.5cm]{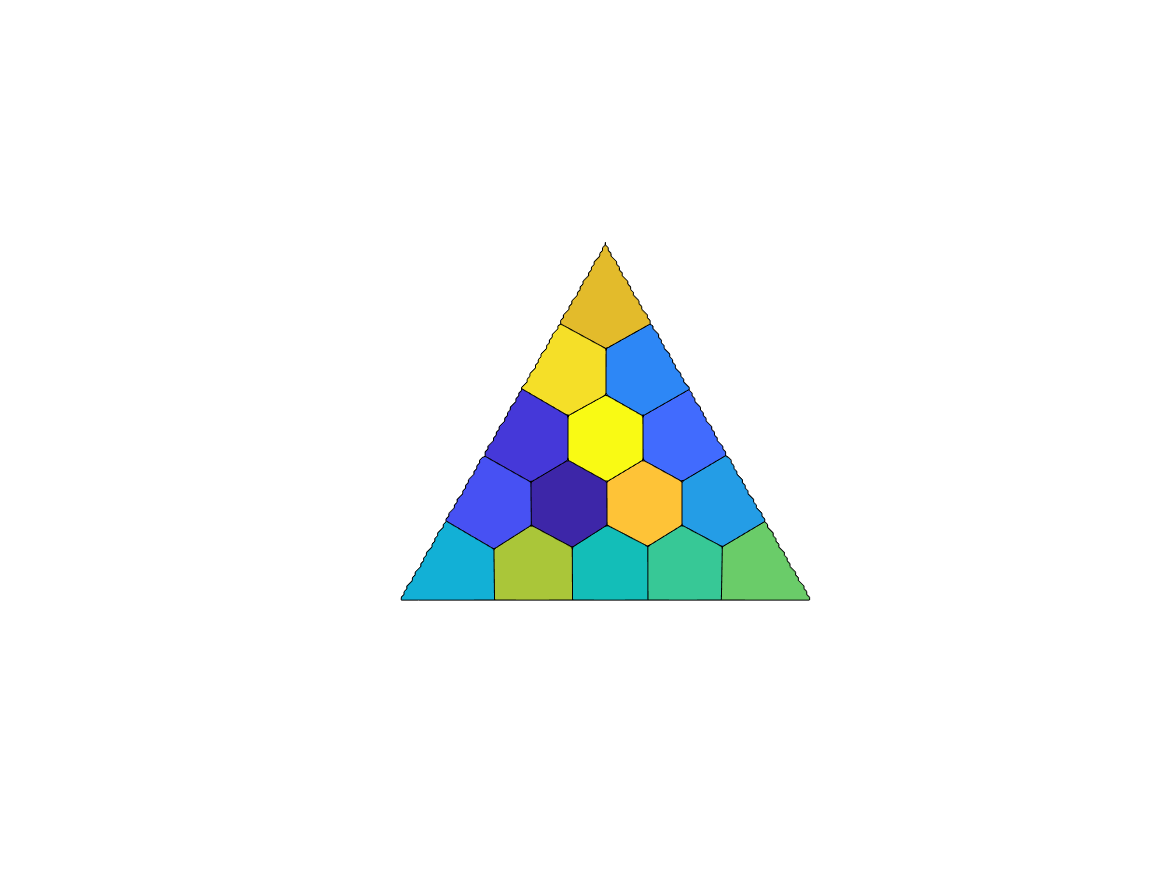}
\includegraphics[width = 0.23\textwidth, clip, trim = 6cm 4cm 5.5cm 3.5cm]{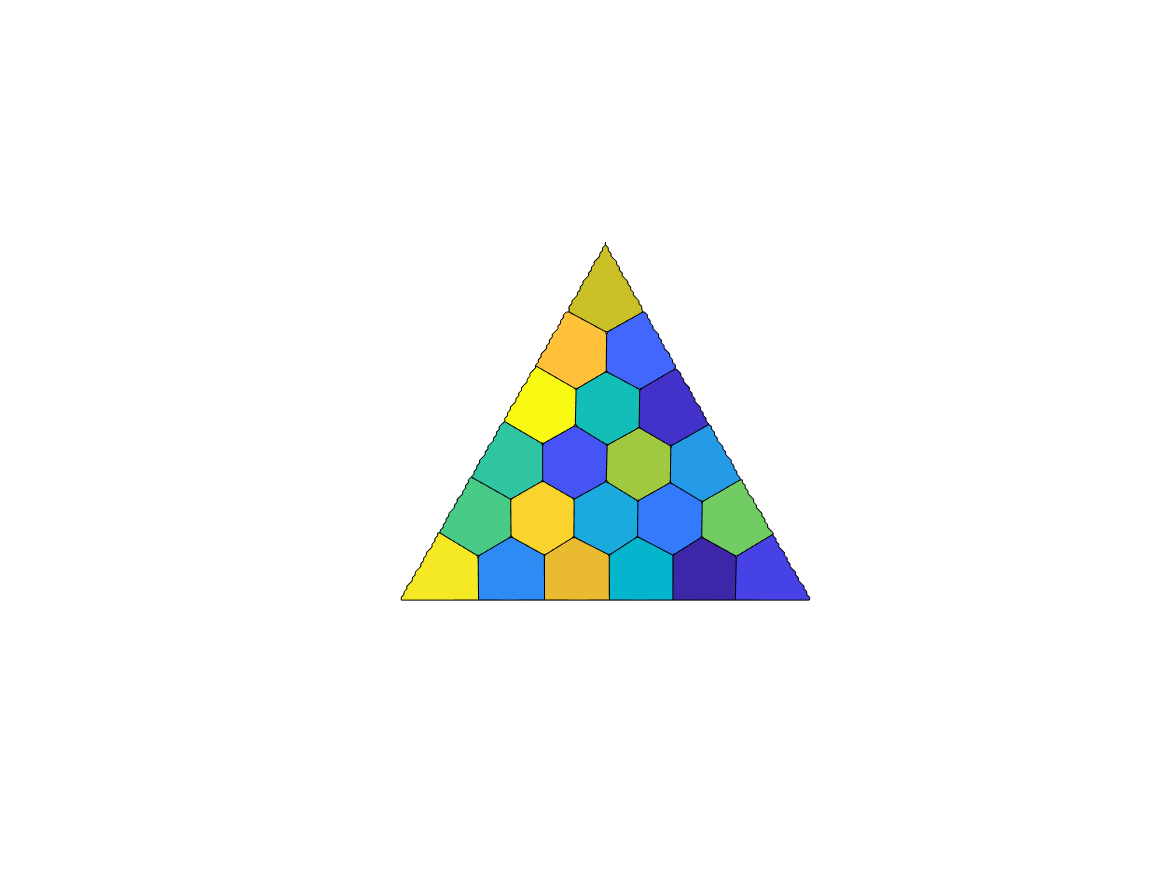}
\includegraphics[width = 0.23\textwidth, clip, trim = 6cm 4cm 5.5cm 3.5cm]{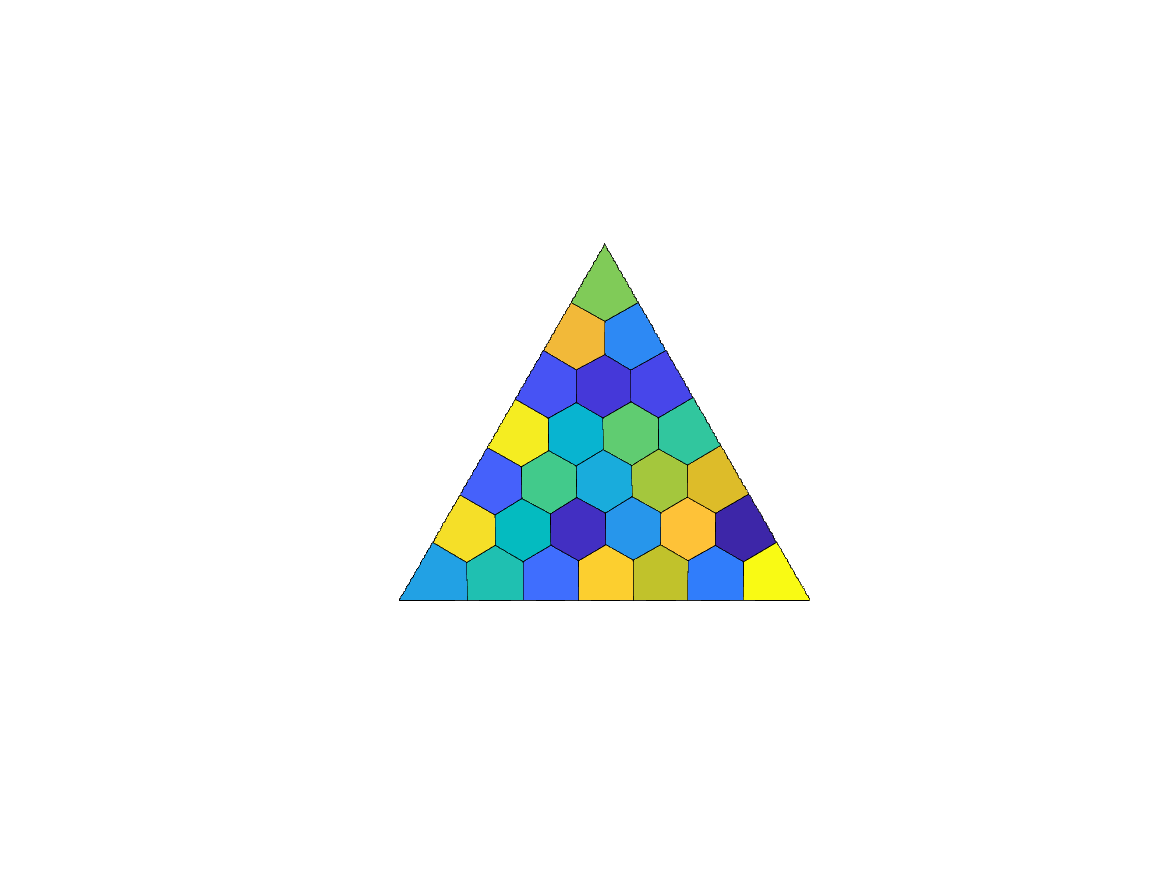}
\includegraphics[width = 0.23\textwidth, clip, trim = 6cm 4cm 5.5cm 3.5cm]{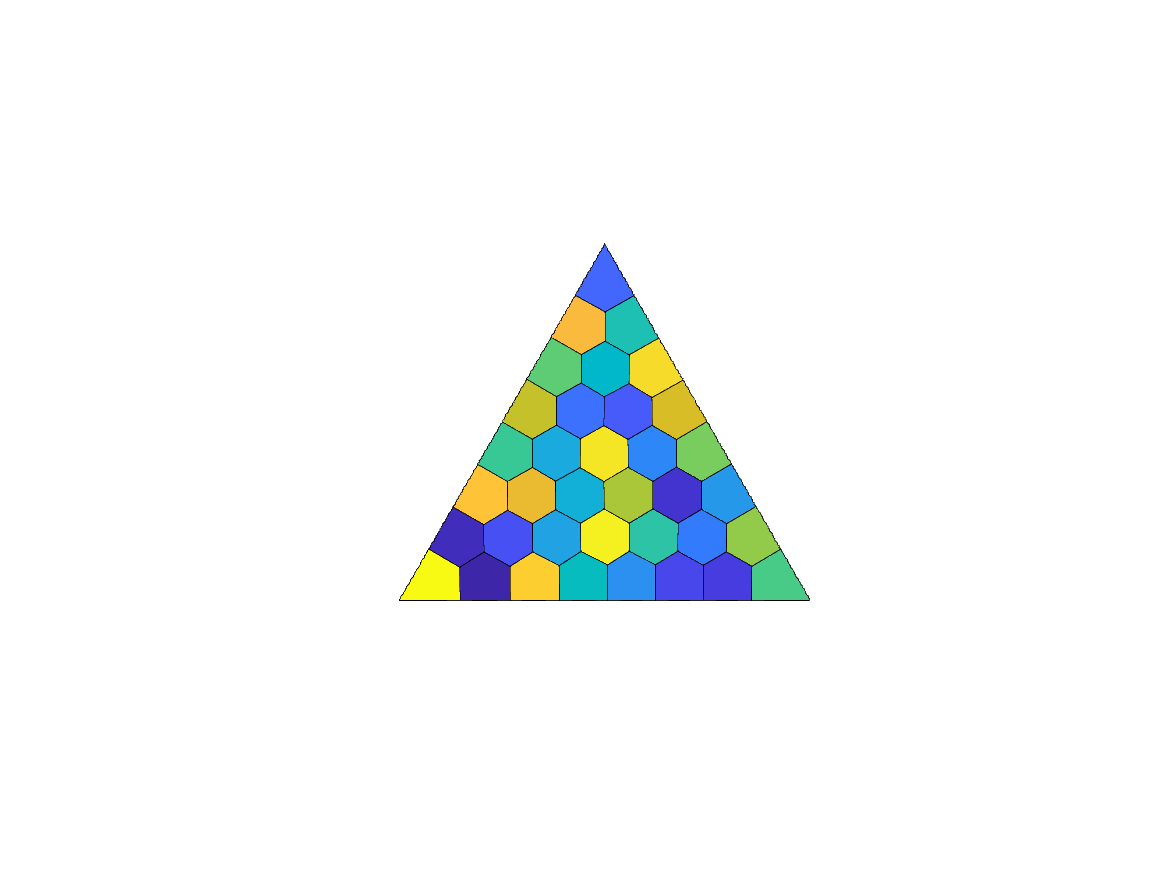}
\includegraphics[width = 0.23\textwidth, clip, trim = 6cm 4cm 5.5cm 3.5cm]{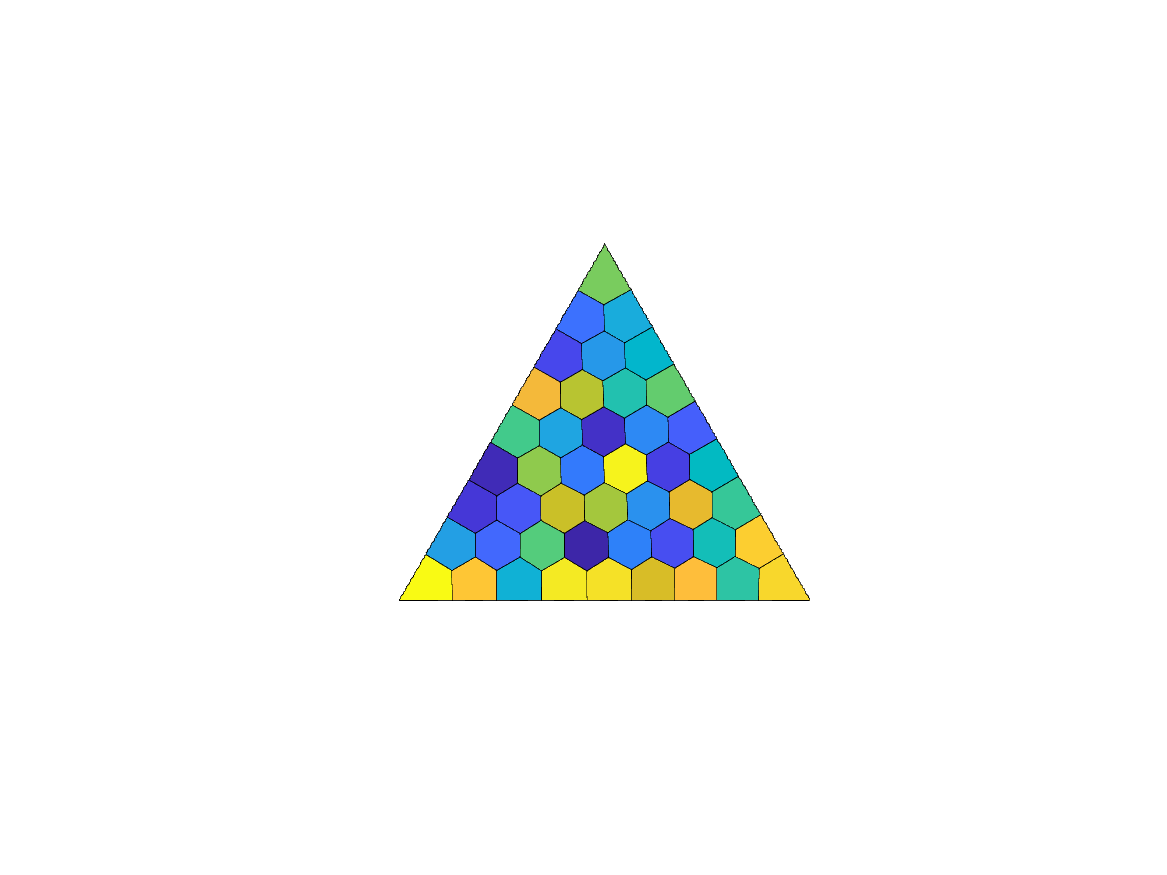}

\medskip 

{\footnotesize
\begin{tabular}{|c|c|c|c|c|c|c|c|}
\hline
k=2 & k=3  & k=4 & k=5& k=6 & k=7& k=8 & k=9  \\ 
\hline 
20.45 & 35.72 & 61.16 & 87.71 &113.76 & 148.74 & 184.59 & 220.78   \\
\hline
\hline
k=10 & k=12  & k=13 & k=15& k=21 & k=28& k=36 & k=45  \\ 
\hline 
256.33 & 342.54 & 385.97 & 473.18 & 768.24 & 1142.01 & 1592.94  & 3379.16   \\
\hline\end{tabular}
}

\caption{$k$-partitions in a triangle domain for $k = 2-10$, $12$, $13$, $15$, $21$, $28$, $36$, and $45$. The table lists all approximate eigenvalues. See Section~\ref{sec:2darbitrary}.} \label{fig:2darbitrary1}
\end{figure}
Figure~\ref{fig:2darbitrary4} lists the Dirichlet $k$-partitions in a square domain, a pentagon domain, a regular hexagon domain, a disk domain, a three-fold domain, and a five-fold domain for $k=2-10$. All results agree with the computational results in \cite{Chu_2021} and \cite{Bogosel_2016} for the reported cases. However, all average computational time is less than 30 {\it seconds} starting with random initial guesses. In particular, the computational time of the level set based method proposed in \cite{Chu_2021} for the five-fold star cases for $k=3-10$ are $23, 27, 31, 35, 40, 44, 49$ and $54$ {\it minutes} respectively with a $100^2$ discretization of the domain. However, we find the same results with random initial guesses (with the computational domain $[-\pi,\pi]^2$ discretized by $512^2$ grid points) only in $5.8, 9.2, 9.5, 5.4,18.4,15.8, 20.8, 19.1$ {\it seconds}, respectively. It achieves more than $100$ times acceleration. The corresponding approximate values are listed in the table in Figure~\ref{fig:2darbitrary4}.

\begin{figure}[ht]
\centering
\includegraphics[width = 0.1\textwidth, clip, trim = 7cm 4.5cm 6cm 4cm]{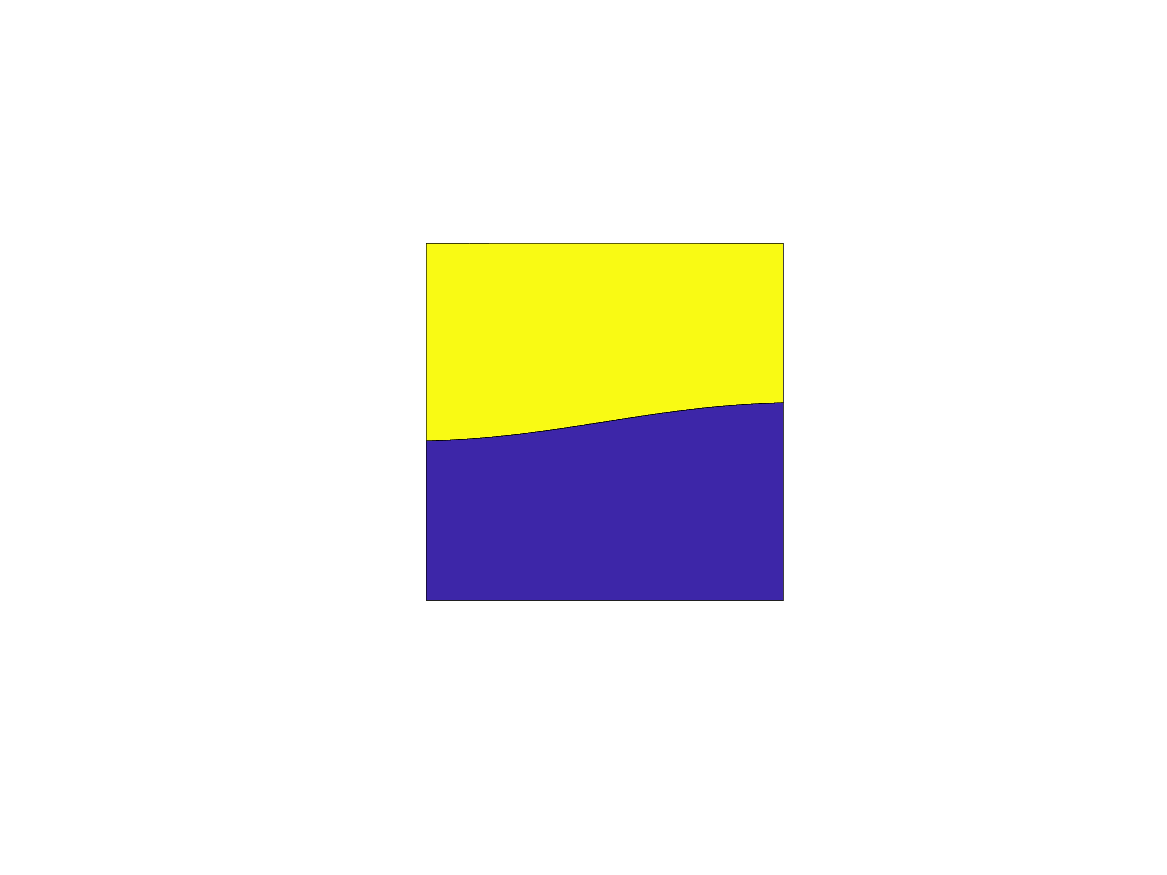}
\includegraphics[width = 0.1\textwidth, clip, trim = 7cm 4.5cm 6cm 4cm]{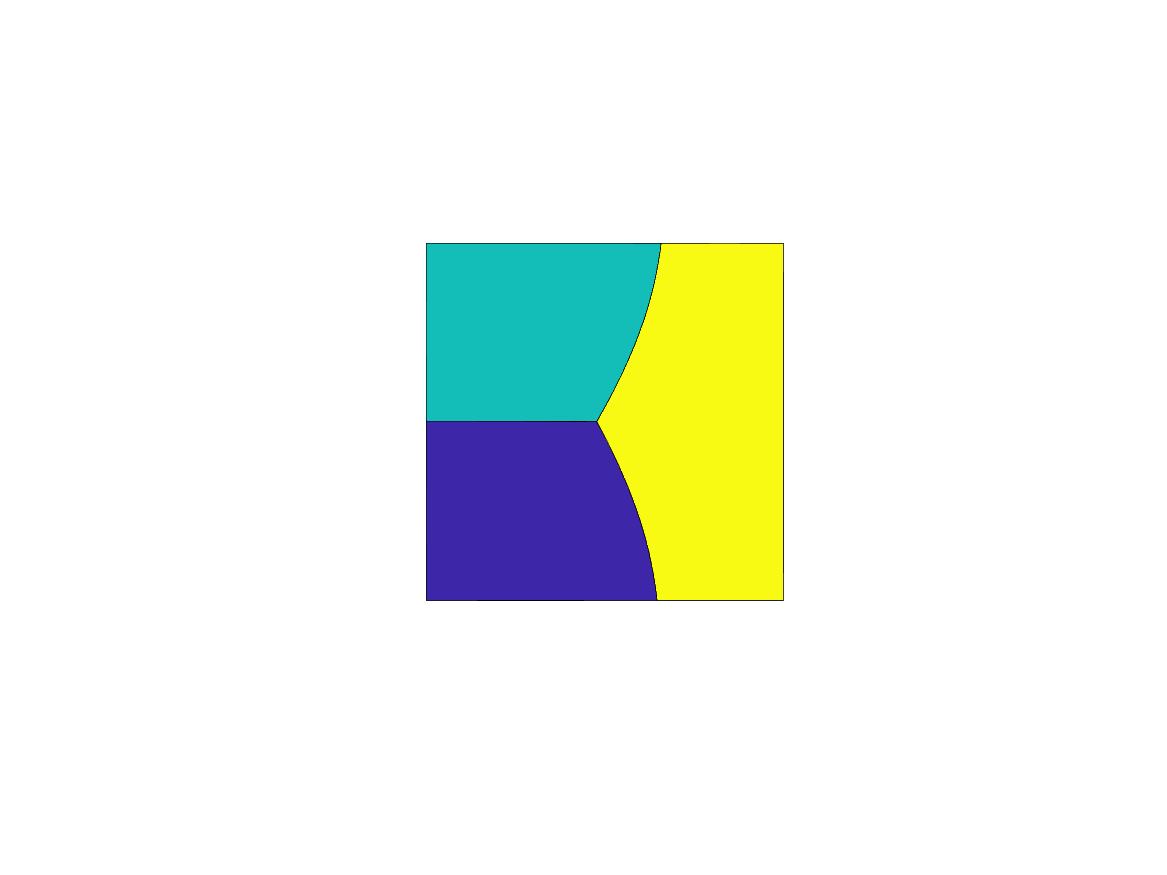}
\includegraphics[width = 0.1\textwidth, clip, trim = 7cm 4.5cm 6cm 4cm]{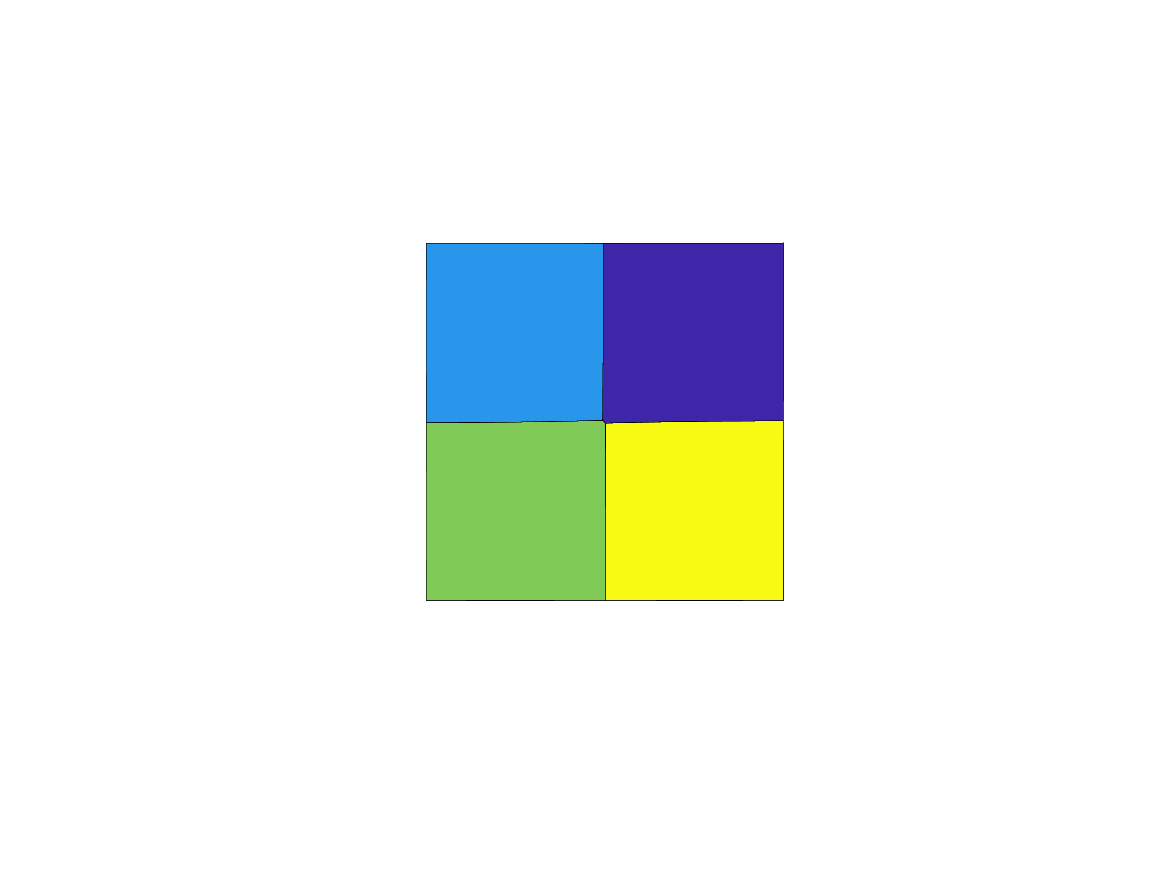}
\includegraphics[width = 0.1\textwidth, clip, trim = 7cm 4.5cm 6cm 4cm]{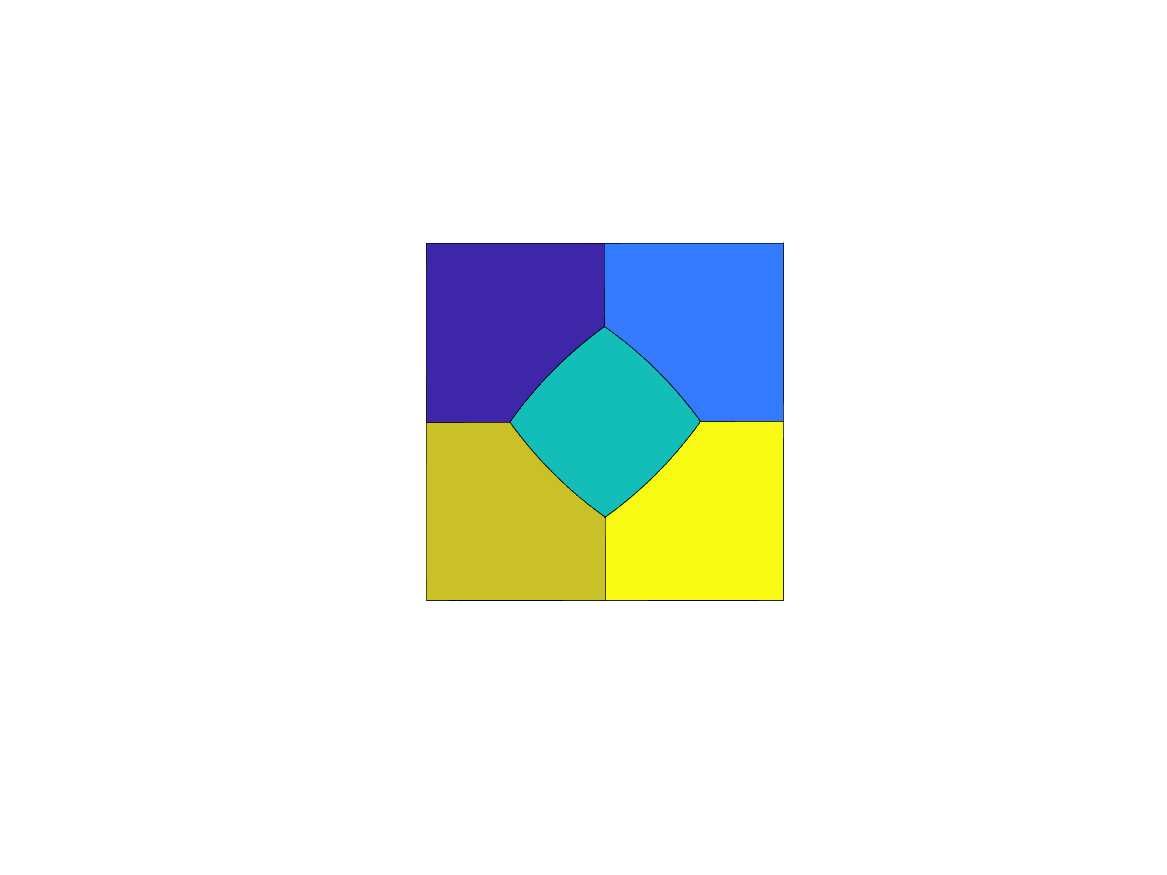}
\includegraphics[width = 0.1\textwidth, clip, trim = 7cm 4.5cm 6cm 4cm]{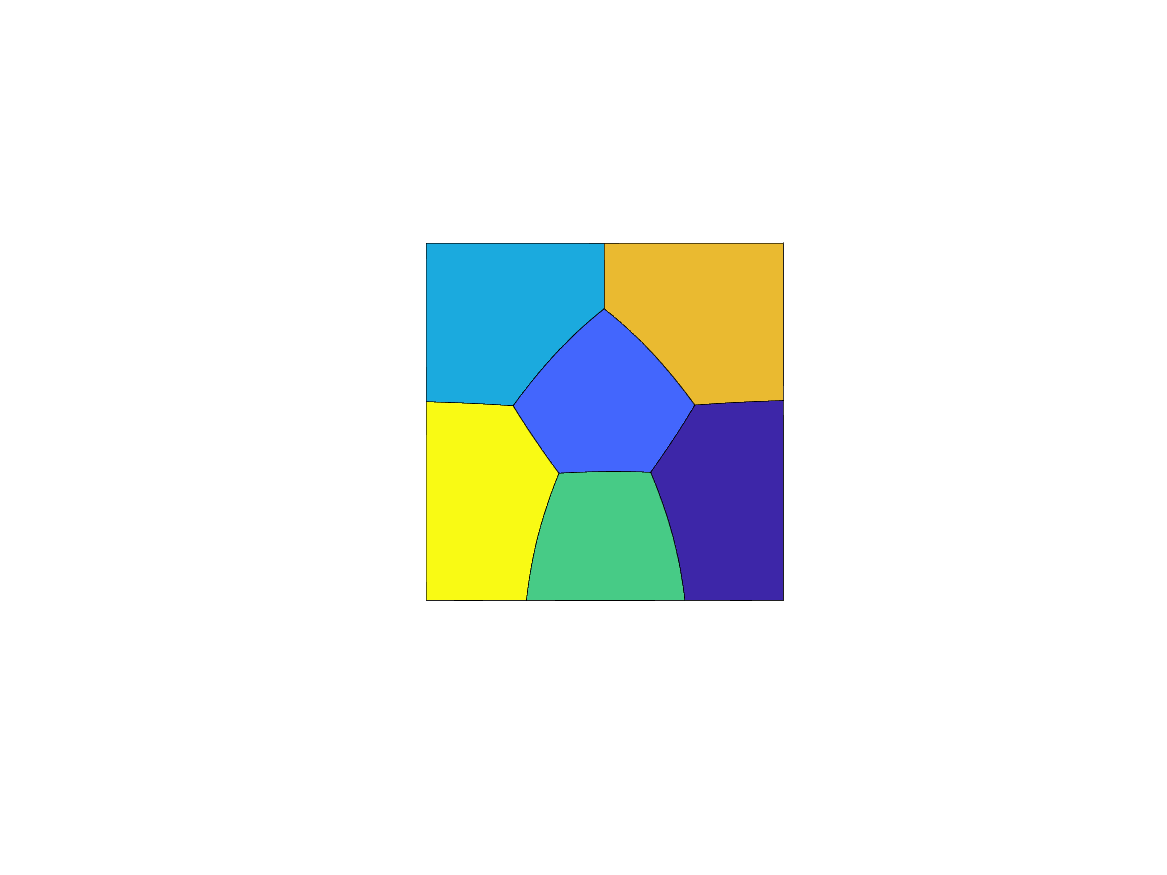}
\includegraphics[width = 0.1\textwidth, clip, trim = 7cm 4.5cm 6cm 4cm]{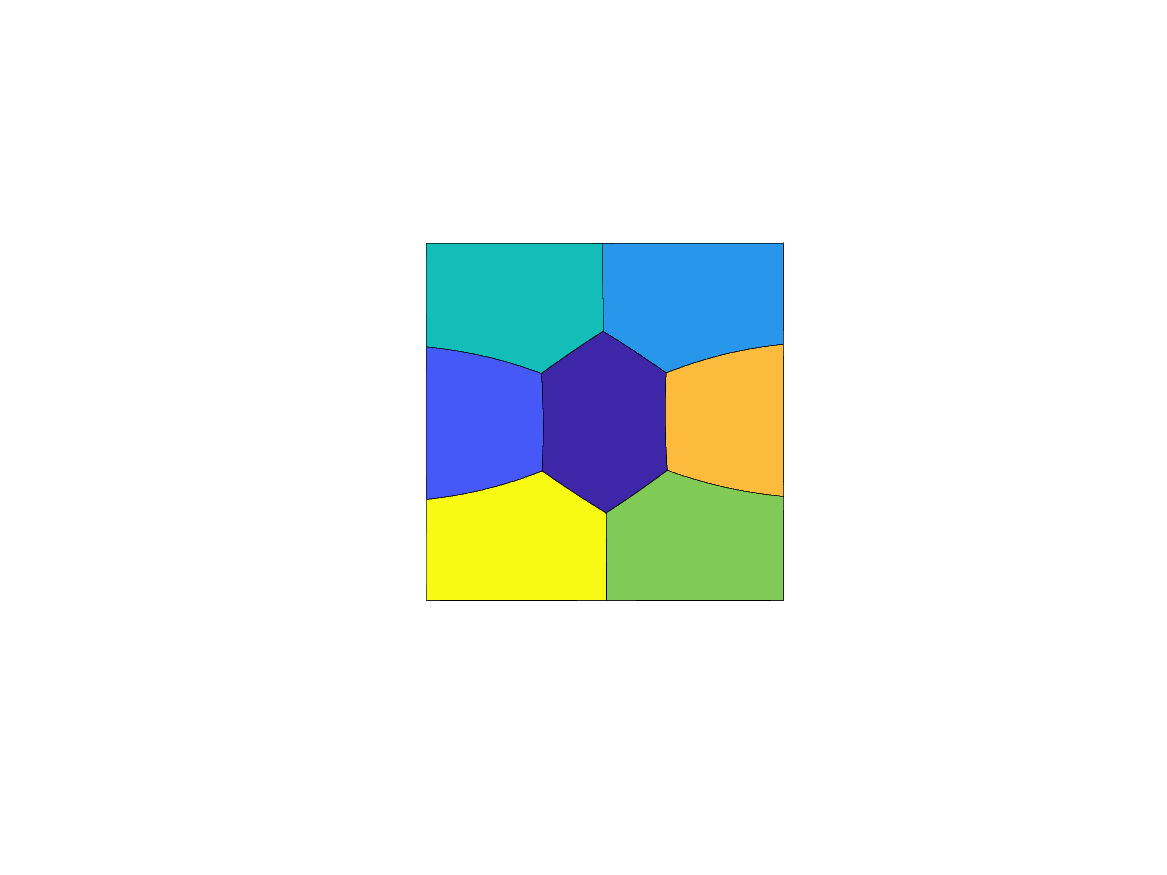}
\includegraphics[width = 0.1\textwidth, clip, trim = 7cm 4.5cm 6cm 4cm]{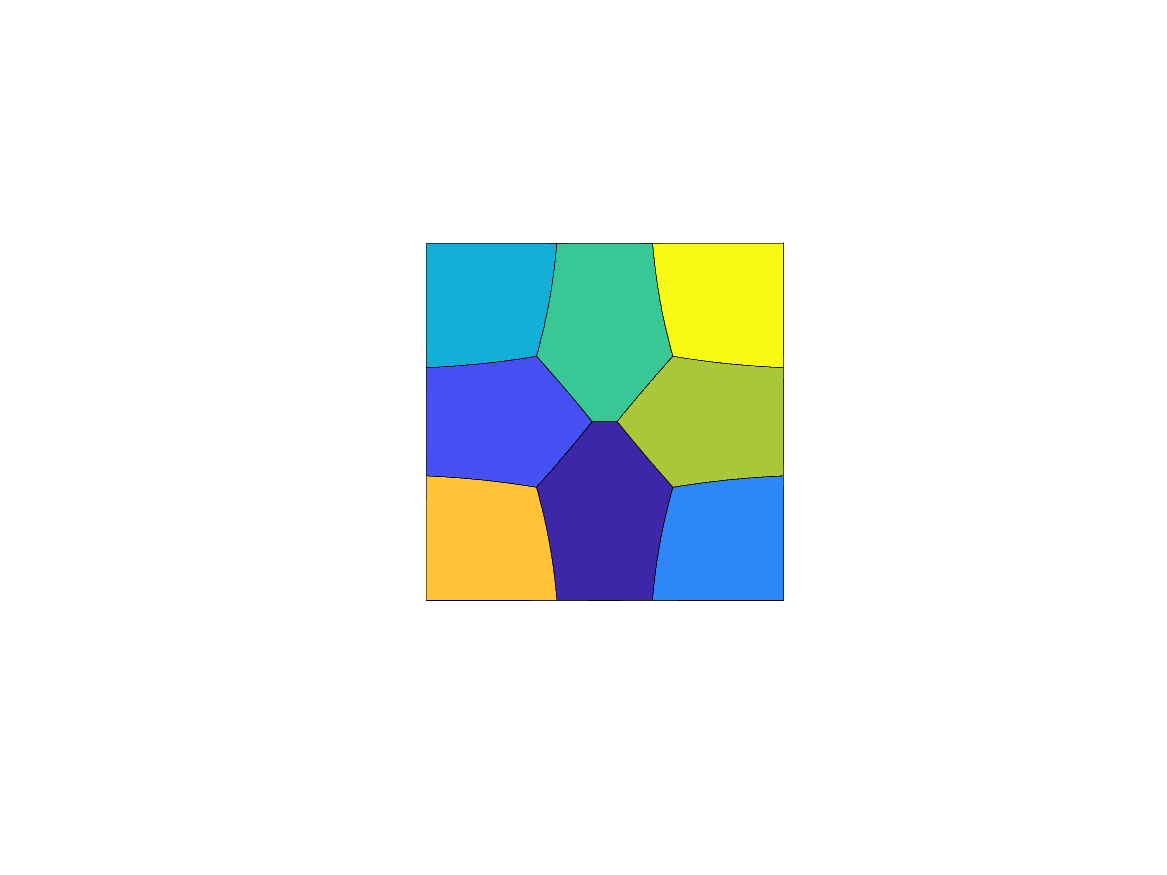}
\includegraphics[width = 0.1\textwidth, clip, trim = 7cm 4.5cm 6cm 4cm]{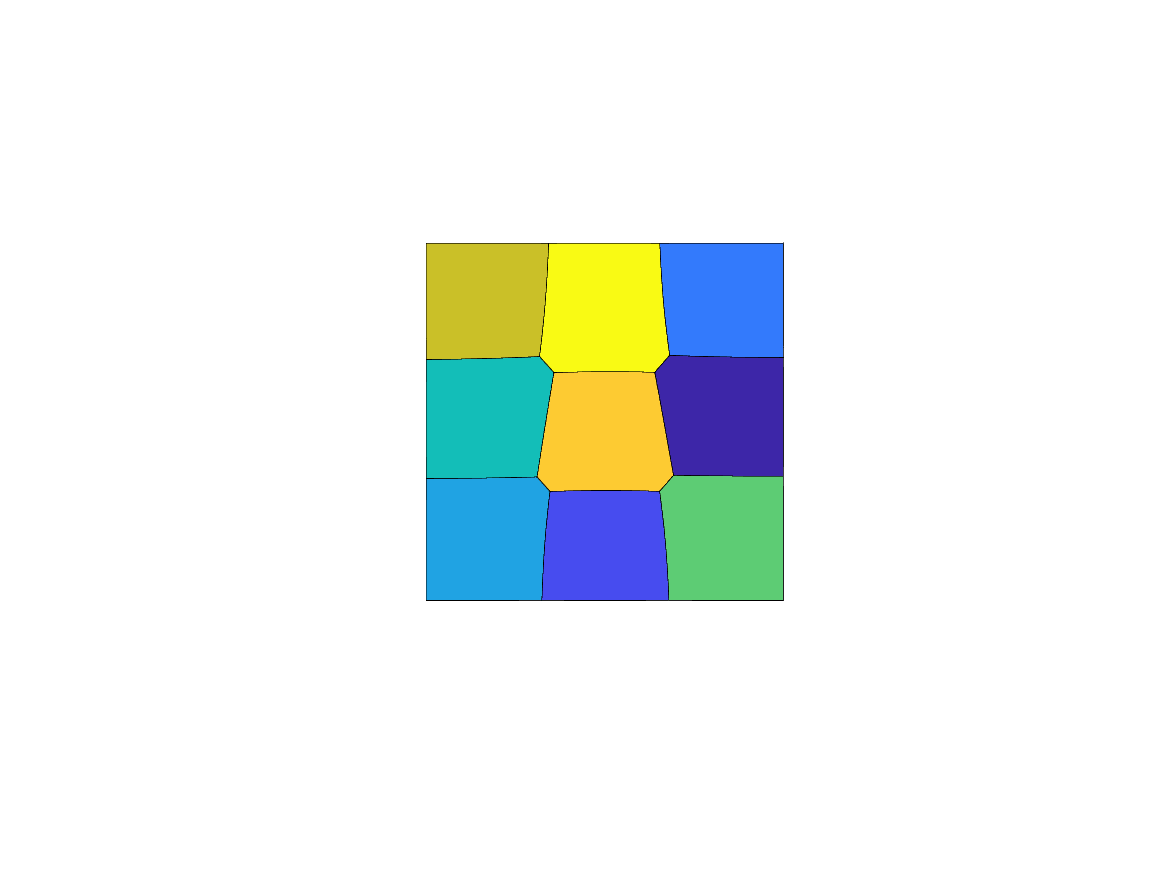}
\includegraphics[width = 0.1\textwidth, clip, trim = 7cm 4.5cm 6cm 4cm]{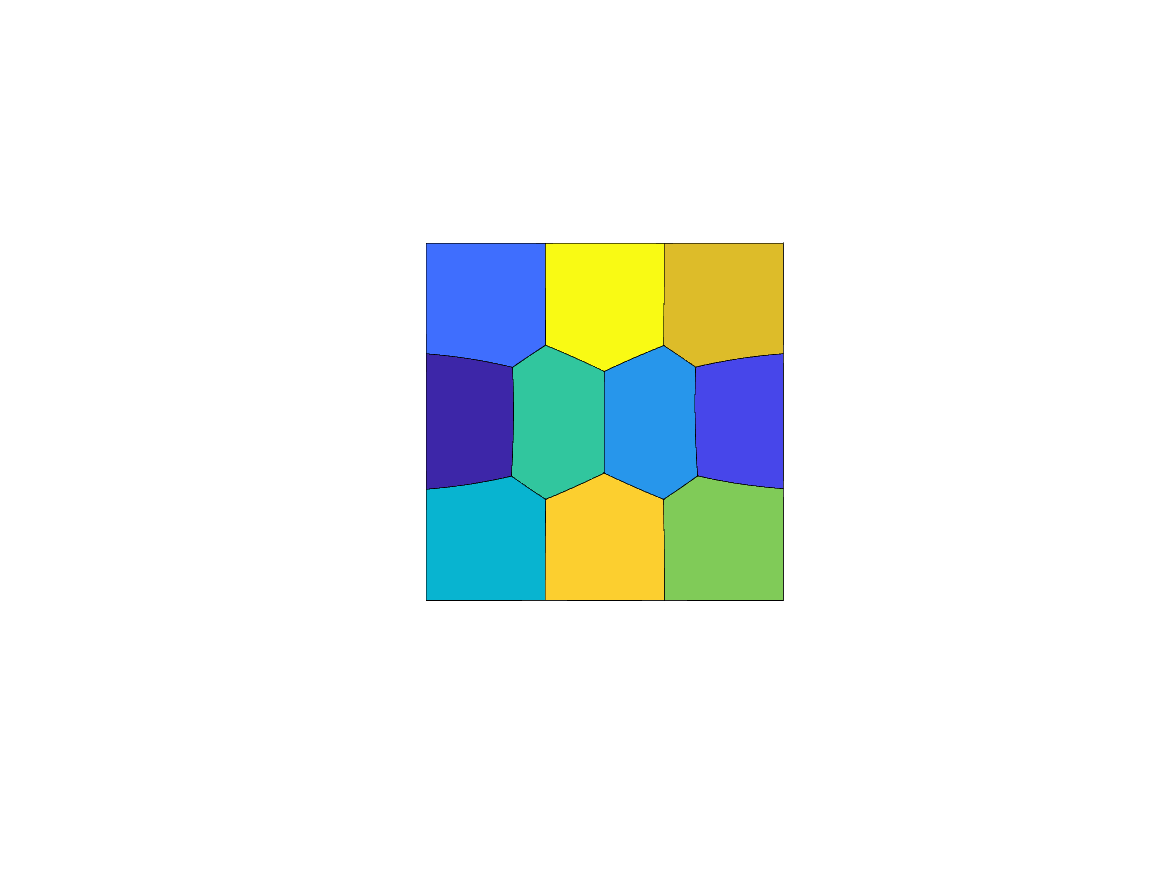}
\smallskip
\includegraphics[width = 0.1\textwidth, clip, trim = 7cm 5cm 6.5cm 4cm]{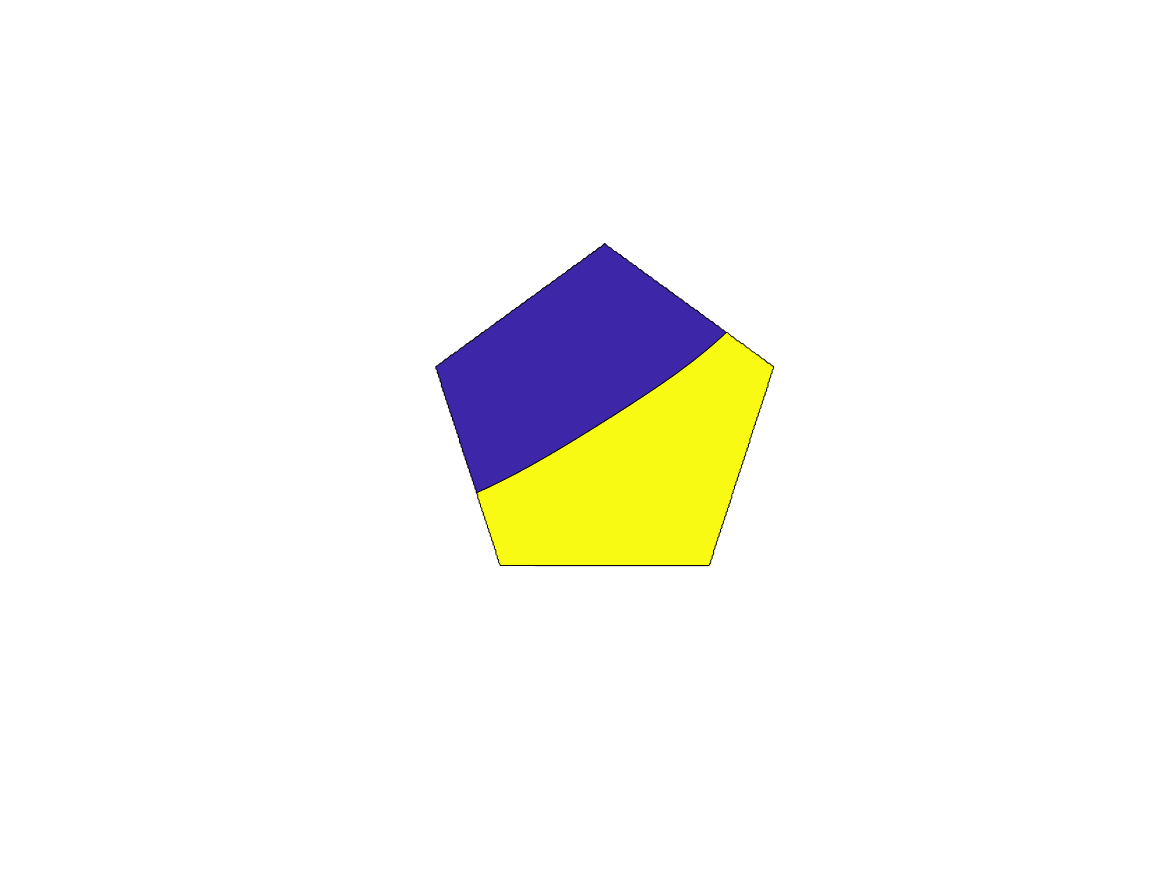}
\includegraphics[width = 0.1\textwidth, clip, trim = 7cm 5cm 6.5cm 4cm]{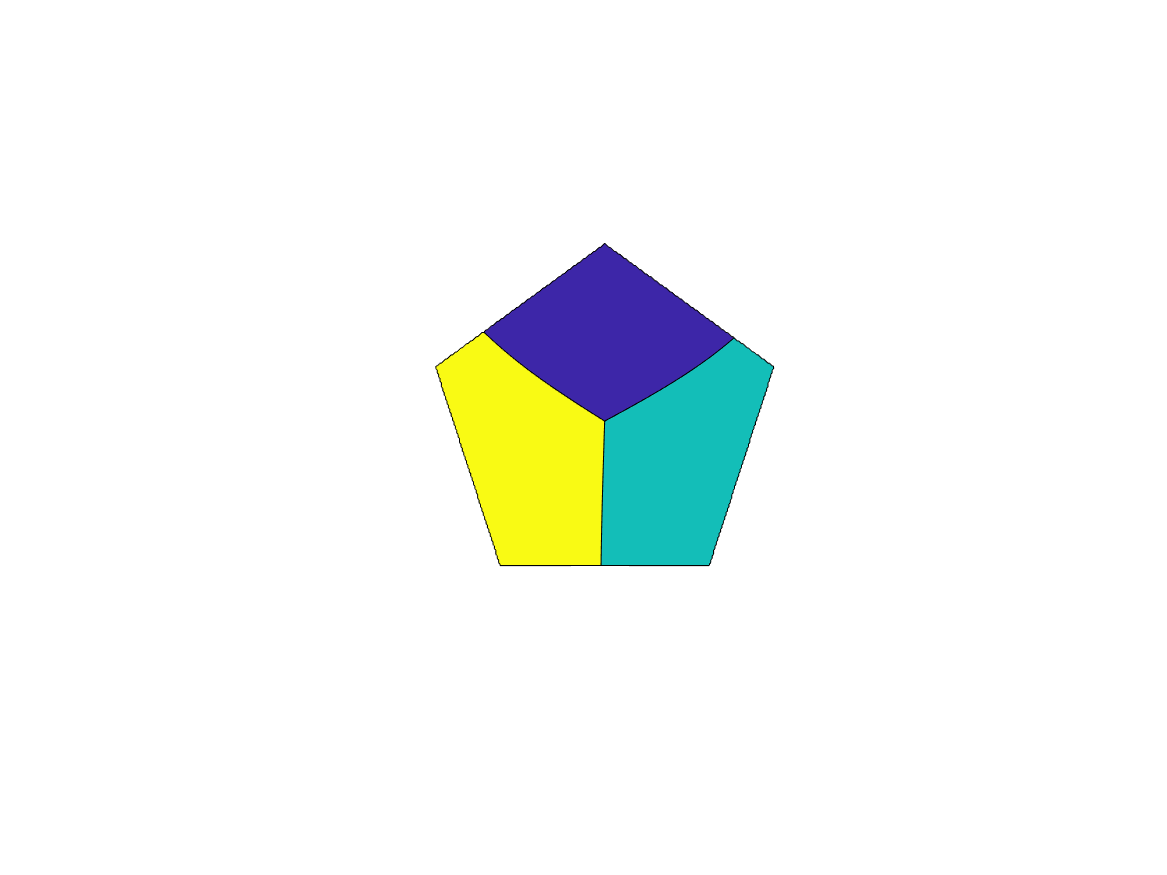}
\includegraphics[width = 0.1\textwidth, clip, trim = 7cm 5cm 6.5cm 4cm]{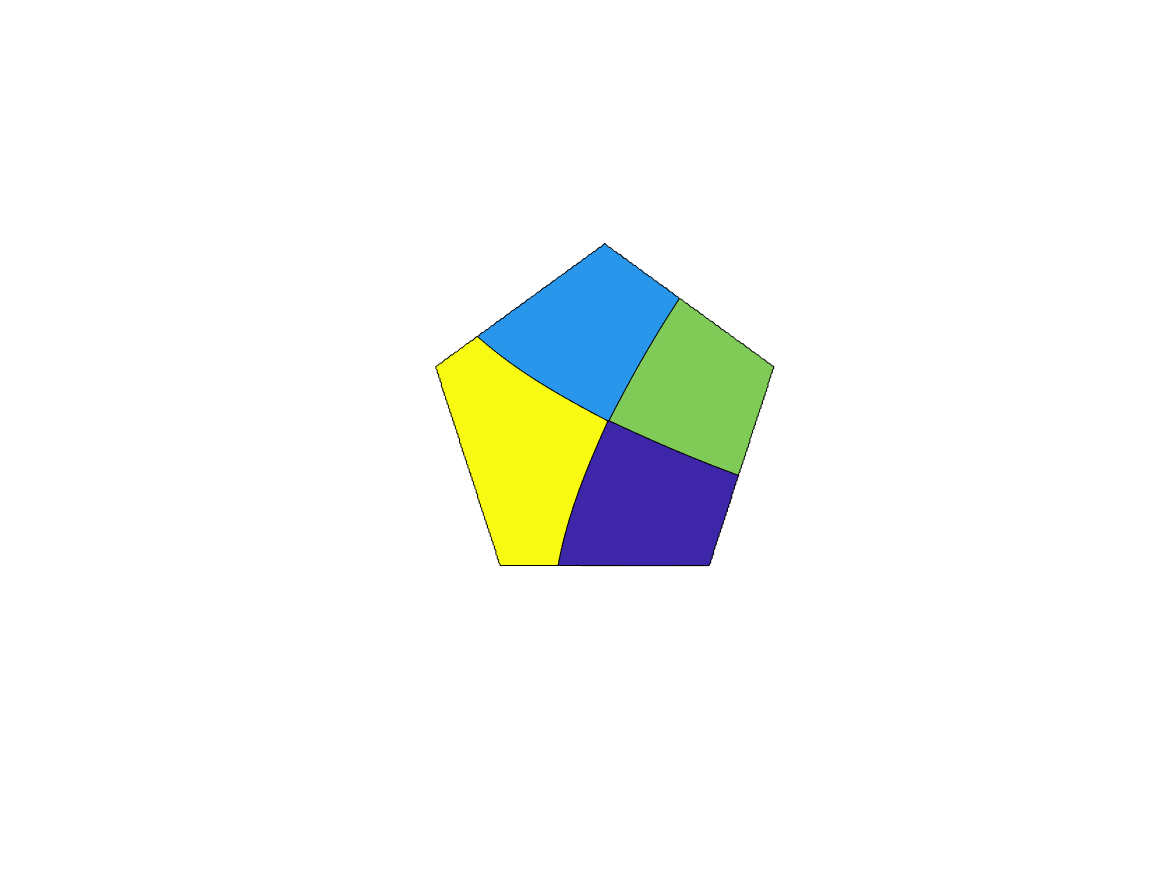}
\includegraphics[width = 0.1\textwidth, clip, trim = 7cm 5cm 6.5cm 4cm]{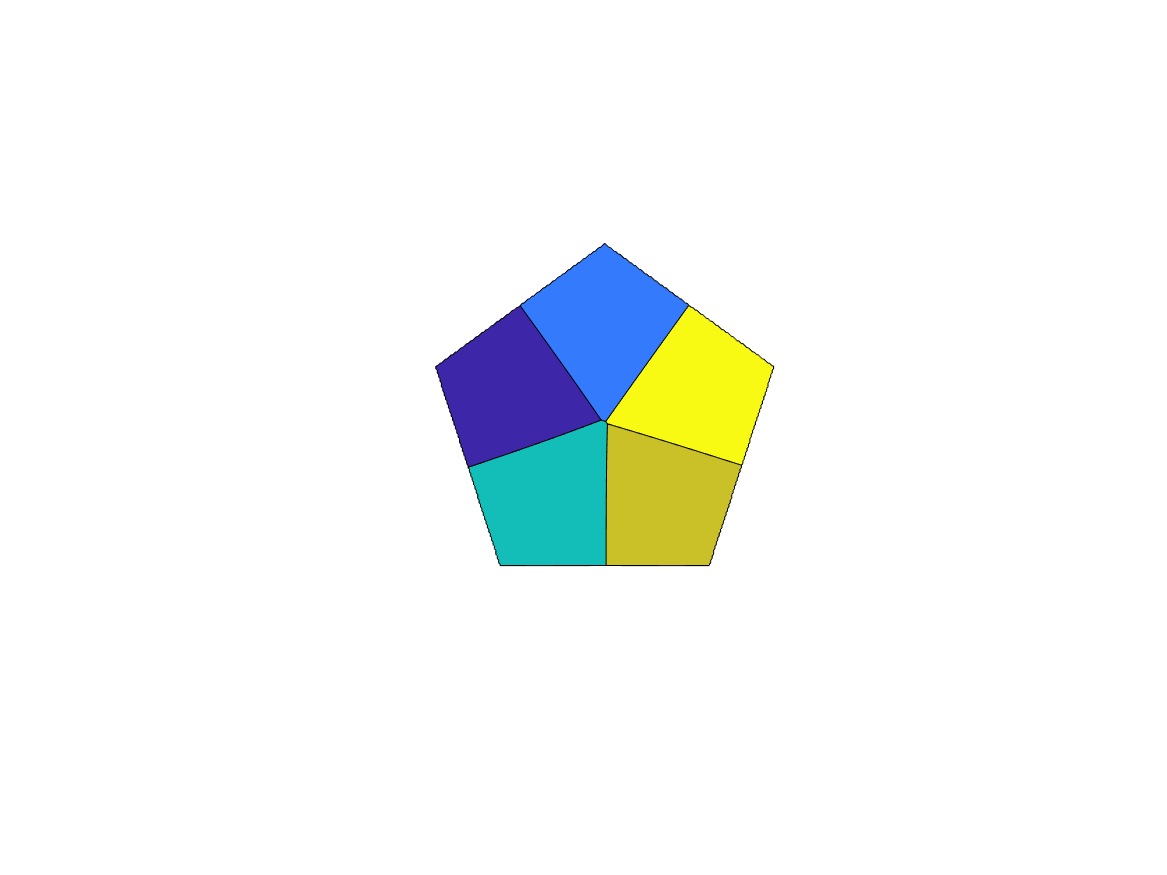}
\includegraphics[width = 0.1\textwidth, clip, trim = 7cm 5cm 6.5cm 4cm]{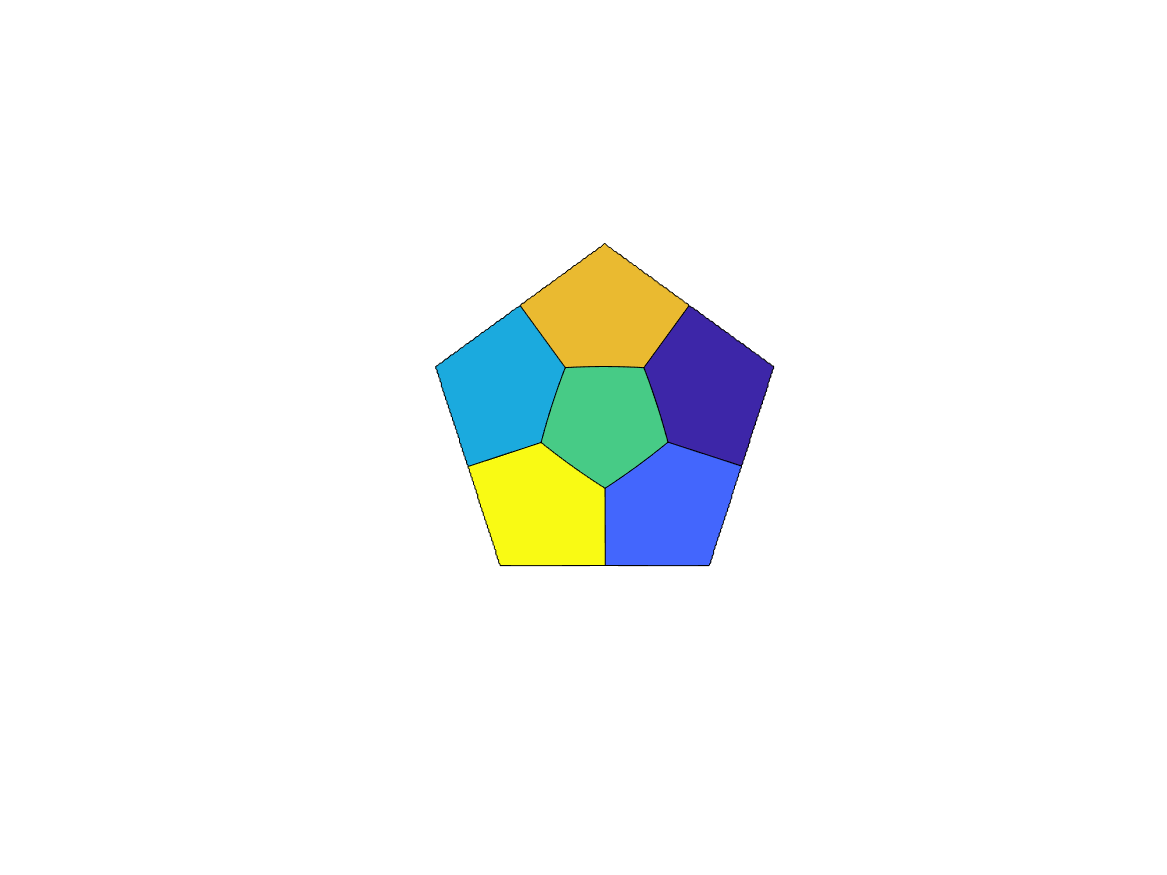}
\includegraphics[width = 0.1\textwidth, clip, trim = 7cm 5cm 6.5cm 4cm]{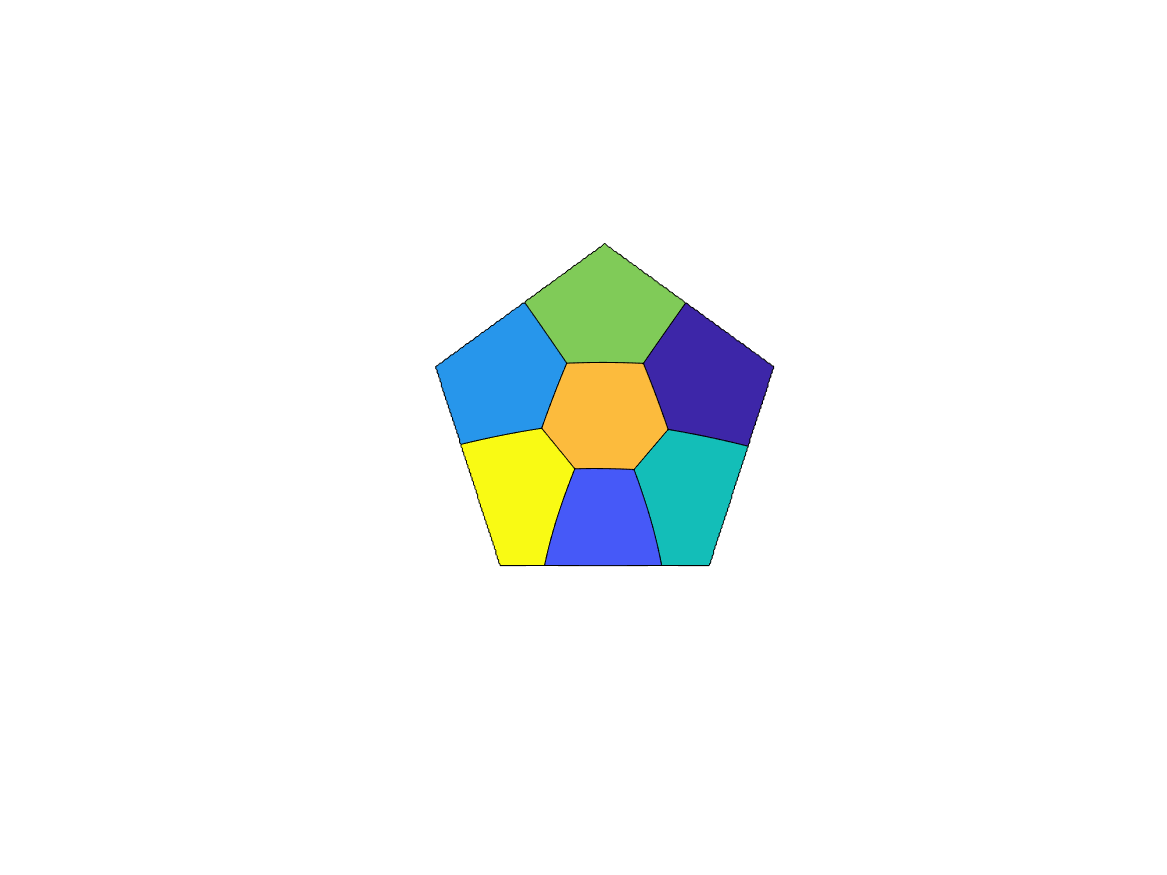}
\includegraphics[width = 0.1\textwidth, clip, trim = 7cm 5cm 6.5cm 4cm]{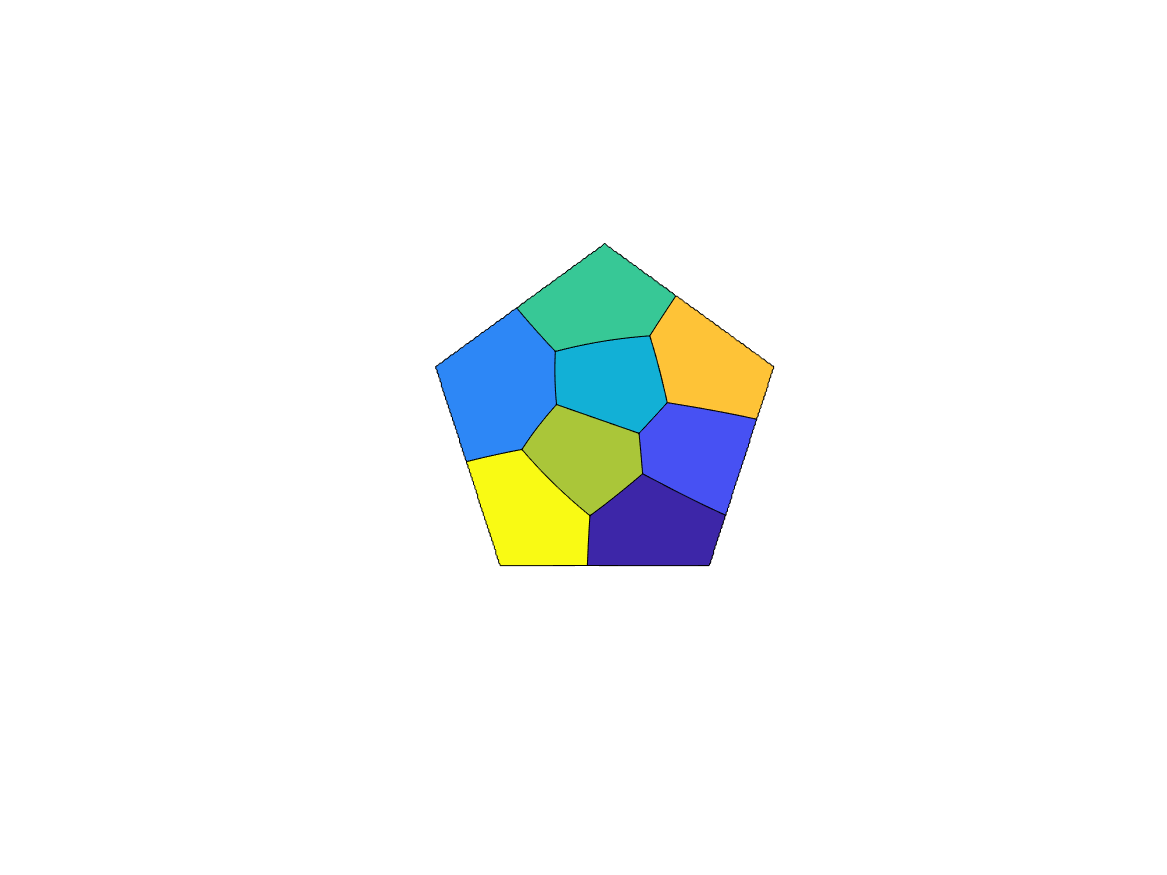}
\includegraphics[width = 0.1\textwidth, clip, trim = 7cm 5cm 6.5cm 4cm]{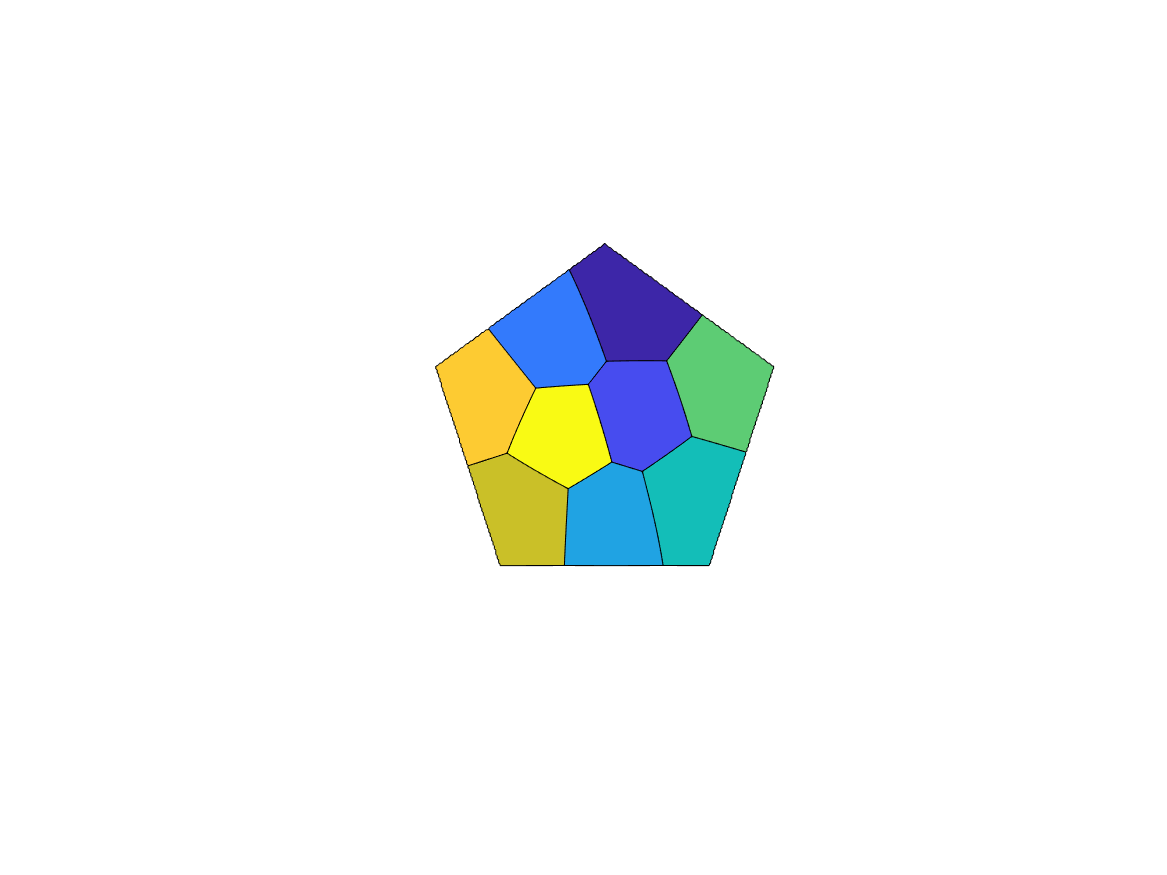}
\includegraphics[width = 0.1\textwidth, clip, trim = 7cm 5cm 6.5cm 4cm]{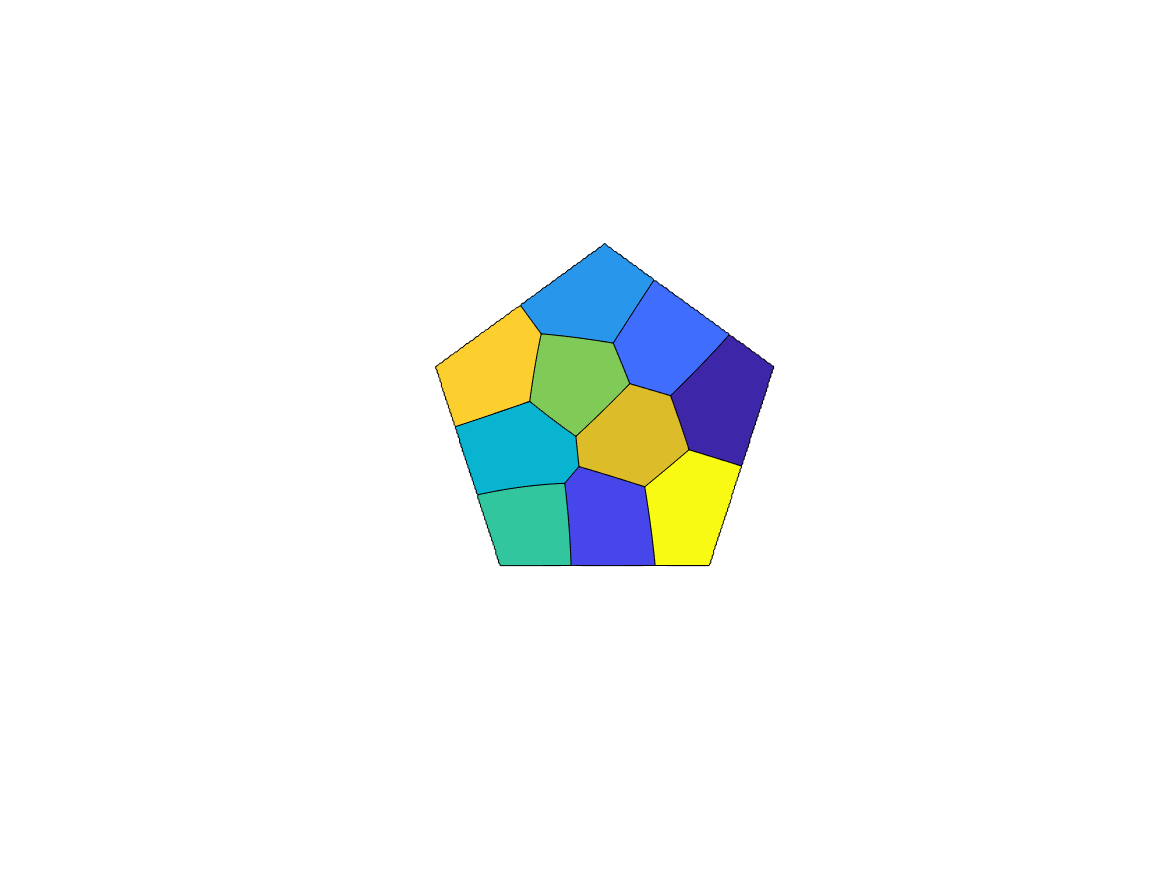}
\smallskip
\includegraphics[width = 0.1\textwidth, clip, trim = 7cm 5cm 6.5cm 4cm]{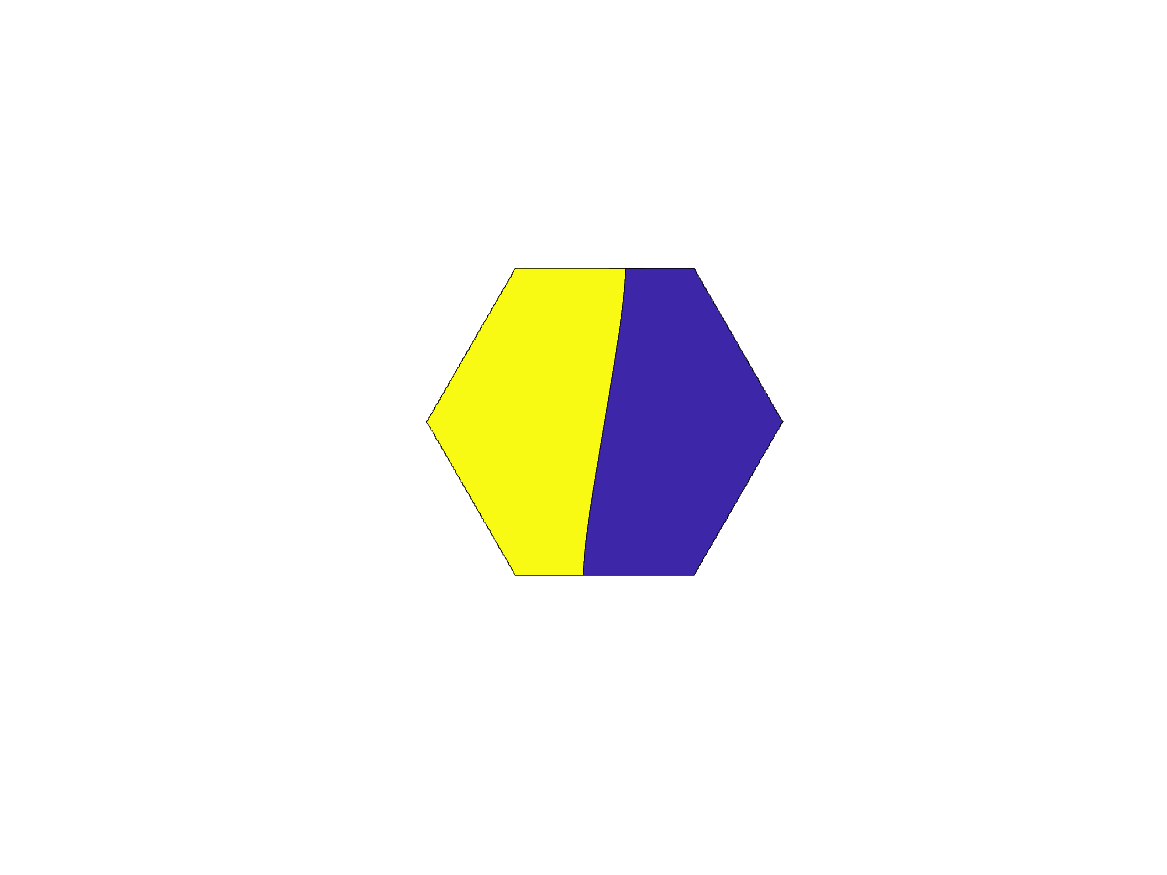}
\includegraphics[width = 0.1\textwidth, clip, trim = 7cm 5cm 6.5cm 4cm]{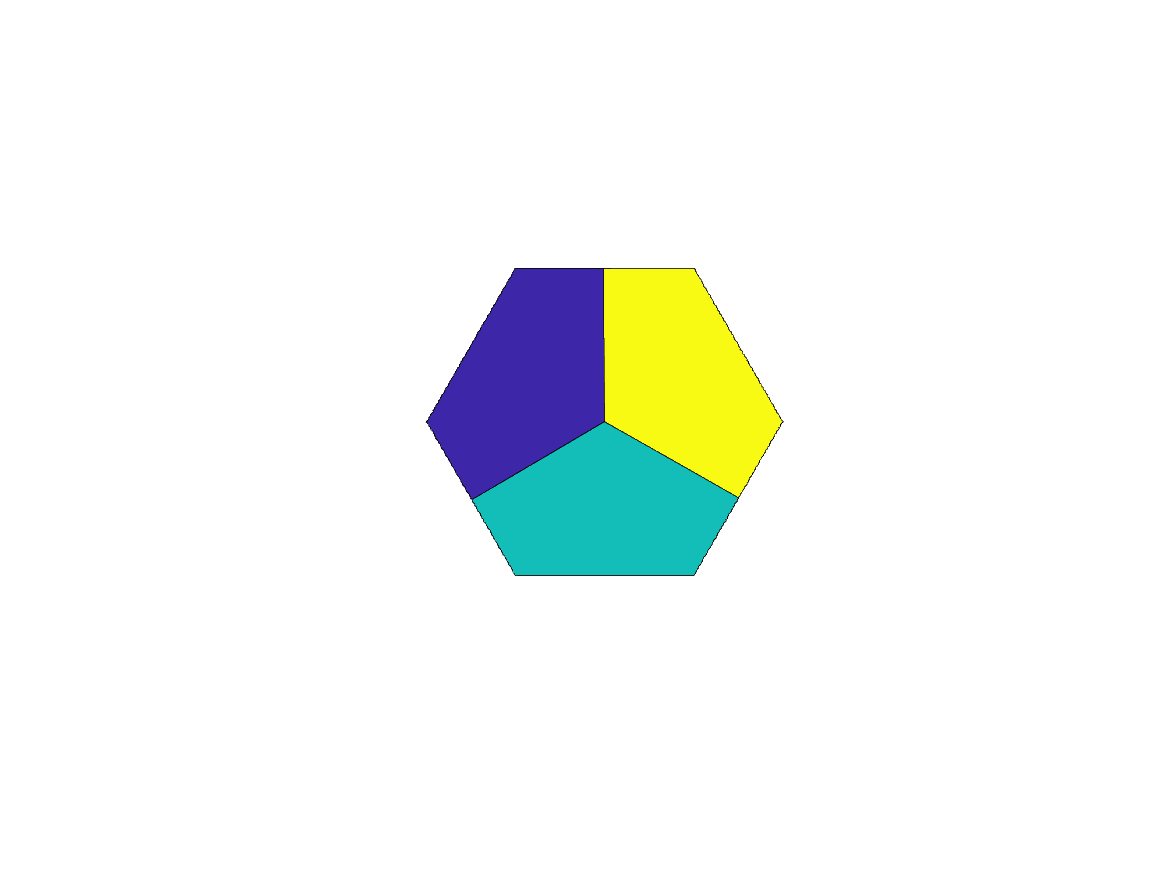}
\includegraphics[width = 0.1\textwidth, clip, trim = 7cm 5cm 6.5cm 4cm]{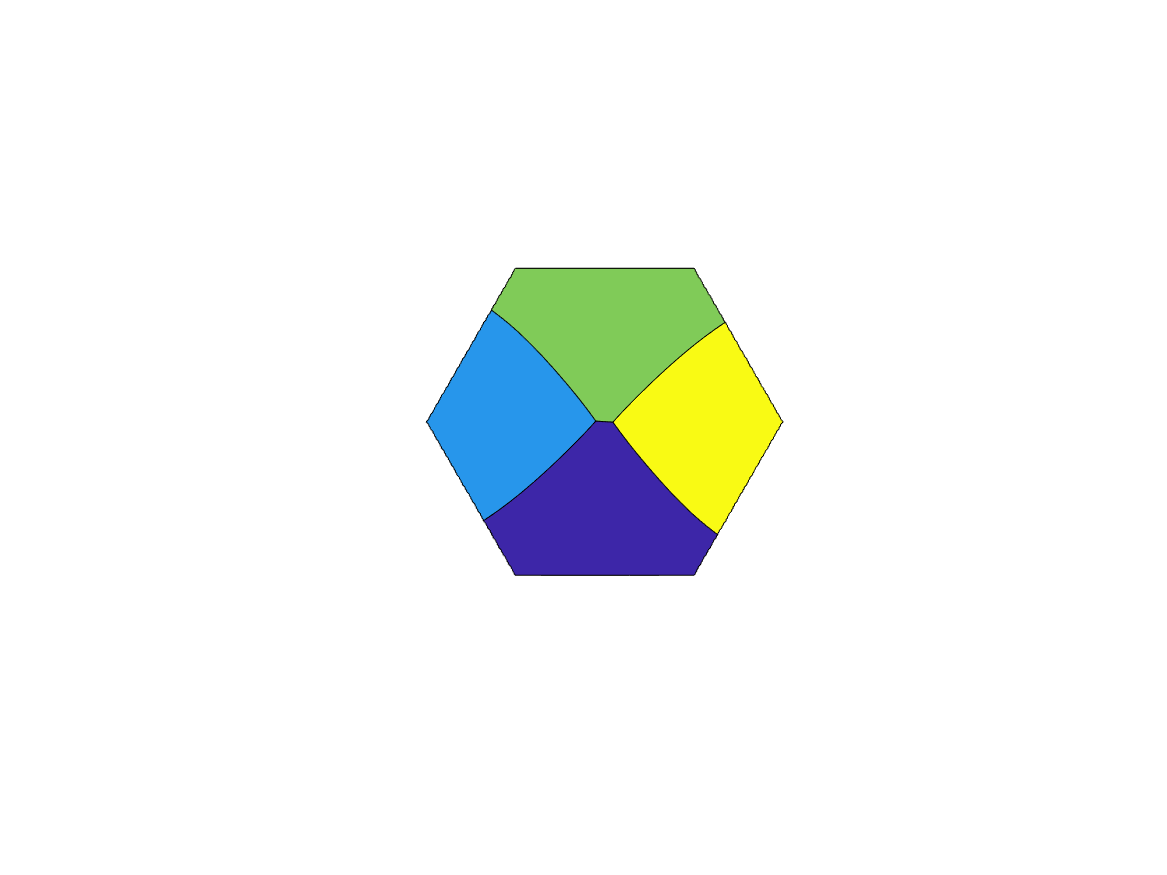}
\includegraphics[width = 0.1\textwidth, clip, trim = 7cm 5cm 6.5cm 4cm]{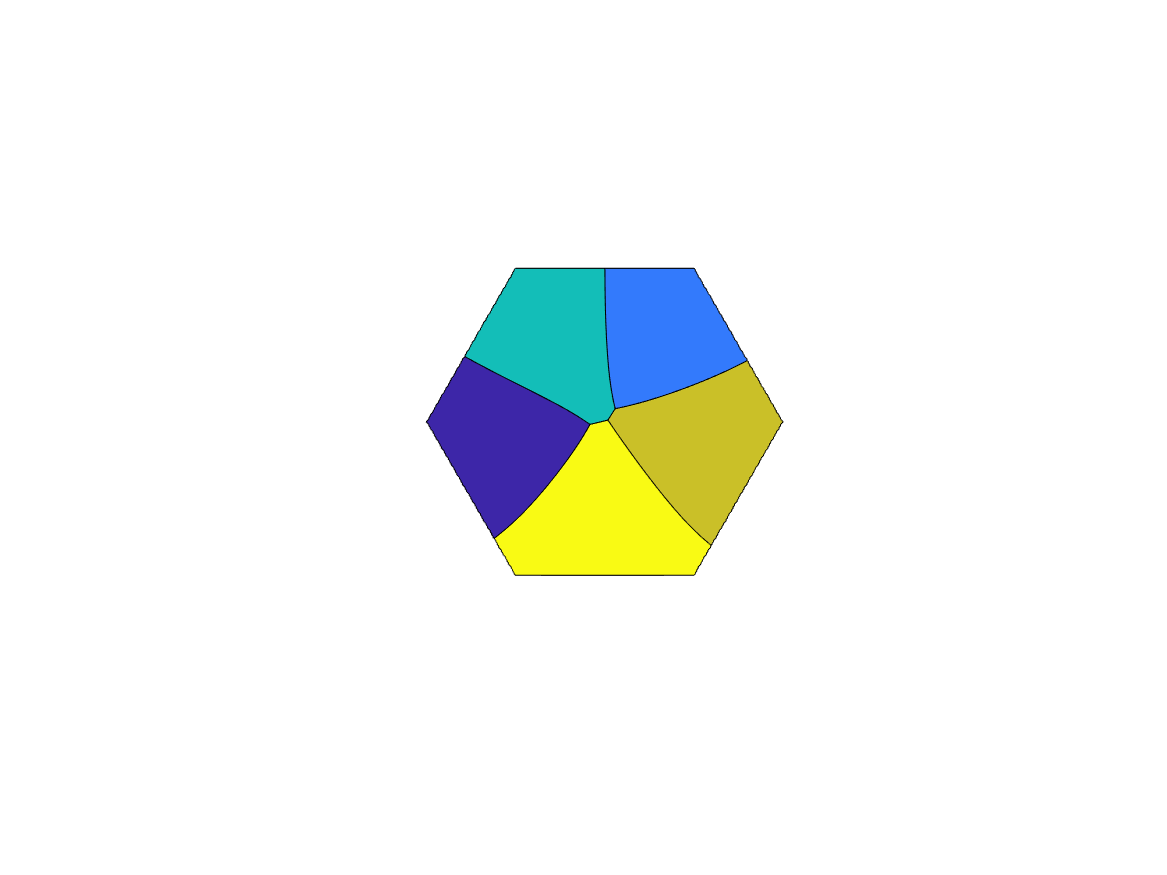}
\includegraphics[width = 0.1\textwidth, clip, trim = 7cm 5cm 6.5cm 4cm]{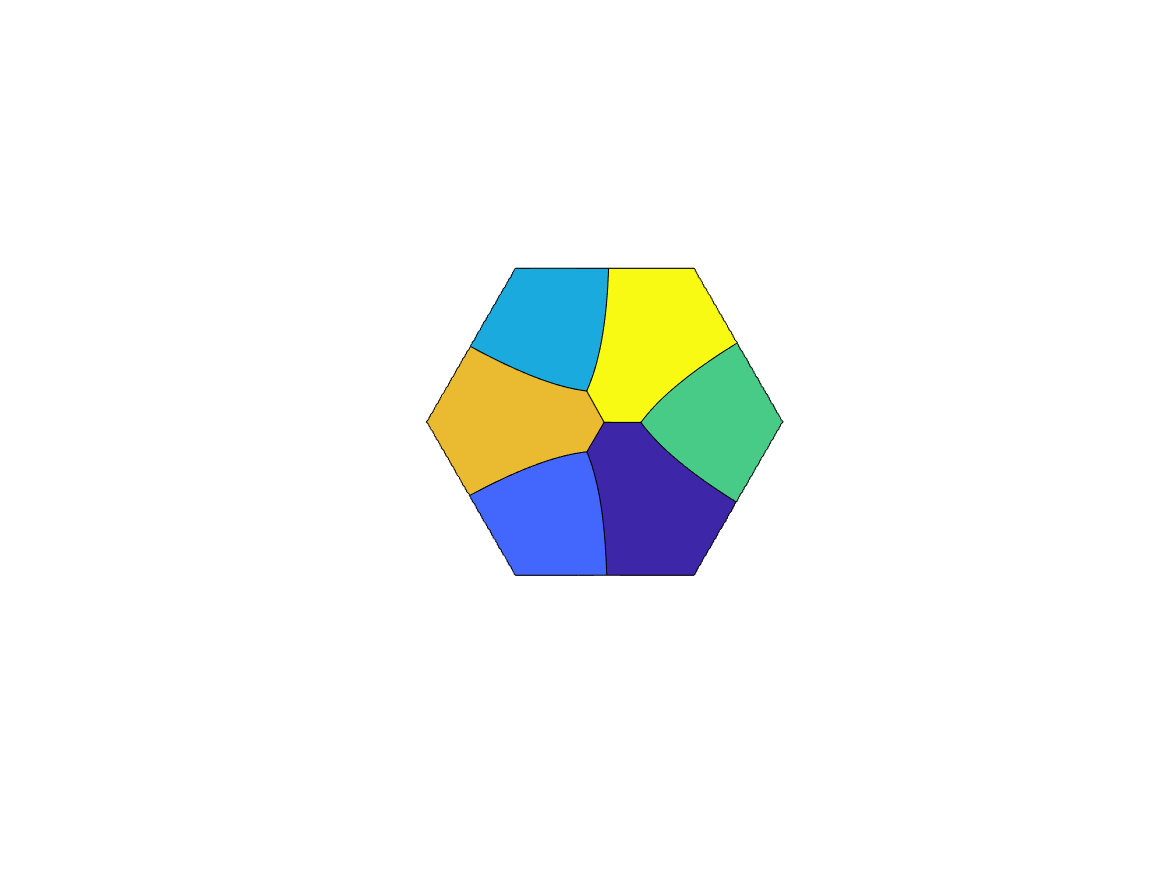}
\includegraphics[width = 0.1\textwidth, clip, trim = 7cm 5cm 6.5cm 4cm]{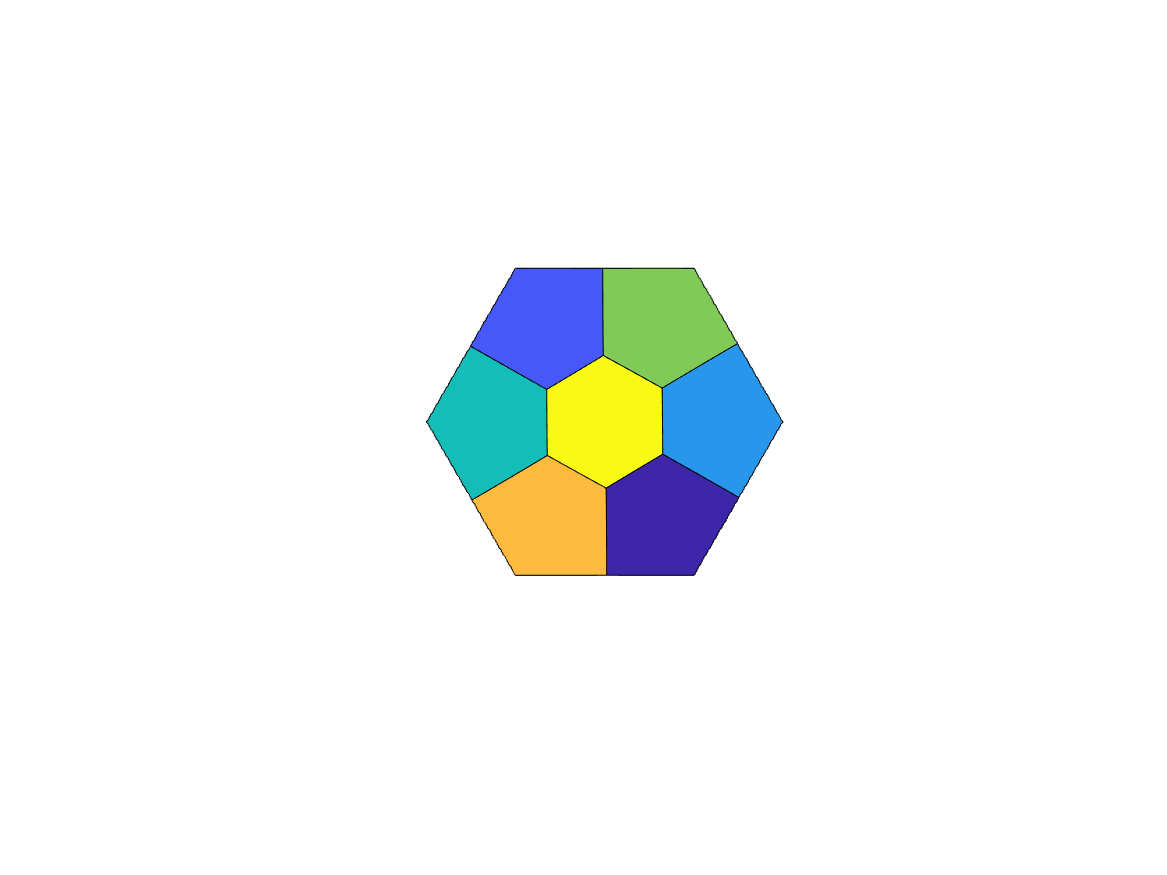}
\includegraphics[width = 0.1\textwidth, clip, trim = 7cm 5cm 6.5cm 4cm]{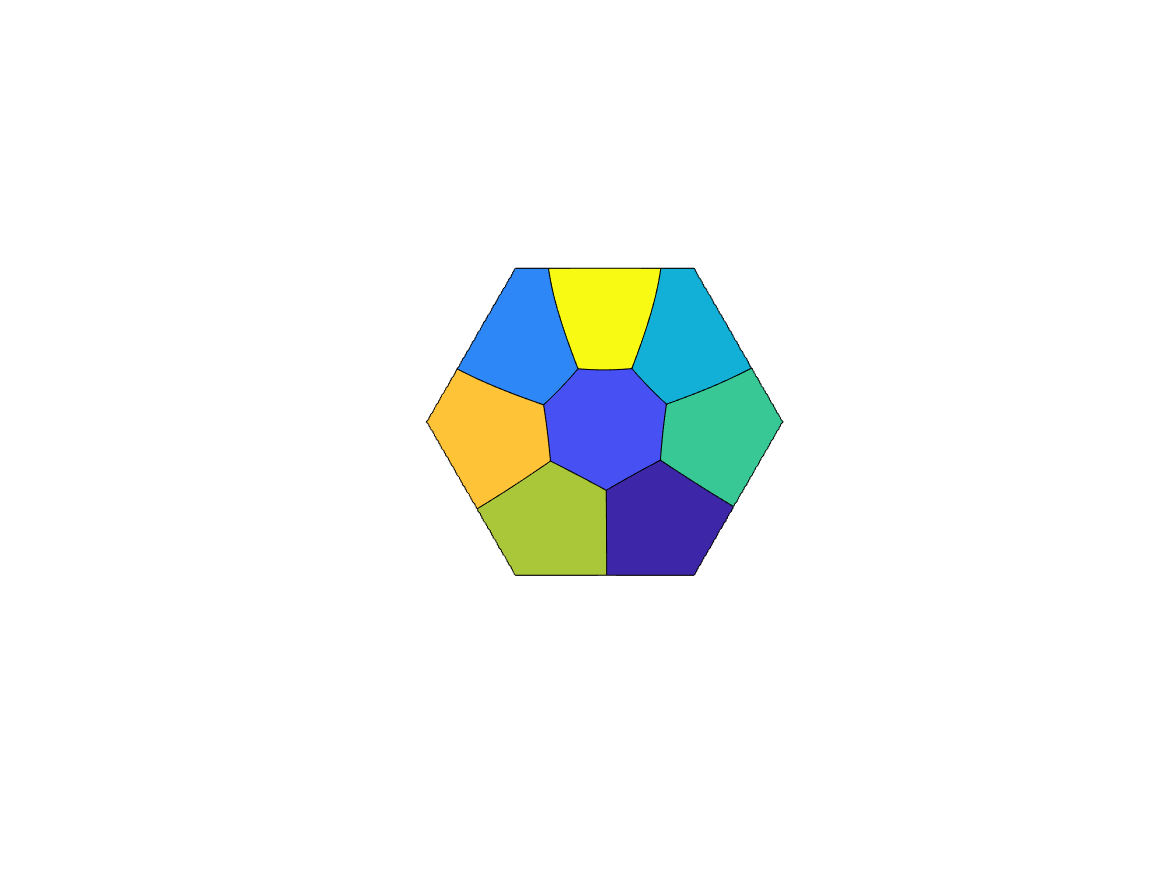}
\includegraphics[width = 0.1\textwidth, clip, trim = 7cm 5cm 6.5cm 4cm]{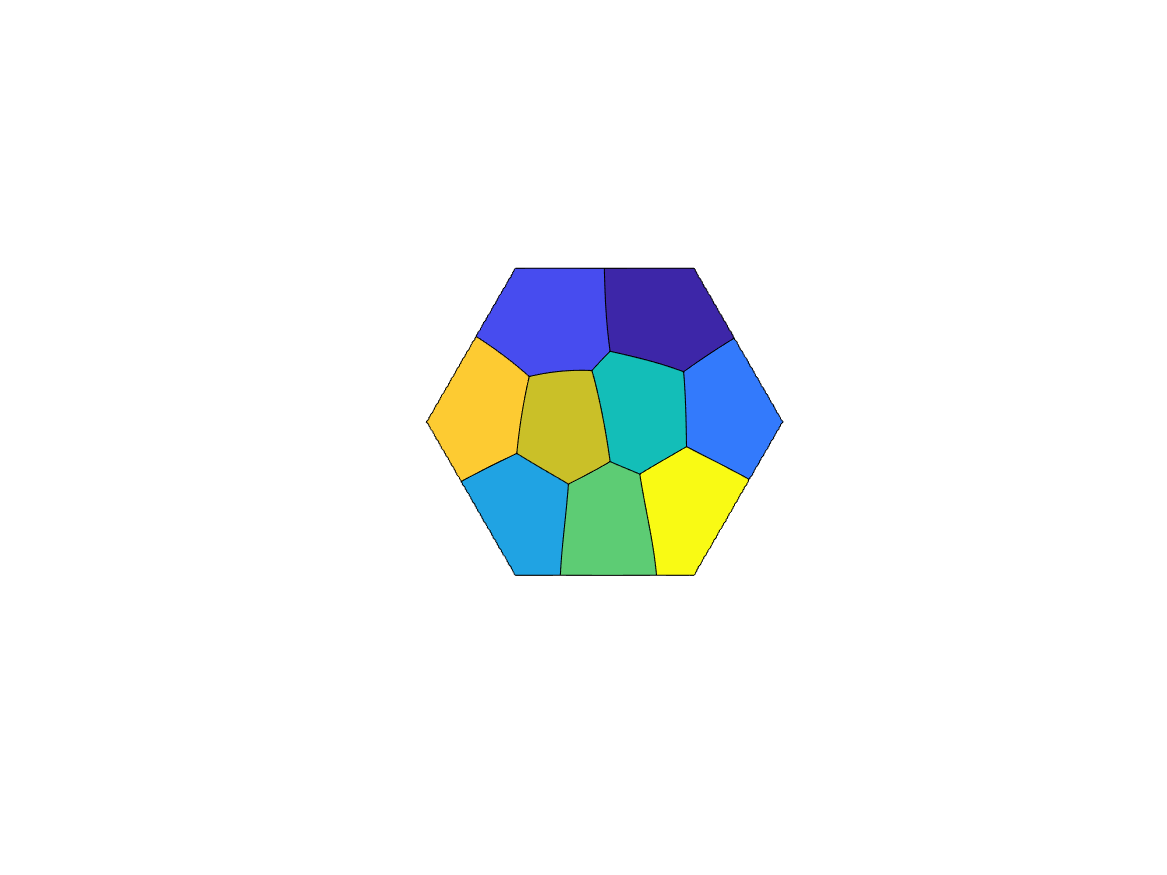}
\includegraphics[width = 0.1\textwidth, clip, trim = 7cm 5cm 6.5cm 4cm]{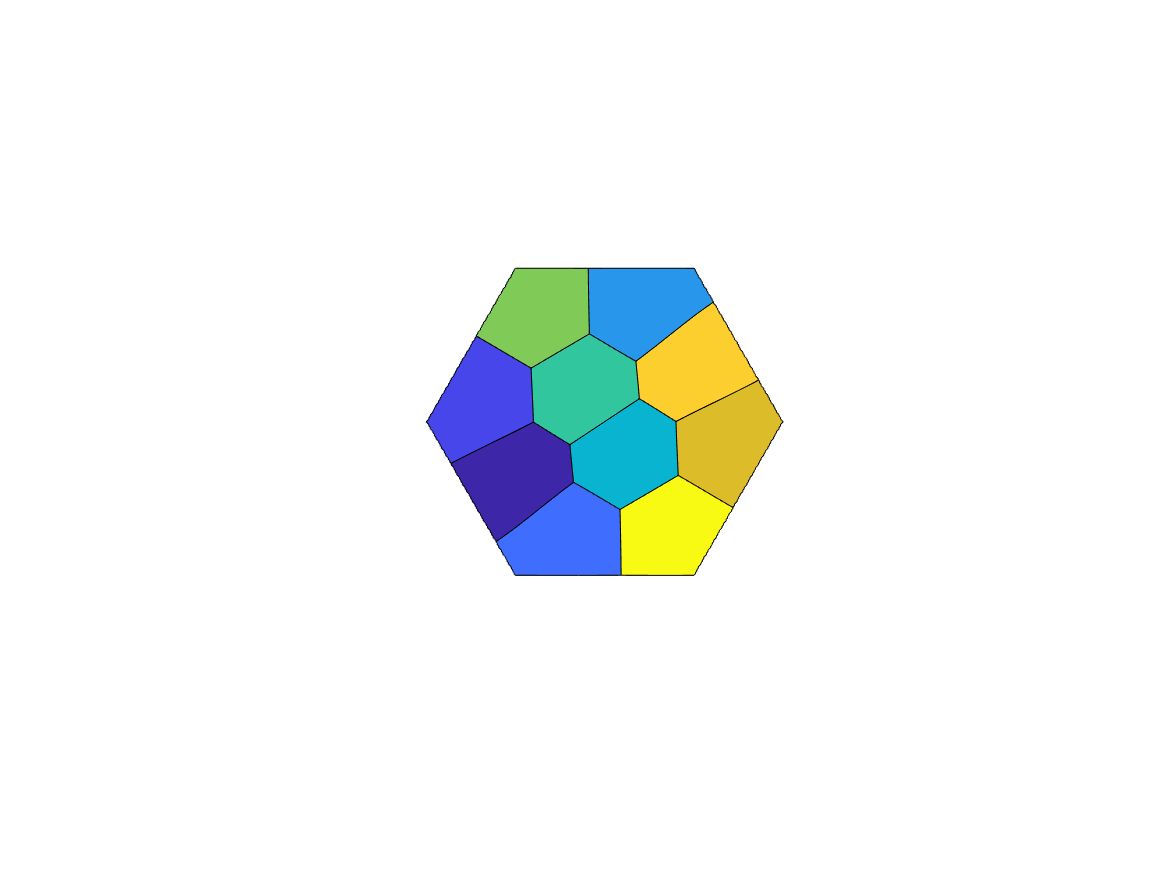}\\
\smallskip
\includegraphics[width = 0.1\textwidth, clip, trim = 7cm 4.5cm 6cm 4cm]{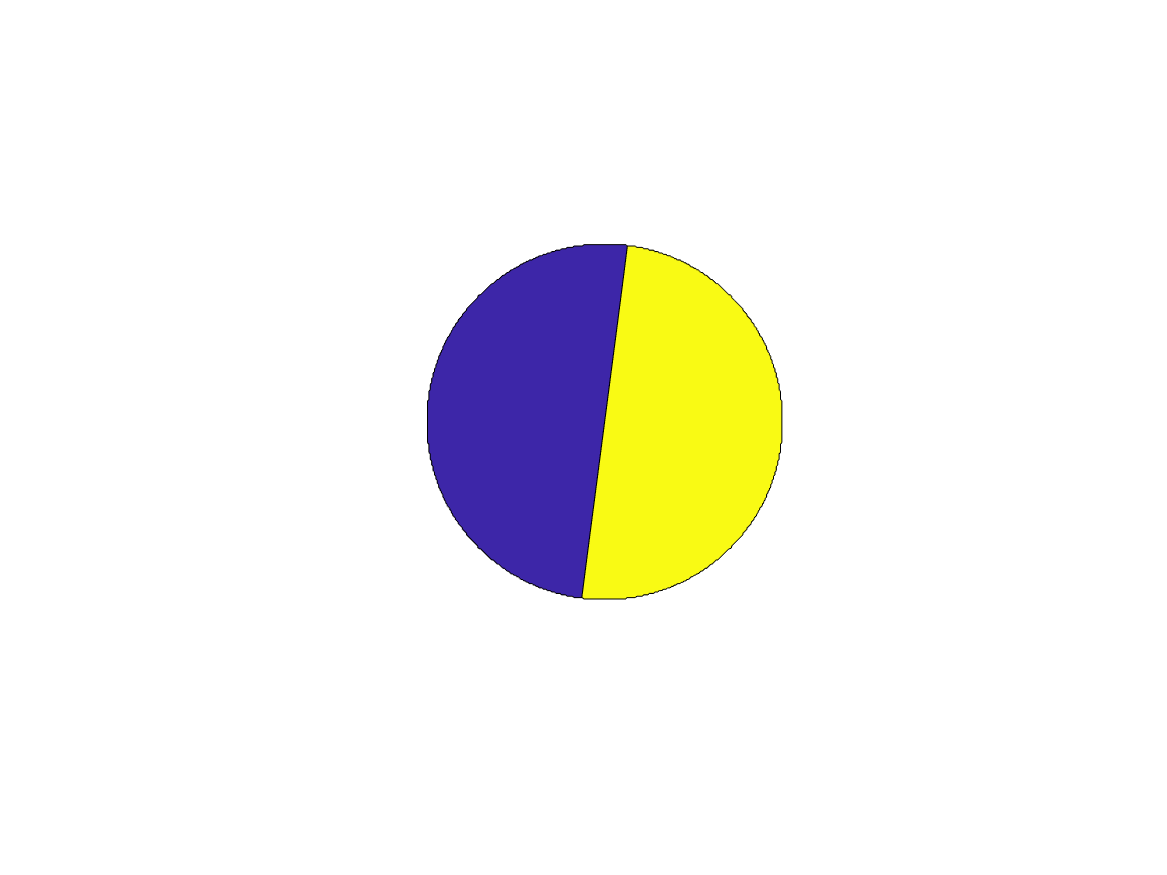}
\includegraphics[width = 0.1\textwidth, clip, trim = 7cm 4.5cm 6cm 4cm]{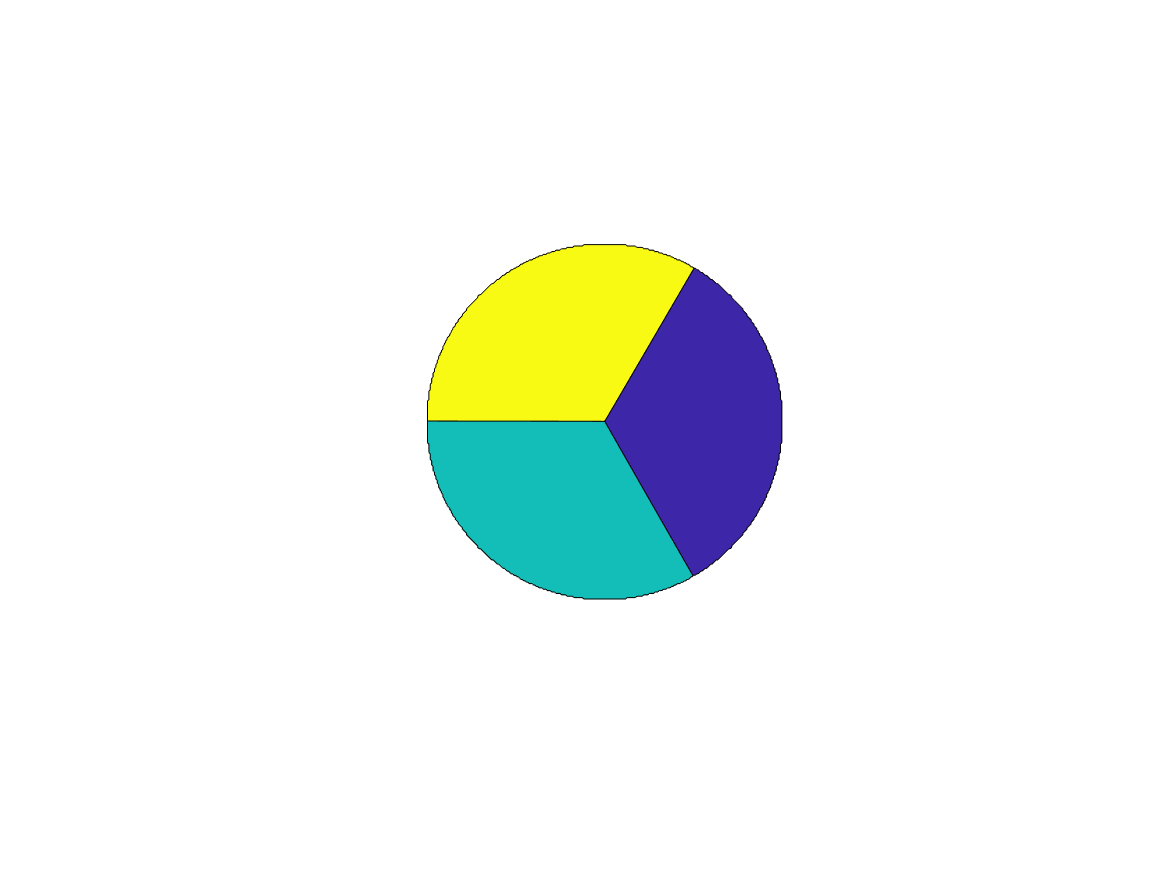}
\includegraphics[width = 0.1\textwidth, clip, trim = 7cm 4.5cm 6cm 4cm]{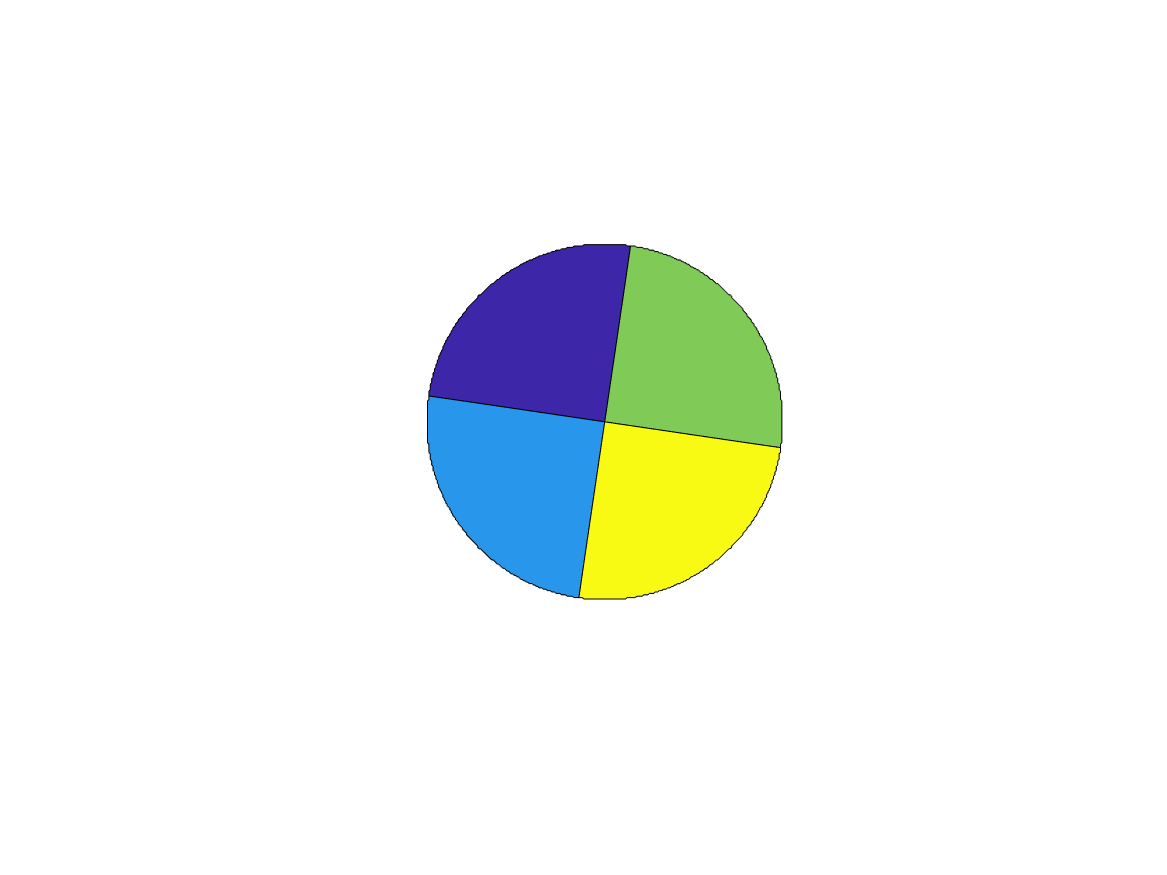}
\includegraphics[width = 0.1\textwidth, clip, trim = 7cm 4.5cm 6cm 4cm]{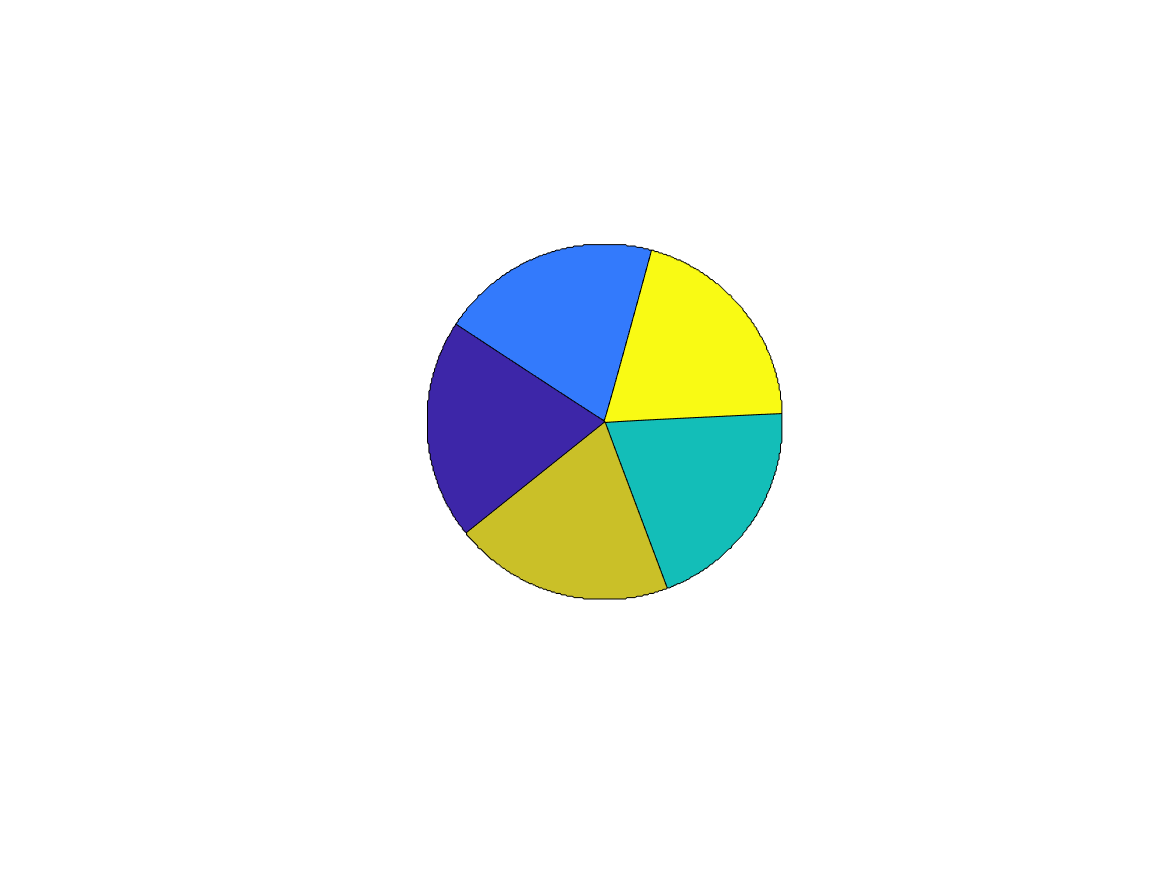}
\includegraphics[width = 0.1\textwidth, clip, trim = 7cm 4.5cm 6cm 4cm]{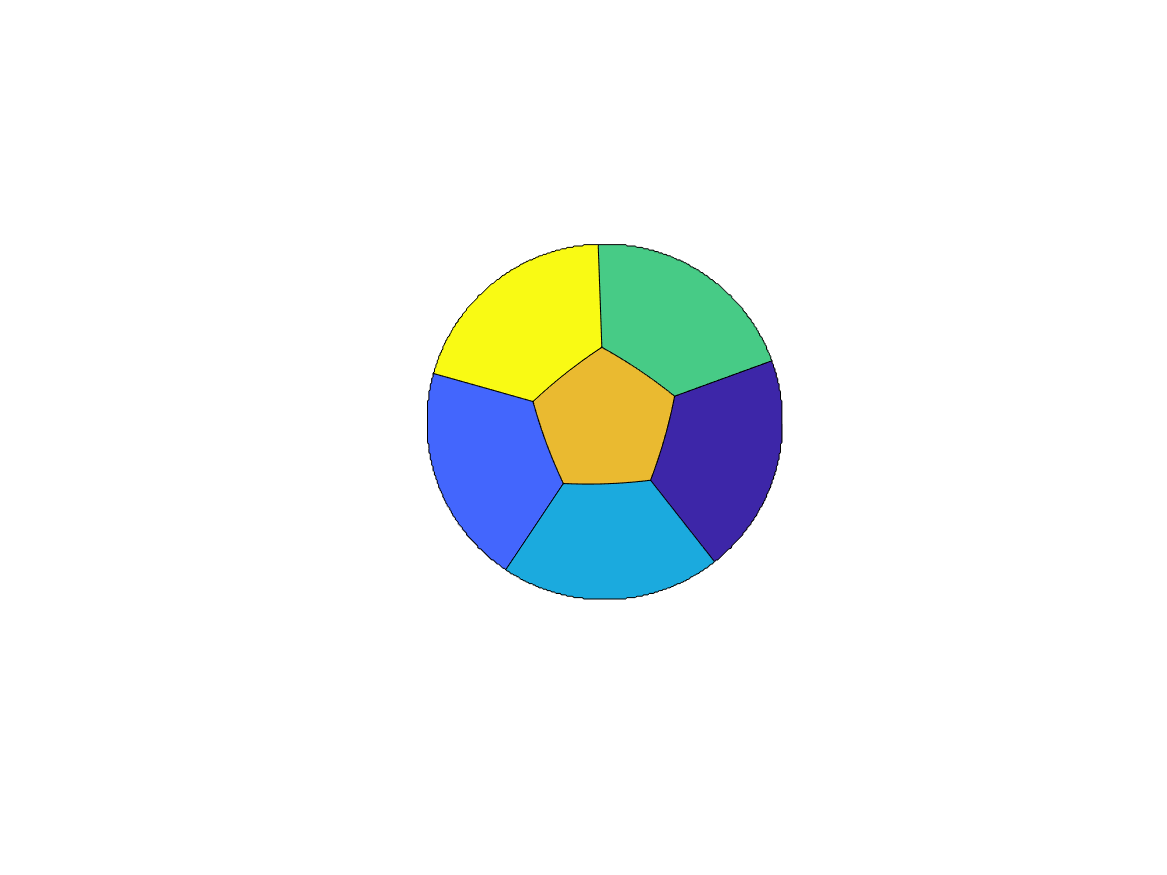}
\includegraphics[width = 0.1\textwidth, clip, trim = 7cm 4.5cm 6cm 4cm]{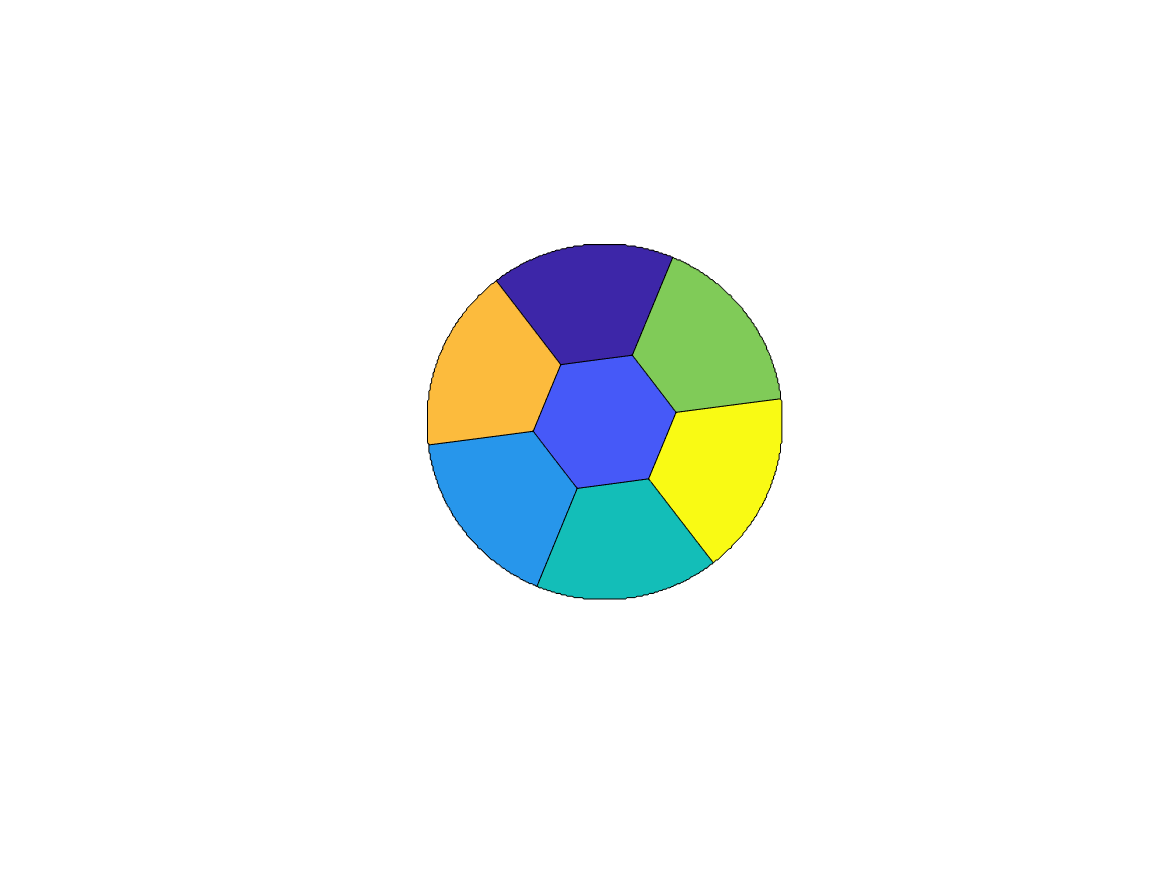}
\includegraphics[width = 0.1\textwidth, clip, trim = 7cm 4.5cm 6cm 4cm]{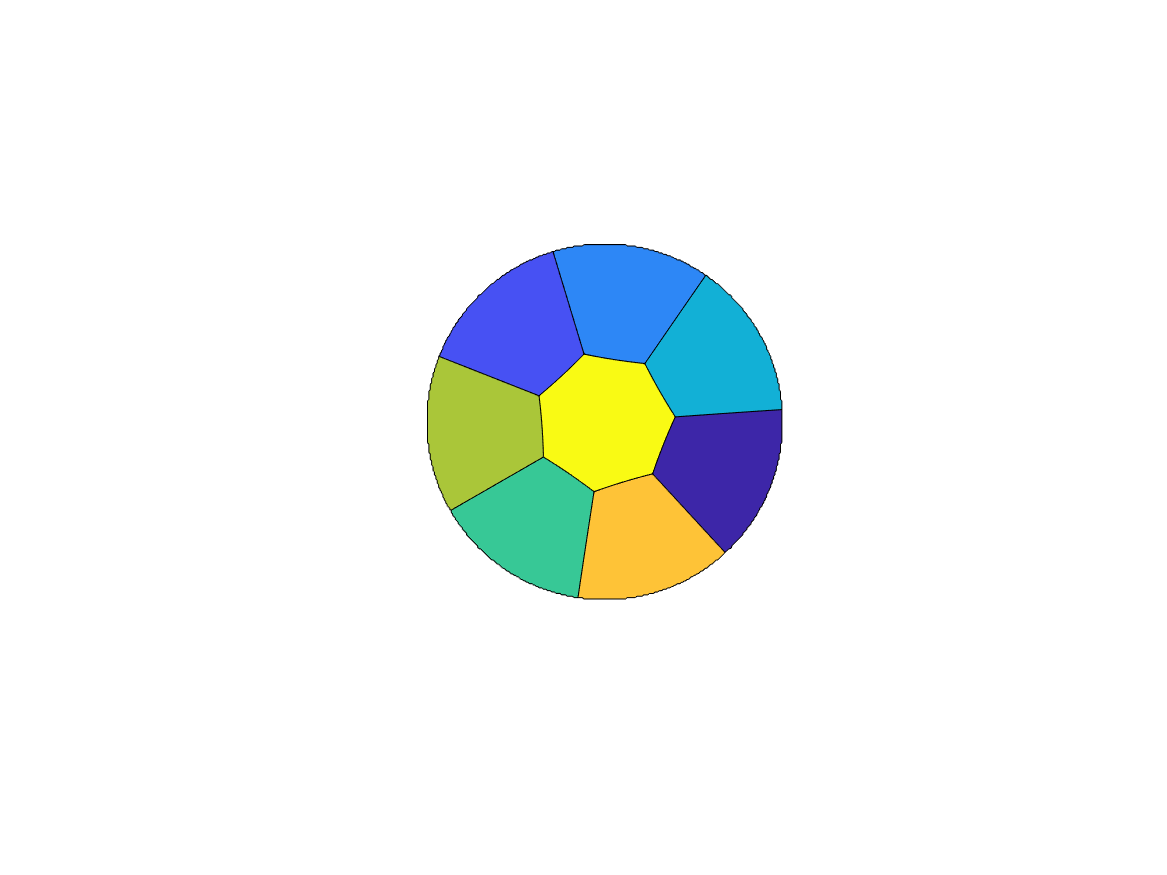}
\includegraphics[width = 0.1\textwidth, clip, trim = 7cm 4.5cm 6cm 4cm]{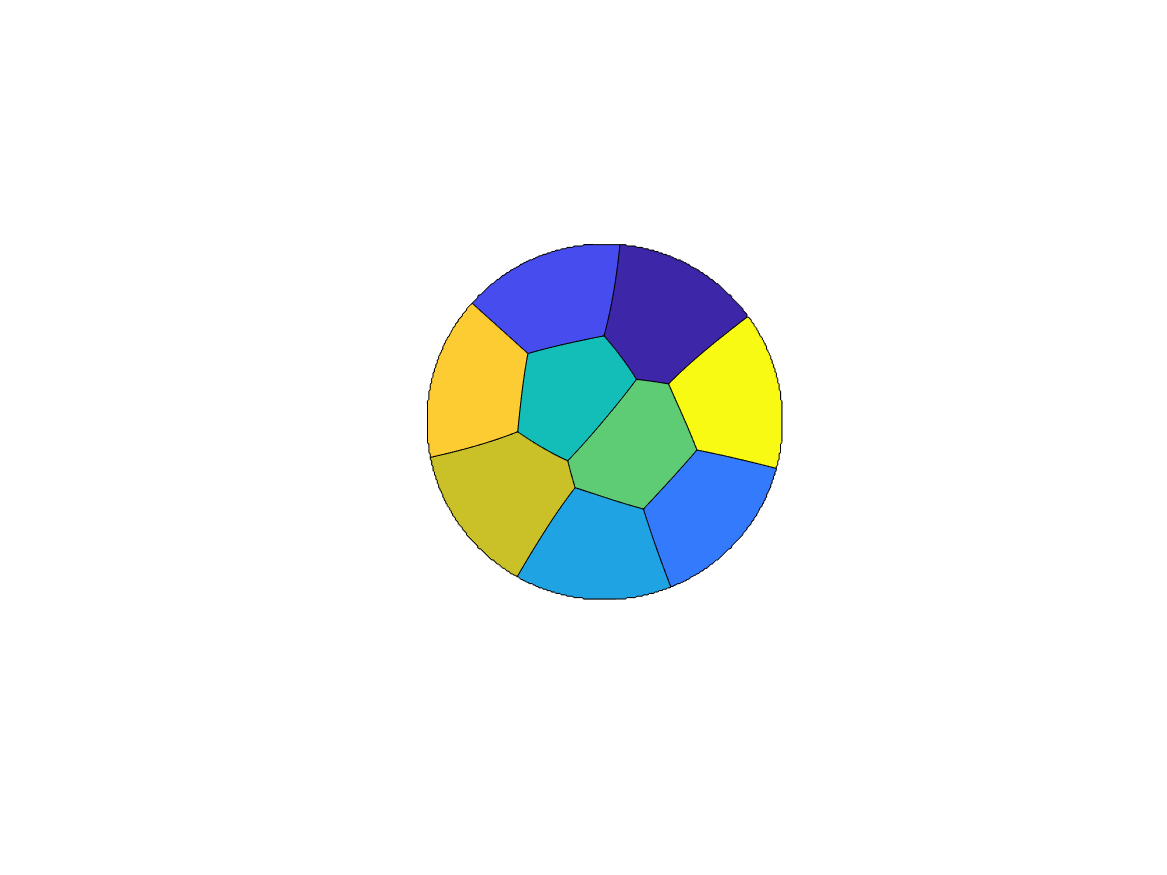}
\includegraphics[width = 0.1\textwidth, clip, trim = 7cm 4.5cm 6cm 4cm]{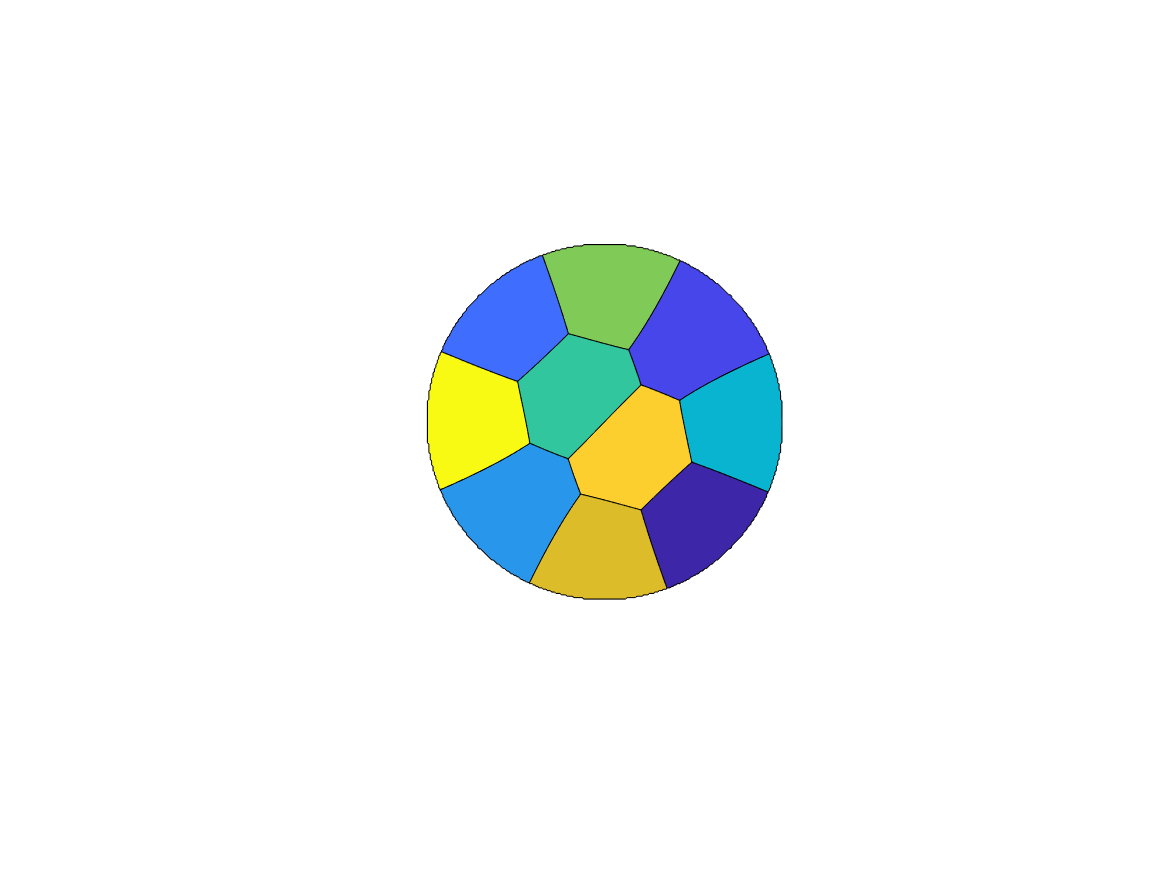}\\
\smallskip
\includegraphics[width = 0.1\textwidth, clip, trim = 6.5cm 3.5cm 6cm 4cm]{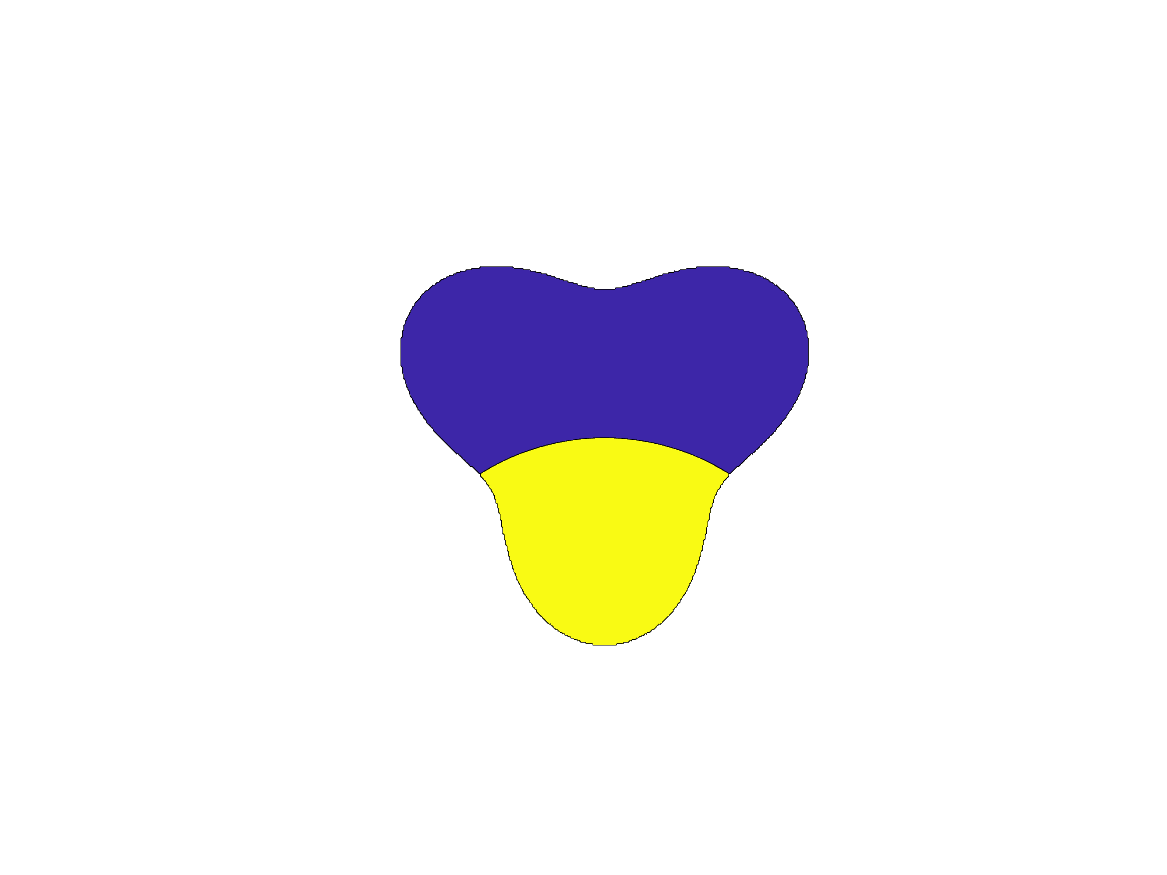}
\includegraphics[width = 0.1\textwidth, clip, trim = 6.5cm 3.5cm 6cm 4cm]{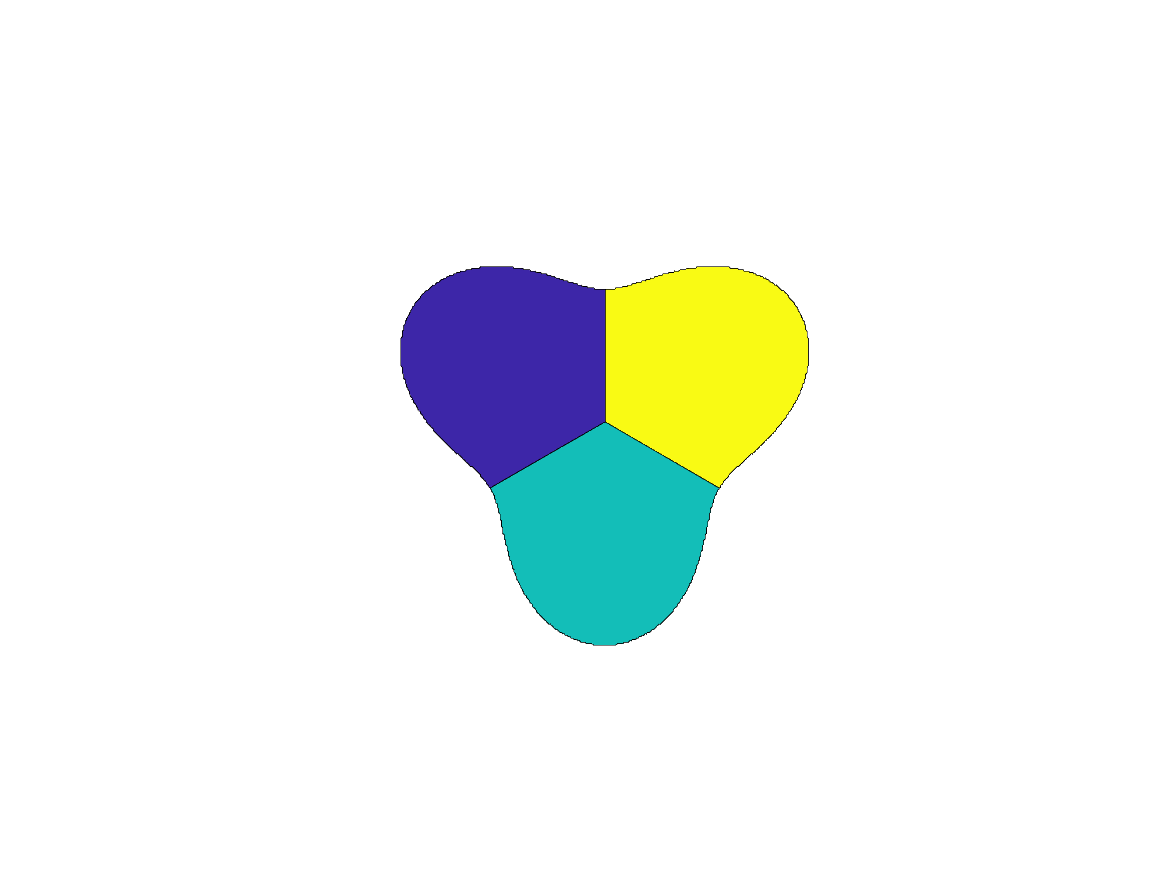}
\includegraphics[width = 0.1\textwidth, clip, trim = 6.5cm 3.5cm 6cm 4cm]{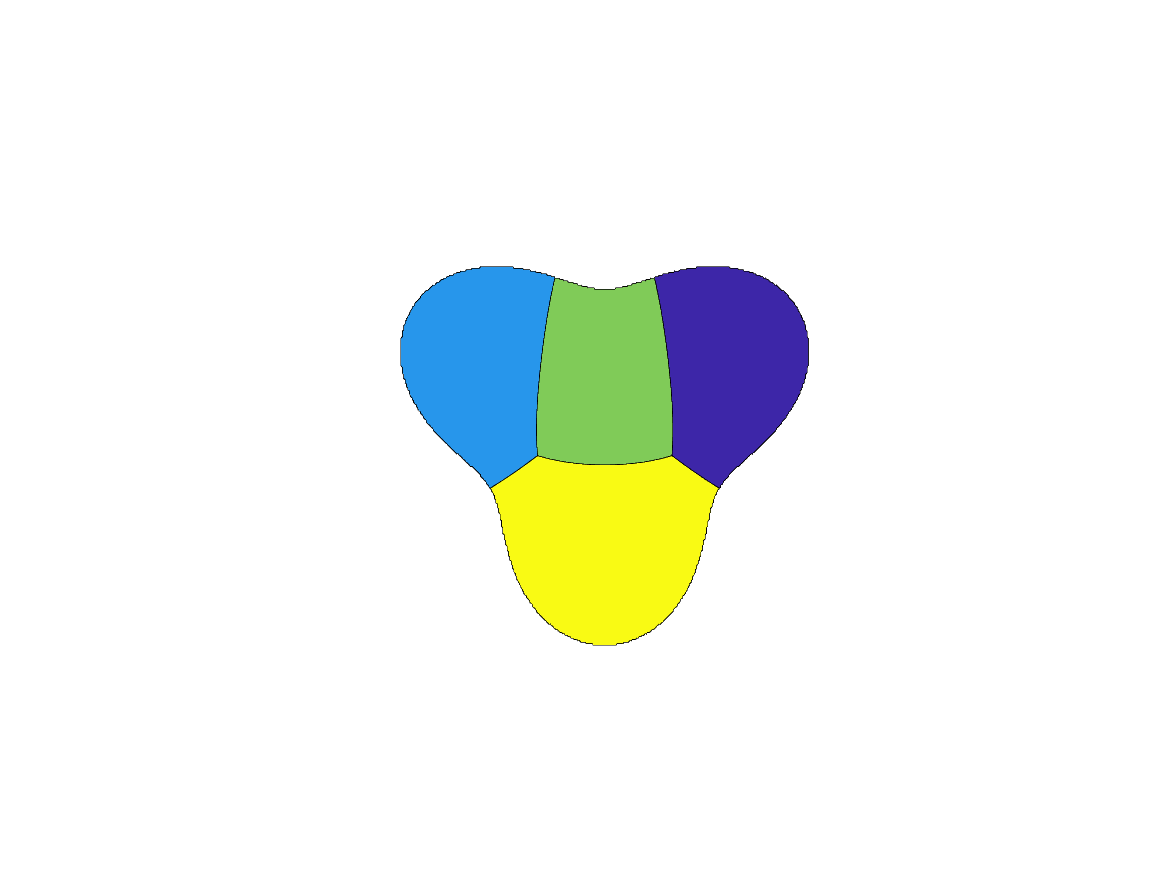}
\includegraphics[width = 0.1\textwidth, clip, trim = 6.5cm 3.5cm 6cm 4cm]{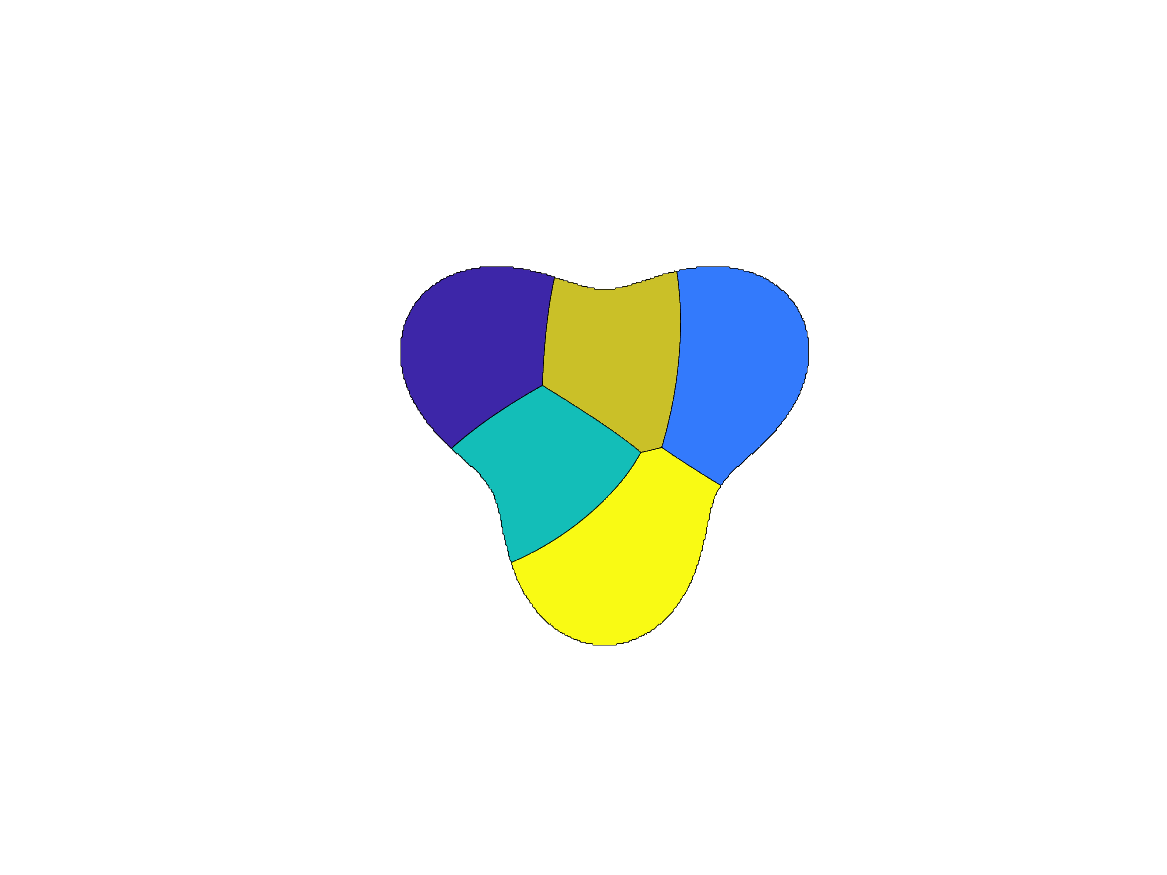}
\includegraphics[width = 0.1\textwidth, clip, trim = 6.5cm 3.5cm 6cm 4cm]{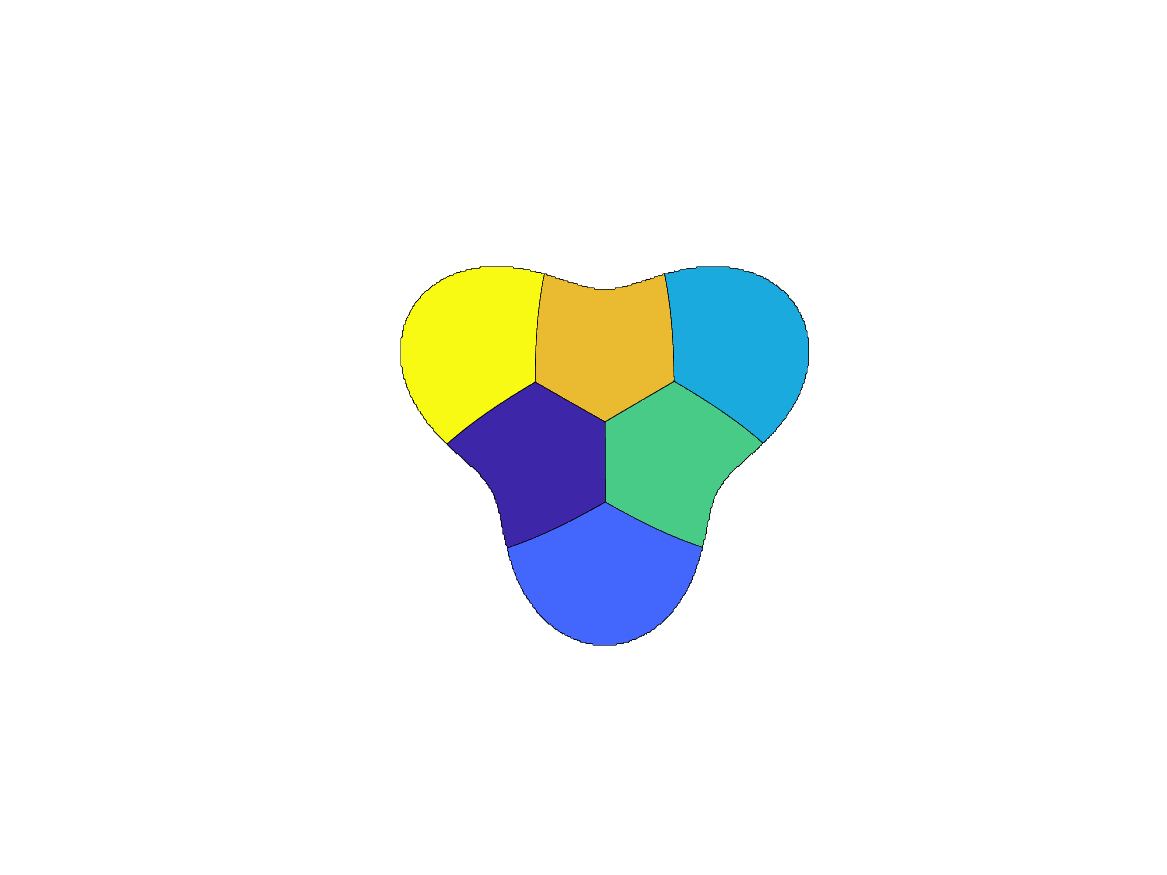}
\includegraphics[width = 0.1\textwidth, clip, trim = 6.5cm 3.5cm 6cm 4cm]{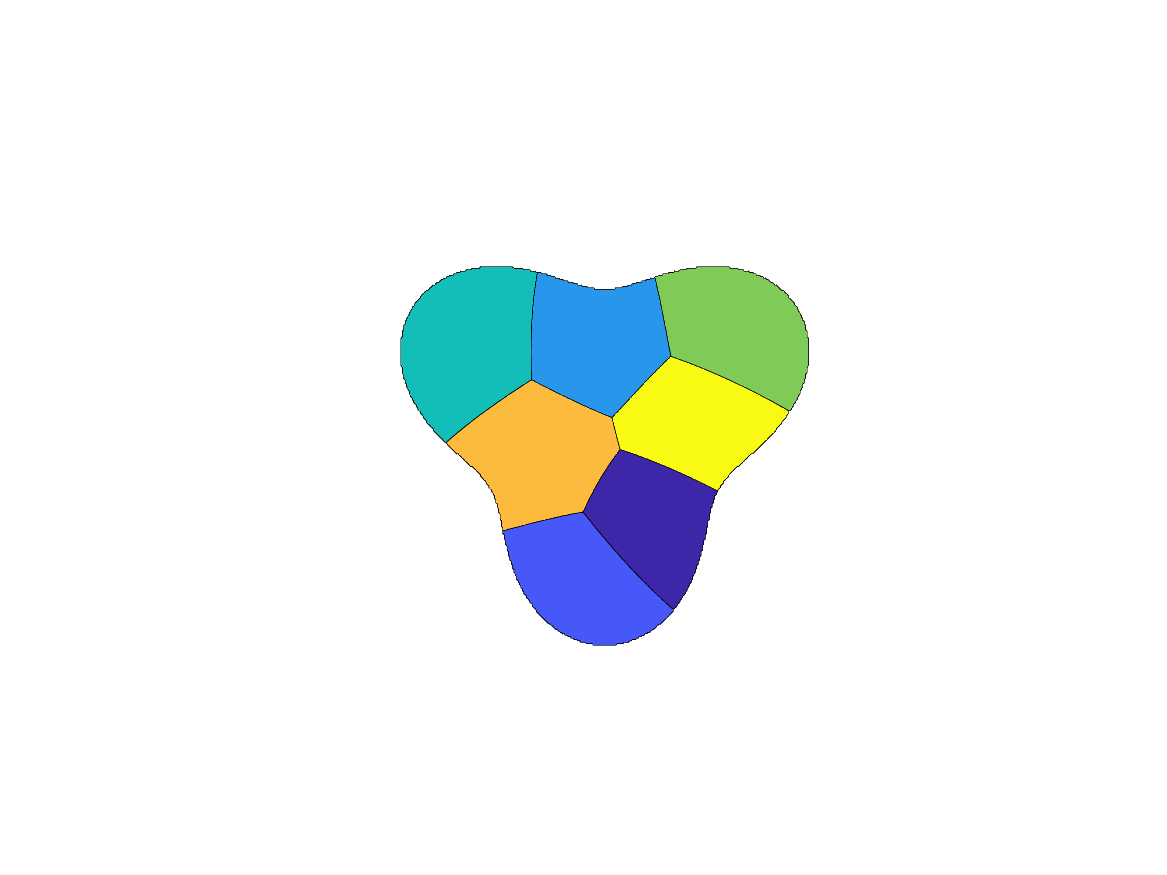}
\includegraphics[width = 0.1\textwidth, clip, trim = 6.5cm 3.5cm 6cm 4cm]{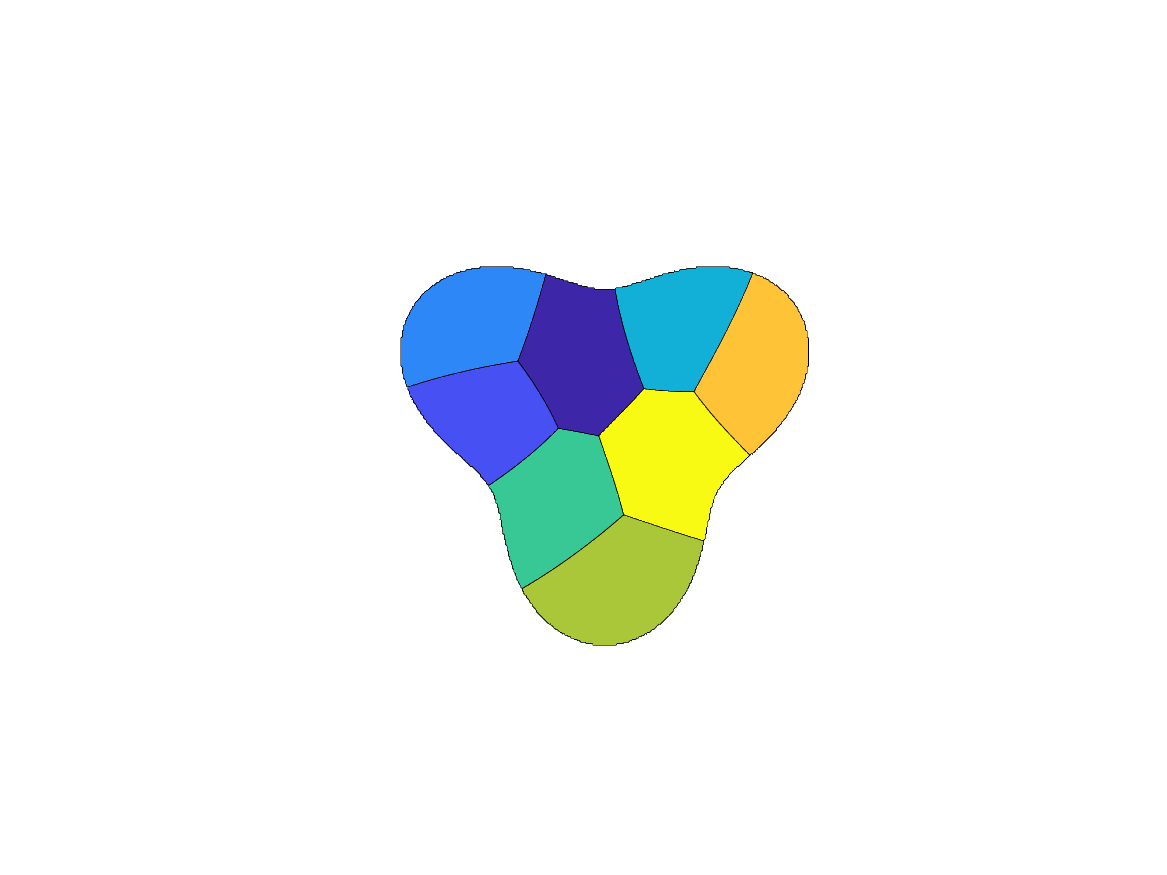}
\includegraphics[width = 0.1\textwidth, clip, trim = 6.5cm 3.5cm 6cm 4cm]{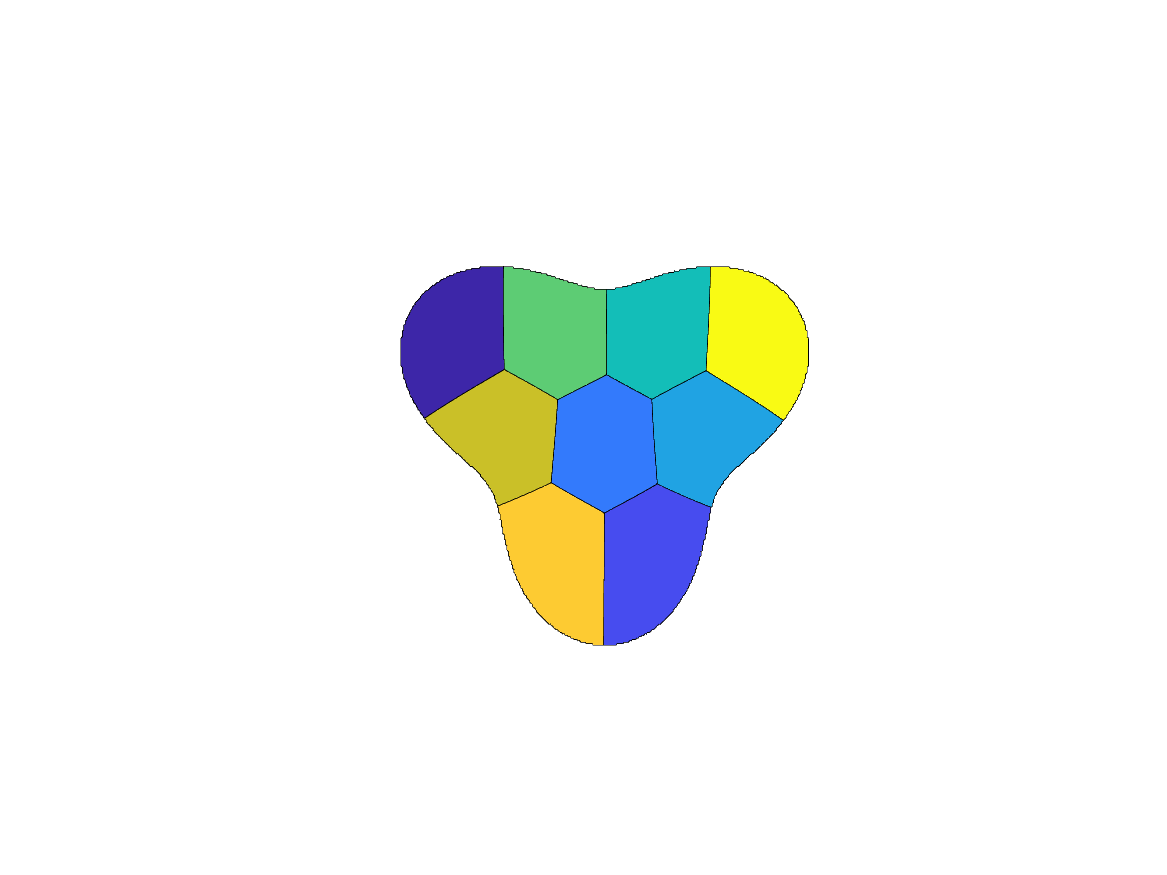}
\includegraphics[width = 0.1\textwidth, clip, trim = 6.5cm 3.5cm 6cm 4cm]{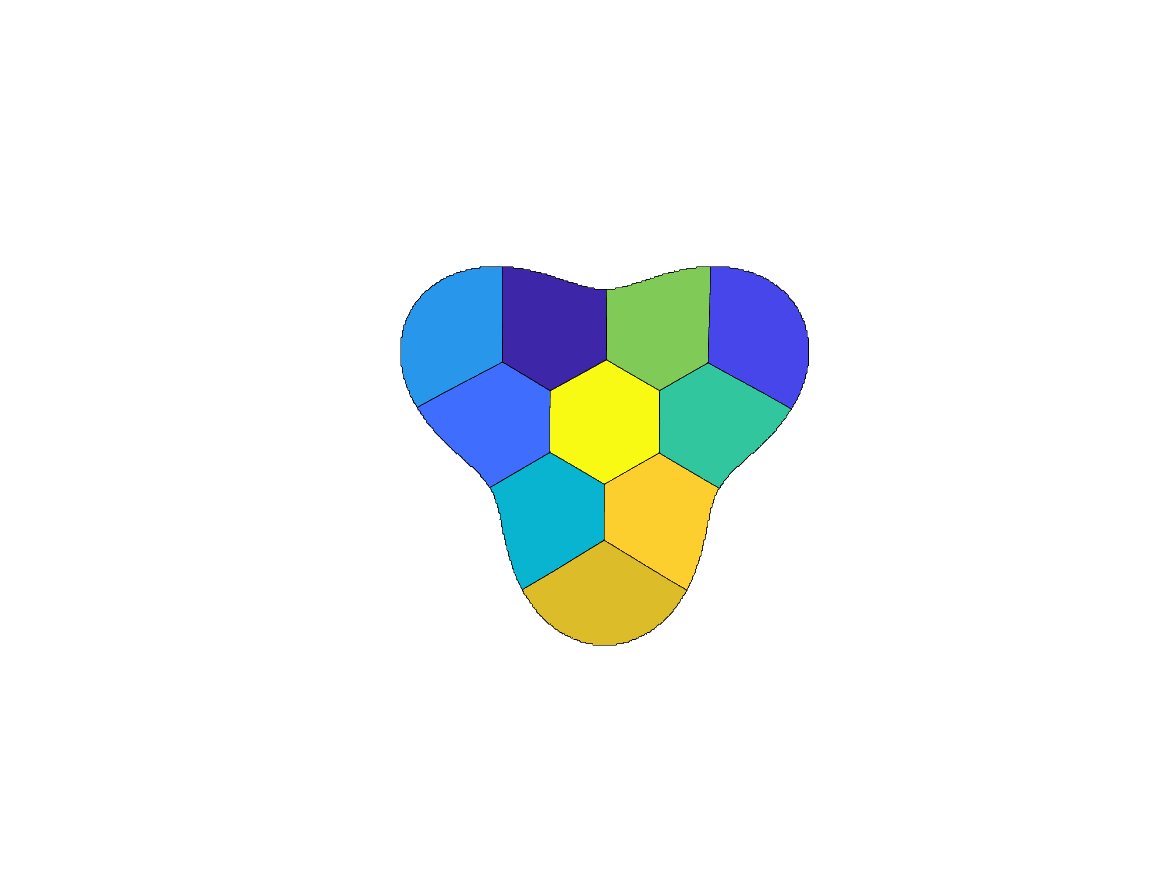}\\
\smallskip
\includegraphics[width = 0.1\textwidth, clip, trim = 6.5cm 4.5cm 5.5cm 3cm]{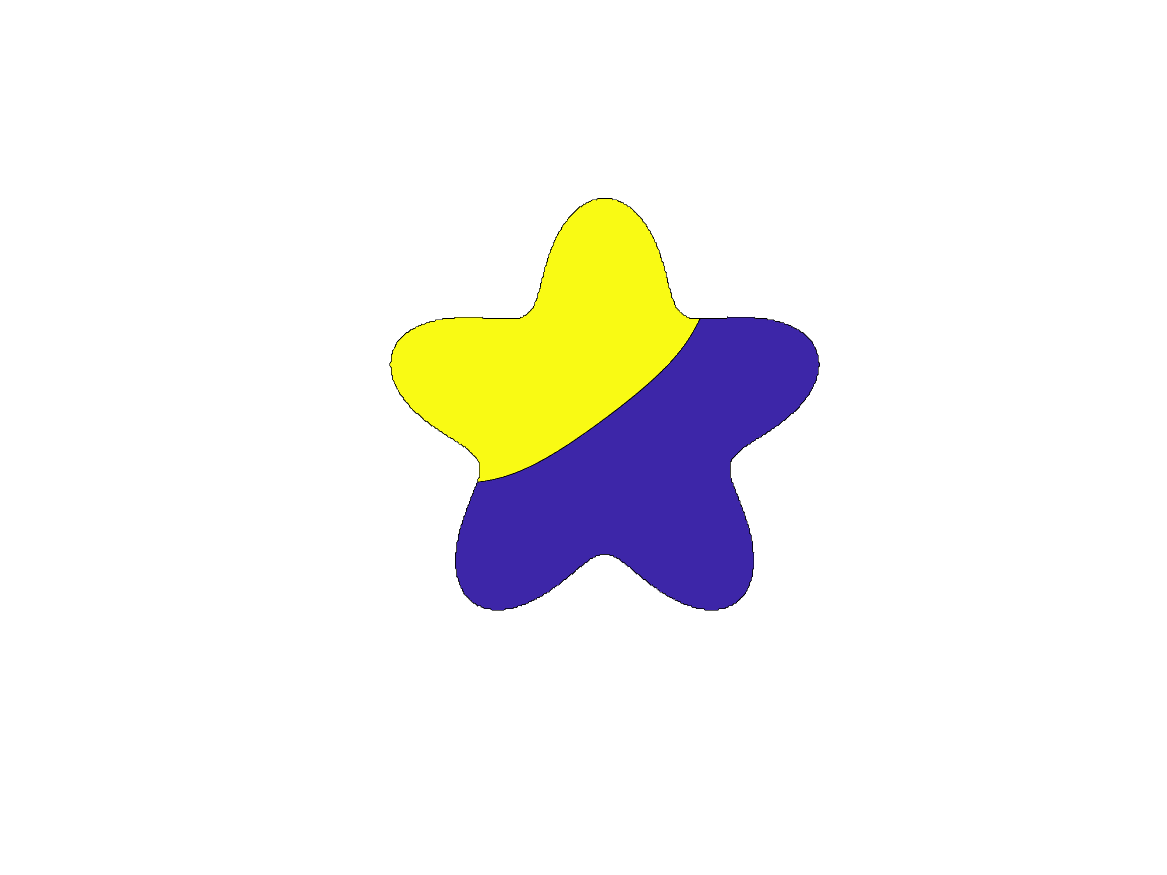}
\includegraphics[width = 0.1\textwidth, clip, trim = 6.5cm 4.5cm 5.5cm 3cm]{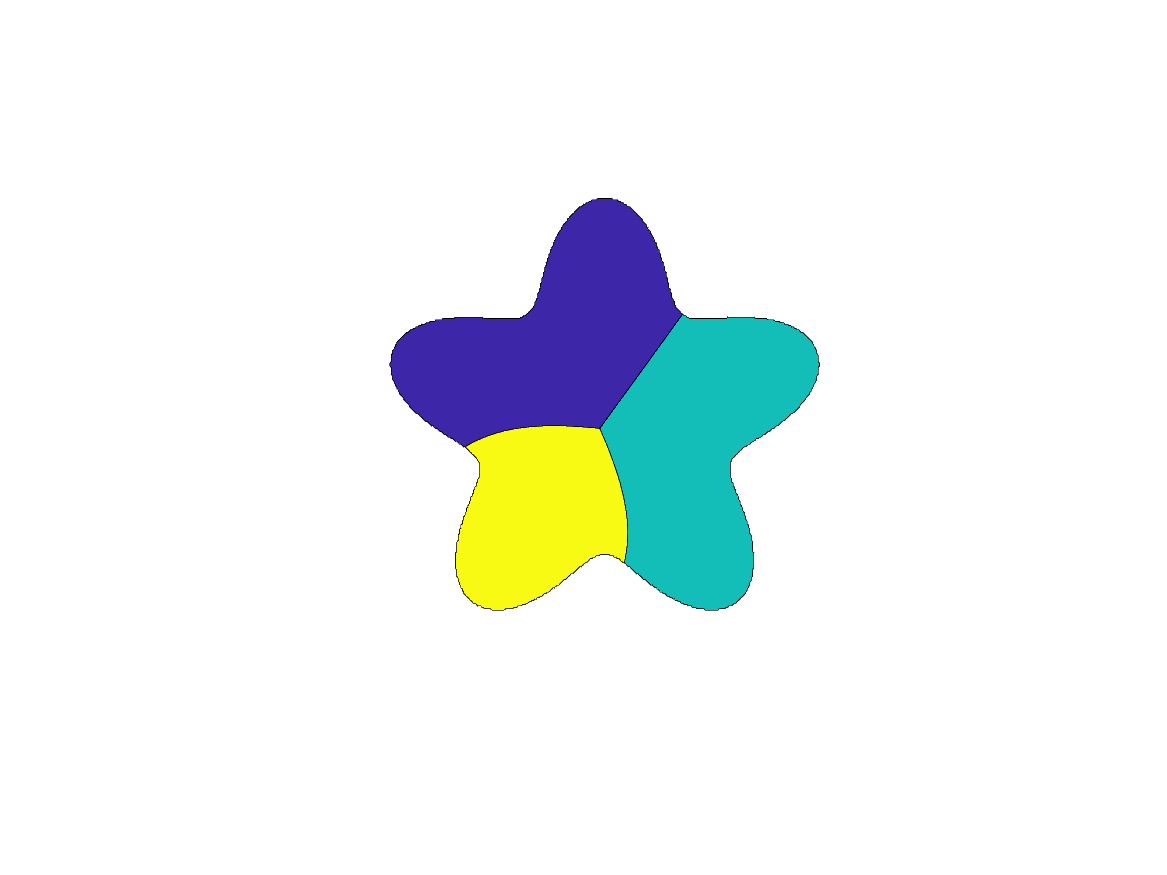}
\includegraphics[width = 0.1\textwidth, clip, trim = 6.5cm 4.5cm 5.5cm 3cm]{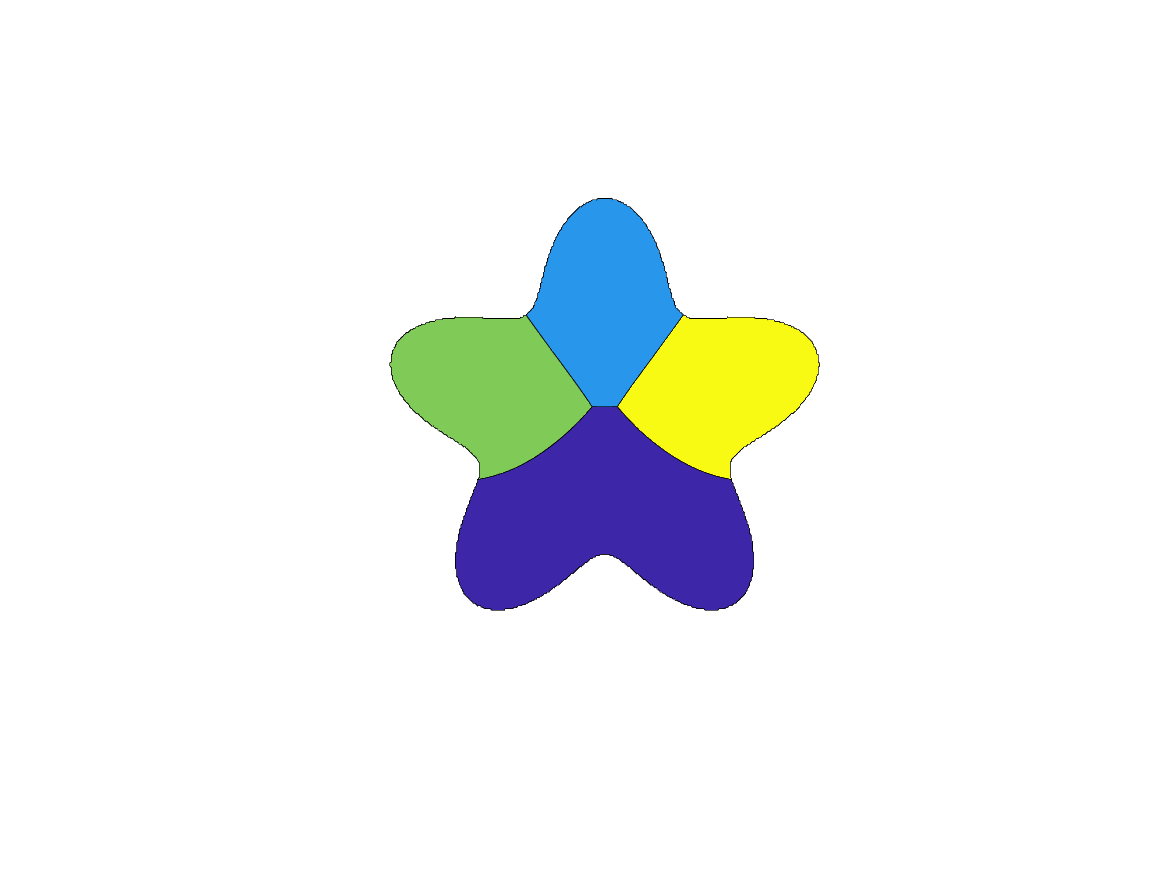}
\includegraphics[width = 0.1\textwidth, clip, trim = 6.5cm 4.5cm 5.5cm 3cm]{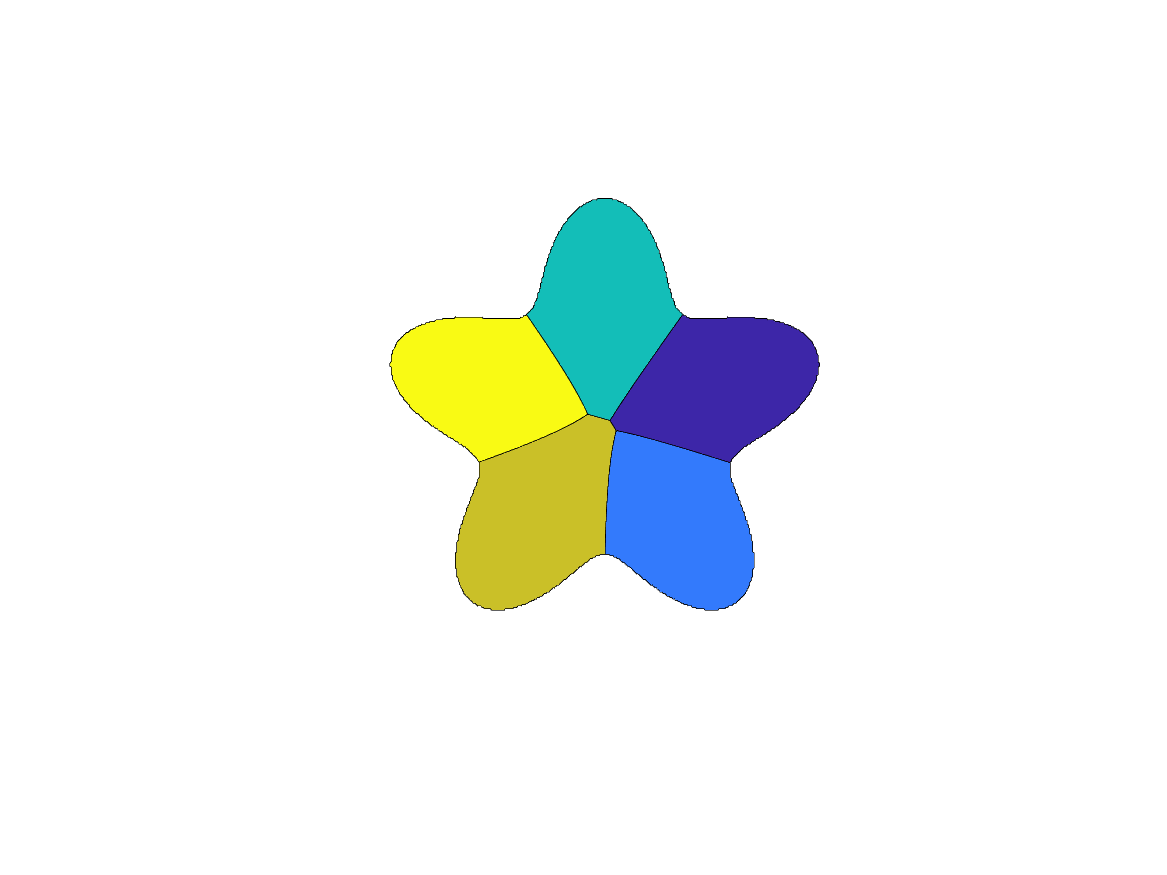}
\includegraphics[width = 0.1\textwidth, clip, trim = 6.5cm 4.5cm 5.5cm 3cm]{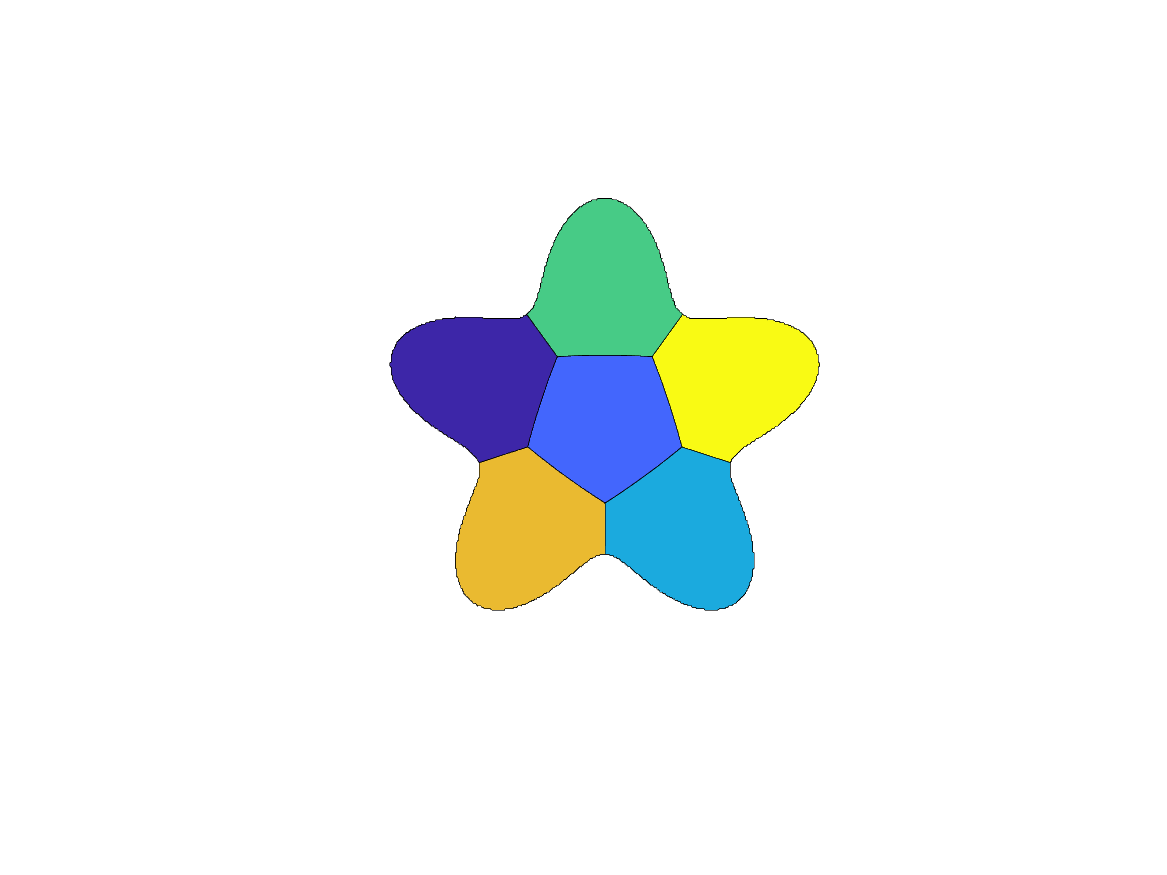}
\includegraphics[width = 0.1\textwidth, clip, trim = 6.5cm 4.5cm 5.5cm 3cm]{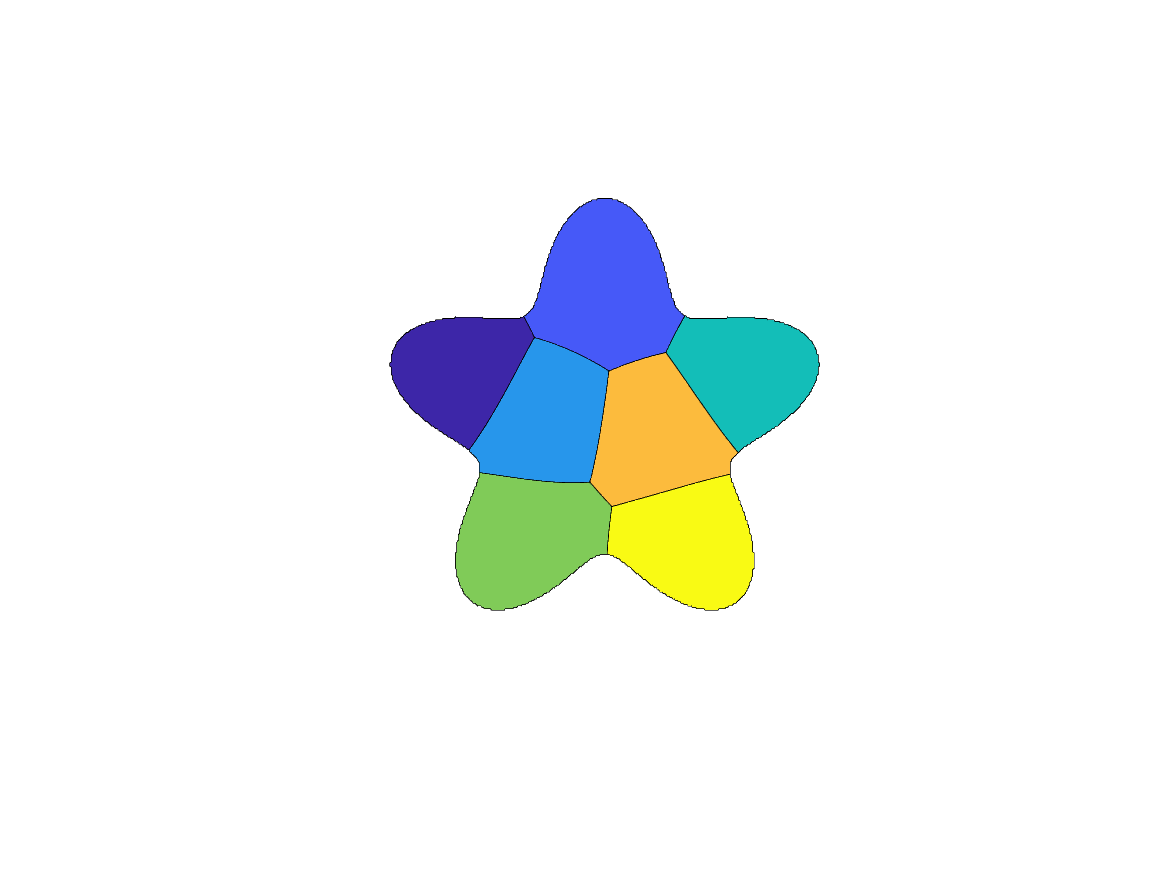}
\includegraphics[width = 0.1\textwidth, clip, trim = 6.5cm 4.5cm 5.5cm 3cm]{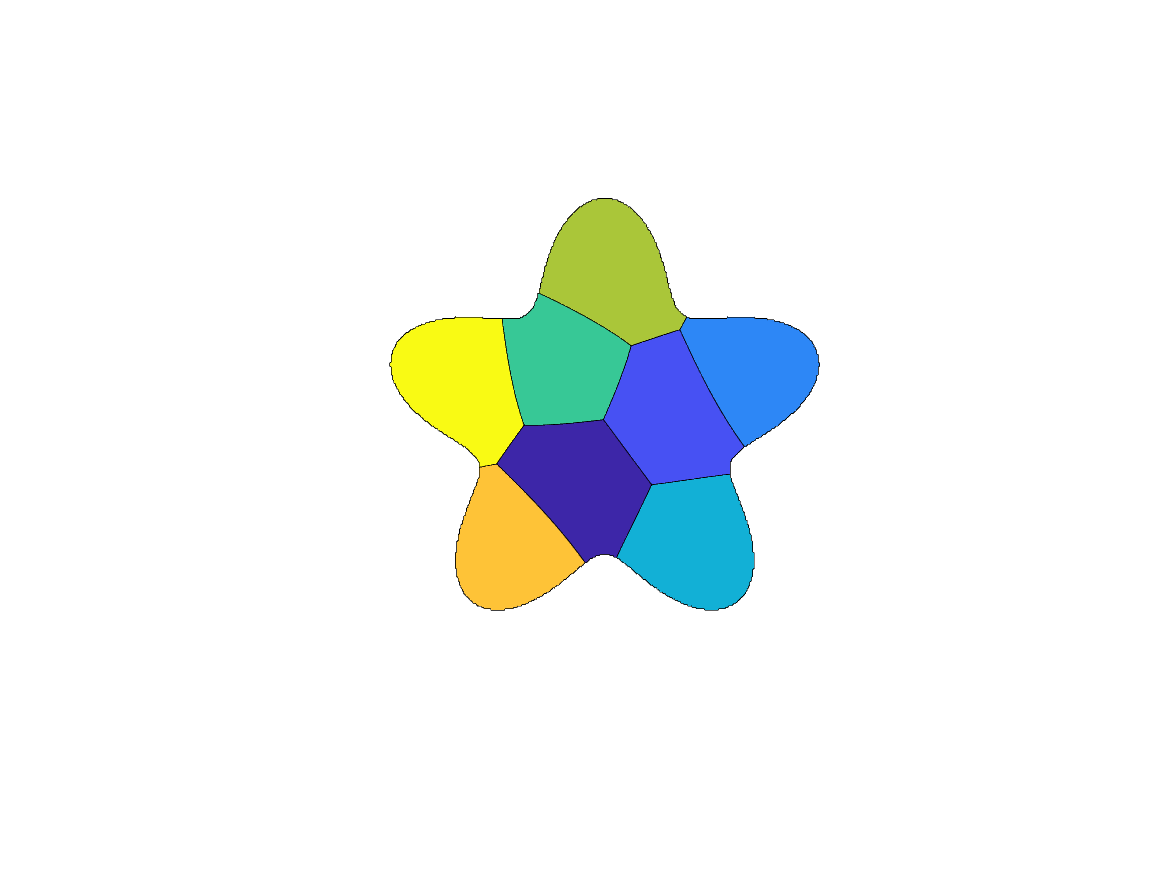}
\includegraphics[width = 0.1\textwidth, clip, trim = 6.5cm 4.5cm 5.5cm 3cm]{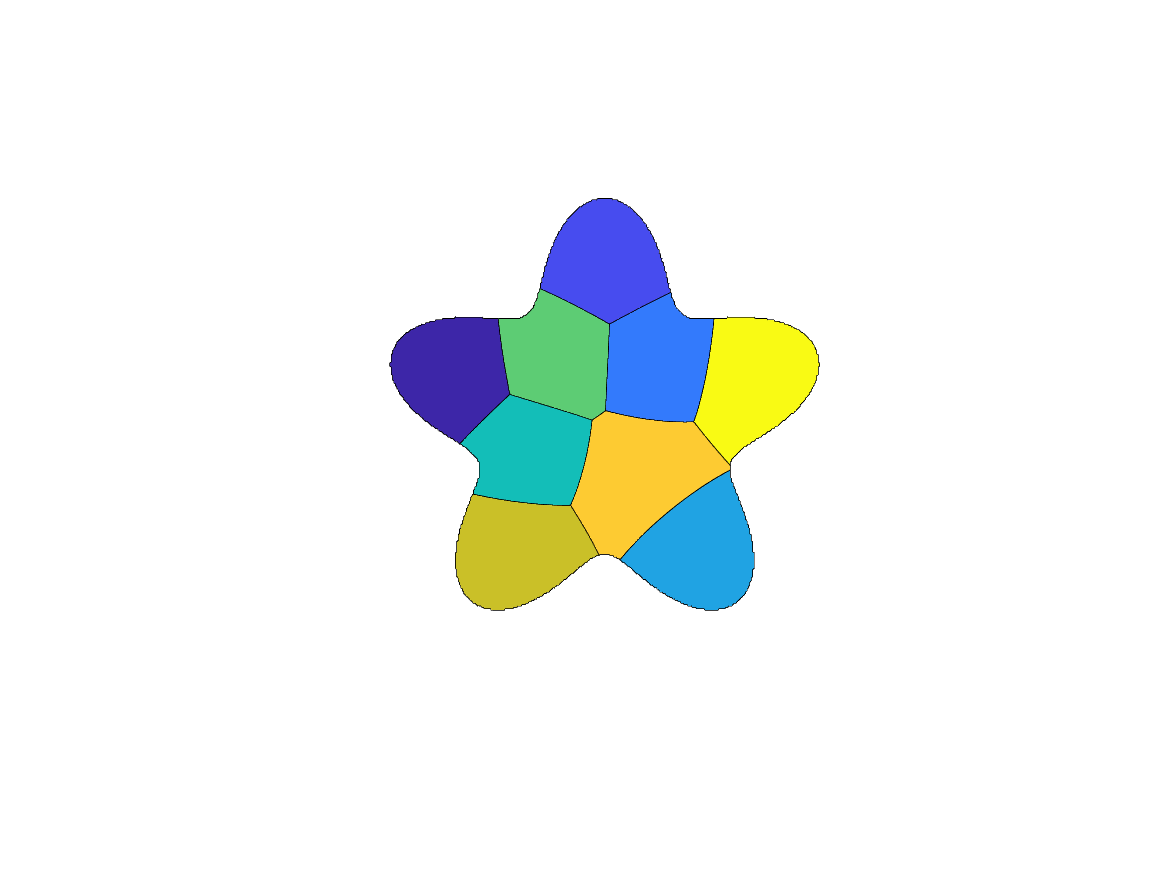}
\includegraphics[width = 0.1\textwidth, clip, trim = 6.5cm 4.5cm 5.5cm 3cm]{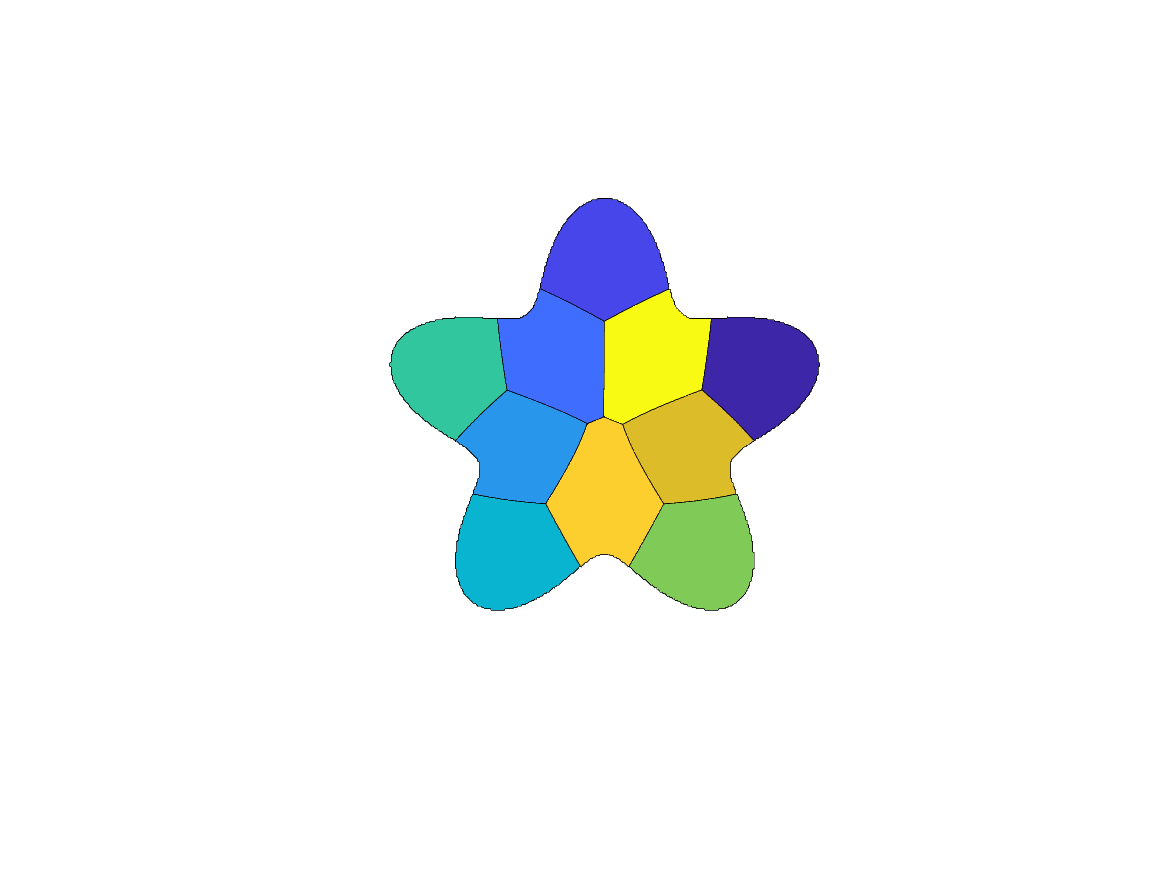}\\

\medskip

{\footnotesize
\begin{tabular}{c|c|c|c|c|c|c|c|c|c}
\hline
Approximate &   & & &  &  &  &  & & \\ 
eigenvalues & k=2  & k=3 & k=4& k=5 & k=6& k=7 & k=8 & k=9 & k =10 \\ 
\hline 
Square & 13.91& 26.90 & 41.69 & 64.35 & 87.31 & 112.78 & 139.03 & 172.64 & 200.69 \\
\hline 
Petagon & 11.84 & 23.25 & 38.08 & 54.67 & 74.08 & 98.30 & 123.78 & 151.43 & 179.79\\
\hline 
Hexgon & 9.29 & 18.33 & 35.52 & 52.78 & 71.70 & 89.89 & 115.59 & 142.47 & 169.15 \\
\hline
Disk & 9.17 & 18.33 & 20.48& 45.27 & 61.69 & 79.72& 100.26 & 124.67 & 148.40 \\
\hline
Three-fold star& 9.29 & 16.28 & 29.40 & 44.12 & 58.81 & 79.05 & 100.08 & 121.51 & 143.37 \\
\hline
Five-fold star& 10.69 &20.39  &31.62 & 43.17 & 58.09 & 78.46 &98.62 & 120.13 & 142.30 \\
\hline
\end{tabular}
}
\caption{Computed $k$-partition in a square domain, a pentagon domain, a regular hexagon domain, a disk domain, a three-fold star domain, and a five-fold star domain for $k = 2-10$ together with corresponding approximate eigenvalues. See Section~\ref{sec:2darbitrary}.} \label{fig:2darbitrary4}
\end{figure}

\subsection{$k$-partition for a 3-dimensional periodic domain}\label{sec:3dperiodic}
In this section, we show the efficiency of the proposed algorithm for a 3-dimensional periodic domain. In all experiments, we discretize the periodic computational domain $[-\pi,\pi]^3$ by $128^3$ uniform grid points and simply fix $\tau = \pi/16$.

For $k = 4$ and initialization using a random tessellation, we obtain a partition
of a 3D flat torus by four identical rhombic dodecahedron structures as displayed in
Figure~\ref{fig:rhombic} with the approximate eigenvalue $6.96$. The result agrees with those reported in \cite{Wang_2019, Chu_2021}.  The CPU time for this experiment without parallel computing is only 60 {\it seconds} and the CPU time reported in \cite{Chu_2021} for the same case with a $100^3$ uniform discretization is 588 {\it minutes}.

\begin{figure}[ht!]
\centering
\includegraphics[width = 0.3\textwidth]{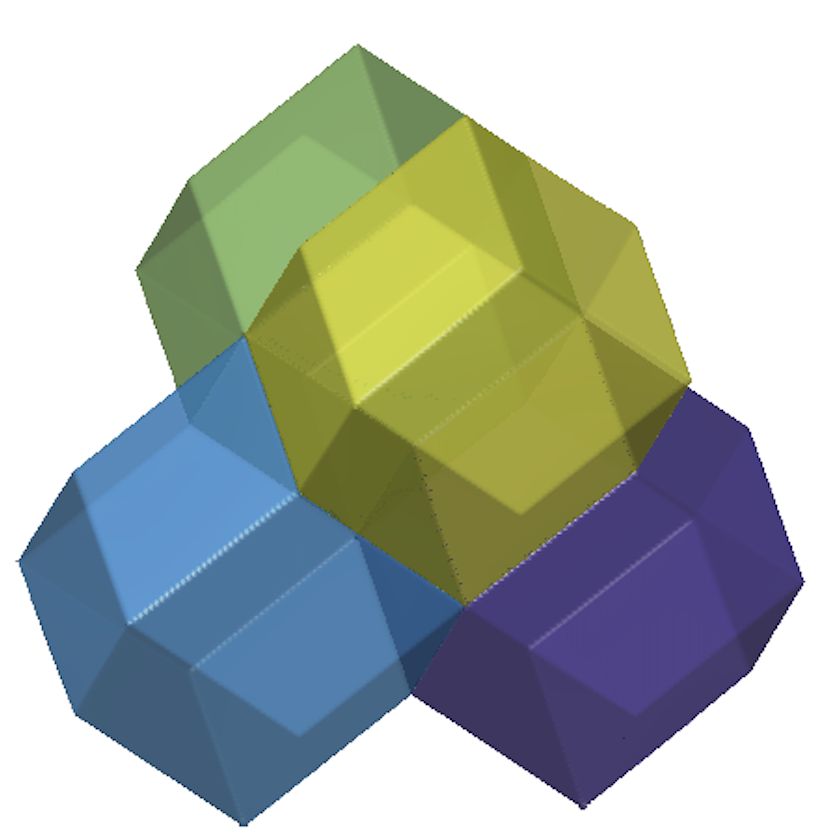}
\includegraphics[width = 0.3\textwidth]{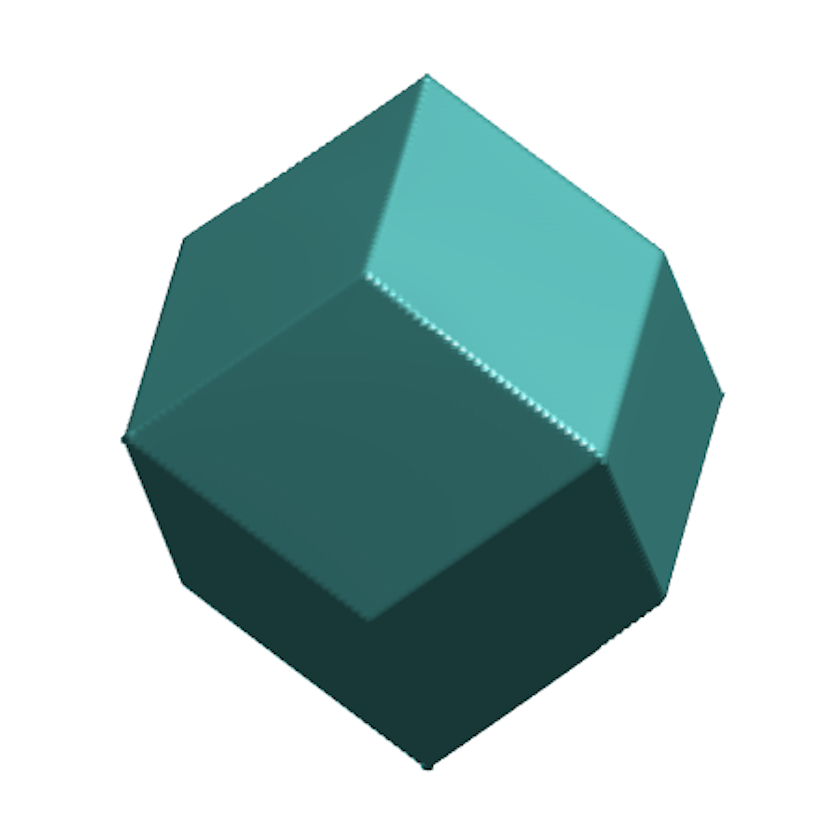}
\includegraphics[width = 0.3\textwidth]{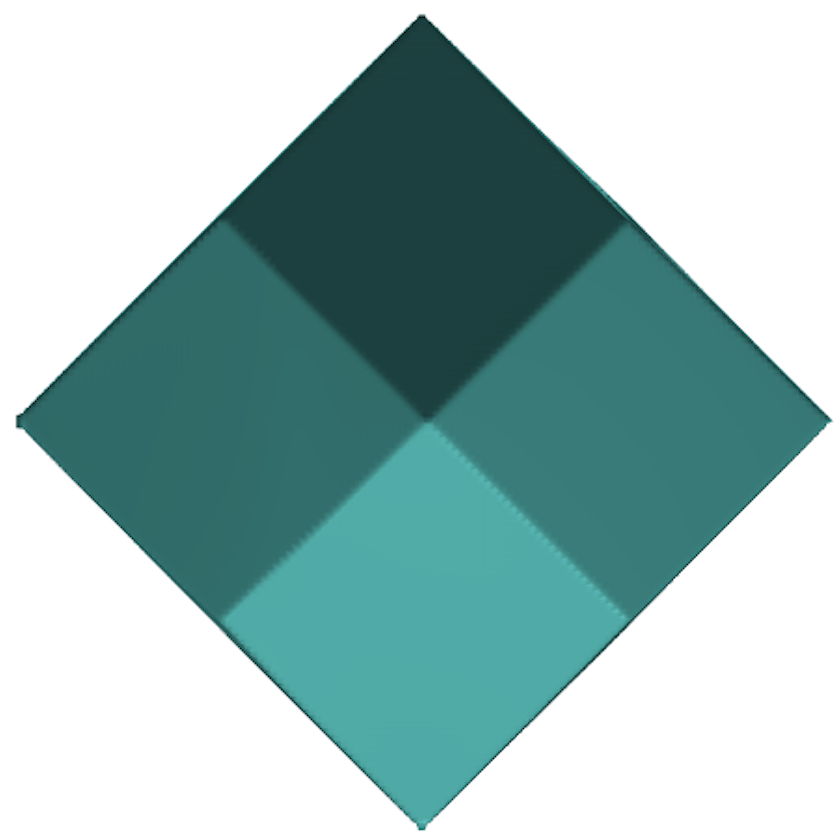} \\
\caption{Left: A periodic extension on the $4$-partition in a 3-dimensional flat torus. Middle: The rhombic dodecahedron structure. Right: The xy-view of the rhombic dodecahedron structure. The approximate eigenvalue is $6.96$. The CPU time is 60 {\it seconds}. See Section~\ref{sec:3dperiodic}.} \label{fig:rhombic}
\end{figure}

For $k=8$, we obtain the well-known Weaire–Phelan structure which is a structure representing a foam of equal-sized shapes as shown in the second and third row of Figure~\ref{fig:8-periodic}. Among eight partitions, two regions have the first type of shape which consists of 12 pentagonal faces while the other six regions are the second type of shape which has 2 hexagonal faces surrounded by 12 pentagonal faces. The overall packing is shown in the first row of Figure~\ref{fig:8-periodic}. The approximate eigenvalue is $21.97$. The CPU time for this computation from a random initialization with a $128^3$ uniform discretization is only 374 {\it seconds} while the time reported in \cite{Chu_2021} is 1246 {\it minutes} for a $100^3$ discretization.
\begin{figure}[ht!]
\centering
\includegraphics[width = 0.2 \textwidth,clip, trim = 12cm 10cm 12cm 6cm]{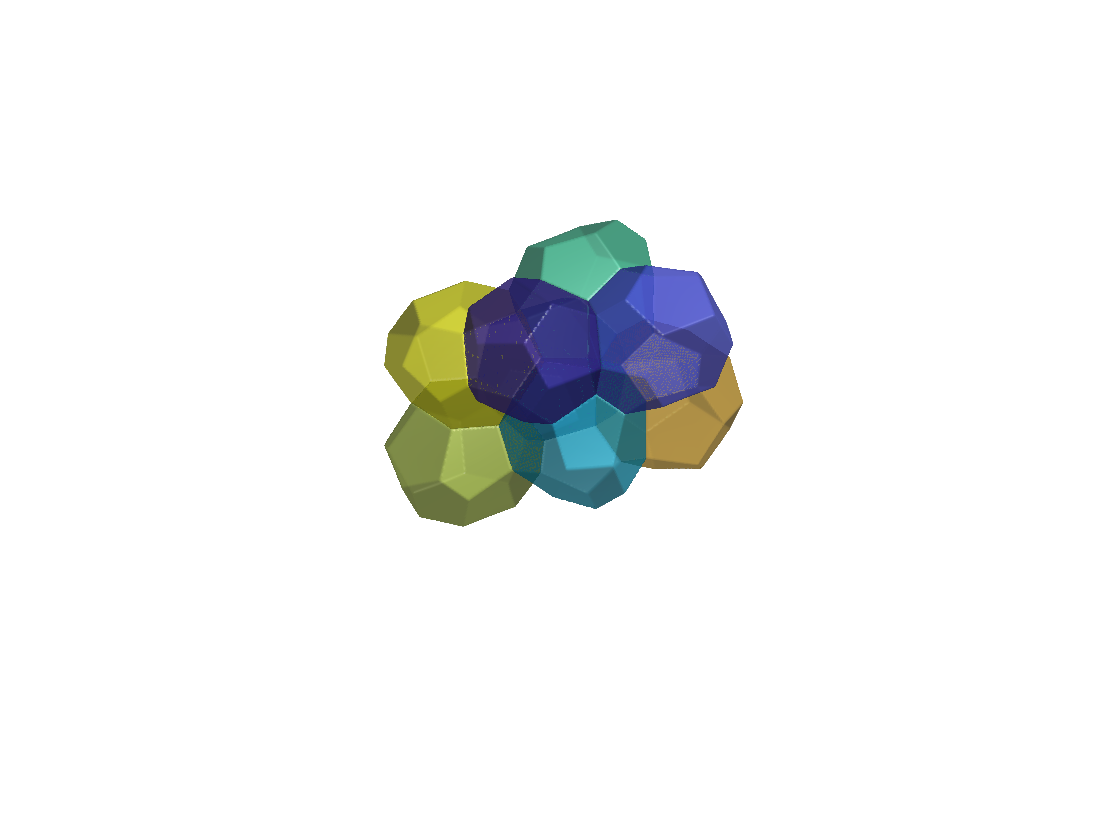}
\includegraphics[width = 0.2 \textwidth,clip, trim = 10cm 6cm 10cm 4cm]{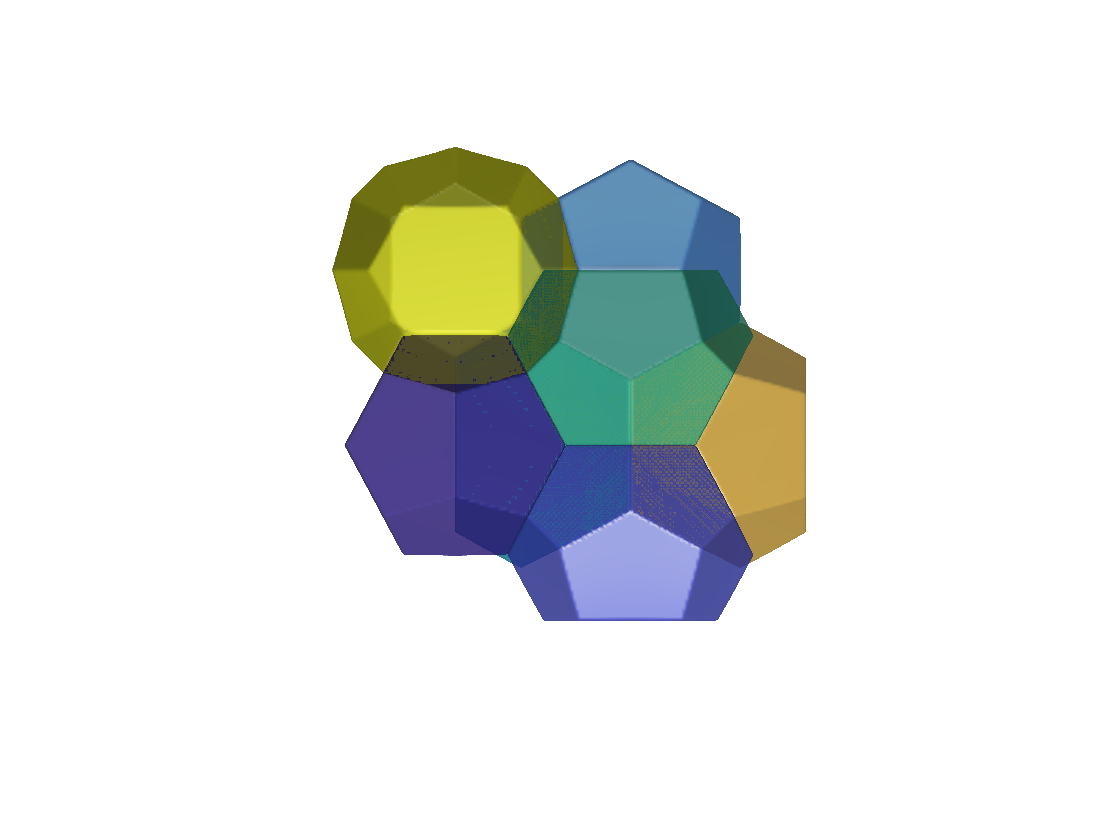}
\includegraphics[width = 0.2 \textwidth,clip, trim =10cm 6cm 10cm 4cm]{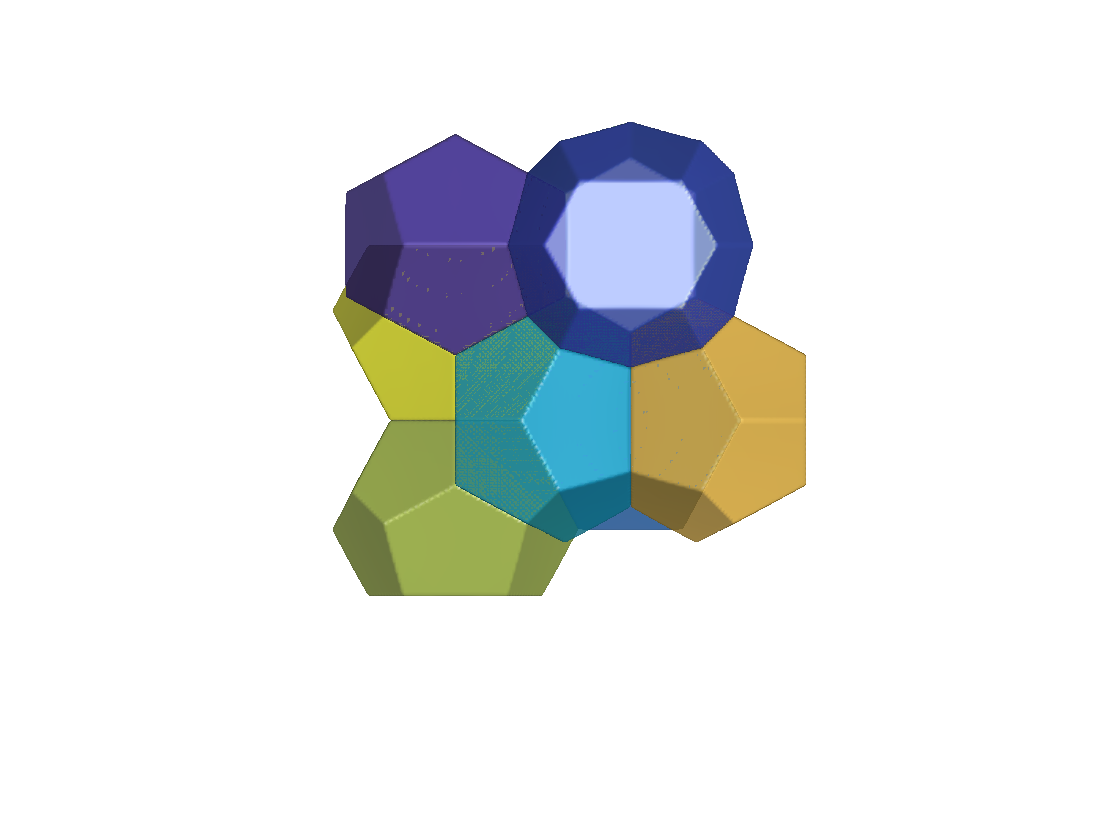}
\includegraphics[width = 0.2 \textwidth,clip, trim = 11cm 6cm 9cm 4cm]{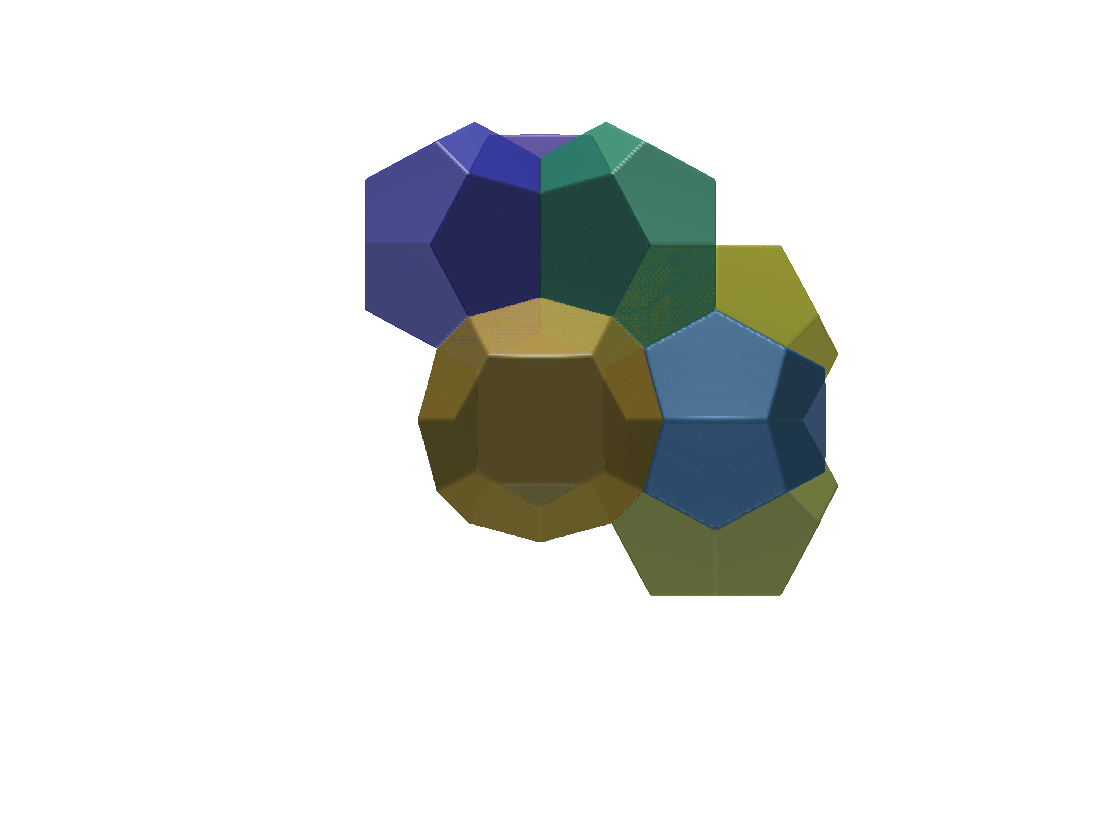}
\includegraphics[width = 0.2 \textwidth,clip, trim = 12cm 6cm 10cm 6cm]{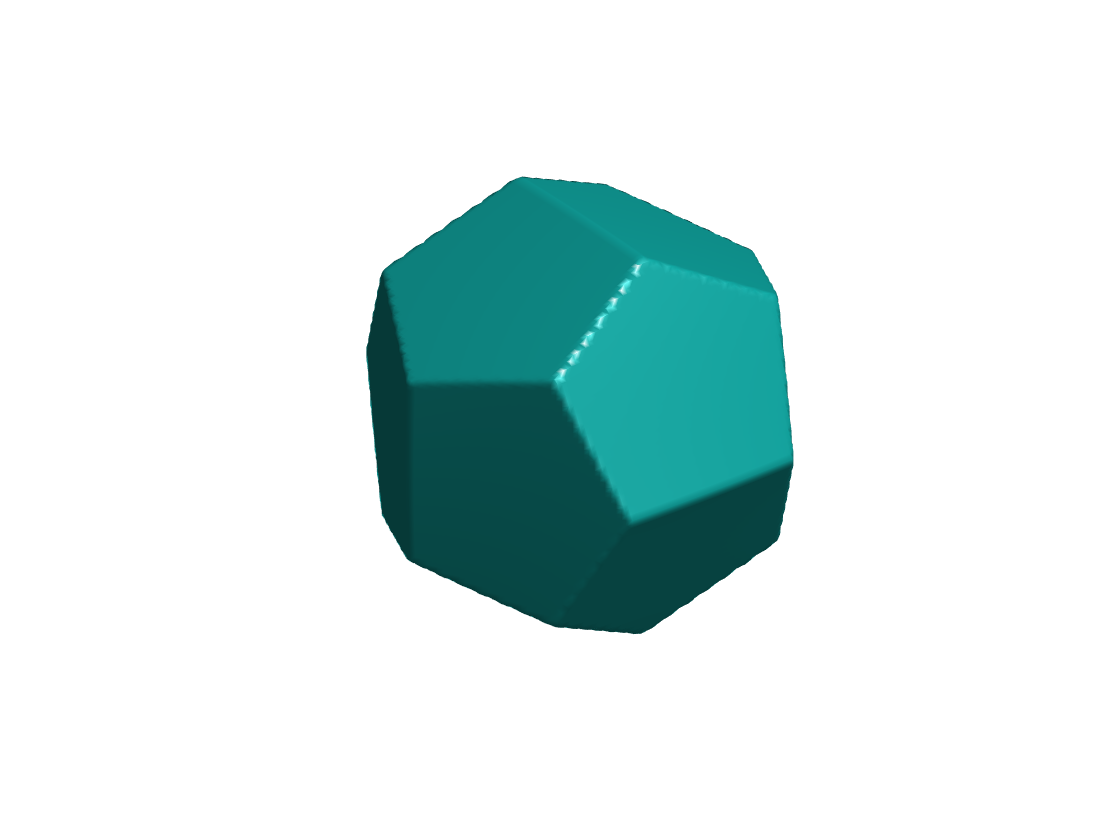}
\includegraphics[width = 0.2 \textwidth,clip, trim = 8cm 0cm 6cm 0cm]{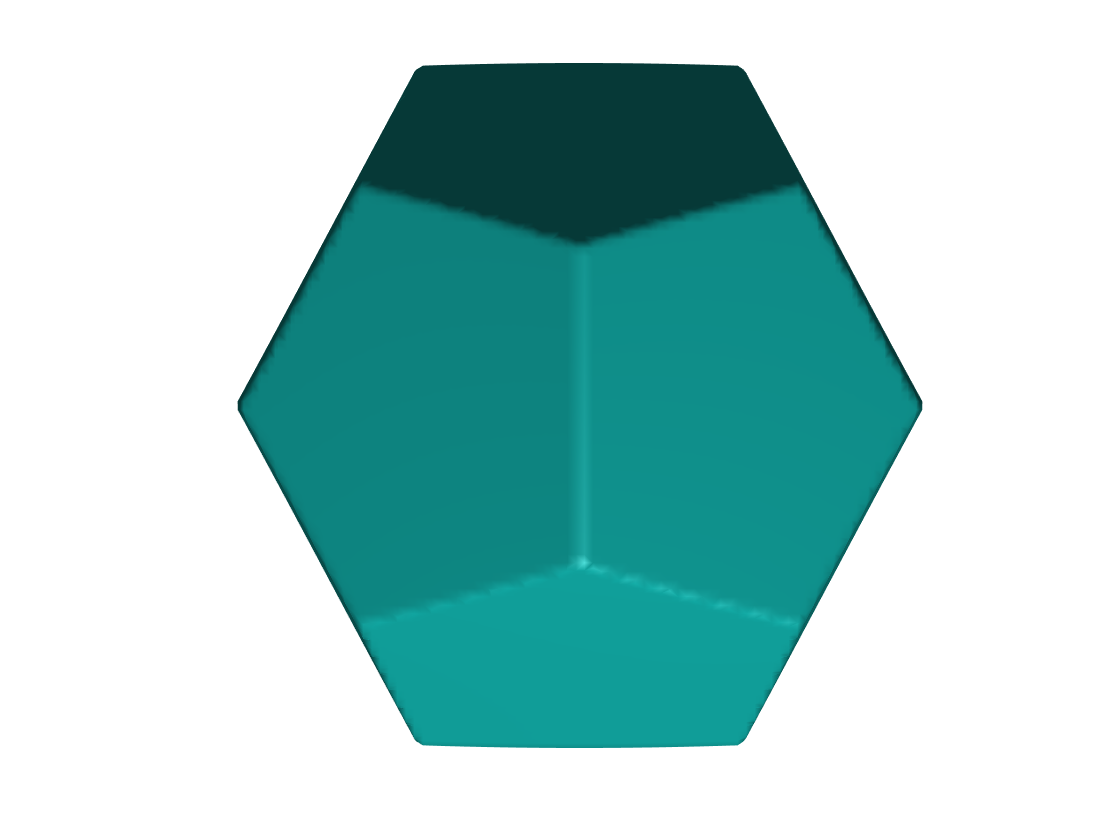}
\includegraphics[width = 0.2 \textwidth,clip, trim = 8cm 0cm 6cm 0cm]{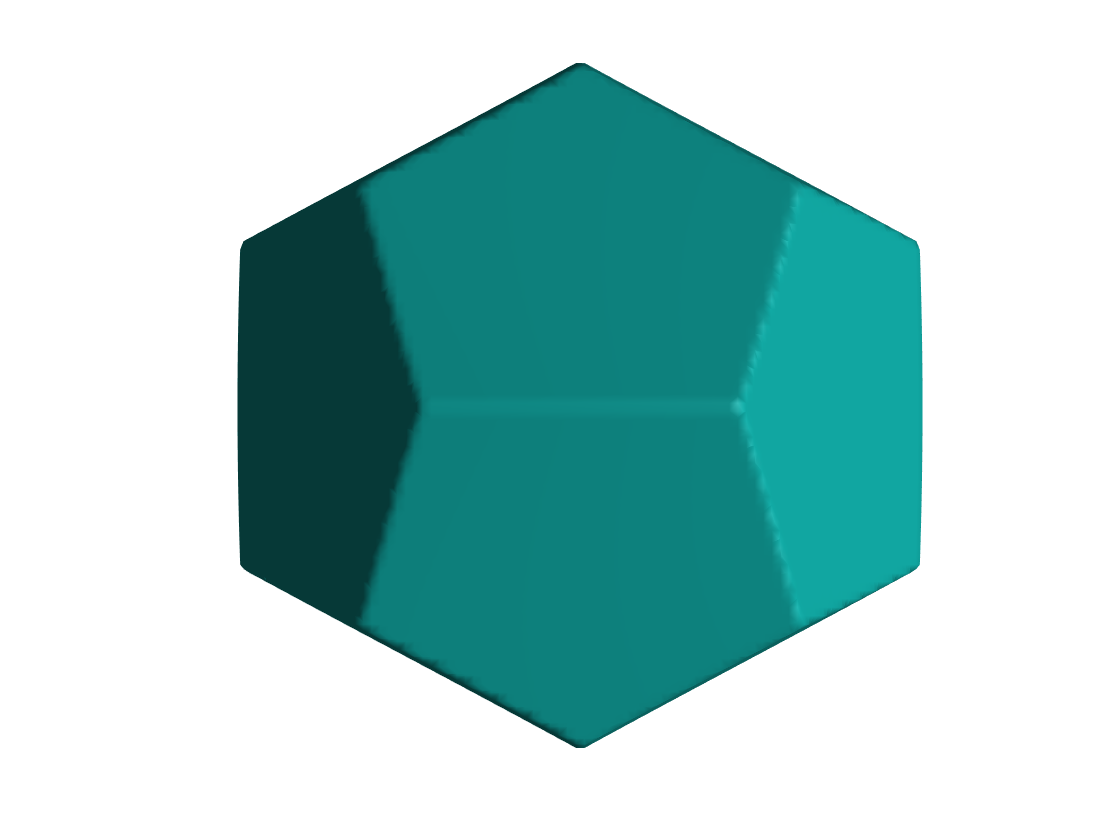}
\includegraphics[width = 0.2 \textwidth,clip, trim = 8cm 0cm 6cm 0cm]{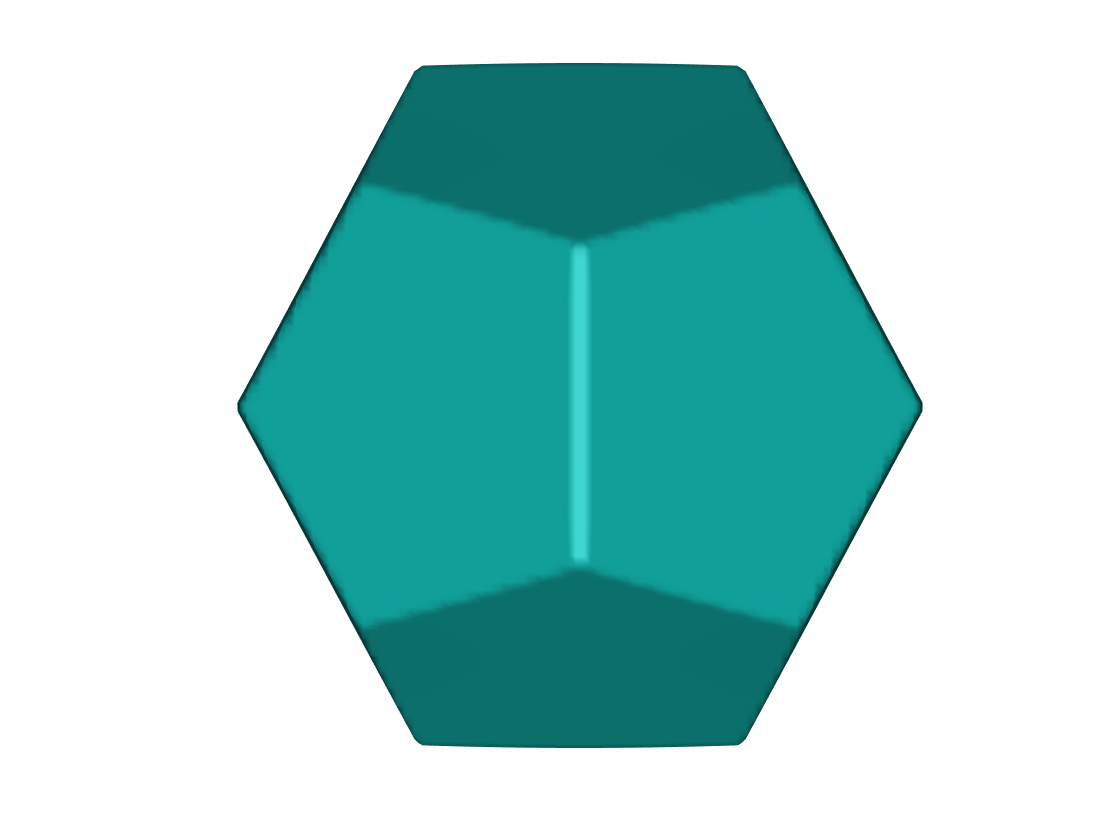}
\includegraphics[width = 0.2 \textwidth,clip, trim = 12cm 6cm 10cm 6cm]{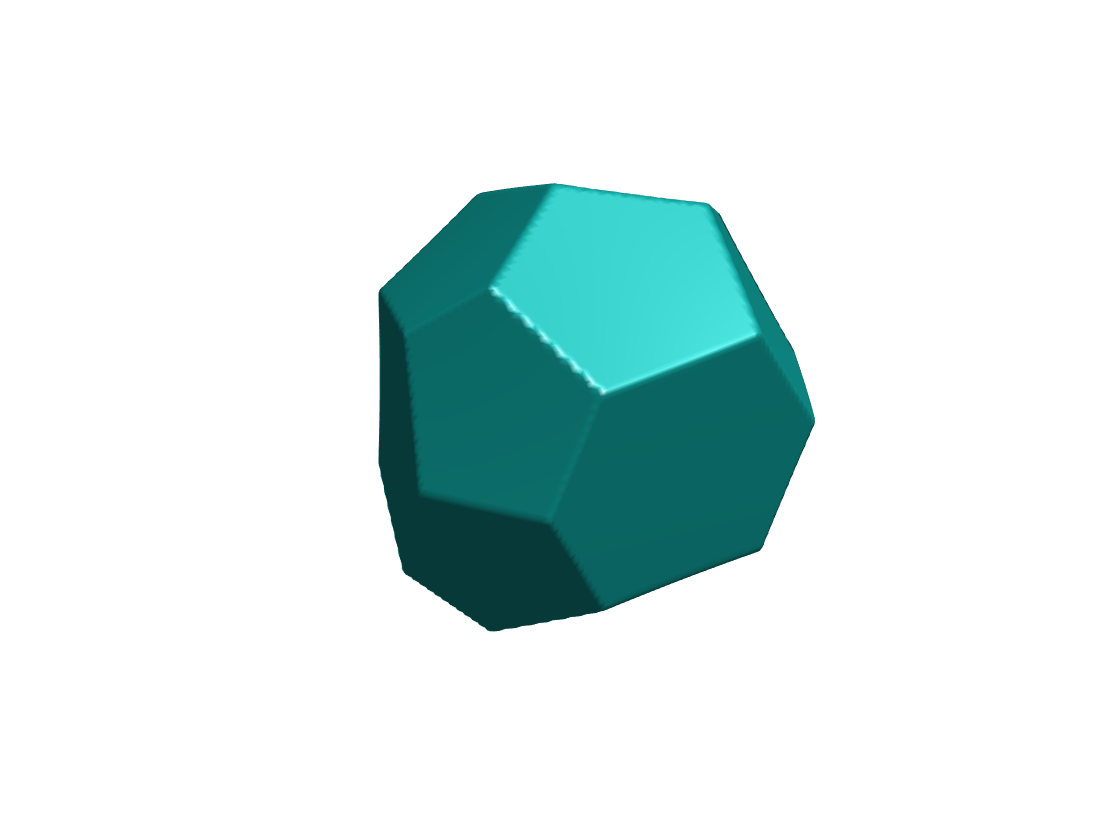}
\includegraphics[width = 0.2 \textwidth,clip, trim = 3cm 0cm 3cm 0cm]{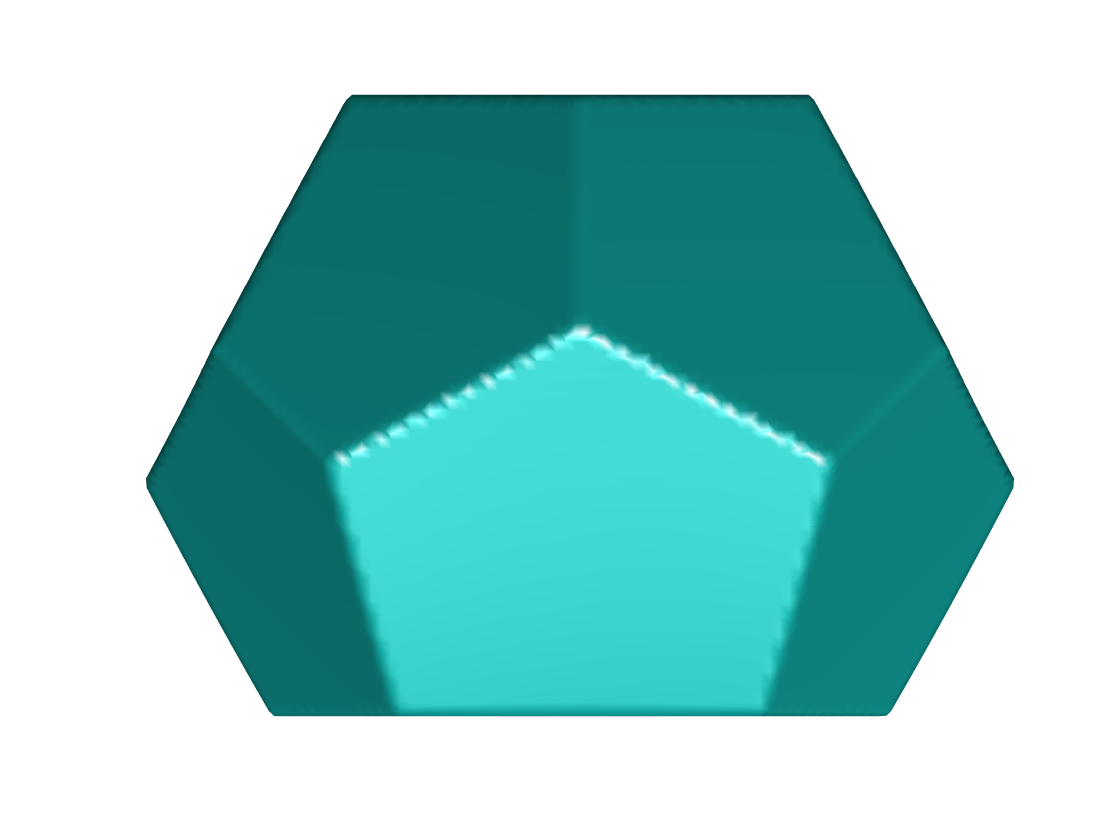}
\includegraphics[width = 0.2 \textwidth,clip, trim = 6cm 0cm 6cm 0cm]{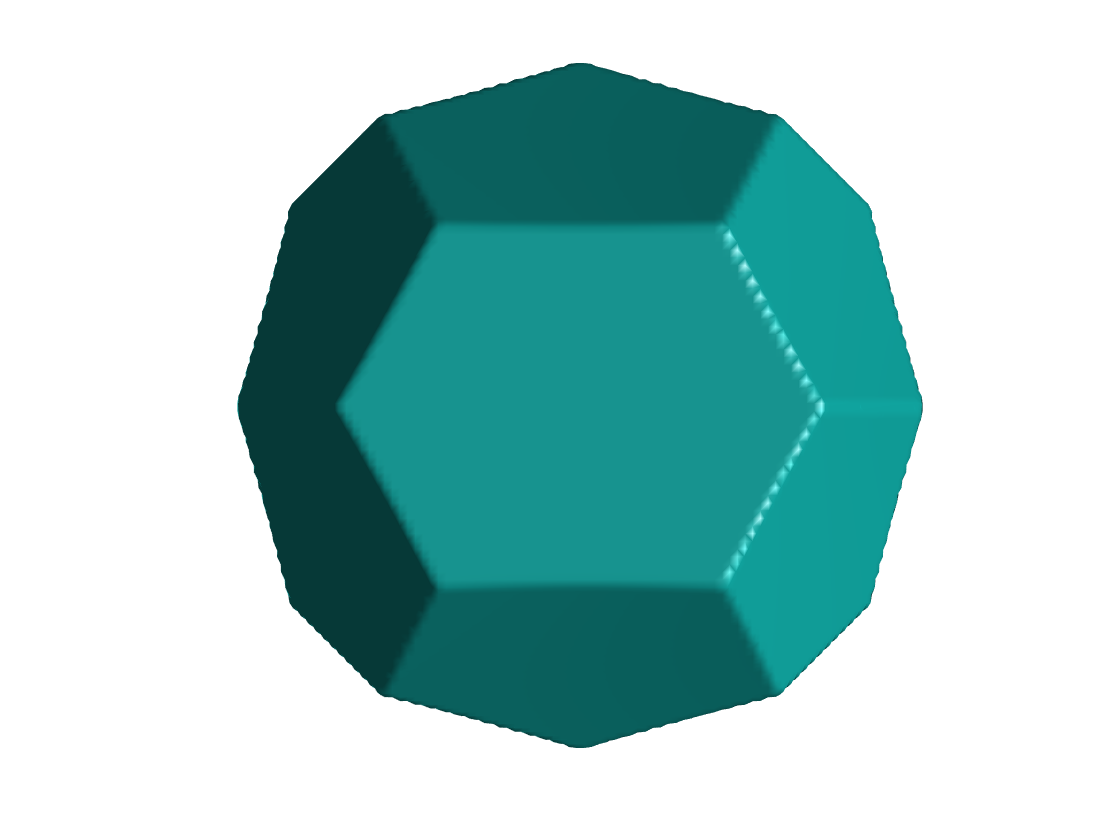}
\includegraphics[width = 0.2 \textwidth,clip, trim = 8cm 0cm 6cm 0cm]{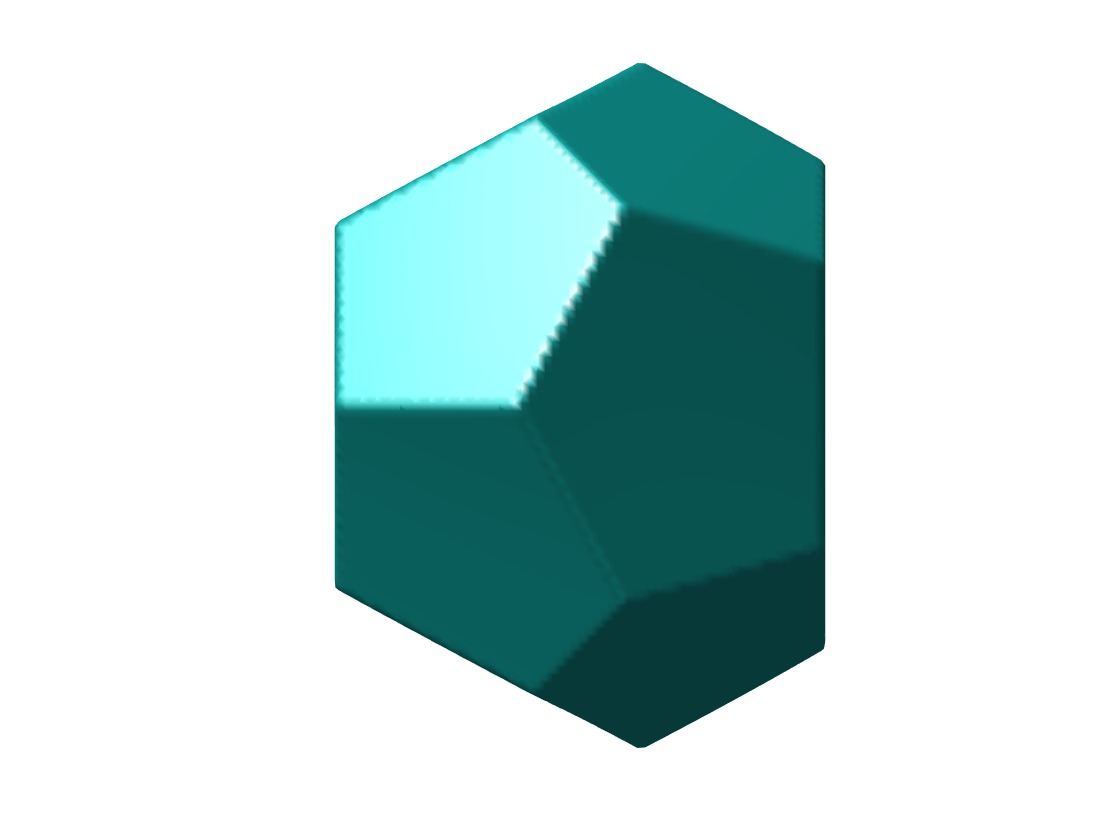}

\caption{First row: The 8-partition for a periodic cube with three different views. Second and third row: Two types of shapes. The shape in the second row consists of 12 pentagonal faces and the shape in the third row has 2 hexagonal faces surrounded by 12 pentagonal faces. The approximate eigenvalue is $21.97$. The CPU time is 374 {\it seconds}. See Section~\ref{sec:3dperiodic}.}\label{fig:8-periodic}
\end{figure}

For $k=16$, we obtain the well-known Kelvin Structure which consists of $16$ exactly same shapes as shown in Figure~\ref{fig:16-periodic} with the approximate eigenvalue $69.65$. This shape is usually called truncated octahedron which is a space-filling convex polyhedron with 6 square faces and 8 hexagonal faces. The CPU time for this computation from a random initialization with $128^3$ uniform discretization is only 656 {\it seconds}.

\begin{figure}[ht!]
\centering
\includegraphics[width = 0.17 \textwidth,clip, trim = 14cm 6cm 12cm 5cm]{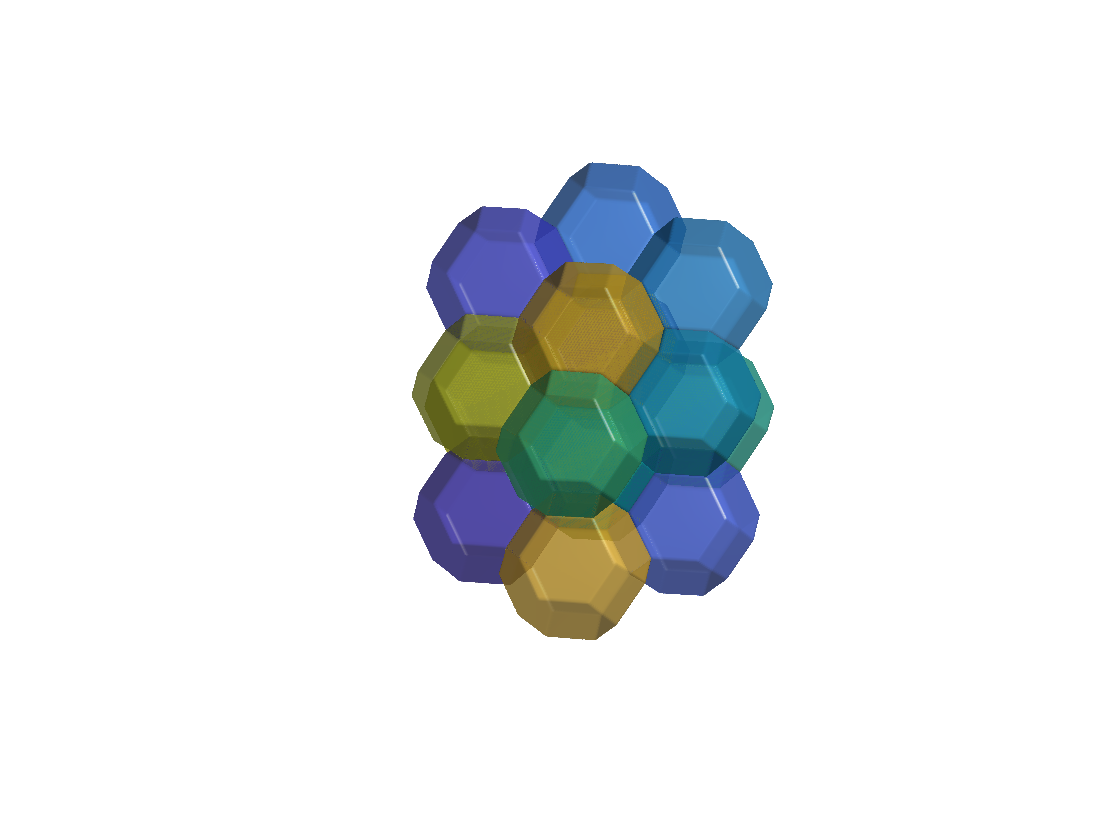}
\includegraphics[width = 0.17 \textwidth,clip, trim = 10cm 4cm 7cm 3cm]{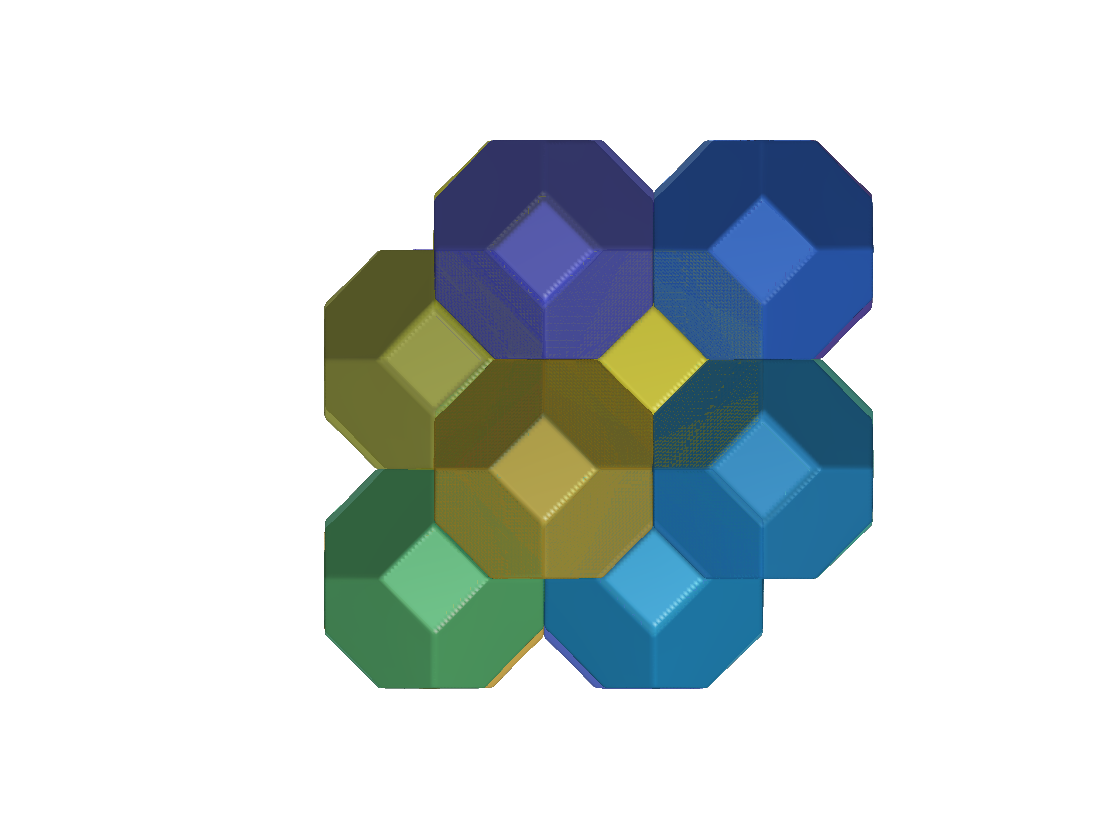}
\includegraphics[width = 0.17 \textwidth,clip, trim = 10cm 4cm 7cm 3cm]{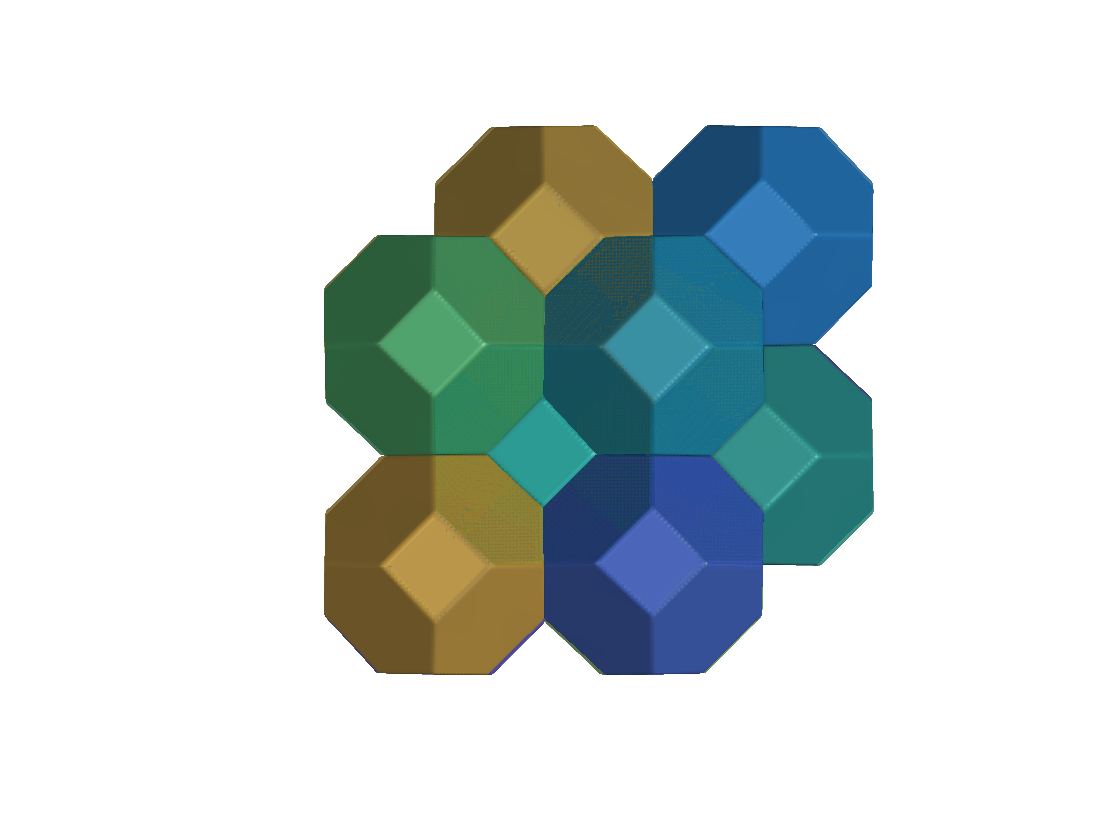}
\includegraphics[width = 0.17 \textwidth,clip, trim = 13cm 10cm 12cm 6cm]{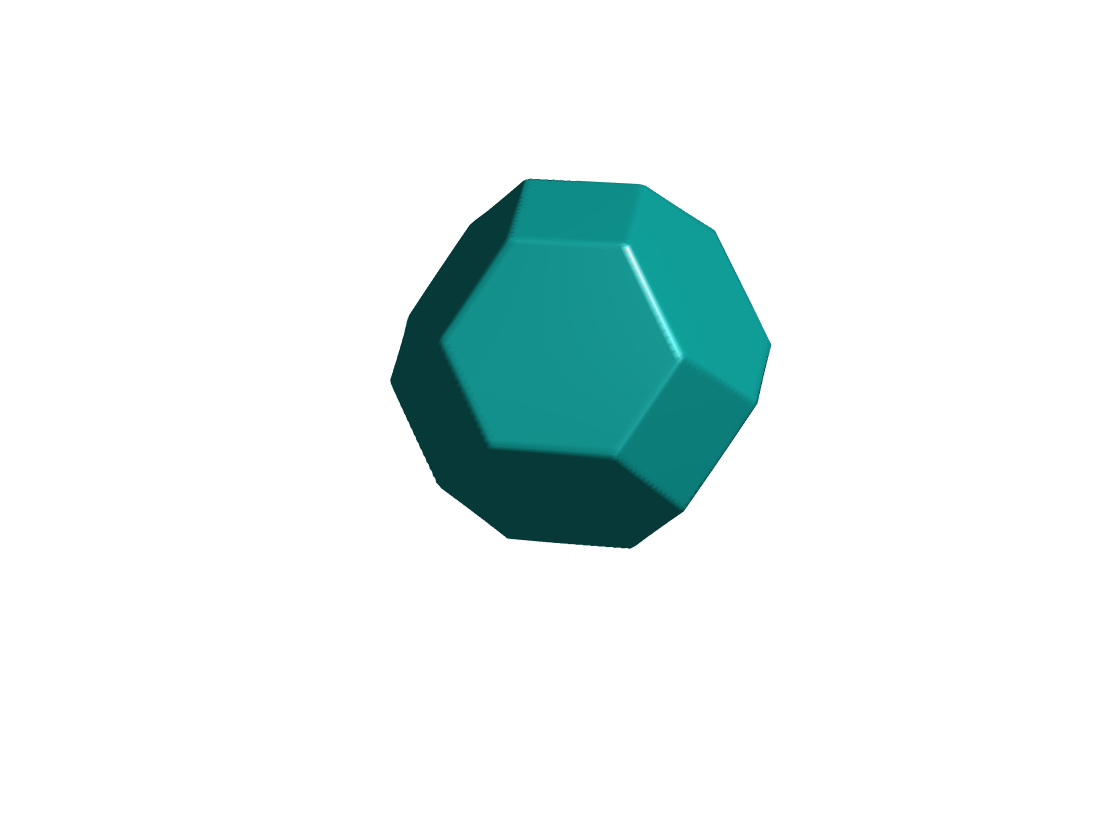}
\includegraphics[width = 0.17 \textwidth,clip, trim = 10cm 6cm 7cm 2cm]{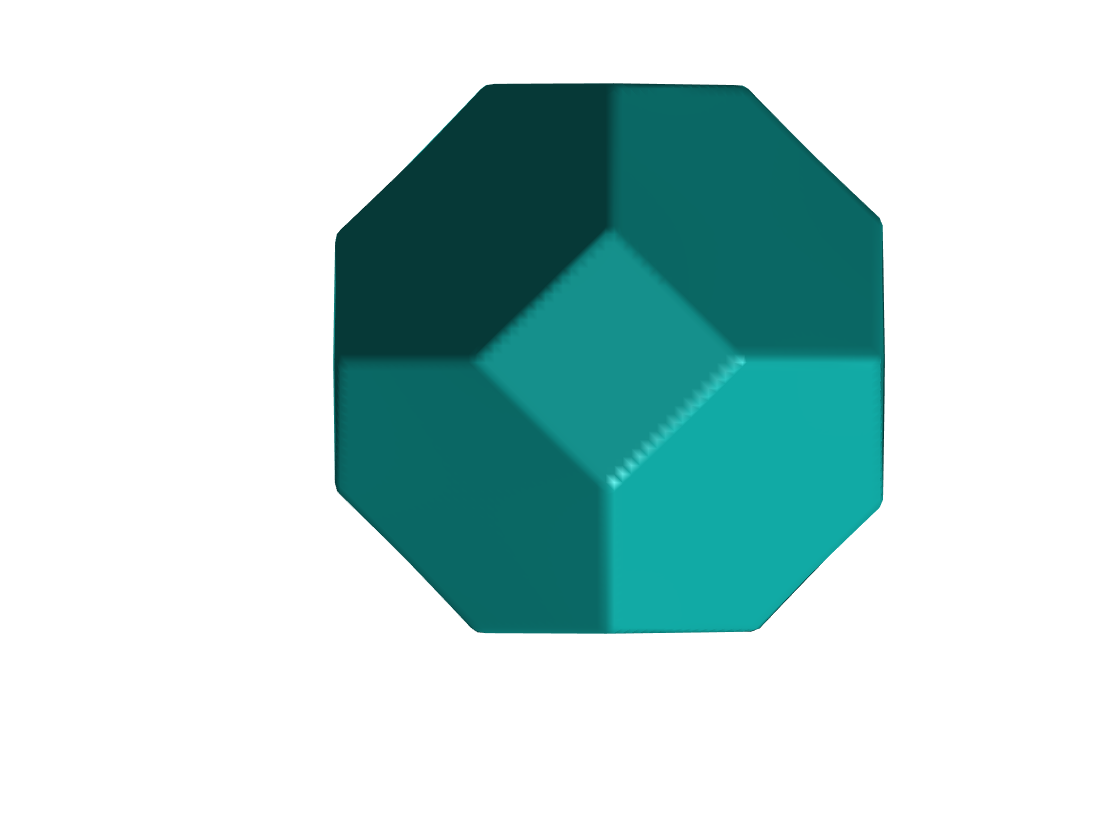}
\caption{The 16-partition for a periodic cube with different views and the shape of the truncated octahedron. Left to right: the 16-partition for a periodic cube, the xy-view of the structure (same as the yz-view), the xz-view of the structure, the truncated octahedron, and the xy-view of the truncated octahedron (same as the xz- and yz- views). The approximate eigenvalue is $69.65$. The CPU time for this computation from a random initialization with $128^3$ uniform discretization is only 656 {\it seconds}. See Section~\ref{sec:3dperiodic}.} \label{fig:16-periodic}
\end{figure}

\subsection{$k$-partition in arbitrary 3-dimensional domains}\label{sec:3darbitrary}
In this section, we consider the $k$-partition in several 3-dimensional domains to show the performance of the proposed method in 3-dimensional arbitrary domains with Dirichlet boundary conditions. In all following experiments, we set the computational domain as $[-\pi,\pi]^3$ discretized by $128^3$ uniform grid points and $\tau = \pi/16$. We consider the following three domains: $[-\pi/2,\pi/2]^3 \subset [-\pi,\pi]^3$, a ball centered at the origin with radius $\pi/2$, and a regular tetrahedron centered at the origin with radius of circumsphere $3\pi/4$. We mainly compare the results with the results computed from the method in \cite{bourdin2010optimal} and reported by Bogosel\footnote{\url{http://www.cmap.polytechnique.fr/~beniamin.bogosel/eig_part3D.html}}.

In Figure~\ref{fig:square}, we list the $k$-partitions in a cube for $k = 3-6$, $8$, $9$ and $14$ with some dissections to expose the interior shapes. All results agree with those reported by Bogosel. The approximate eigenvalues are  $ 23.23, 33.62, 48.02, 62.20, 95.41, 117.07$, and $244.02 $. The CPU time to obtain these results from random initial guesses are $64$, $41$, $157$, $108$, $57$, $233$, $368$ {\it seconds}, respectively.

\begin{figure}[ht!]
\centering
\includegraphics[width = 0.22 \textwidth, clip, trim = 4cm 1cm 3cm 1cm]{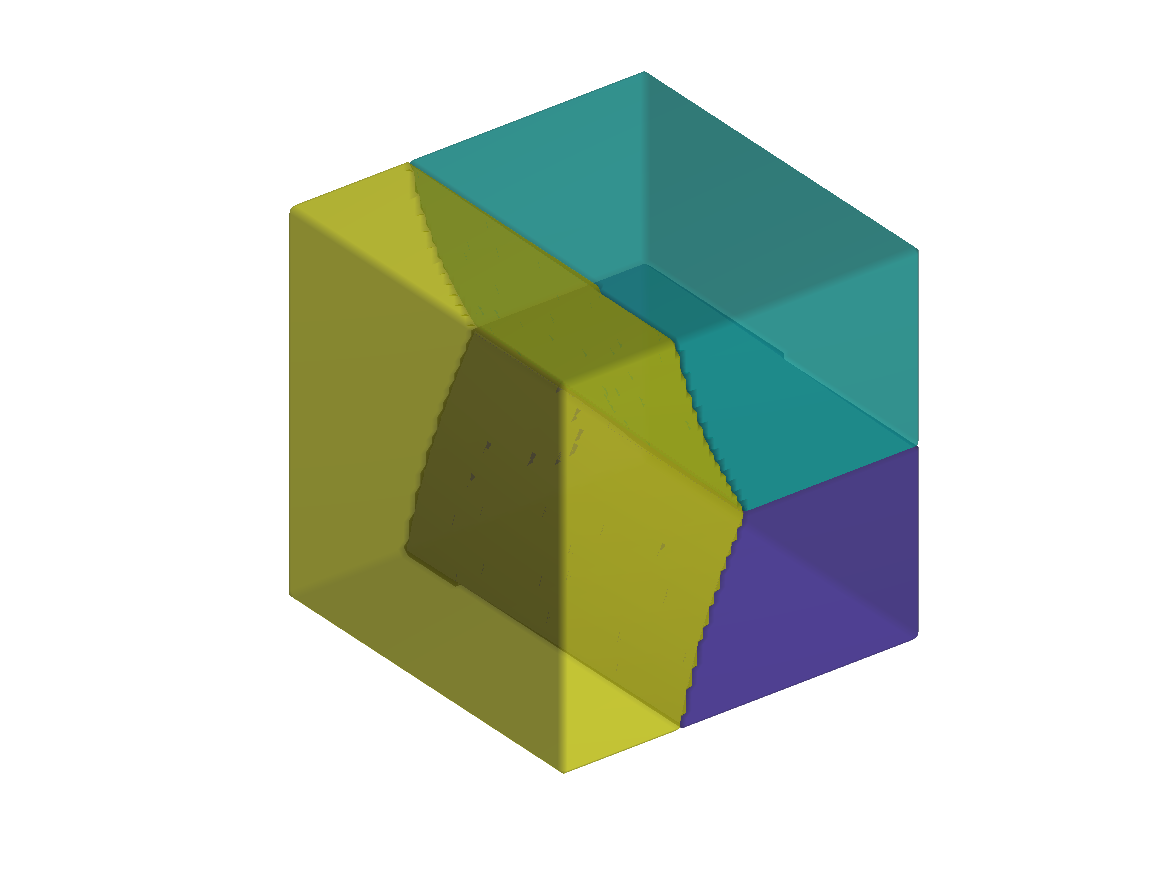}
\includegraphics[width = 0.22 \textwidth, clip, trim = 4cm 1cm 3cm 1cm]{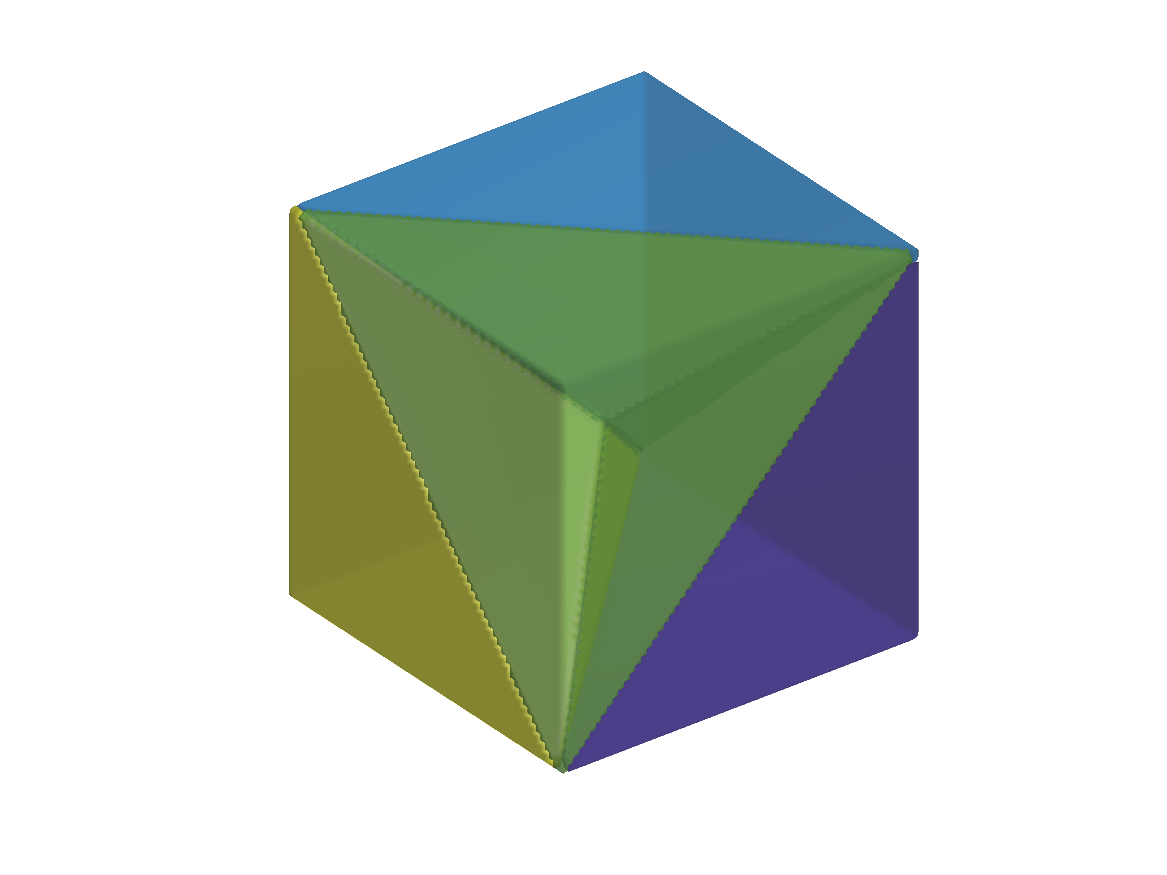}
\includegraphics[width = 0.22 \textwidth, clip, trim = 4cm 1cm 3cm 1cm]{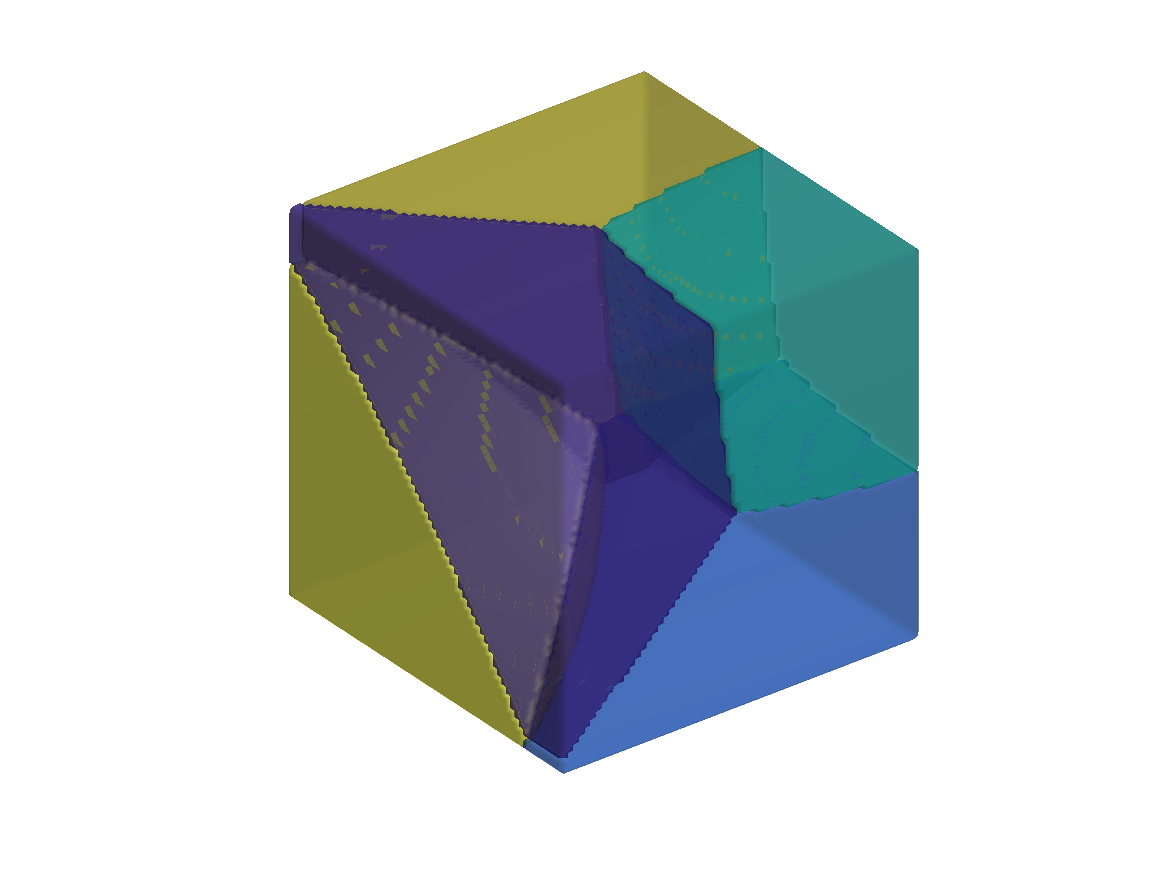}
\includegraphics[width = 0.22 \textwidth, clip, trim = 4cm 1cm 3cm 1cm]{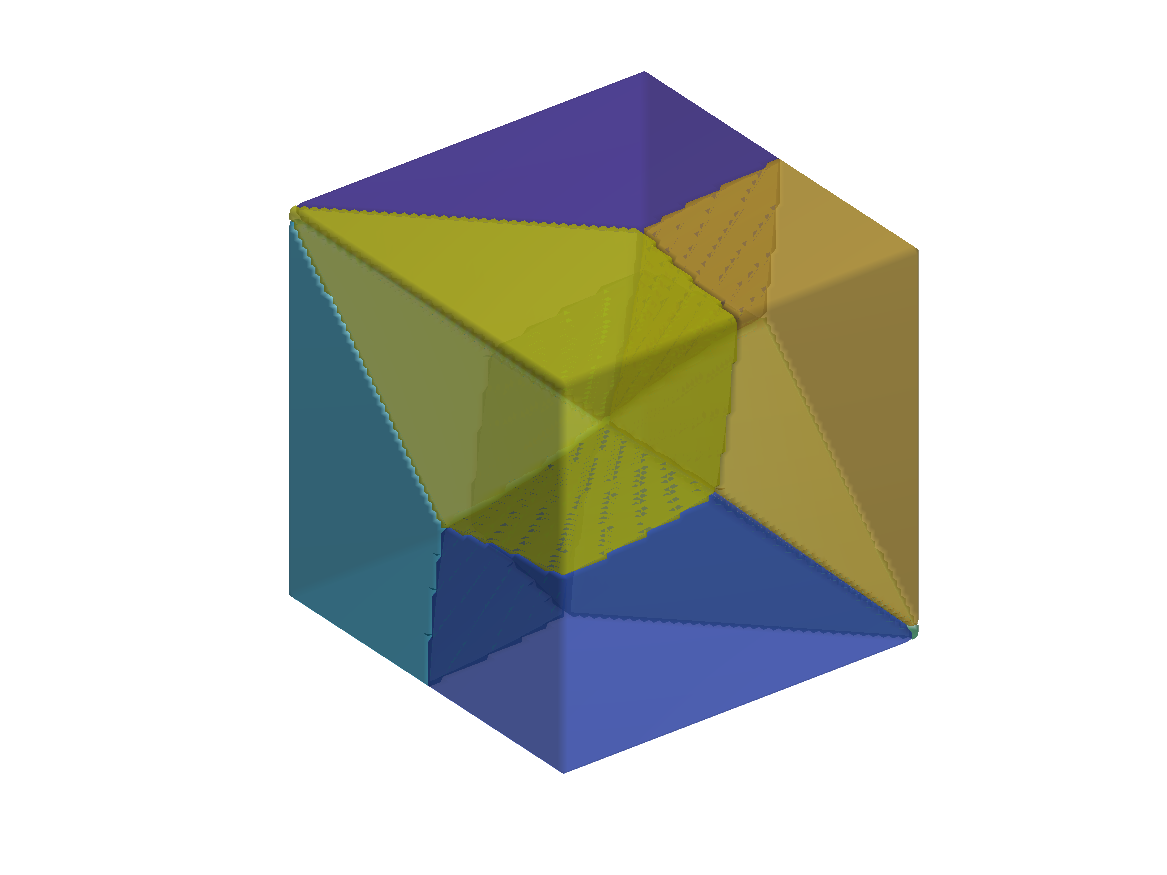}
\includegraphics[width = 0.22 \textwidth, clip, trim = 4cm 1cm 3cm 1cm]{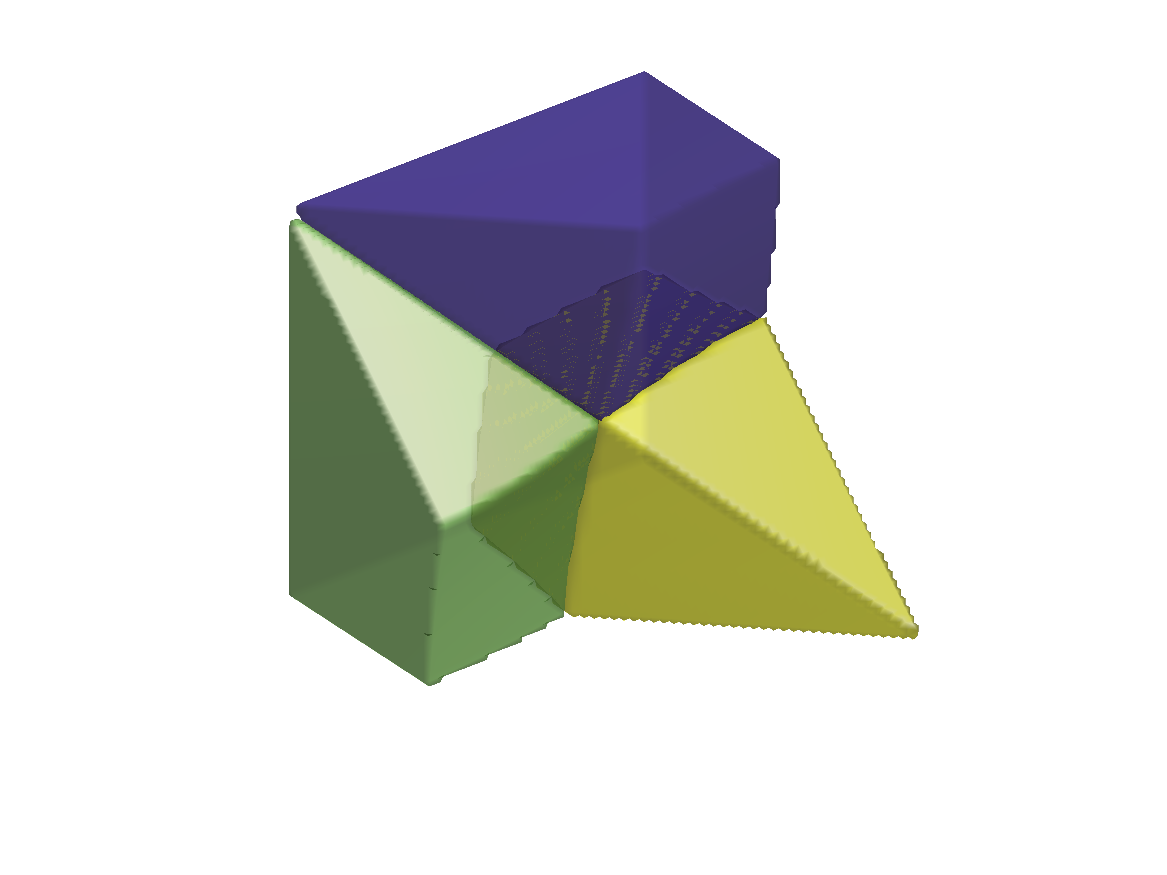}
\includegraphics[width = 0.22 \textwidth, clip, trim = 4cm 1cm 3cm 1cm]{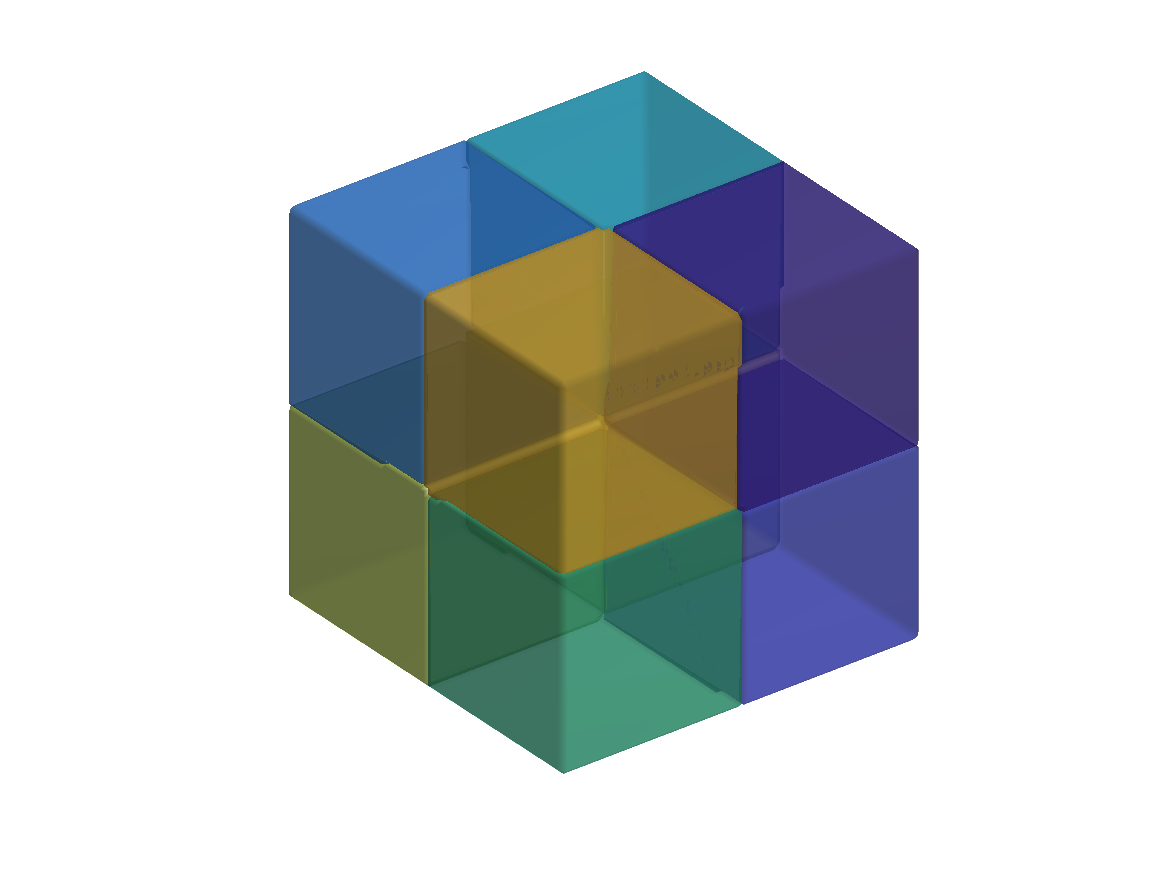}
\includegraphics[width = 0.22 \textwidth, clip, trim = 4cm 1cm 3cm 1cm]{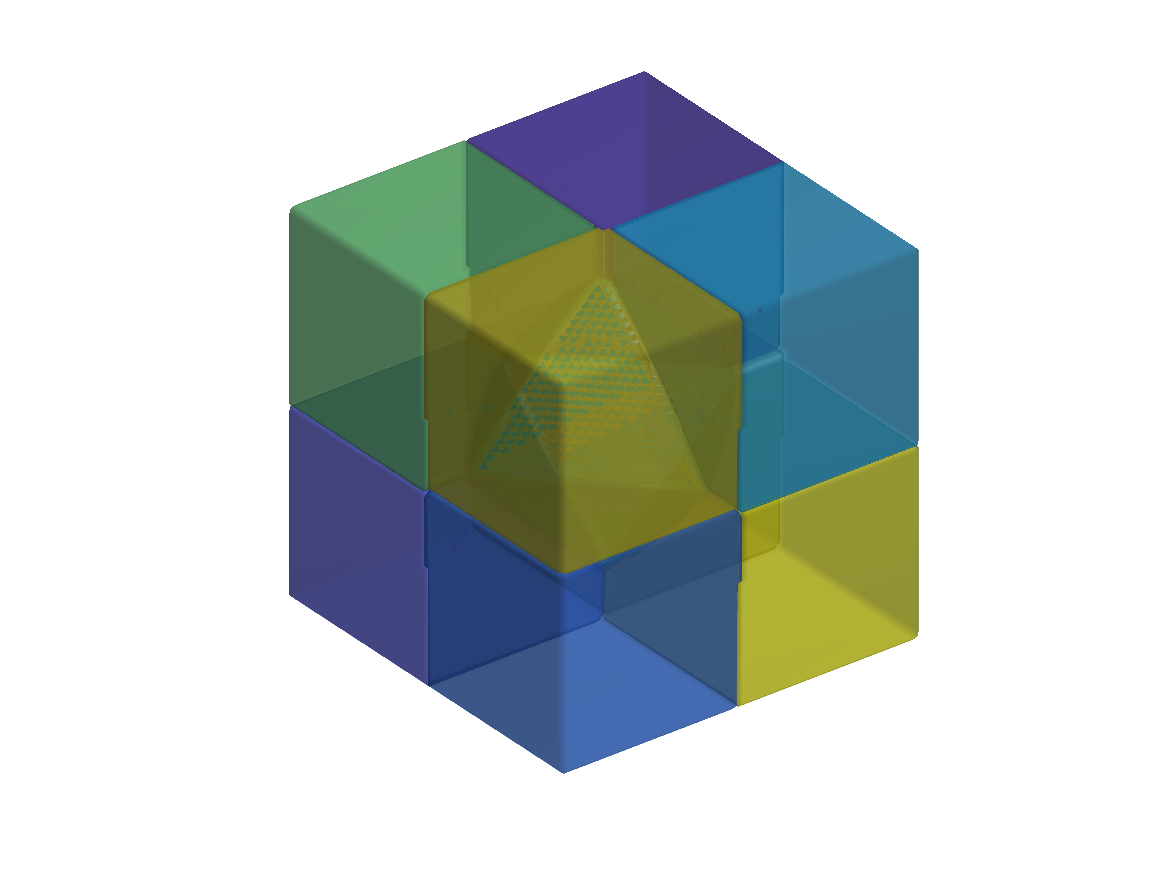}
\includegraphics[width = 0.22 \textwidth, clip, trim = 4cm 1cm 3cm 1cm]{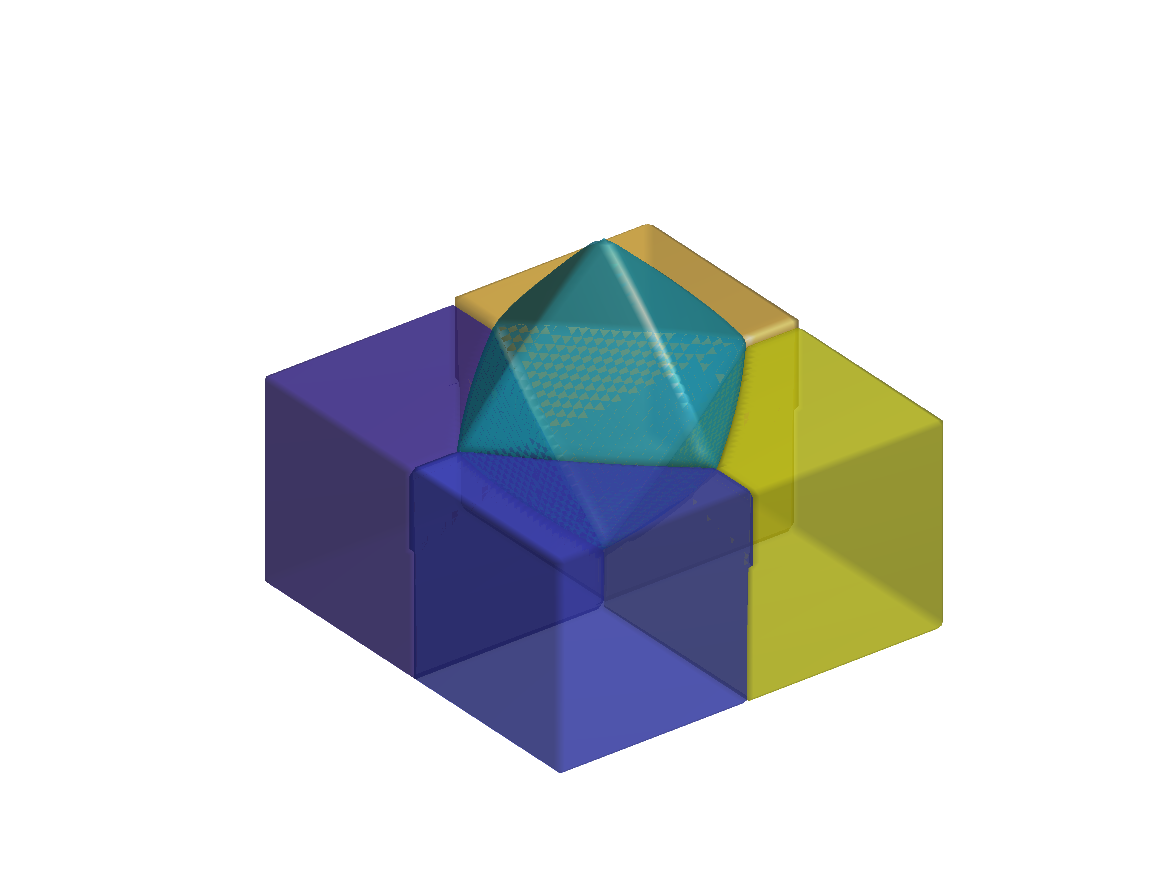}
\includegraphics[width = 0.22 \textwidth, clip, trim = 4cm 1cm 3cm 1cm]{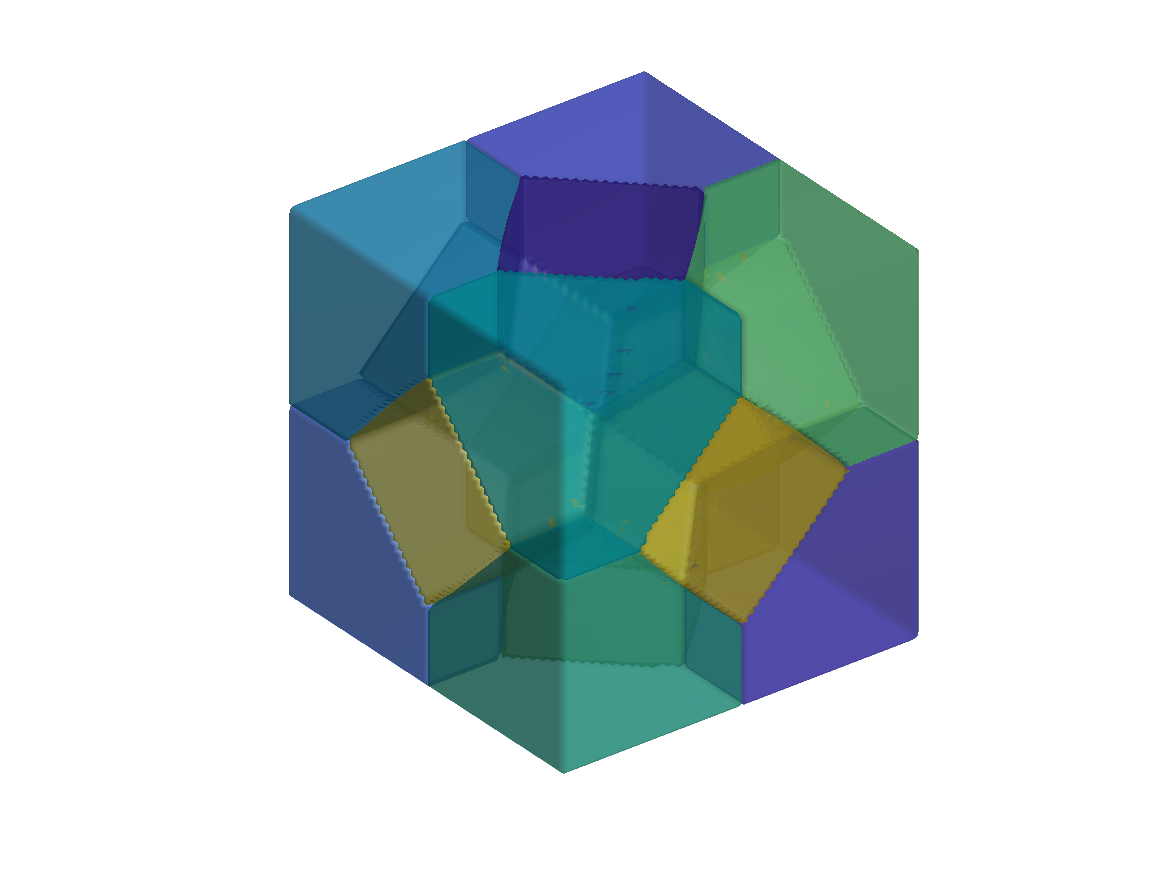}
\includegraphics[width = 0.22 \textwidth, clip, trim = 4cm 1cm 3cm 1cm]{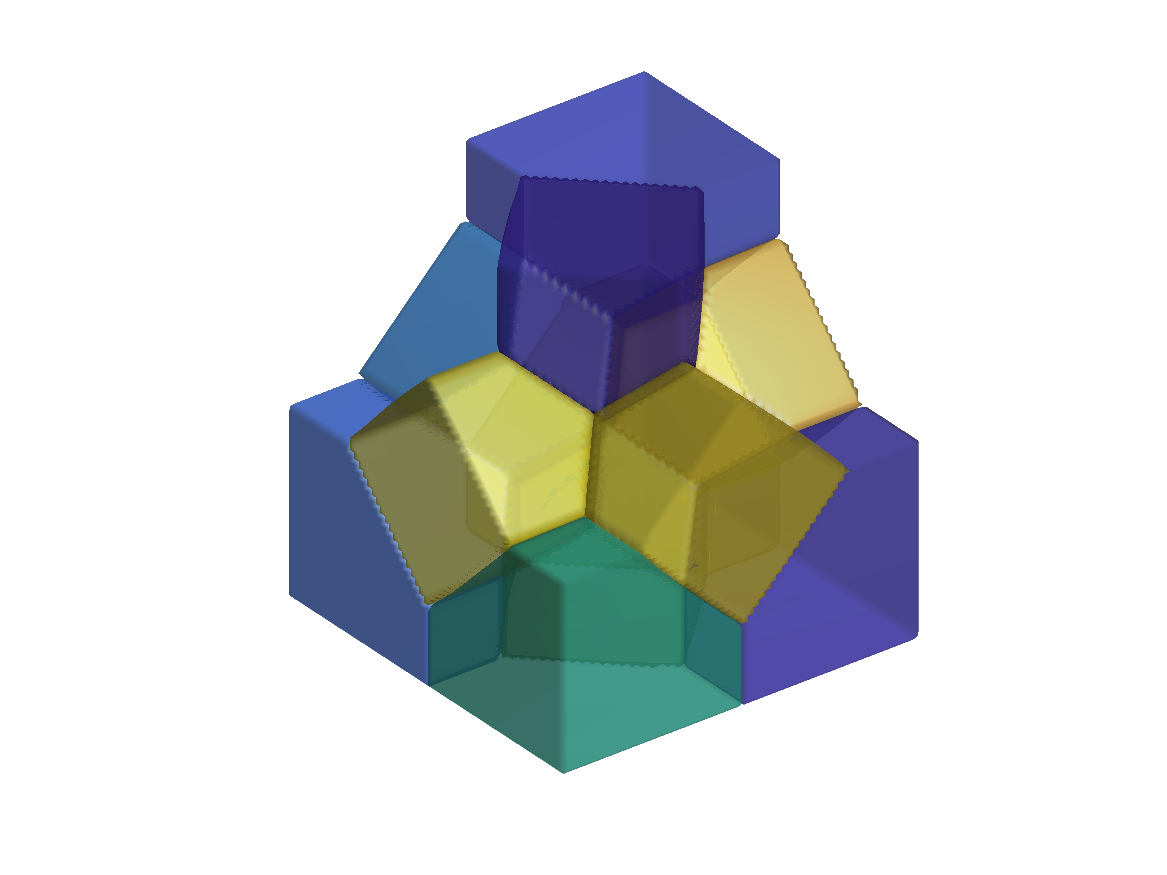}
\includegraphics[width = 0.22 \textwidth, clip, trim = 4cm 1cm 3cm 1cm]{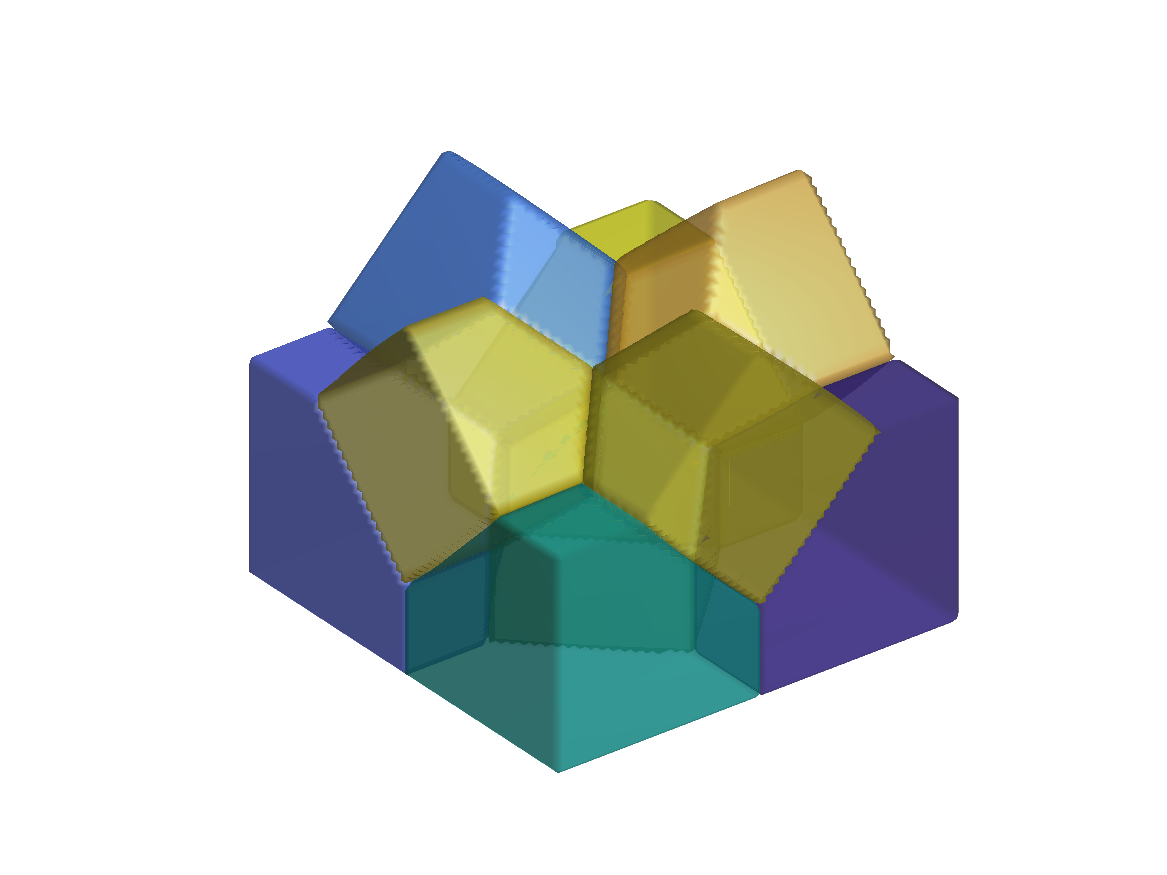}
\includegraphics[width = 0.22 \textwidth, clip, trim = 2.6cm 1cm 2cm 0cm]{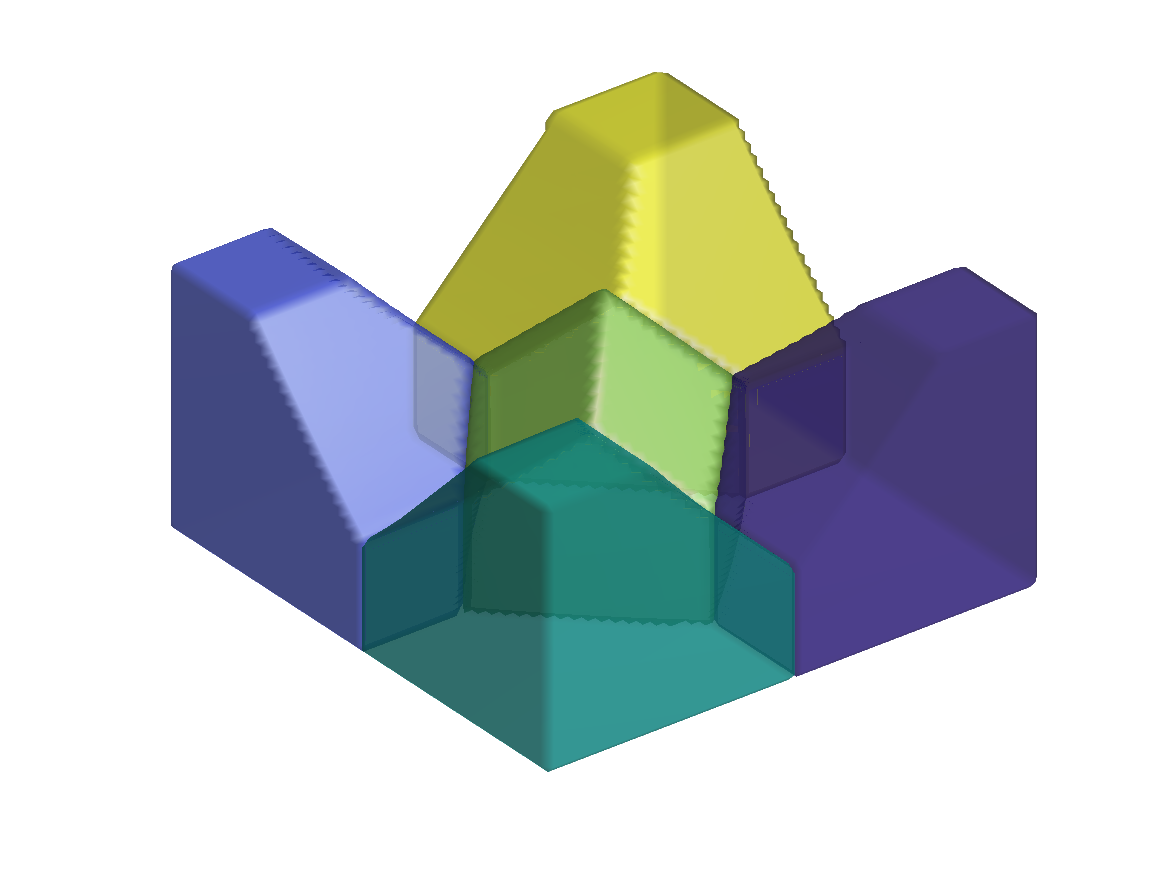}

\medskip

{
\begin{tabular}{c|c|c|c|c|c|c}
\hline
  k=3  & k=4 & k=5& k=6 & k=8& k=9 & k=14  \\ 
\hline 
  23.23 & 33.62 & 48.02 &62.20  & 95.41 & 117.07 & 244.02 \\
\hline 
\end{tabular}
}

\caption{First row: $3-6$-partitions, second row: a dissection of $6$-partition,  $8$-partition,  $9$-partition, and a dissection of  $9$-partition, third row: $14$-partition and dissections of the $14$-partition. The table lists the approximate eigenvalues for different $k$. The CPU time to obtain these results from random initial guesses are $64$, $41$, $157$, $108$, $57$, $233$, $368$ {\it seconds}, respectively. See Section~\ref{sec:3darbitrary}.} \label{fig:square}
\end{figure}

\begin{figure}[ht!]
\centering
\includegraphics[width = 0.22 \textwidth, clip, trim = 12cm 7cm 11cm 6cm]{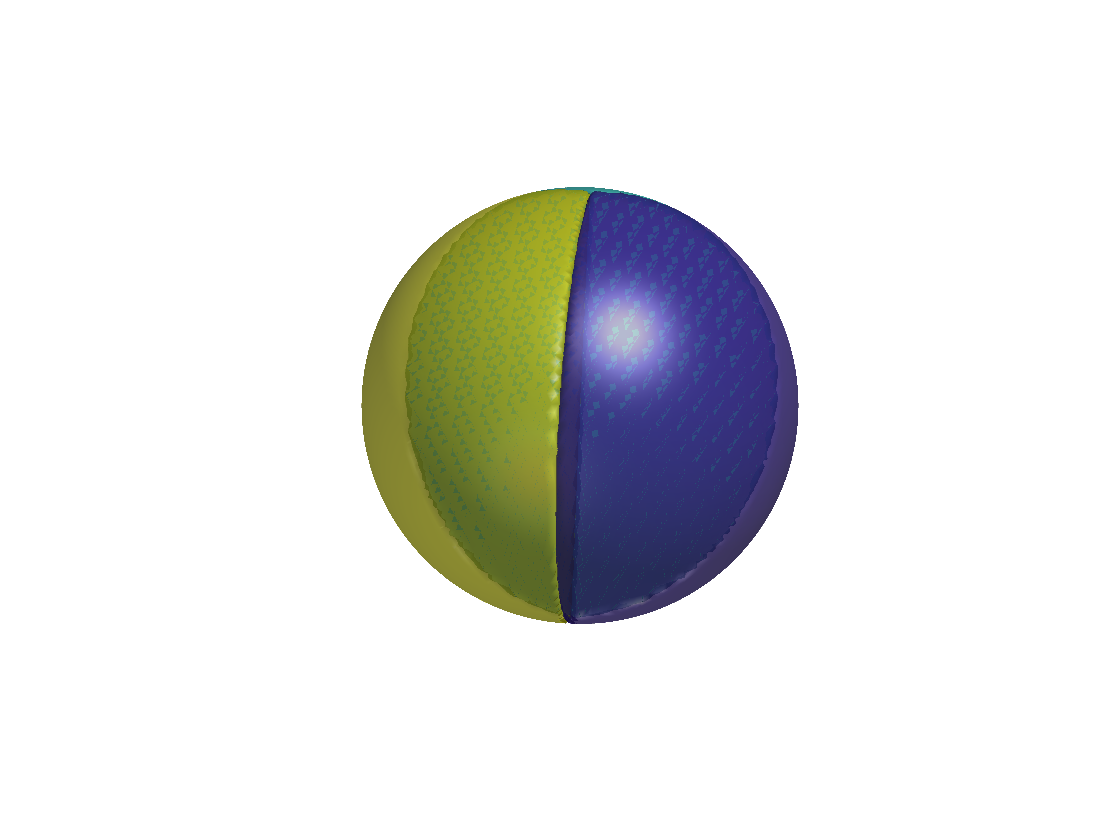}
\includegraphics[width = 0.22 \textwidth, clip, trim = 12cm 7cm 11cm 6cm]{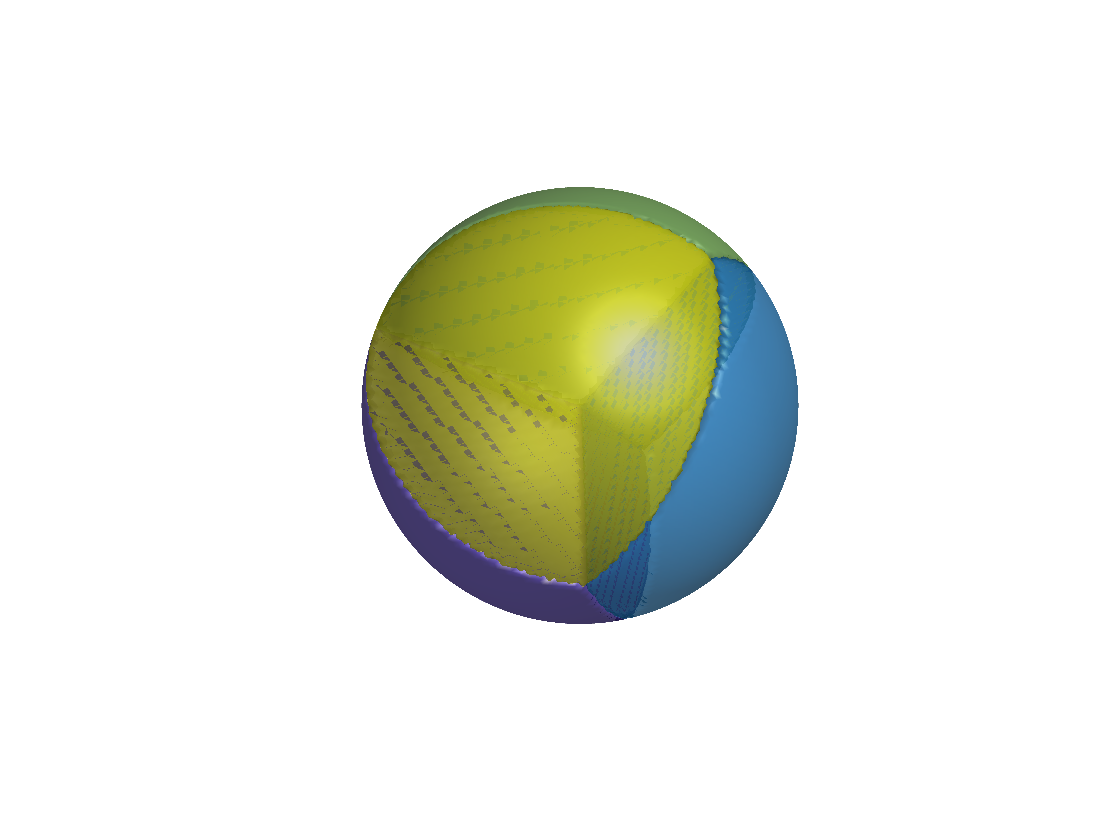}
\includegraphics[width = 0.22 \textwidth, clip, trim = 12cm 7cm 11cm 6cm]{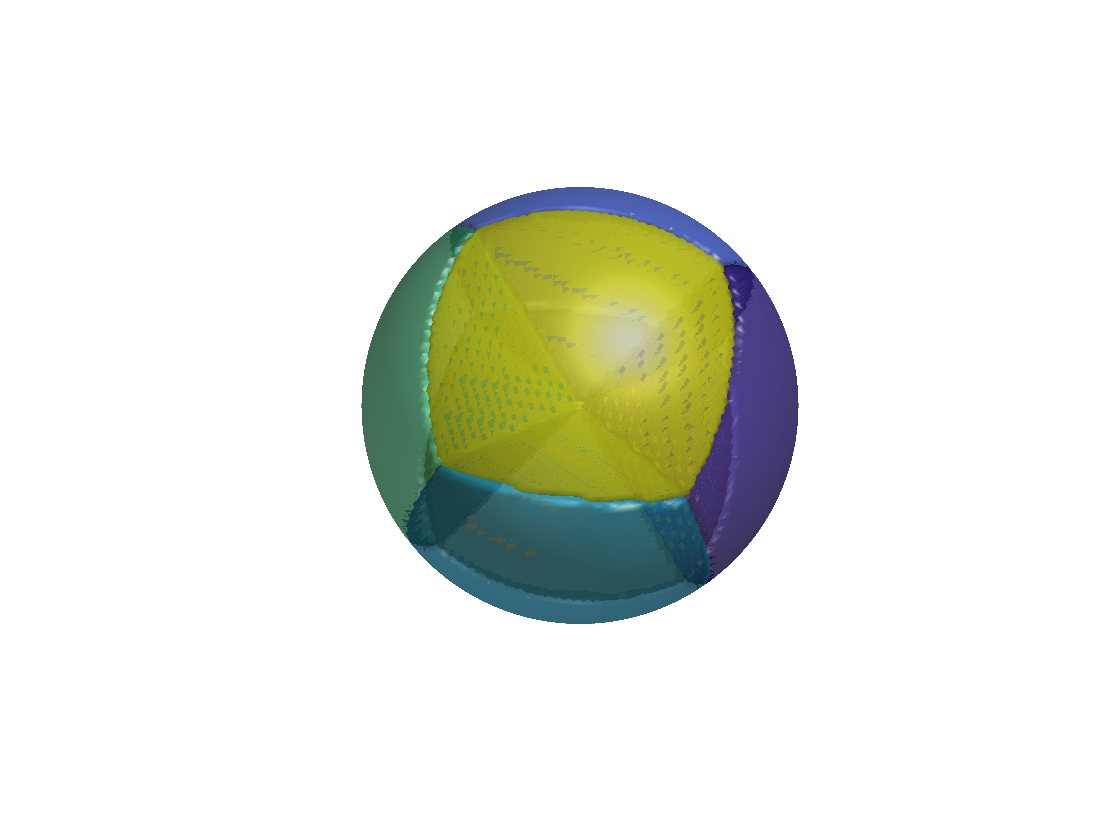}
\includegraphics[width = 0.22 \textwidth, clip, trim = 12cm 7cm 11cm 6cm]{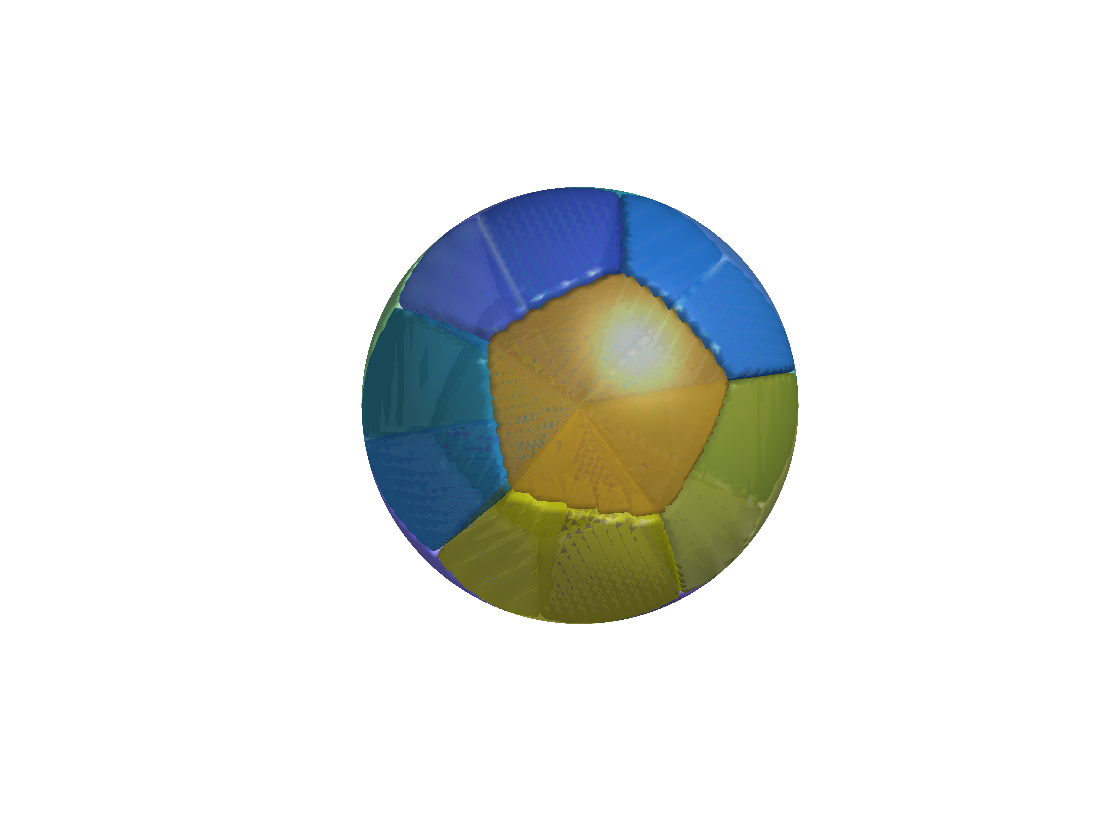}
\includegraphics[width = 0.22 \textwidth, clip, trim = 12cm 7cm 11cm 6cm]{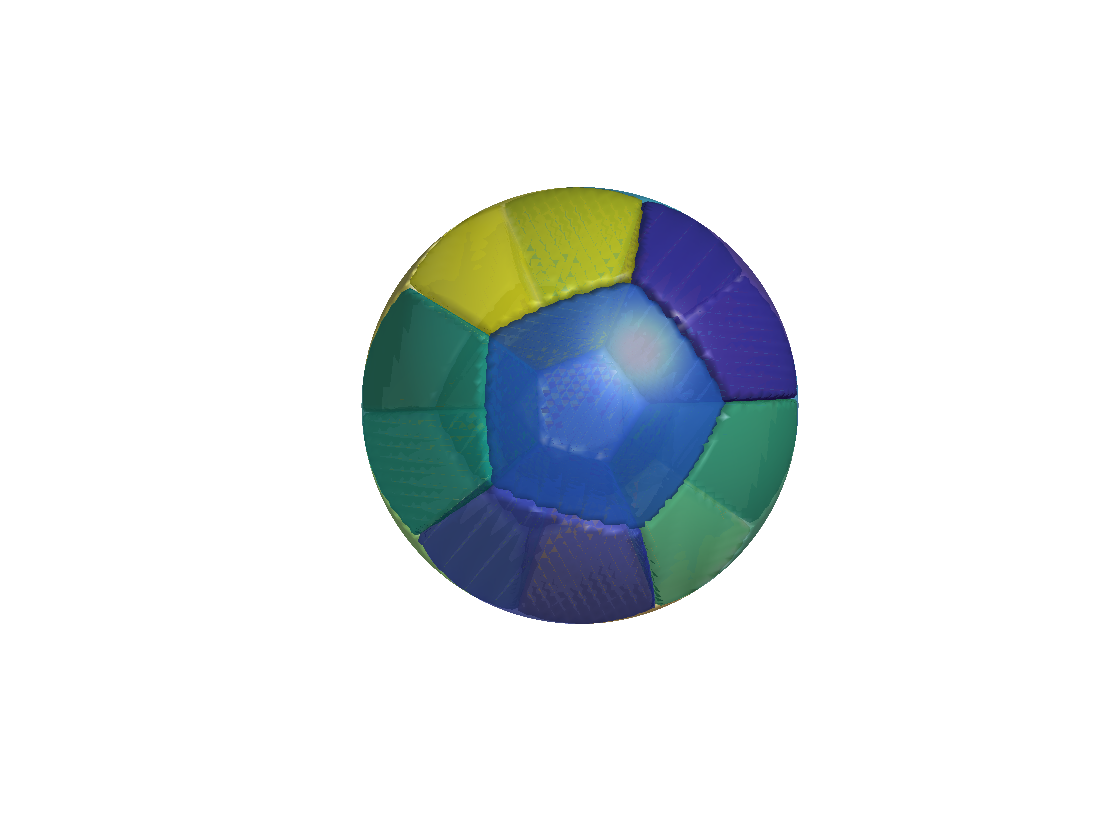}
\includegraphics[width = 0.22 \textwidth, clip, trim = 9cm 4cm 9cm 4cm]{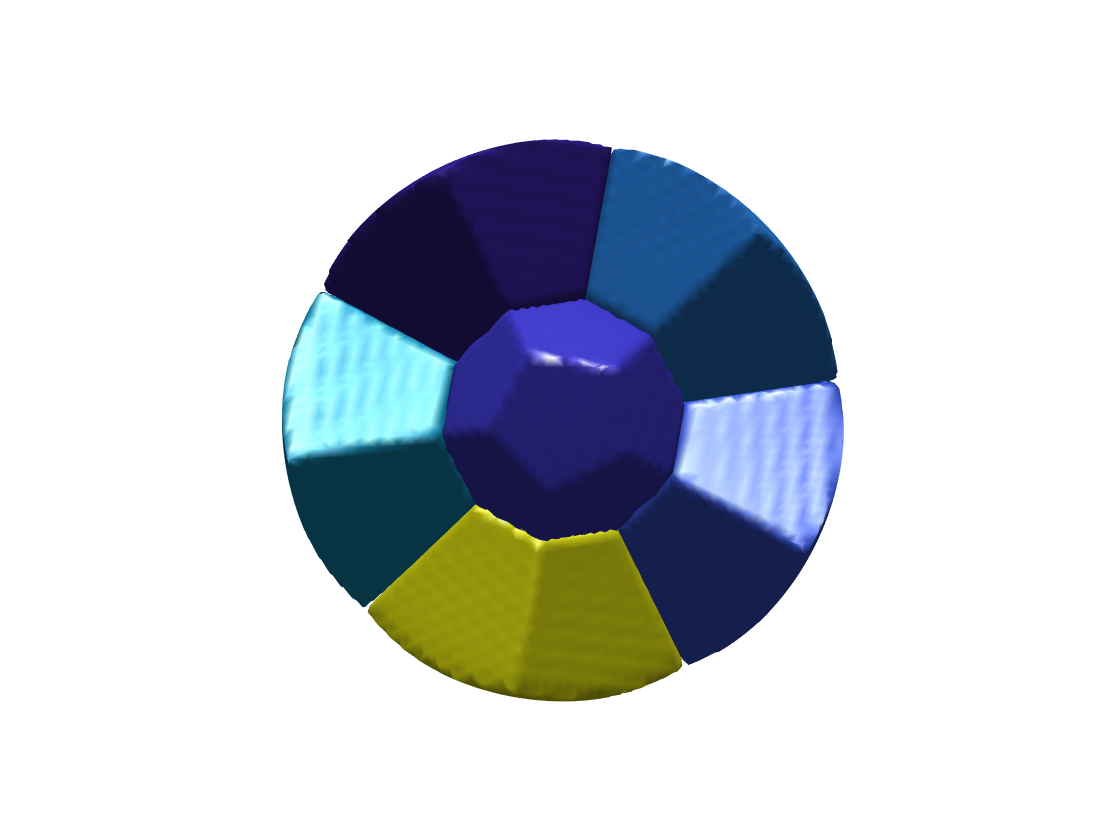}
\includegraphics[width = 0.22 \textwidth, clip, trim = 12cm 7cm 11cm 6cm]{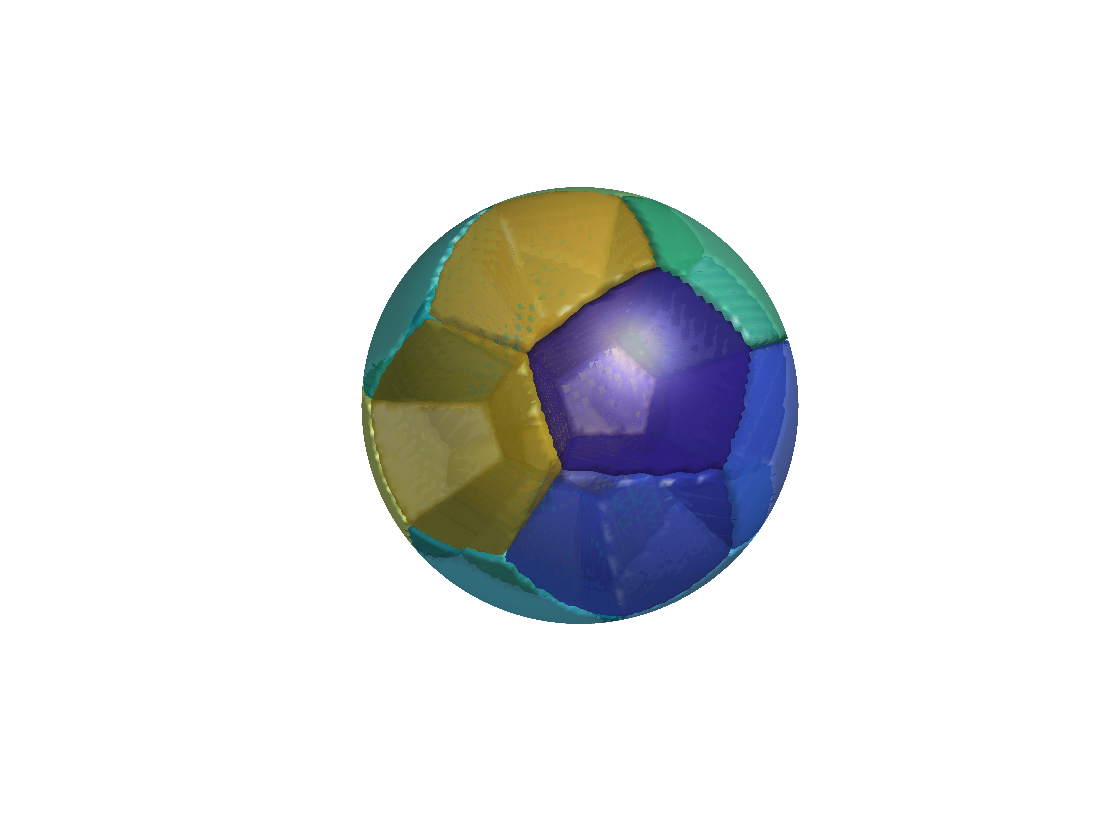}
\includegraphics[width = 0.22 \textwidth, clip, trim = 9cm 4cm 9cm 4cm]{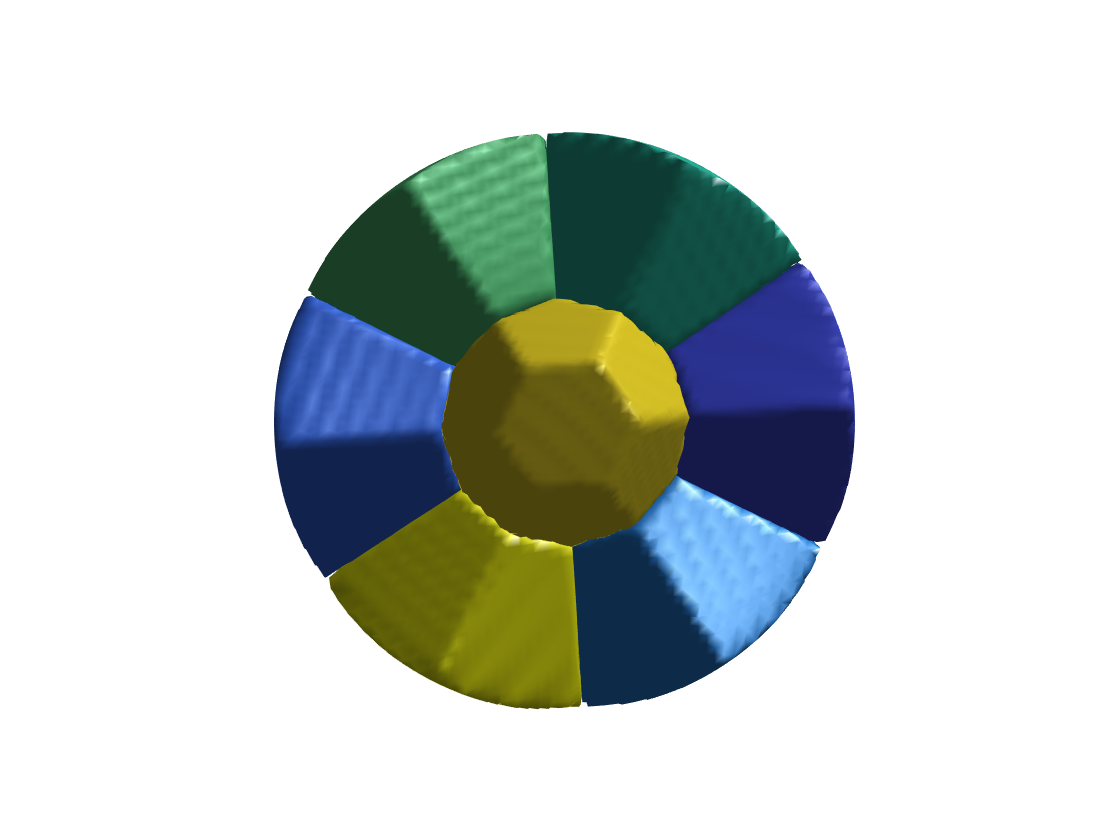}

\medskip

{
\begin{tabular}{c|c|c|c|c|c}
\hline
  k=3  & k=4 & k=6& k=12 & k=13& k=15  \\ 
\hline 
 31.68 &  49.07 & 92.95 & 285.36  & 320.59 &  405.72    \\
\hline 
\end{tabular}
}
\caption{The optimal $k$-partition in a ball with $k = 3,4,6$, $12$, $13$, and $15$ and dissections of $13$-partition and $15$-partition to expose the interior shapes. The table lists the approximate eigenvalues. The CPU time for $k = 3,4,6,12,13,$ and $15$ are $20, 41,193, 296,143,$ and $388$ {\it seconds}, respectively. See Section~\ref{sec:3darbitrary}.}\label{fig:ball}
\end{figure}

Figure~\ref{fig:ball} lists the optimal $k$-partition in a ball with $k = 3,4,6$, $12$, $13$, and $15$. The approximate eigenvalues are $31.68, 49.07, 92.95, 285.36,  320.59 $, and $405.72$, respectively. For $k=3, 4, 6, 12 $, and $13$, all results agree with the results reported by Bogosel. The $13$-partition is very regular and composed of one interior region and 12 regions that are on the boundary. Interestingly, the interior bubble is very similar to a regular dodecahedron as shown in Figure~\ref{fig:ball}.
Furthermore, we observe that the $15$-partition is also very regular and is composed of one interior region and 14 regions on the boundary.  The interior shape is very similar to the truncated hexagonal trapezohedron that appears in the Weaire–Phelan structure similar to the second shape showed in Figure~\ref{fig:8-periodic}. The shapes on the boundary consist of twelve rounded truncated pentagonal trapezohedron and two rounded truncated hexagonal trapezohedron as shown in Figure~\ref{fig:ball}. These results are also similar to the optimal structure of foam bubbles in the sense of minimizing the total surface area reported in \cite{Wang_2019b}. The CPU time for $k = 3,4,6,12,13,$ and $15$ are $20, 41,193, 296,143,$ and $388$ {\it seconds}, respectively.

When the domain is a tetrahedron, we list the results for $k = 2,4,10$, and $20$ in Figure~\ref{fig:tetrahedron}.  For $k = 2, 4, 10$, and $20$, we recover the results in 14, 14, 97, and 244 {\it seconds} with approximate eigenvalues $24.02, 58.71, 254.21$, and $776.74$, respectively.

\begin{figure}[ht!]
\centering
\includegraphics[width = 0.22 \textwidth, clip, trim = 8cm 3cm 8cm 6cm]{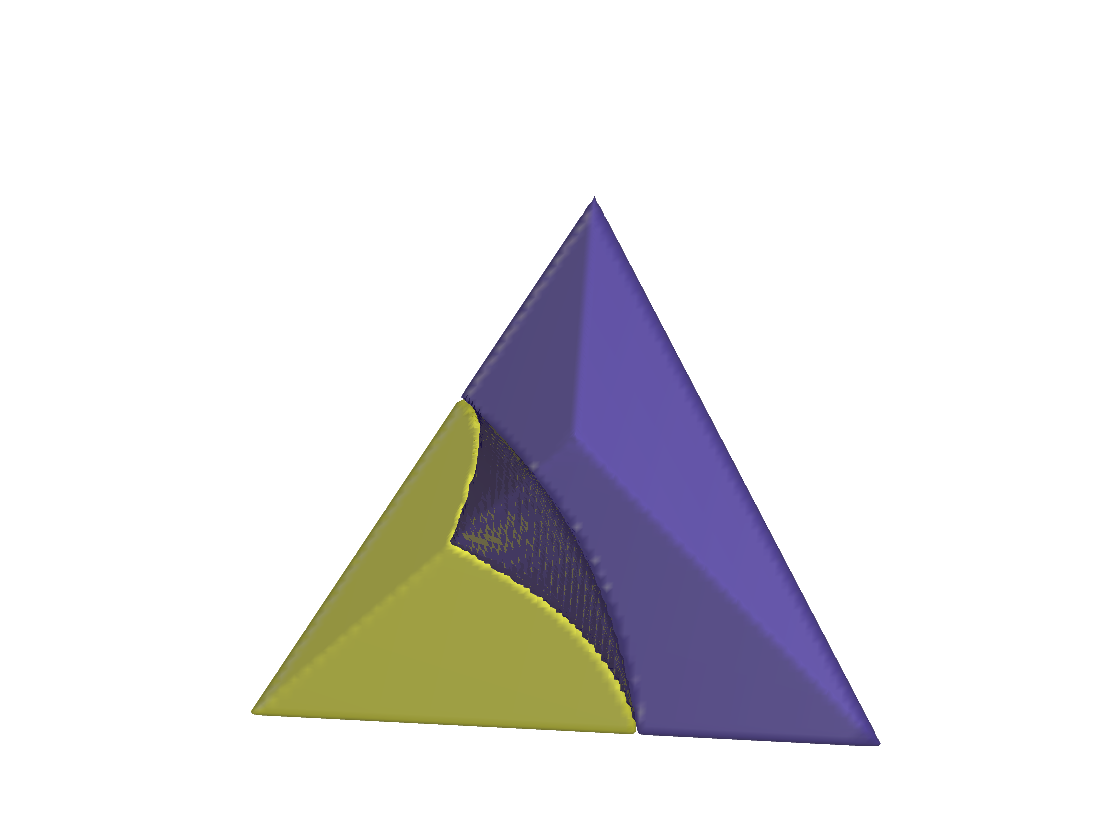}
\includegraphics[width = 0.22 \textwidth, clip, trim = 8cm 3cm 8cm 6cm]{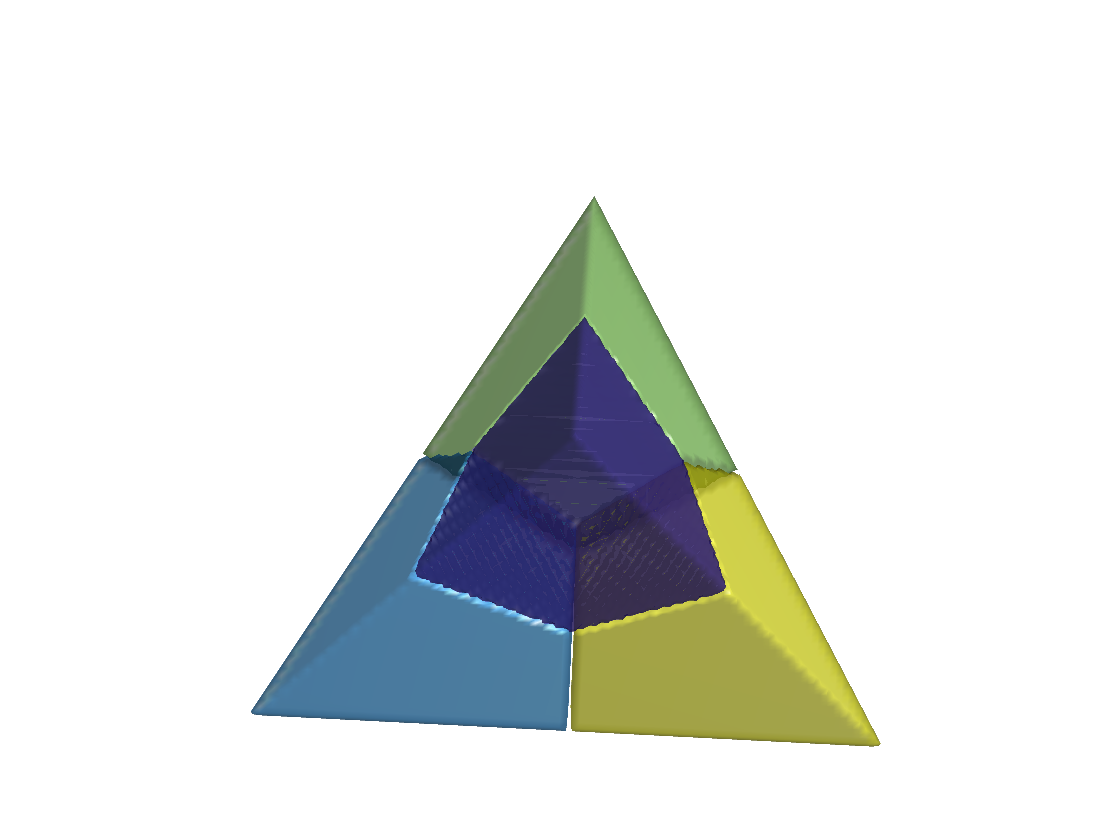}
\includegraphics[width = 0.22 \textwidth, clip, trim = 8cm 3cm 8cm 6cm]{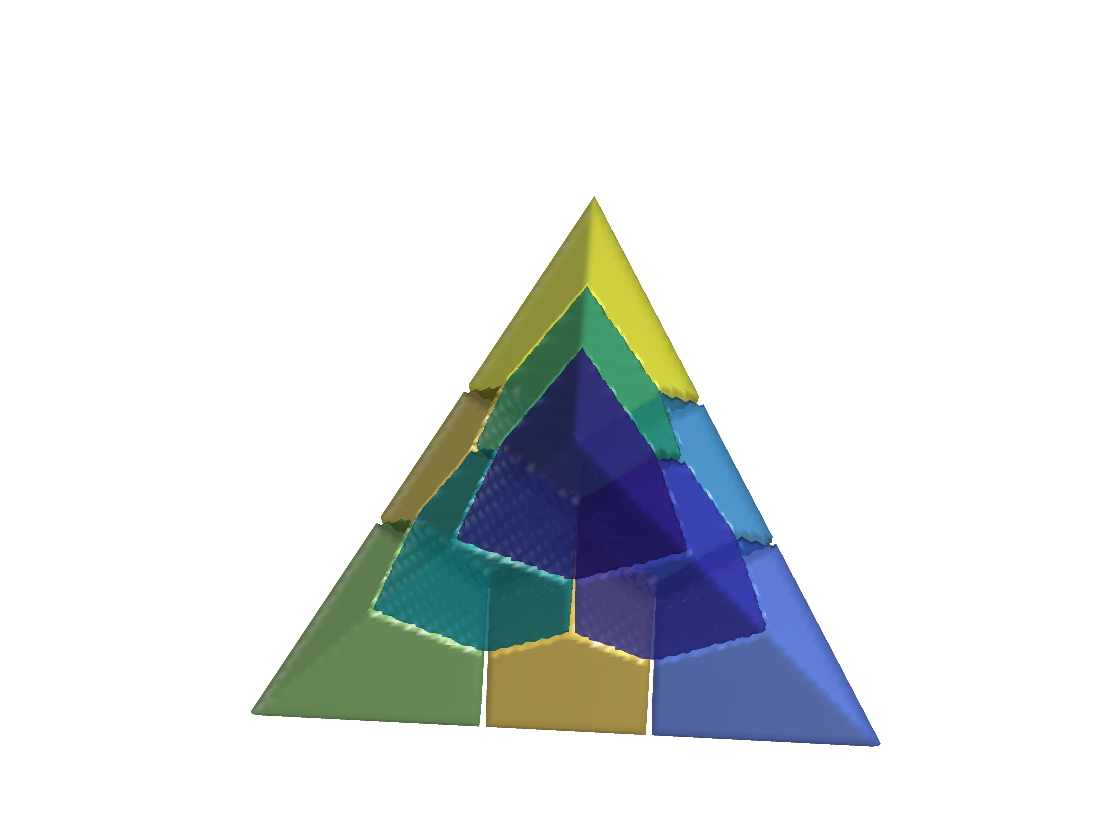}
\includegraphics[width = 0.22 \textwidth, clip, trim = 7cm 3cm 7cm 4cm]{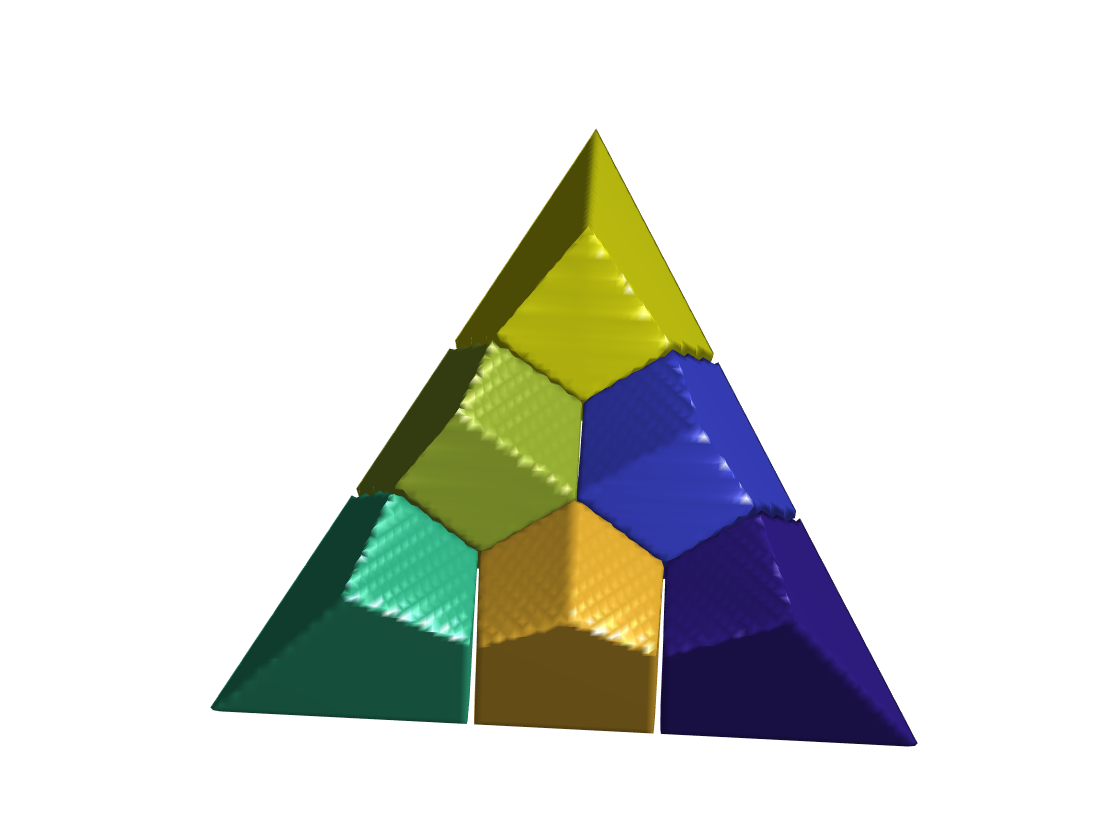} 
\includegraphics[width = 0.22 \textwidth, clip, trim = 8cm 3cm 8cm 6cm]{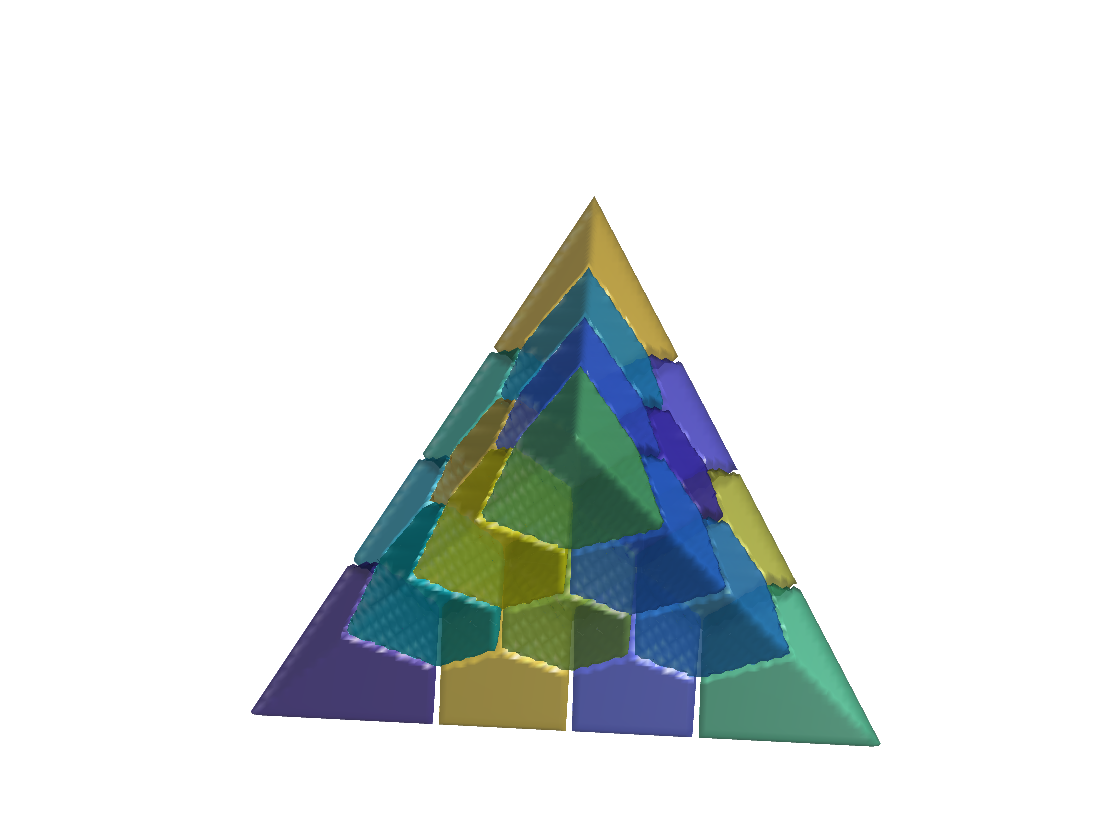}
\includegraphics[width = 0.22 \textwidth, clip, trim = 7cm 3cm 7cm 4cm]{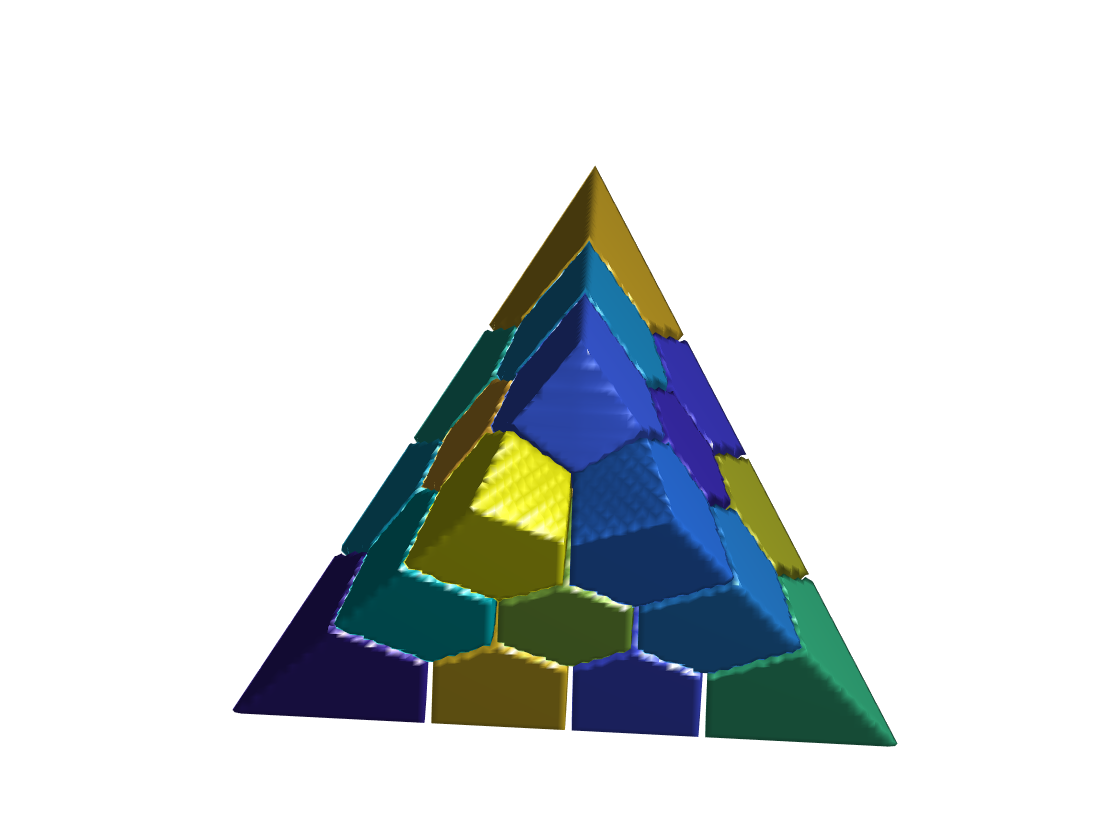}
\includegraphics[width = 0.22 \textwidth, clip, trim = 7cm 3cm 7cm 4cm]{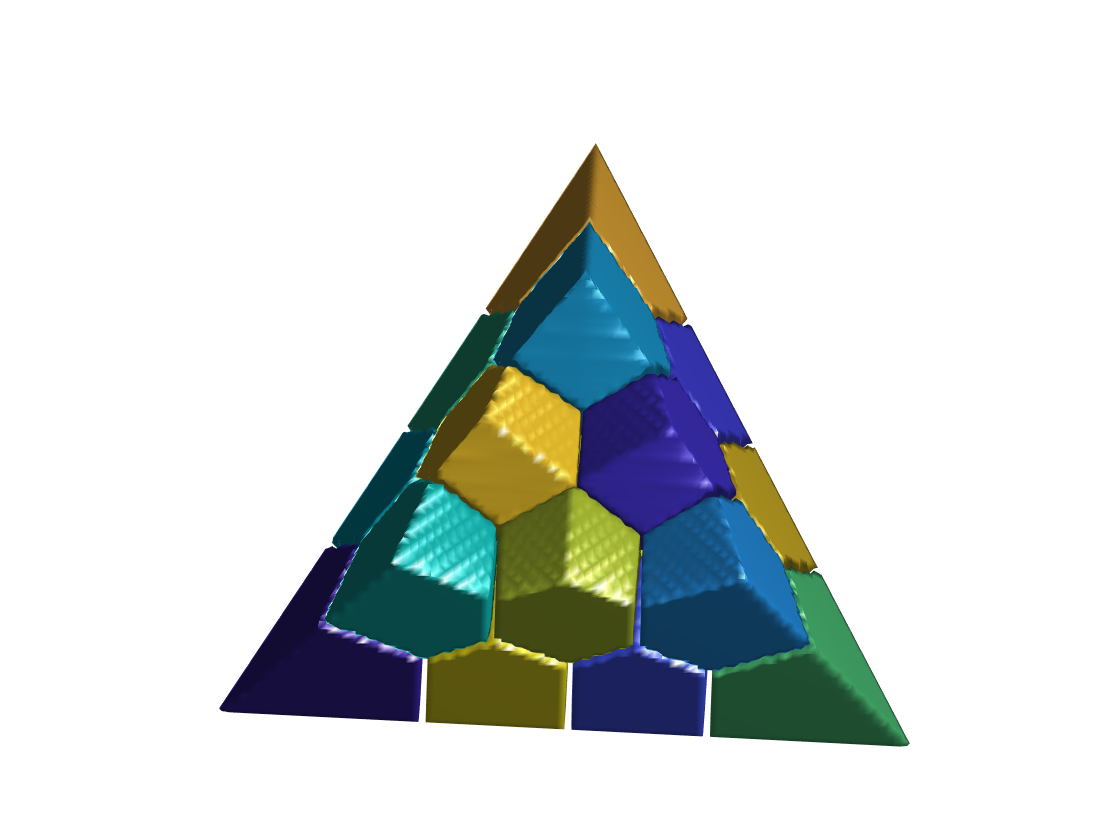}
\includegraphics[width = 0.22 \textwidth, clip, trim = 7cm 3cm 7cm 4cm]{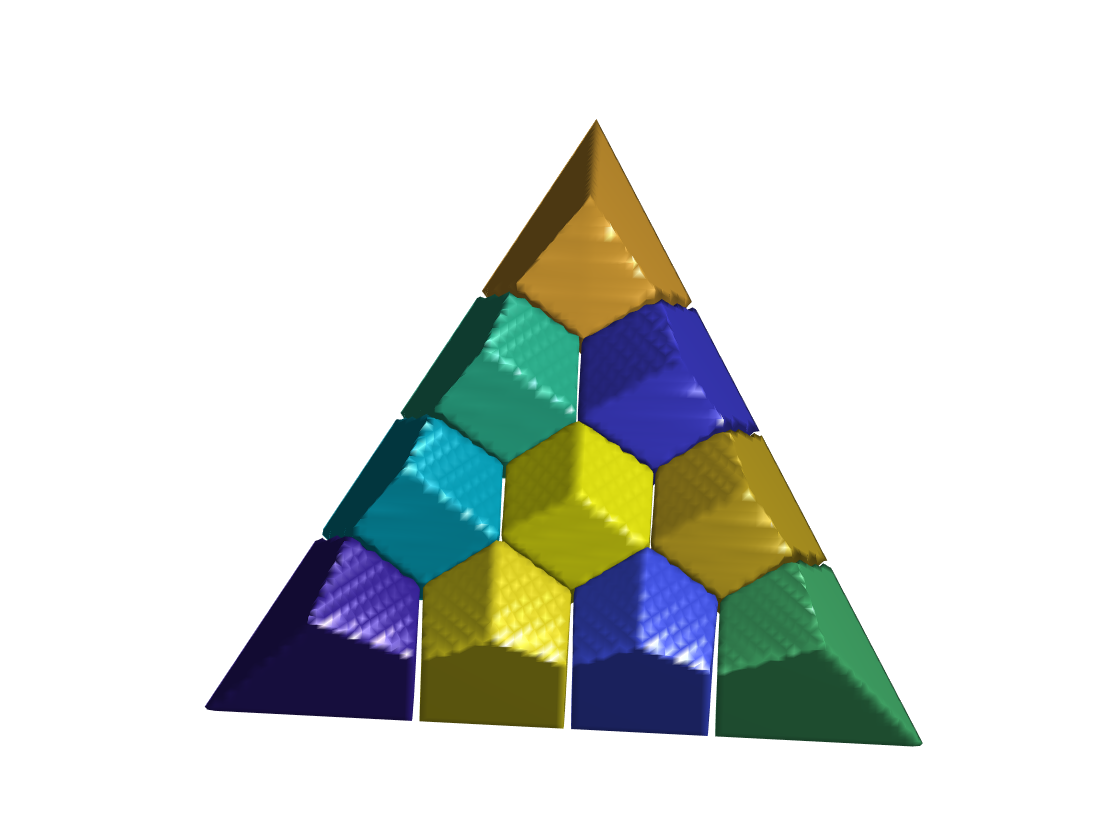}

\caption{Left to right, top to bottom: The optimal $k$-partition (obtained by the proposed method) in a tetrahedron with $k = 2,4,10$, the dissection of $10$-partition, the $20$-partition, three dissections of the $20$-partition. The CPU time for $k = 2,4,10,$ and $20$ are $14, 14, 97$, and $244$ {\it seconds}  and the approximate eigenvalues are $24.02, 58.71, 254.21$, and $776.74$, respectively. See Section~\ref{sec:3darbitrary}.}\label{fig:tetrahedron}
\end{figure}

\section{Conclusion and discussions}\label{sec:diss}

In this paper, we proposed a new relaxation of Dirichlet $k$-partition problems in arbitrary domains and derived a novel algorithm for computing Dirichlet $k-$partitions. The algorithm is very efficient and insensitive to domains. We theoretically proved the monotonically decaying property of the approximate energy. Numerical results show that the proposed method can achieve more than hundreds of times acceleration. 

To our knowledge, this is the first paper on relaxing the Dirichlet $k$-partition via using concave functionals and auxiliary indicator functions. A rigorous proof of the convergence of the new approximation as $\tau \rightarrow 0$ is needed for the theoretical guarantee of the new approximation. Besides, in this work, we compute the convolution using FFT by an extension of the domain of interest. Because the values out of the domain are all 0, one can also implement the algorithm by fast Multipole methods or Non-uniform fast Fourier transform based approaches \cite{jiang_2017,WJW19} to further accelerate the algorithm. These are out of the scope of this work and will be reported elsewhere. 

\subsection*{Acknowledgement}

D. Wang would like to thank Shihua Gong, Shingyu Leung, Yutian Li, Braxton Osting and Xiao-Ping Wang for helpful suggestions and discussions.

\bibliographystyle{siamplain}
\bibliography{refs}
\end{document}